\documentclass[11pt,a4paper]{amsart}
\usepackage{amsmath, enumerate, MnSymbol}
\usepackage{amscd,mathrsfs,epic,empheq,float}
\usepackage[all]{xy}
\usepackage{wasysym}
\usepackage{todonotes}
\usepackage[vcentermath]{youngtab}
\usepackage[hmarginratio={1:1},vmarginratio={1:1},width=370pt,tmargin=114pt]{geometry}
\usepackage{bbm}
\usepackage{nicefrac}
\usepackage{faktor}
\usepackage{tabularx,tabu}
\usepackage{phaistos}

\definecolor{orchid}{RGB}{143,40,194}
\definecolor{lava}{RGB}{207,16,32}

\usepackage{graphicx,tikz}
\tikzset{anchorbase/.style={baseline={([yshift=-0.5ex]current bounding box.center)}}}
\usetikzlibrary{shapes.geometric}
\usetikzlibrary{arrows}

\numberwithin{equation}{section}
\usepackage[hypertexnames=false]{hyperref}
\usepackage{cleveref,bookmark}
\hypersetup{
    pdftoolbar=true,        
    pdfmenubar=true,        
    pdffitwindow=false,     
    pdfstartview={FitH},    
    pdftitle={Nazarov--Wenzl algebras, coideal subalgebras and categorified skew Howe duality},    
    pdfauthor={Michael Ehrig and Catharina Stroppel},     
    pdfsubject={},   
    pdfcreator={Michael Ehrig and Catharina Stroppel},   
    pdfproducer={Michael Ehrig and Catharina Stroppel}, 
    pdfkeywords={}, 
    pdfnewwindow=true,      
    colorlinks=true,       
    linkcolor=lava,          
    citecolor=blue,        
    filecolor=magenta,      
    urlcolor=orchid,          
    linkbordercolor=lava,
    citebordercolor=blue,
    urlbordercolor=orchid,  
    linktocpage=true
}

\keywords{category $\cO$, coideal algebras, Kazhdan-Lusztig polynomials, categorification, skew Howe duality} 
\allowdisplaybreaks

 \newlength{\baseunit}               % the basic unit length
 \newcount{\numlines}                % depth of picture (in number of lines)
\newtheorem{theorem}{Theorem}[section]

\newtheorem{lemma}[theorem]{Lemma}
\newtheorem{prop}[theorem]{Proposition}
\newtheorem{corollary}[theorem]{Corollary}

\theoremstyle{definition}
\newtheorem{definition}[theorem]{Definition}
\newtheorem{define}[theorem]{Definition}
\newtheorem{remark}[theorem]{Remark}

\newtheorem{ex}[theorem]{Example}

\newcommand{\hint}{{\scalebox{.75}{\halfnote}}}
\newcommand{\dd}{d}
\newcommand{\kk}{k}

\newcommand{\diamondb}{{\scriptstyle \blacklozenge}}
\newcommand{\cA}{{\mathcal A}}
\newcommand{\cB}{{\mathcal B}}
\newcommand{\cC}{{\mathcal C}}

\newcommand{\cE}{{\mathcal E}}
\newcommand{\cF}{{\mathcal F}}
\newcommand{\hcF}{\hat{\mathcal F}}

\newcommand{\cH}{{\mathcal H}}

\newcommand{\cK}{{\mathcal K}}

\newcommand{\cO}{{\mathcal O}}
\newcommand{\cP}{{\mathcal P}}
\newcommand{\cR}{{\mathcal R}}

\newcommand{\cS}{{\mathcal S}}
\newcommand{\cT}{{\mathcal T}}
\newcommand{\cU}{{\mathcal U}}

\newcommand{\La}{\Lambda}
\newcommand{\Ga}{\Gamma}
\newcommand{\bla}{{\mathbf{\la}}}
\newcommand{\bmu}{{\mathbf{\mu}}}
\newcommand{\dht}{\op{ht}_\de}
\newcommand{\pr}{\op{pr}}
\def\bi{\text{\boldmath$i$}}
\newcommand{\ei}{\operatorname{e}(\bi)}

\def\down{{\scriptstyle\vee}}
\def\up{{\scriptstyle\wedge}}
\def\cross{{\scriptstyle\times}}
\newcommand{\Br}{\operatorname{Br}}

\def\down{\vee}
\def\up{\wedge}

\newcommand{\ml}{\mathfrak{l}}
\newcommand{\mg}{\mathfrak{g}}
\newcommand{\mh}{\mathfrak{h}}
\newcommand{\mb}{\mathfrak{b}}

\newcommand{\mn}{\mathfrak{n}}

\newcommand{\pp}{\mathfrak{p}}

\newcommand{\mN}{\mathbb{N}}

\newcommand{\mV}{\mathbb{V}}
\newcommand{\mW}{\mathbb{W}}
\newcommand{\mC}{\mathbb{C}}

\newcommand{\mQ}{\mathbb{Q}}
\newcommand{\mX}{\mathbb{X}}
\newcommand{\mM}{\mathbb{M}}
\newcommand{\mB}{\mathbb{B}}
\newcommand{\mZ}{\mathbb{Z}}

\newcommand{\GL}{\mathfrak{gl}}

\newcommand{\om}{{\omega}}
\newcommand{\updb}{up-down-bitableau}
\newcommand{\updbd}{up-down-$d$-bitableau}
\newcommand{\upd}{up-down-tableau}
\newcommand{\updd}{up-down-$d$-tableau}

\newcommand{\half}{{\nicefrac{1}{2}}}
\newcommand{\vpath}{\text{\scalebox{.45}{\PHpedestrian}}}
\newcommand{\wt}{\mathrm{wt}}
\newcommand{\bd}{{\underline{d}}}
\newcommand{\kw}{\scalebox{1.5}{$\curlywedge$}}

\newcommand{\la}{\lambda}

\newcommand{\HOM}{\operatorname{Hom}}

\newcommand{\END}{\operatorname{End}}

\newcommand{\op}{\operatorname}

\newcommand{\Op}{\mathcal{O}^\pp}
\newcommand{\de}{{\underline\delta}}
\newcommand{\Md  }{M\otimes V^{\otimes d}}

\newcommand{\MdV }{M^\pp(\de)\otimes V^{\otimes d}}
\newcommand{\MdVmtwo}{M(\de)\otimes V^{\otimes {d-2}}}

\newcommand{\Vla  }{M(\la)}

\newcommand{\Mde  }{M^\pp(\de)}

\newcommand{\VW}{\bigdoublevee}
\newcommand{\VWd}{{\bigdoublevee}_d}

\DeclareMathOperator{\Hom}{Hom}   
   \DeclareMathOperator{\Id}{Id}

\begin{document}
\title[Nazarov--Wenzl algebras and skew Howe duality]{Nazarov--Wenzl algebras, coideal subalgebras and categorified skew Howe duality}

\author[M. Ehrig]{Michael Ehrig}
\address{M.E.: School of Mathematics \& Statistics, Carslaw Building, University of Sydney, NSW 2006, Australia}
\email{michael.ehrig@sydney.edu.au}

\author[C. Stroppel]{Catharina Stroppel}
\address{C.S.: Mathematisches Institut, Universit\"at Bonn, Endenicher Allee 60, Room 4.007, 53115 Bonn, Germany}
\email{stroppel@math.uni-bonn.de}

\thanks{M.E. was financed by the DFG Priority program 1388. This material is based on work supported by the National Science Foundation under Grant No. 0932078 000, while the authors were in residence at the MSRI in Berkeley, California.}
%\date{\today}

\begin{abstract}
We describe how certain cyclotomic Nazarov--Wenzl algebras occur as endomorphism rings of projective modules in a parabolic version of BGG category $\cO$ of type $\mathrm{D}$. Furthermore we study a family of subalgebras of these endomorphism rings which exhibit similar behaviour to the family of Brauer algebras even when they are not semisimple. The translation functors on this parabolic category $\cO$ are studied and proven to yield a categorification of a coideal subalgebra of the general linear Lie algebra. Finally this is put into the context of categorifying skew Howe duality for these subalgebras.
\end{abstract}
\maketitle
\setcounter{tocdepth}{1}
\tableofcontents

\section*{Introduction}
\renewcommand{\thetheorem}{\Alph{theorem}}

In \cite{BKKL} a remarkable connection between the representation theory of Hecke algebras and Lusztig's canonical bases was established by showing that cyclotomic Hecke algebras are isomorphic to cyclotomic quotients of quiver Hecke algebras introduced in \cite{KL}, \cite{KLII}, where also a connection between the representation theory of these algebras and Lusztig's geometric construction of canonical bases was predicted. This prediction was verified, \cite{VV}, and therefore a connection between representations of cyclotomic Hecke algebras and canonical bases was established. As a result cyclotomic Hecke algebras inherit an interesting grading which in type $ \mathrm{A}$ can also be obtained from the graded versions of parabolic category $\cO$'s and hence be described in terms of type $ \mathrm{A}$ Kazhdan-Lusztig polynomials, see \cite{BSIII}, \cite{MathasHuQuiv}. In other types however, Schur--Weyl duality connects the Lie algebra with a centralizer algebra (Brauer algebra) different from the group algebra of any Weyl group. In this paper we investigate relations between parabolic category $\cO^{\pp}(\mathfrak{so}_{2n})$, Brauer algebras and their degenerate affine versions $\VWd=\VWd(\Xi)$, depending on a parameter set $\Xi$, and their cyclotomic quotients, \cite{AMR}. The algebras $\VWd(\Xi)$ were introduced in \cite{Nazarov}, we call them affine {\it VW-algebras}.\footnote{Keeping in mind that (affine) VW can be seen as a degeneration of (affine) BMW, the affine Birman-Murakawi-Wenzl algebras, \cite{DRVcenter}.}
These families are the Brauer algebra analogues of the cyclotomic Hecke algebras, but in contrast to the latter, they are not well understood and so far slightly neglected; mostly due to a lack of a good combinatorial description and geometric realization. The main goal of the paper is to connect these algebras to category $\cO$ and its  Kazhdan-Lusztig combinatorics, and in this way obtain canonical bases of certain representations for coideal subalgebras in quantum groups. 

We start with a type $ \mathrm{D}$ analogue of the Arakawa-Suzuki theorem from \cite{AS}:

\begin{theorem} \label{thm:first}
Let $M$ be a highest weight module in $\cO(\mathfrak{so}_{2n})$. With an appropriate choice of $\Xi$ there is an algebra homomorphism
\abovedisplayskip0.3em
\belowdisplayskip0.3em
\begin{gather*}
\Psi_M\,:\,{\VWd}\,\,\longrightarrow\,\,\END_\mg(M \otimes (\mC^{2n})^{\otimes d})^{\op{opp}}.
\end{gather*}
\end{theorem}

In general, this morphism is not surjective and it is hard to fully describe the kernel. We study in detail the case of a certain class of parabolic Verma modules, i.e $M=M^\pp(\la)\in\cO^{\pp}(\mathfrak{so}_{2n})$ for a maximal parabolic subalgebra $\pp$ of type $ \mathrm{A}$ inside type $ \mathrm{D}$. We show that cyclotomic quotients $\VWd(\alpha,\beta)$ of level $2$ occur as endomorphism rings for the special choice of $\la=\delta\omega_0$, i.e. an appropriate multiple of a fundamental weight (the one corresponding to the simple root not attached to $\pp$):

\begin{theorem}[see Theorem~\ref{iso}]\label{thm:second}
If $n \geq 2d$ and $\delta\in\mZ$ then
\abovedisplayskip0.3em
\belowdisplayskip0.3em
\[
\VWd(\alpha,\beta) \,\,\cong\,\, \END_\mg(M^\pp(\underline{\delta}) \otimes (\mC^{2n})^{\otimes d})^{\op{opp}}.
\]
\end{theorem}
We deduce that the $\VWd(\alpha,\beta)$'s inherit a positive Koszul grading from the graded version $\hat{\cO}$ of category $\cO$, hence a geometric interpretation; in terms of first perverse sheaves on isotropic Grassmannians, \cite{ES_diagrams}, and second topological Springer fibres, \cite{ES_springer}, \cite{Arik} via the Khovanov algebra of type $ \mathrm{D}$. This Khovanov algebra also allows us, \cite{ES_Brauer}, to mimic (with some effort) the approach from \cite{BS_walled_Brauer} to construct a graded version of the Brauer algebra $\op{Br}_d(\delta)$ for an arbitrary integral parameter $\delta$. This requires quite involved computations which can be found in \cite{ES_Brauer}. The crucial player hereby is a subalgebra $z_d\VWd(\alpha,\beta)z_d\subset\VWd(\alpha,\beta)$ of which we show that it shares properties with the Brauer algebra:

\begin{theorem}[see Proposition \ref{prop:dimension_brauer} and Theorem \ref{thm:semisimplicity}]
\label{thm:third}\hfill
\begin{enumerate}
\item The algebra $z_d\VWd(\alpha,\beta)z_d$ has dimension $(2d-1)!!$.
\item The algebra $z_d\VWd(\alpha,\beta)z_d$ is semisimple if and only if $\delta\not=0$ and $\delta\geq d-1$ or $\delta=0$ and $d=1,3,5$.
\end{enumerate}
\end{theorem}

In \cite{ES_Brauer} it is then proved that this algebra is isomorphic to the Brauer algebra $\op{Br}_d(\delta)$ providing  a conceptual explanation for the fact, \cite{CDVMI} and \cite{CDVMII}, that the decomposition numbers of Brauer algebras are given by type $ \mathrm{D}$ Weyl group combinatorics, and branching rules  relate to the set $\vpath_d(\delta)$ of Verma paths.

Theorems \ref{thm:first} and \ref{thm:second} rely on a good understanding of (a graded version of) taking successive tensor products with the natural representation $\mC^{2n}$ on category $\mathcal{O}^\pp$. Since $\mC^{2n}$ is self-dual and hence $\_\otimes\mC^{2n}$ self-adjoint, we do not obtain an action of a quantum group on our categories, but rather of certain {\it coideal subalgebras} $\cH$ and $\cH^\hint$ of $\cU_q(\mathfrak{gl}_\mZ)$. These are quantum group analogues of the symmetric pair $\mathfrak{gl}_\mN \times \mathfrak{gl}_{-\mN}\subset\mathfrak{gl}_\mZ$. The construction goes back to Noumi, Sugitani, and Dijkhuizen, \cite{Noumi}, \cite{NDS}, who used explicit solutions of the reflection equation to obtain analogues of all classical symmetric pairs; for a thorough study from the quantum group point of view see \cite{Letzter}, \cite{LetzterII}, \cite{Kolb}.

Since the integral weights of $\mathfrak{so}_{2n}$ can be partitioned into integer or half-integer weights in the standard $\epsilon$-basis, the  integral part, $\cO^{\pp}(\mathfrak{so}_{2n})$, of parabolic category $\cO$ can be decomposed into two subcategories $\cO^{\pp}_1(\mathfrak{so}_{2n})$ and $\cO^{\pp}_\hint(\mathfrak{so}_{2n})$ stable under taking tensor products with the natural representation. We obtain (as a special case) two categorifications:

\begin{theorem}[see Theorems~\ref{prop:generators_and_functors} and  Theorem~\ref{thm:main}]\label{thm:fourth}\hfill
\begin{enumerate}
\item The $\cH$-module $\bigwedge^n_q \mathbb{Q}(q)^\mZ$ is categorified by $\hat{\cO}^{\pp}_1(\mathfrak{so}_{2n})$.
\item The $\cH^\hint$-module $\bigwedge^n_q \mathbb{Q}(q)^{\mZ+\half}$ is categorified by $\hat{\cO}^{\pp}_\hint(\mathfrak{so}_{2n})$.
\end{enumerate}
\end{theorem}
The classes of Verma modules correspond hereby to the standard tensor product basis. The involved categories have a contravariant duality, hence the categorification equips the modules on the left with a bar-involution.  Mimicking Lusztig's approach for quantum groups, we can define canonical bases on the above modules and show that the classes of simple modules correspond to the dual canonical basis (in the sense of \cite[Section 2]{BSIII}). Although we formulated it here only in the $q=1$ setting, a large part of the paper will deal with the quantum version. 

\begin{theorem}[see Theorem~\ref{prop:bar_exists} and more generally Theorem~\ref{thm:bar_unique}] \label{thm:fourth2}
There \hfill\\exists a duality $\textbf{d}$ on $\hat{\cO}_1^\pp(n)$ respectively on  $\hat{\cO}_\hint^\pp(n)$ inducing a compatible bar-involution on $\bigwedge^n_q \mQ(q)^\mZ$ respectively on  $\bigwedge^n_q \mQ(q)^{\mZ+\half}$. The classes of simple modules correspond via Theorem~\ref{thm:third} to the associated dual canonical basis. 
\end{theorem}

Theorem~\ref{thm:fourth2}  gives a new instance of based categorifications in the context of category $\cO$, but now connecting canonical bases of Hecke algebras with canonical bases of quantum symmetric pairs instead of quantum groups as for instance in \cite{BSII}, \cite{FKS}, \cite{Sartori}, \cite{CTP}, see \cite{Mazorchuk} for an overview. The base change matrix is here given by parabolic type $ \mathrm{D}$ Kazhdan-Lusztig polynomials, see \cite{LS} for explicit formulas, in the Grassmannian case corresponding to $\pp$. Independently, and maybe motivated by our results, Balagovic and Kolb introduced an analogue of Lusztig's bar involution in much more generality for the quantum symmetric pairs themselves (instead of focusing on based representations), see \cite{BalKolb1} establishing a very elegant general theory. Up to an obvious renormalisation, our special case of bar involution for the type $\op{(AIII)}$ coideal agrees with theirs, and was also studied in detail in \cite{BW}. In neither of these two approaches a categorification is given, and we expect that generalizations to other types require a substantially new approach.

In the final part of the paper we investigate generalizations of these modules from the viewpoint of skew Howe duality, \cite{Howe}, and its categorification. For this we consider more general parabolic category $\cO$'s and their block decompositions for Levi subalgebras isomorphic to products of $\mathfrak{gl}_{k}$'s. Denote the sum of these blocks by $\bigoplus_\Gamma \cO_{\Gamma}(n)$. Considering analogous projective functors in this setup we first categorify the two actions of $\mathfrak{gl}_m \times \mathfrak{gl}_m$ and $\mathfrak{gl}_r \times \mathfrak{gl}_r$ on the vector space $\bigwedge(n,m,r)=\bigwedge^n (\mC^m \otimes \mC^2 \otimes \mC^r)$ separately. Using gradings we in fact categorify a quantized version of this involving coideal subalgebras. Then we finally show that the two actions are graded derived equivalent via Koszul duality. Under this identification given by Koszul duality we can finally realize {\it both} commuting actions on the {\it same} category. The projective functors turn into derived Zuckerman functors under this identification. As far as we know this is the first instance,  where both commuting actions involved in a skew Howe duality were categorified simultaneously in a common framework. 

\begin{theorem}[see Theorems~\ref{prop:catskewclass} and ~\ref{thm:skewKoszul}]\label{thm:fifth}
The bimodule $\bigwedge(n,m,r)$ is categorified by $\bigoplus_\kappa \cO_{\kappa}(n)$. The graded version of this category categorifies the quantization. Furthermore, the two categorifications for the left and right action are then Koszul dual to each other.
\end{theorem}

This skew Howe duality can be squeezed in between two type $ \mathrm{A}$ skew Howe dualities as follows. For an integer $m$ let $V_{m}$ be the natural $\mathfrak{gl}_m$-module. Fix integers $n,m,r$ and consider the $\mathfrak{gl}_m\times\mathfrak{gl}_r$-module $\bigwedge^n (V_{m} \otimes V_{r})$. By skew Howe duality, \cite{Howe}, the weight spaces of the $\mathfrak{gl}_m$-action are representations of $\mathfrak{gl}_r$ and vice versa; and the decomposition is the following
\abovedisplayskip0.3em
\belowdisplayskip0.3em
\begin{gather*}
\bigoplus_{\alpha}\bigwedge^{\alpha_1} V_{m}\otimes\cdots\otimes \bigwedge^{\alpha_r}V_m
\,\, \cong \,\, \bigwedge^n (V_{m} \otimes V_{r})\,\,
\cong \,\, \underset{\beta}\bigoplus\bigwedge^{\beta_1} V_r\otimes\cdots\otimes\bigwedge^{\beta_m}V_r
\end{gather*}
where the first isomorphism is as $\mathfrak{gl}_m$-module and the second as $\mathfrak{gl}_r$-module. The sums run over all compositions (possibly with zero parts)  of $n$ with $r$ respectively $m$ parts. Now consider the restriction $V_{2m}=V_m\oplus V_m$ to $\mathfrak{gl}_m$, where $\mathfrak{gl}_m$ is embedded diagonally into $\mathfrak{gl}_m \times \mathfrak{gl}_m\subset\mathfrak{gl}_{2m}$. Theorem~\ref{thm:fifth} gives a categorification of the middle piece of the diagram
\abovedisplayskip0.3em
\belowdisplayskip0.3em
\begin{eqnarray}
\label{introbigdiagram}
\mathfrak{gl}_{2m}\quad\quad{\curvearrowright}&\displaystyle{\bigoplus_{\alpha}} \bigwedge^{\alpha_1} V_{2m}\otimes \bigwedge^{\alpha_2} V_{2m}\otimes\cdots \otimes\bigwedge^{\alpha_r}V_{2m} &{\curvearrowleft}\quad \quad\mathfrak{gl}_r\nonumber\\
\cup\quad\quad\quad&&\quad\quad\quad\cap\quad\nonumber\\
\mathfrak{gl}_m\times \mathfrak{gl}_m\quad{\curvearrowright}&\bigwedge(n,m,r)&{\curvearrowleft}\quad \mathfrak{gl}_r\times \mathfrak{gl}_r\nonumber\\
\cup\quad\quad\quad&&\quad\quad\quad\cap\quad\\
\mathfrak{gl}_m\quad\quad{\curvearrowright}&\displaystyle{\bigoplus_{\gamma}} \bigwedge^{\gamma_1} V_{m}\otimes \bigwedge^{\gamma_2} V_{m}\otimes\cdots \otimes\bigwedge^{\gamma_{2r}}V_{m}&{\curvearrowleft}\quad \quad\mathfrak{gl}_{2r}\nonumber
\end{eqnarray}
The top and bottom parts can be quantized, \cite{LZZ}.  They were already categorified using parabolic-singular category $\cO$'s in \cite{MS} in a way which also fits into the framework of \cite{LQR},  \cite{MW}. In our categorification of the middle part the commuting actions of the two Lie algebras, viewed as specializations of the coideal subalgebras from above, are again given by translation and derived Zuckerman functors respectively but now on category $\cO$ for Lie algebras of type $\mathrm{D}$. Since this construction has a graded lift using $\hat{\cO}$ there is a quantized version of skew Howe duality with the commuting actions of the coideal subalgebras:

\begin{theorem}[see Theorem~\ref{Koszul}]
\label{Koszulintro}
 Fix the equivalence 
\abovedisplayskip0.3em
\belowdisplayskip0.3em
\begin{gather*}
\mathcal{K}:\bigoplus_{\underline{k} \in C(n,r)}
\hspace{-5pt}D^b\left(\hat{\cO}^{\mathfrak{q}_{\underline{k}}}_{\leq m}(n)\right)
\,\,\,\xrightarrow{\phantom{x}\textstyle \sim\phantom{x}} \bigoplus_{\underline{k} \in C(n,m)}
\hspace{-5pt}D^b\left(\hat{\cO}^{\mathfrak{q}_{\underline{k}}}_{\leq r}(n)\right).
\end{gather*}
given by Koszul duality. Then the coideal algebra actions from Theorem~\ref{thm:main} on the two sides turn into commuting actions on the same space, one given by Theorem~\ref{thm:main}, the other by Proposition~\ref{prop:ZuckeronK}.
\end{theorem}

As in the type $\mathrm{A}$ case, skew Howe duality contains in fact some Schur--Weyl duality version (see Theorem~\ref{prop:Schur_Weyl})  pairing the action of the coideal subalgebra with the action of the Hecke algebra corresponding to the Weyl groups of type $\mathrm{B}$ (and also $\mathrm{D}$). This in fact explains the occurrence of the Kazdhan-Lusztig combinatorics and was the starting point of our construction. 

On the uncategorified level this duality can also be found for instance in \cite{BW}, \cite{Watanabe} and implicitly in \cite{Li}. It is the coideal version of the well-known duality for wreath products, \cite{MSSchurWeyl}. We present a categorified version:

\begin{theorem}[see Theorem~\ref{prop:Schur_Weyl}]
The commuting actions of $\cH^\hint$ and $\mathbb{H}_n(\mathrm{B})$ on ${\mV_m^\hint}^{\otimes n}$ are each others centralizers.
\end{theorem}

It is related to for instance \cite{Li} or \cite{Li2} by using instead of convolution algebras of functions on partial flag varieties over finite fields the corresponding (graded version of) categories of perverse sheaves on partial flag varieties defined over the complex numbers, that is the uncategorified version can be obtained by the usual passage from sheaves to functions. 

Combinatorially there is only a small difference between the type $\mathrm{B}$ and type $\mathrm{D}$ version, see \cite[9.7]{ES_diagrams} and \cite{Li} versus \cite{Li2}  for a more precise statements. We prefer here to work with type $\mathrm{D}$, since it is simply laced and in particular its own Langlands dual, a crucial fact for the Koszul self-duality statement in Theorem~\ref{Koszulintro}.  The corresponding Hecke algebra of type $\mathrm{D}$ is realized as a subalgebra of a certain specialisation $\mathbb{H}_{1,q}$ of the 2-parameter Hecke algebra of type $\mathrm{B}$, see Section~\ref{sectionSW} for more details.

Our categorification is in reality an action of a certain 2-category, where $1$-morphisms are graded translation functors and $2$-morphisms are natural transformations. By construction, these natural transformations are given by the affine VW-algebras. These algebras do however not see explicitly the grading, and can also not single out the natural transformations of projective functors passing between two fixed integral blocks of the  underlying category $\cO$. This phenomenon is analogous to the role of the degenerate affine Hecke algebra in the type A situation, and very much in contrast to the KLR algebra which carries exactly this extra information.  For the special case of the coideal $\cH$ from Section~\ref{section:coideal}, a definition of the KLR analogue and the attached graded 2-category was recently given in \cite{BSWW}.

\subsection*{Acknowledgements} Most of the research was done when both authors were visiting University of Chicago in 2012, we acknowledge the excellent working conditions. We thank the organizers of the workshops "Gradings and Decomposition numbers" and  "Representation theory and symplectic algebraic geometry'' for the possibility to present our results. We thank Pavel Etingof, Stefan Kolb, Gail Letzter, Andrei Okounkov, Dmitri Panyushev, Vera Serganova, Daniel Tubbenhauer, Ben Webster for discussions and explanations.
\section{Basics}
\label{section:lie_basics}
\subsection{Lie theory}
\renewcommand{\thetheorem}{\arabic{section}.\arabic{theorem}}
Throughout this paper we denote by $\mathfrak{g}$ the Lie algebra $\mathfrak{so}_{2n}(\mathbb{C})$, by $U(\mathfrak{g})$ its enveloping algebra, and by $Z(U(\mathfrak{g}))$ the center of $U(\mathfrak{g})$. In Section \ref{section:VW} we will fix a specific presentation for $\mathfrak{so}_{2n}$.

Fix a basis $\{\epsilon_i \mid 1 \leq i \leq n\}$ of the dual of the Cartan $\mathfrak{h}^*$ with simple roots $\alpha_i=\epsilon_{i+1}-\epsilon_{i}$, for $1 \leq i \leq n-1$, and $\alpha_0=\epsilon_1 + \epsilon_2$ and the set of all roots given by $\mathrm{R}(\mathfrak{so}_{2n}) = \{ \pm \epsilon_i \pm \epsilon_j \mid i \neq j \}$.

We denote by $\cO$ the \emph{BGG-category} $\cO$ of $\mathfrak{g}$, \cite{Hbook}. For the set of simple roots $\alpha_1,\ldots,\alpha_{n-1}$ we denote the associated parabolic subalgebra of $\mathfrak{g}$ by $\pp$ and by $\cO^{\pp}$ the subcategory of $\cO$ of all $\pp$-finite modules, the \emph{parabolic category} of $\cO$.

A weight $\lambda = (\lambda_1,\ldots,\lambda_n) \in \mh^*$ of $\mathfrak{so}_{2n}$ is \emph{integral} if it is in the $\mZ$-span or the $(\mZ+\half)$-span of the $\epsilon_i$'s. In the former case we say the weight is \emph{supported on the integers}, in the latter that it is \emph{supported on the half-integers}.

We will focus on modules in $\cO$ that have integral weights only. These form a direct summand of $\cO$, the sum of \emph{integral blocks} of $\cO$, which we denote by $\cO(n)$. Similarly we denote by $\cO^\pp(n)$ the subcategory of $\cO^\pp$ having only modules with integral weights. The combinatorics of $\cO^\pp(n)$ were studied in \cite{ES_diagrams} and we recall some of the notations used therein below.

For $\la\in\mh^*$ let $M^\pp(\la)$ denote the \emph{parabolic Verma module} of highest weight $\la$, i.e. the maximal locally $\pp$-finite quotient of the ordinary Verma module $M(\lambda)$. Explicitly, $M^\pp(\la)=0$ if $\la$ is not $\pp$-dominant, i.e. dominant with respect to the sub root system generated by $\alpha_1,\ldots,\alpha_{n-1}$ and otherwise
\abovedisplayskip0.3em
\belowdisplayskip0.3em
\begin{gather}
\label{parabolic-Verma}
M^\pp(\la)\,=\,U(\mg)\otimes_{U(\pp)}E(\la)
\end{gather}
where $E(\la)$ is the finite dimensional $\mathfrak{l}$-module of highest weight $\la$ for the Levi subalgebra $\mathfrak{l}$ of $\pp$ inflated trivially to a $\pp$-module. If no confusion can arise, we call parabolic Verma modules simply {\it Verma modules}.

The set of $\pp$\emph{-dominant (integral) weights} will be denoted by
\begin{align}
\label{DefLa}
\begin{aligned}
\La&\,=\,\left\lbrace \la\in\mh^*\text{ integral}\mid\la=\sum_{i=1}^n\la_i\epsilon_i \text{ where } \la_1\leq \la_2\leq\cdots\leq \lambda_n \right\rbrace \\
&\,=\, \left\lbrace \la\in\mh^*\text{ integral}\mid\la+\rho=\sum_{i=1}^n\la_i'\epsilon_i \text{ where } \la_1'< \la_2'<\cdots< \lambda_n'\right\rbrace
\end{aligned}
\end{align}
where $\rho$ denotes the half-sum of positive roots, i.e. $\rho=(0,1,2,\dots, n-1)$ in the chosen basis. It decomposes into a disjoint union $\La^1 \cup \La^\hint$, the weights supported on the integers, denoted $\La^1$, and those supported on half-integers, denoted $\La^\hint$. The decomposition into generalized common eigenspaces for $Z(U(\mathfrak{g}))$ induces 
\abovedisplayskip0.3em
\belowdisplayskip0.3em
\[
\cO(n)\,=\,\bigoplus_\chi\cO_\chi(n)
\]
a decomposition of $\cO(n)$ indexed (via the Harish-Chandra isomorphism) by orbits of integral weights under the dot action $w\cdot\la=w(\la+\rho)-\rho$ of the Weyl group $W_n$ of $\mg$. The decomposition of $\cO(n)$ induces a decomposition of the subcategory $\Op(n)$. We denote by $\Op_\la(n)$ the summand containing $M^\pp(\la)$, i.e. the summand corresponding combinatorially to the (ordinary) orbit through $\lambda + \rho$. Note that the ordinary action and the dot-action of $W_n$ on integral weights preserve the lattices of weights supported on integers resp. on half-integers. Hence, we have a decomposition
\abovedisplayskip0.3em
\belowdisplayskip0.3em
\begin{gather}
\label{evenodd}
\Op(n)=\Op_1(n)\oplus\Op_\hint(n),
\end{gather}
where $\Op_1(n)$, respectively $\Op_\hint(n)$, is the direct sum of all $\Op_\la(n)$, such that $\la$ is supported on the integers, respectively half-integers.
We will equip the Grothendieck groups of $\Op_1(n)$ and $\Op_\hint(n)$ with the structure of a representation of a quantum symmetric pair,  see Section \ref{section:coideal}, induced by the action of translation functors.

Given an abelian category $\cA$, we denote by $K_0(\cA)$ the scalar extension by  $\mQ$ of the Grothendieck group of $\cA$, and call it by abuse of language just {\it Grothendieck group}. For an object $M\in\cA$ let $[M]\in K_0(\cA)$ its class, and for an exact functor $F$ let $[F]$ denote the induced $\mQ$-linear map between the Grothendieck groups. If  $\cA$ is additionally graded, its Grothendieck group is naturally a $\mZ[q,q^{-1}]$-module where $q$ corresponds to the grading shift $\langle 1\rangle$ up by $1$ and $K_0(\cA)$ is  the  $\mQ(q)$-module obtained by scalar extension.

\subsection{Diagrammatics} \label{section:diagrammatics}
To make the action on the Grothendieck groups explicit we use the combinatorics
from \cite[Section 2.2]{ES_diagrams} to identify $\La$ with diagrammatic weights, allowing us to describe the blocks of category $\cO^\pp(n)$ combinatorially. Since we want to distinguish the combinatorics for the two subcategories in \eqref{evenodd} we slightly modify the indexing sets for diagrammatic weights from \cite{ES_diagrams}.

\begin{definition}
We denote by $\mX_n$ the set of sequences $a = (a_i)_{i \in \mZ_{\geq 0}}$ such that $a_i \in \left\{ \up, \down, \times, \circ, \diamondb \right\}$, $a_i \neq \diamondb$ for $i \neq 0$, $a_0 \in \{\circ,\diamondb\}$, and
$$ \#\{ a_i \mid a_i \in \{\up,\down,\diamondb\}\} + 2 \#\{ a_i \mid a_i = \times\} = n.$$
Elements of $\mX_n$ are called \emph{diagrammatic weights supported on the integers}.
\end{definition}

\begin{remark}
The translation from a diagrammatic weight $a \in \mX$ to a diagrammatic weight $\widetilde{a}$ in the sense of \cite{ES_diagrams} is done by setting
$\widetilde{a}_i=a_{i-1}$, except for $i=1$ and $a_0=\diamondb$ where we choose $\widetilde{a}_1 \in \{\up,\down\}$ such that the total number of $\down$'s is even. In the language of \cite{ES_diagrams} we have thus always fixed the even parity for these blocks where  $a_0=\diamondb$.
\end{remark}

Similarly we have a set-up for half-integers.

\begin{definition}
We denote by $\mX^\hint_n$ the set of sequences $a = (a_i)_{i \in {\mZ_{\geq 0}+\half}}$ such that $a_i \in \{ \up, \down, \times, \circ\}$ and $ \#\{ a_i \mid a_i \in \{\up,\down\}\} + 2 \#\{ a_i \mid a_i = \times\} = n.$ We call these elements \emph{diagrammatic weights supported on the half-integers}.
\end{definition}

\begin{remark}
Translating a diagrammatic weight $a \in \mX^\hint$ to a diagrammatic weight $\widetilde{a}$ in the sense of \cite{ES_diagrams} is done by putting
$\widetilde{a}_i=a_{i-\half}$. The shift in the index is just convenient for later use.
\end{remark}

Obviously a diagrammatic weight $a$ is uniquely determined by the sets
$$\mathrm{P}_?(a) = \{ i \mid a_i = ? \},$$
for $? \in \{\up, \down, \times, \circ, \diamondb \}$. For $\lambda = (\lambda_1,\ldots,\lambda_n)\in \Lambda$, let $\lambda'=(\lambda_1',\ldots,\lambda_n')=\lambda + \rho$ and denote by $a_\lambda$ the diagrammatic weight defined by
\abovedisplayskip0.3em
\belowdisplayskip0.3em
\begin{gather*}
\begin{aligned}
\mathrm{P}_\down(\lambda) =&\, \{ \lambda_i' \mid \lambda_i' > 0 \text{ and } -\lambda_i' \text{ does not appear in } \lambda'\}, \\
\mathrm{P}_\up(\lambda) =&\, \{ -\lambda_i' \mid \lambda_i' < 0 \text{ and } -\lambda_i' \text{ does not appear in } \lambda'\},\\
\mathrm{P}_\times(\lambda) =&\, \{ \lambda_i' \mid \lambda_i' > 0 \text{ and } -\lambda_i' \text{ appears in } \lambda'\},\\
\mathrm{P}_\diamondb(\lambda) =&\, \{ \lambda_i' \mid \lambda_i' = 0\},\\
\mathrm{P}_\circ(\lambda) =&\, \begin{cases}
\mZ_{\geq 0} \setminus (\mathrm{P}_\down \cup \mathrm{P}_\up \cup \mathrm{P}_\times \cup \mathrm{P}_\diamondb) & \text{ if } \lambda \text{ is supported on integers,}\\
(\mZ_{\geq 0} + \half) \setminus (\mathrm{P}_\down \cup \mathrm{P}_\up \cup \mathrm{P}_\times) & \text{ if } \lambda \text{ is supported on half-integers.}
\end{cases}
\end{aligned}
\end{gather*}
The assignment $\lambda \mapsto a_\lambda$ defines a bijection $\Lambda\cong\mX_n \cup \mX_n^\hint$
between $\pp$-dominant weights and diagrammatic weights. In the following we will not distinguish between a weight and a diagrammatic weight and denote both by $\lambda$. We use the following identifications
\abovedisplayskip0.3em
\belowdisplayskip0.3em
\begin{gather} \label{eq:identify}
K_0\left({\cO}^{\pp}_1(n)\right) \cong \left\langle \mX_n \right\rangle_{\mQ} \quad \text{and} \quad K_0\left({\cO}^{\pp}_\hint(n)\right) \cong \left\langle \mX_n^\hint \right\rangle_{\mQ},
\end{gather}
between the $\mQ$-vector spaces on basis $\mX_n$ resp. $\mX_n^\hint$ and the Grothendieck group scalar extended to $\mQ$ by identifying the class of a parabolic Verma module of highest weight $\la$ with the weight $\lambda$.

Our weight dictionary induces an action of the Weyl group $W_n$ on diagrammatic weights corresponding to the dot-action on weights for $\mg$. Two diagrammatic weights are in the same orbit (and thus the corresponding Verma modules have the same central character) if and only if one is obtained from the other by a finite sequence of changes of the following form: swapping a $\down$ with an $\up$ or replacing two $\down$'s with two $\up$'s or two $\up$'s with two $\down$'s (keeping all $\cross$'s and $\circ$'s untouched), respectively interchanging an $\up$ and a $\down$ at position $1$ if a
diamond $\diamondb$ is present, see \cite{ES_diagrams}. Orbits of diagrammatic weights, called \emph{diagrammatic blocks} are given by fixing the positions of the $\times$'s and $\circ$'s and the parity of $\#\down+\#\cross$ of any of its weights. They are indexed by block diagrams, see \cite[Section 2.2]{ES_diagrams}, using the symbol $\bullet$ in place of all $\up$, $\down$ and $\diamondb$ as well as the afore mentioned parity. Moving $\down$'s to the left or turning two $\up$'s into two $\down$'s makes the weight bigger (with respect to the standard ordering on weights) as well as changing an $\up$ at position $1$ into a $\down$ if a diamond $\diamondb$ is present.

Note that diagrammatic blocks correspond precisely to blocks of $\Op(n)$ by sending $\Ga$ to the summand $\Op_\Ga(n)$ containing all Verma modules with highest weight in $\Ga$ which is indeed a block of the category.

\section{Affine VW-algebras}
\label{section:VW}

The main purpose of this section is a generalization of the Arakawa-Suzuki action, \cite{AS}, to the Lie algebra of type $\mathrm{D}_n$. The replacement of the degenerate affine Hecke algebra is the following affine VW-algebra $\VWd(\Xi)$.

\begin{define}
\label{def:VW}
Let $d \in \mN$ and fix a set $\Xi =(w_k)_{k \geq 0}$ of complex parameters. Then the \emph{affine Nazarov--Wenzl algebra} $\VWd=\VWd(\Xi)$, short \emph{affine $VW$-algebra}, is generated by $s_i, e_i, y_j \quad 1 \leq i \leq d-1, 1 \leq i \leq d, k \in \mN$,
subject to the following relations (for $1
\leq a,b \leq d-1$, $1 \leq c < d-1$, and $1 \leq i,j \leq d$):
\begin{enumerate}[(VW.1)]
\item $s_a^2 = 1$ \label{1}
\item 	
\begin{enumerate} 	
\item $s_as_b = s_bs_a$ for $\mid a-b \mid > 1$ 	 \label{2a}
\item $s_c s_{c+1} s_c = s_{c+1} s_c s_{c+1}$ 	 \label{2b}
    \item $s_ay_i = y_is_a$ for $i \not\in \{a,a+1\}$ \label{2c}
    \end{enumerate}
\item $e_a^2 = w_0 e_a$ \label{3}
\item $e_1y_1^ke_1 = w_ke_1$ for $k \in \mN$ \label{4}
\item 	
\begin{enumerate}
\item $s_ae_b = e_bs_a$ and $e_ae_b = e_be_a$ for $\mid a-b \mid > 1$ \label{5a}	 
\item $e_ay_i = y_ie_a$ for $i \not\in \{a,a+1\}$ \label{5b}	
\item $y_iy_j = y_jy_i$ \label{5c}	
\end{enumerate}
\item
\begin{enumerate} 	
\item $e_as_a = e_a = s_ae_a$ 	\label{6a}
\item $s_ce_{c+1}e_c = s_{c+1}e_c$ and
    $e_ce_{c+1}s_c = e_cs_{c+1}$ \label{6b}	
\item $e_{c+1}e_cs_{c+1} = e_{c+1}s_c$ and $s_{c+1}e_ce_{c+1} = s_ce_{c+1}$ \label{6c}	
\item $e_{c+1}e_ce_{c+1} =
    e_{c+1}$ and $e_ce_{c+1}e_c = e_c$ 	\label{6d}
\end{enumerate}
\item $s_ay_a - y_{a+1}s_a = e_a - 1$ and $y_as_a - s_ay_{a+1} = e_a - 1$ \label{7}
\item 	
    \begin{enumerate} 	\item $e_a(y_a+y_{a+1}) = 0$ \label{8a}	
    \item $(y_a+y_{a+1})e_a = 0$ \label{8b}	
\end{enumerate}
\end{enumerate}
\end{define}

\begin{remark} \label{Wop}
Relations \eqref{6b}, \eqref{6c}, \eqref{6d}, \eqref{7} come in pairs and it is in fact sufficient
to either require the first set of relations or the second, the other is then satisfied automatically.
All relations are symmetric or come in
symmetric pairs, thus we have a canonical isomorphism $\VWd(\Xi)\cong \VWd(\Xi)^{\op{opp}}$.
\end{remark}

\subsection{Centralisers}

Fix an integer $n\geq 4$ and set $N=2n$. Let $\mathbb{I}^+=\{1,\ldots,n\}$, $\mathbb{I}^-=-\mathbb{I}^+$, and $\mathbb{I}=\mathbb{I}^+ \cup \mathbb{I}^-$. We denote by $V$ the vector space with basis $\{v_i \mid i \in \mathbb{I}\}$ and by $\GL(\mathbb{I})$ its corresponding Lie algebra of endomorphisms, viewed as the matrices with respect to the chosen basis. Let $\mathbf{J}$ be the matrix such that $\mathbf{J}_{kl} = \delta_{k,-l}$ for $k,l \in \mathbb{I}$ with respect to the chosen basis. If we order columns and rows decreasing from top to bottom and left to right this is the matrix with ones on the anti-diagonal and zeros elsewhere.

\begin{definition} \label{matrixJ}
The Lie algebra $\mg=\mathfrak{so}_{2n}$ is the Lie subalgebra of $\GL(V)$ of all matrices $A$ satisfying $\mathbf{J}A+A^t\mathbf{J} = 0$; that is all matrices which  are skew-symmetric with respect to the anti-diagonal, $A_{i,j} = - A_{-j,-i}$. In terms of the bilinear form $\left\langle -,- \right\rangle$ on $V$ defined by $\mathbf{J}$ we thus have $\left\langle Xv,w \right\rangle + \left\langle v,Xw \right\rangle = 0$.
\end{definition}

Fix the Cartan subalgebra $\mh\subset\mg$ given by all diagonal matrices and a basis $\{\epsilon_i \mid i\in \mathbb{I}^+\}$ for $\mh^*$ such that the weight of $v_i$ is $\epsilon_i$ if $i\in \mathbb{I}^+$ and the weight of $v_i$ is $-\epsilon_{-i}$ if $i \in \mathbb{I}^-$. For every $\alpha \in \mathrm{R}(\mathfrak{so}_{2n})$ fix a root vector $X_\alpha$ of weight $\alpha$ and for $i \in \mathbb{I}^+$ let $X_i$ be the element in $\mh$ dual to $\epsilon_i$. 

Then $\{X_\gamma \mid \gamma \in \mathbb{B}(\mathfrak{so}_{2n})\}$ with $\mathbb{B}(\mathfrak{so}_{2n}) = \mathrm{R}(\mathfrak{so}_{2n}) \cup \mathbb{I}^+$ forms a basis of $\mathfrak{so}_{2n}$. We set $\mn^+ = \left\langle X_{\epsilon_i \pm \epsilon_j} \mid i > j \right\rangle$, and $\mn^- = \left\langle X_{-(\epsilon_i \pm \epsilon_j)} \mid i > j \right\rangle$ and fix the Borel subalgebra $\mb= \mn^+ \oplus \mh$. In this notation the natural representation $V$ is the irreducible representation $L(\epsilon_n)$ with highest weight $\epsilon_n$, the fundamental weight corresponding to $\alpha_{n-1} = \epsilon_n - \epsilon_{n-1}$. Furthermore, the $\{ X_{\pm(\epsilon_i-\epsilon_j)} \mid i> j\}$ together with $\mathfrak{h}$ form a Levi subalgebra $\ml$ isomorphic to $\mathfrak{gl}_n$ with corresponding standard parabolic subalgebra $\pp=\ml+\mn^+$ from Section~\ref{section:lie_basics}.

For $i \in \mathbb{I}$ denote by $v_i^* = v_{-i}$ the basis element dual to $v_i$ with respect to $\left\langle -,- \right\rangle$ and for $X_\gamma$ denote by $X_\gamma^*$ the element dual to $X_\gamma$ with respect to the Killing form of $\mathfrak{so}_{2n}$.

\begin{definition}
Let $M$ be a $\mg$-module. For $d\geq 0$ consider $\Md$. The linear endomorphisms $\tau, \sigma: V \otimes V \longrightarrow V \otimes V$ defined as\\[-0.3cm]
\noindent\begin{tabularx}{0.9\textwidth}{XXX}
\begin{equation}\hspace{-7cm}\label{Deftau}
\tau :\,\, v \otimes w \,\mapsto\, \left\langle v,w \right\rangle \sum_{i \in \mathbb{I}} v_i \otimes v_i^* \hspace*{-1cm}
\end{equation} & &
\begin{equation}\hspace{-9cm}\label{Defs}
\sigma :\,\, v \otimes w \,\mapsto\, w \otimes v \hspace*{-1cm}
\end{equation}
\end{tabularx}\\[0.1cm]
%\begin{eqnarray}
%\tau :\quad v \otimes w &\mapsto& \left\langle v,w \right\rangle \sum_{i \in \mathrm{I}} v_i \otimes v_i^*,\label{Deftau}\\
%\sigma :\quad v \otimes w &\mapsto& w \otimes v.\label{Defs}
%\end{eqnarray}
induce the following endomorphisms $s_i$, $e_i$ of $M \otimes V^{\otimes d}$ for $1\leq i\leq d-1$
\abovedisplayskip0.3em
\belowdisplayskip0.3em
\begin{gather*}
e_i \,=\, \Id \otimes \Id^{\otimes (i-1)} \otimes \tau \otimes \Id^{\otimes (d-i-3)} \quad \text{and} \quad
s_i \,=\, \Id \otimes \Id^{\otimes (i-1)} \otimes \sigma \otimes \Id^{\otimes (d-i-3)}.
\end{gather*}
By definition of the comultiplication it is obvious that $s_i$ is a $\mg$-homomorphism and using the compatibility of $\mg$ and the bilinear form it immediately follows that $e_i$ is as well (see also Remark \ref{omega_sigma_tau}), hence both are in $\END_\mg \left(M \otimes V^{\otimes d}\right)$.
\end{definition}

\begin{definition}
The \emph{pseudo Casimir element} in $\cU(\mg)\otimes \cU(\mg)$ is defined as
\abovedisplayskip0.3em
\belowdisplayskip0.3em
\begin{gather}
\label{DefOmega}
\Omega \,=\, \sum_{\gamma \in \mathbb{B}(\mathfrak{so}_{2n})} X_\gamma \otimes X_\gamma^*.
\end{gather}
It is connected to the ordinary \emph{Casimir element} $C = \sum_{\gamma \in \mathbb{B}(\mathfrak{so}_{2n})} X_\gamma X_\gamma^*$ via
\abovedisplayskip0.3em
\belowdisplayskip0.3em
\[
\Omega = \frac{1}{2}(\Delta(C)- C \otimes 1 - 1 \otimes C),
\]
where $\Delta$ denotes the comultiplication of $\mg$. For $0\leq i<j\leq d$ we define
\abovedisplayskip0.3em
\belowdisplayskip0.3em
\begin{gather}
\label{DefOmegaij}
\Omega_{ij} \,=\, \sum_{\gamma \in B_n} 1 \otimes \ldots \otimes X_{\gamma} \otimes 1 \otimes \ldots \otimes 1 \otimes X_{\gamma}^* \otimes 1 \otimes \ldots \otimes 1,
\end{gather}
where $X_{\gamma}$ is at position $i$ and $X_{\gamma}^*$ is at position $j$. Multiplication with $\Omega_{ij}$ defines an element $\Omega_{ij}\in\END_\mg(M \otimes V^{\otimes d})$ and we finally set for $1 \leq i \leq d$
\abovedisplayskip0.3em
\belowdisplayskip0.3em
\begin{gather}
\label{Defy}
y_i\,=\,\sum_{0 \leq k < i} \Omega_{ki} + \left( \frac{2n-1}{2}\right) \Id ,
\end{gather}
By definition of $\Omega$ it is clear that $y_i$ is a $\mg$-endomorphism on $M \otimes V^{\otimes d}$.
\end{definition}

Recall that a \emph{highest weight module} for $\mg$ is a $\mg$-module $M$ which is generated by a non-zero vector $m\in M$ satisfying $\mn^+ m=0$ and $\mh m\subseteq\mC m$. We say that $M$ is highest weight for short. Note that it satisfies $\END_\mg(M)=\mC$. Denote by $c_{\lambda}$ the value by which $C$ acts on a module of highest weight $\lambda$. 

\begin{remark} \label{omega_sigma_tau}
The representation $V \otimes V$ decomposes as a $\mg$-module into the irreducible representations $L(0)$, $L(2\epsilon_n)$, and $L(\epsilon_n+\epsilon_{n-1})$. A small computation shows that on $V \otimes V$ the following equation holds
\abovedisplayskip0.3em
\belowdisplayskip0.3em
\[
\Omega = - {\rm pr}_{L(0)} ( c_{\epsilon_n} {\rm id} + \sigma ) + \sigma.
\]

Note that $\tau$ from \eqref{Deftau} is a quasi-projection from $V \otimes V$ onto the copy of the trivial representation $L(0)$ inside $V \otimes V$.
A quick computation using the explicit form of $\Omega$ given above shows that multiplication with $\Omega$ on $V \otimes V$ is equal to the morphism
$\sigma - \tau$.
\end{remark}

Associated with $\la\in\mh^*$ and the corresponding $1$-dimensional module $\mC_\la$ we have the (ordinary) Verma module $\Vla=U(\mg)\otimes_{U(\mb)}\mC_\la$ of highest weight $\la$ and its irreducible quotient $L(\la)$. For any highest weight module $M$ we have:

\begin{lemma} \label{relation4} There exist complex numbers $(a_l(M))_{l \geq 0}$ such that $e_1 y_1^l e_1 = a_l(M)e_1$ for $l \geq 0$ in $\END_\mg(M \otimes V^{\otimes d})$.
\end{lemma}
\begin{proof}

The case $k=0$ follows directly from the definitions. Recalling Remark~\ref{omega_sigma_tau}, we first let $d=2$ and consider the composition
\abovedisplayskip0.3em
\belowdisplayskip0.3em
$$f: M=M \otimes L(0) \longrightarrow M \otimes V \otimes V \stackrel{y_1^k}\longrightarrow M \otimes V \otimes V \stackrel{e_1}\longrightarrow M \otimes
L(0)=M,$$
where the first map is the canonical inclusion. Now $f$ is an endomorphism of $M$, hence must be a multiple, say $a_k(M)$, of the identity. By pre-composing with $e_1$ we obtain $e_1 y_1^k e_1 = a_k(M)e_1.$
This identity also holds for $d>2$ since we just act by identities on the following tensor factors.
\end{proof}

\begin{theorem} \label{thm:actionofvw}
Let $M$ be highest weight and $\Xi_M=\{a_k(M)\mid k\geq 0\}$ as in Lemma \ref{relation4}. Then there is a well-defined right action of $\VWd(2n)$ on $M \otimes V^{\otimes d}$ defined by $w.s_i=s_i(w)$, $w.e_i=e_i(w)$, $w.y_j=y_j(w)$, for $w\in M \otimes V^{\otimes d}$, $1 \leq i \leq r-1$, $1 \leq j \leq r$ and $k\in\mZ$. In particular, we obtain an algebra homomorphism
\abovedisplayskip0.3em
\belowdisplayskip0.3em
\begin{gather}
\label{Thmaction}
\Psi_M=\Psi_M^{d,n}\,:\,{\VWd}(\Xi_M)\,\,\longrightarrow\,\,\END_\mg(M \otimes V^{\otimes d})^{\op{opp}}.
\end{gather}
\end{theorem}
\begin{proof}
We need to show that the assignment respects the relations of the affine $\VW$-algebra. This will be done in separate Lemmas in the Appendix. Relation \eqref{1} is
obvious, as are relations \eqref{2a} and \eqref{2b}, while \eqref{2c} follows from Lemma~\ref{relation2c}. Relation \eqref{3} is obvious as well and \eqref{4} follows from
Lemma \ref{relation4}. The relations \eqref{6a}-\eqref{6d} follow from Lemma \ref{relation6a6b6c6d}, relation \eqref{5a} is trivial, while \eqref{5b}) follows from Lemma~\ref{relation5b} and \eqref{5c} from Lemma~\ref{relation5c}. Finally relation \eqref{7} follows from Lemma \ref{relation7} and relations \eqref{8a}-\eqref{8b} from Lemma~\ref{relation8a8b}.
\end{proof}

\begin{remark}
As seen in the proof of Theorem \ref{thm:actionofvw} the parameters $\Xi_M$ of $\VWd(\Xi_M)$
depend on $n$ and the highest weight of $M$. This is a substantial difference to the type $ \mathrm{A}$ situation. There, the degenerate affine Hecke algebra acts on the endofunctor $\_ \otimes V^{\otimes r}$ of $\cO(\GL_n)$ for any $n$, with $V$ being the natural representation of $\GL_n$. This property plays an important role in the context of categorification of modules over quantum groups, \cite{Rouquier}, \cite{BK1}, \cite{BSIII}. 
To achieve a similar situation one needs to enlarge the algebra $\VWd(\Xi_M)$ to include the $w_k$ for $k \geq 0$ as central generators. Similar actions are described in \cite[2.2]{DRV1} and for Birman-Murakawi-Wenzl algebras in \cite{OR}.
\end{remark}

\begin{remark}
The action defined here can also be modified to give an action of $\VWd(\Xi_M)$ for a highest weight module $M$ for a Lie algebra of types $\mathrm{B}$ or $\mathrm{C}$ and their respective defining representation as $V$.
\end{remark}

\subsection{Cyclotomic quotients and admissibility}
Recall from \cite[Definition 2.10]{AMR} that the parameters $w_a$, $a\geq 0$ are {\it admissible} if they satisfy the following
\abovedisplayskip0.3em
\belowdisplayskip0.3em
\begin{gather}
\label{admissibility}
w_{2a+1}\,+\,\tfrac{1}{2} w_{2a}\,-\,\tfrac{1}{2}\sum_{b=1}^{2a}(-1)^{b-1}
w_{b-1} w_{2a-b+1}\,=\,0.
\end{gather}
We will call this recursion relation the \emph{admissibility conditions}.

\begin{definition} Given $\mathbf{u}=(u_1, u_2,\ldots, u_l)\in \mathbb{C}^l$ we denote 
\abovedisplayskip0.3em
\belowdisplayskip0.3em
\[
\VWd(\Xi,{\bf u}) = {\VWd(\Xi)}\Big/_{\prod_{i=1}^l(y_1-u_i)}
\]
and call it the \emph{cyclotomic VW-algebra of level $l$} with parameters $\bf{u}$.
\end{definition}

\begin{ex}\label{exBrauer}
If $M=L(0)$ is the trivial representation, then the action from Theorem \ref{thm:actionofvw} factors through the cyclotomic quotient 
\abovedisplayskip0.3em
\belowdisplayskip0.3em
\[
\Br_d(N)\,=\,\VWd(\Xi_{L(0)})\Big/_{\left(y_1-\frac{(N-1)}{2}\right)},
\]
with the notation already indicating that this is the usual Brauer algebra $\Br_d(N)$ introduced in \cite{Brauer}, see \cite[(2.2)]{Nazarov}. In this case $w_a=N\left({\scriptstyle \frac{N-1}{2}}\right)^a$ for $a\geq 0$ and they form an admissible sequence.
\end{ex}

More generally, for $\delta\in\mathbb{C}$, the cyclotomic quotient 
\abovedisplayskip0.3em
\belowdisplayskip0.3em
\[
\Br_d(\delta)\,=\,\VWd(\Xi)\Big/_{\left(y_1-\frac{(\delta-1)}{2}\right)} \qquad (\text{for } \Xi = \left( \delta\left(\scriptstyle \frac{\delta-1}{2}\right)^a \right)_{a\geq 0})
\] 
is the {\it Brauer algebra} of rank $d$ with parameter $\delta$, see e.g. \cite{GW}, \cite{KX}. It has a description as a diagram algebra
with basis $\mathbb{B}(\Br_d)$ consisting of the {\it Brauer diagrams} which display partitions of $2d$ ordered elements into precisely $d$ subsets of order $2$ (see e.g. \eqref{basiscycdia} ignoring the decorations), with a diagrammatic multiplication rule involving $\delta$. It obviously has dimension $(2d-1)!!=1\cdot3\cdot \ldots \cdot (2d-1)$. As an algebra it is generated by the following Brauer diagrams for $1\leq i\leq d-1$. 
\begin{gather} \label{Brgenerators}
\begin{gathered}
\begin{tikzpicture}[anchorbase,scale=0.7,thick,>=angle 90]
\node at (0,.5) {$s_i=$};
\draw (.6,0) -- +(0,1);
\draw [dotted] (1,.5) -- +(1,0);
\draw (2.4,0) -- +(0,1);
\draw (3,0) node[below] {\tiny i} -- +(.6,1);
\draw (3.6,0) node[below] {\tiny i+1} -- +(-.6,1);
\draw (4.2,0) -- +(0,1);
\draw [dotted] (4.6,.5) -- +(1,0);
\draw (6.2,0) -- +(0,1);
\begin{scope}[xshift=8cm]
\node at (0,.5) {$e_i=$};
\draw (.6,0) -- +(0,1);
\draw [dotted] (1,.5) -- +(1,0);
\draw (2.4,0) -- +(0,1);
\draw (3,0) node[below] {\tiny i} to [out=90,in=-180] +(.3,.3) to [out=0,in=90] +(.3,-.3) node[below] {\tiny i+1};
\draw (3,1) to [out=-90,in=-180] +(.3,-.3) to [out=0,in=-90] +(.3,.3);
\draw (4.2,0) -- +(0,1);
\draw [dotted] (4.6,.5) -- +(1,0);
\draw (6.2,0) -- +(0,1);
\end{scope}
\end{tikzpicture}
\end{gathered}
\end{gather}
In fact, any Brauer diagram is a monomial in these generators.

\begin{remark} \label{semisimplicity}
In case $M=L(0)$, $N\geq r$, Theorem~\ref{thm:actionofvw} turns then into the classical faitful action of the Brauer algebra on tensor space, see e.g. \cite{GW}, in particular $\Br_d(N)$ is semisimple.  In fact, $\Br_d(\delta)$ is generically semisimple, \cite{Wenzl}; for $\delta\geq0$ it is semisimple in precisely the following cases, see \cite{Brown} or \cite{Rui}, \cite{AST}:
$\delta\not=0, \text{ and } \delta\geq d-1$ or $\delta=0,\text{ and }d=1,3,5.$
\end{remark}

Admissibility ensures the existence of a nice basis of $\VWd$ as follows: For each Brauer diagram $b\in \mathbb{B}(\Br_d)$ we fix a monomial presentation in \eqref{Brgenerators} and consider the corresponding expression $b$ in the affine VW-algebra. Then given $\gamma,\eta\in\mathbb{Z}_{\geq 0}^r$ and $b \in \mathbb{B}(\Br_d)$ we have the element  $y_1^{\gamma_1}y_2^{\gamma_2}\cdots y_d^{\gamma_d}b y_1^{\eta_1}y_2^{\eta_2}\cdots y_d^{\eta_d}\in\VWd$. A monomial of this form (in the generators of $\VWd$) is \emph{regular} if $\gamma_i\not=0$ implies that $i$ is the left endpoint of a horizontal arc in $b$, and $\eta_i=0$ if  $i^*$ is the left endpoint of a horizontal arc in $b$, see \eqref{basiscycdia}. These monomials form a basis for $\VWd$ by \cite[Theorem 4.6]{Nazarov} in case the parameters are admissible. Under some  additional admissibility condition, \cite[Theorem A, Prop. 2.15]{AMR}, cyclotomic quotients inherit a basis:
\begin{prop}
\label{VWbasis}
If the  $w_a$, $a\geq 0$ are ${\bf u}$-admissible then $\VWd(\Xi,{\bf u})$ has dimension $l^d(2d-1)!!$. The regular monomials with $0\leq \gamma_i,\eta_i< l$ for $1\leq i\leq r$ form a basis $\mathbb{B}(\VWd(\Xi,{\bf u}))$.
\end{prop}

For the definition of $\bf{u}$-admissibility see \cite[Def. 3.6]{AMR}. We only need the special example from \cite[Lemma 3.5]{AMR}:
\begin{ex}
\label{uadmissibility}
Assume the entries of ${\bf u}$ are pairwise distinct and non-zero. Then the $w_a=\sum_{i=1}^l(2u_i-(-1)^l)u_i^a\prod_{1\leq j\not=i\leq l} \frac{u_i+u_j}{u_i-u_j}$ for $a\geq 0$ form a ${\bf u}$-admissible sequence.
\end{ex}

\subsection{A special case}
In the following we study Theorem \ref{thm:actionofvw} for special choices for  $M$, namely certain parabolic Verma modules. Recall our choice of parabolic $\pp$ from Section \ref{section:lie_basics}. To $\delta\in\mZ$ we associate the weight
\abovedisplayskip0.3em
\belowdisplayskip0.3em
\[
\de\,=\,\delta\om_0\,=\,\nicefrac{\delta}{2}\sum_{i=1}^n{\epsilon_i},
\]
with $\om_0$ the fundamental weight of $\mg$ corresponding to $\alpha_0$. In particular, $\de\in\La$ and we have the Verma module $\Mde$.

\begin{prop}\label{Verma}
For $\la \in \La$,
$M^\pp(\la) \otimes V$ has a filtration with sections isomorphic to $M^\pp(\la \pm \epsilon_j)$ for all $j \in \mathbb{I}^+$ such that $\la \pm \epsilon_j \in \La$ and each of these Verma modules appearing exactly once.
\end{prop}

\begin{proof}
This is a standard consequence of  \eqref{parabolic-Verma} and the tensor identity,
see \cite[Theorem 3.6]{Hbook}, using that $V$ has weights $\pm\epsilon_j$ for $j\in \mathbb{I}^+$ with multiplicity $1$.
\end{proof}

Applying Proposition \ref{Verma} iteratively, we obtain a bijection between Verma modules $M^\pp(\mu)$ appearing as subquotients in a Verma filtration of $\MdV$ and {\it $d$-admissible weight sequences} or {\it Verma paths} ending at $\mu$ where the latter is defined as follows: For fixed $d\geq 1$ and $\delta\geq 0$ a weight $\mu\in\La$ is called \emph{$d$-admissible for $\delta$} if there is a sequence
\abovedisplayskip0.3em
\belowdisplayskip0.3em
\begin{gather}
\label{Vermaseq}
\de\,=\,\bla^1 \rightarrow \bla^2 \rightarrow \cdots \rightarrow \bla^d \,=\,\mu
\end{gather}
of length $d$, starting at $\de$ and ending at $\mu$, of weights in $\La$ such that $\la^{i+1}$ differs from $\la^{i}$ by adding precisely one weight of $V$, i.e.  $\la^{i+1} = \la^i \pm \epsilon_j$ for some $j \in \mathbb{I}^+$. For instance, there are eight $2$-admissible weight sequences for $\de$.
\abovedisplayskip0.3em
\belowdisplayskip0.3em
\[
\begin{array}{ll}
\de \rightarrow \de - \epsilon_1 \rightarrow \de - 2 \epsilon_1, & \qquad \de \rightarrow \de - \epsilon_1 \rightarrow \de - \epsilon_1 - \epsilon_2, \\
\de \rightarrow \de - \epsilon_1 \rightarrow \de, & \qquad
\de \rightarrow \de - \epsilon_1 \rightarrow \de - \epsilon_1 + \epsilon_n, \\
\de \rightarrow \de + \epsilon_n \rightarrow \de + 2 \epsilon_n, & \qquad \de \rightarrow \de + \epsilon_n \rightarrow \de + \epsilon_n + \epsilon_{n-1}, \\
\de \rightarrow \de + \epsilon_n \rightarrow \de, & \qquad \de \rightarrow \de + \epsilon_n \rightarrow \de + \epsilon_n - \epsilon_1.
\end{array}
\]
Let $\vpath_d(\delta)$ (resp,  $\vpath_d(\delta)(\mu)$ ) be the set of all such Verma paths (ending at $\mu$).

\begin{prop}
\label{prop:eigenvalues}
As $\mathfrak{g}$-modules $M(\de)\otimes V\cong M(\de-\epsilon_1)\oplus M(\de+\epsilon_n)$ and this is an eigenspace decomposition for the action of $y_1$. The eigenvalues are $\alpha=\half(1-\delta)$ and $\beta=\half(\delta+N-1)$.
\end{prop}

\begin{proof}
By Proposition~\ref{Verma}, $M^\pp(\de)\otimes V$ has a Verma flag of length two with the asserted Verma modules appearing. The filtration obviously splits since they have different central character. The Casimir $C$ acts on a highest weight module with highest weight $\la$ by $c_\la=\langle \la,\la+2\rho\rangle$, see e.g. \cite[Lemma 8.5.3]{Musson} and on the tensor product $M^\pp(\de)\otimes V$ as $\Delta(C)=C\otimes 1+1\otimes C+2\Omega_{0,1}$. Hence $y_1-\frac{N-1}{2}$ acts on the summands $M^\pp(\de+\nu)$ of $M^\pp(\de)\otimes V$ by
\abovedisplayskip0.3em
\belowdisplayskip0.3em
\begin{gather*}
\half\left(\langle \delta+\nu,\delta+\nu+2\rho\rangle-
\langle\delta,\delta+2\rho\rangle-\langle\epsilon_n,\epsilon_n+2\rho\rangle\right)
=
\begin{cases}
-\nicefrac{\delta}{2}-(n-1)&\text{if $\nu=-\epsilon_1$}\\
\phantom{-}\nicefrac{\delta}{2}&\text{if $\nu=\epsilon_n$}
\end{cases}
\end{gather*}
The statement follows now from the definition of $\alpha$ and $\beta$.
\end{proof}

\begin{definition}
\label{defalphabeta}
From now on set $\alpha=\half(1-\delta)$ and $\beta=\half(\delta+N-1)$ and abbreviate $\VWd=\VWd(\Xi_{M^\pp(\de)})$ and $\VWd(\alpha,\beta)=\VWd(\Xi_{M^\pp(\de)};\alpha,\beta)$ in the notation  from Theorem~\ref{thm:actionofvw}.
\end{definition}

\begin{corollary}
For $M=M^\pp(\de)$, the action \eqref{Thmaction} factors through the cyclotomic quotient with parameters $(\alpha,\beta)$ inducing an algebra homomorphism
\abovedisplayskip0.3em
\belowdisplayskip0.3em
\begin{gather*}
\Psi_{M^\pp(\de)}^{d,n}\,:\,\VWd(\alpha,\beta)\,\,\longrightarrow\,\,\END_\mg(\MdV)^{\op{opp}}.
\end{gather*}
\end{corollary}

\begin{lemma}
\label{cor:recursion}
The elements $w_a$ of $\Xi_{M^\pp(\de)}$, $a\geq 0$, satisfy the recursion formula
$w_0=N$, $w_1=N\left({\scriptstyle \frac{N-1}{2}}\right)$, and for $a\geq 2$
\begin{gather}
\label{recursion}
w_a\,=\,(\alpha+\beta)w_{a-1}\,-\,\alpha\beta w_{a-2}.
\end{gather}
\end{lemma}

\begin{proof}
By definition we have $e_1y_1^0e_1=e_1^2=Ne_1$, hence $w_0=N$. On the other hand for any $j\in \mathbb{I}^+$ we have
\abovedisplayskip0.3em
\belowdisplayskip0.3em
\begin{gather*}
e_1y_1e_1(m\otimes v_j\otimes v_j^*) \,\,=\,\, \sum_{k \in \mathbb{I}} e_1y_1(m\otimes v_k\otimes v_{k}^*)
\end{gather*}
Recalling \eqref{DefOmega}, \eqref{Defy} this is equal to
\begin{gather*}
\begin{aligned}
&e_1 \sum_{k \in \mathbb{I}} \left( \sum_{i \in \mathbb{I}^+} X_i^* m\otimes X_i v_k \otimes v_{k}^* +\hspace{-12pt}\sum_{\alpha \in \mathrm{R}(\mathfrak{so}_{2n})} \hspace{-12pt} X_\alpha m \otimes X_\alpha^* v_k \otimes v_k^* \right) + \tfrac{N(N-1)}{2} e_1 m\otimes v_j\otimes v_{j}^* \\
=\,\,&e_1 \sum_{i \in \mathbb{I}^+} \left( \sum_{k \in \mathbb{I}^+} X_i^* m \otimes \epsilon_k(X_i) v_k \otimes v_{k}^* - \sum_{k \in \mathbb{I}^-} X_i^* m \otimes \epsilon_k(X_i) v_k \otimes v_{k}^*\right)\\
&+ {\frac{N(N-1)}{2}} e_1 m\otimes v_j\otimes v_{j}^*.
\end{aligned}
\end{gather*}
We obtain $w_1=N\left(\scriptstyle \frac{N-1}{2}\right)$. Finally, $y_1^2=(\alpha+\beta)y_1-\alpha\beta$ by Proposition~\ref{prop:eigenvalues}, and hence $e_1y_1^ne_1=(\alpha+\beta)e_1y_1^{n-1}e_1-\alpha\beta e_1y_1^{n-2}e_1$.
\end{proof}

\begin{lemma}
\label{lem:omegan}
The $w_a$ from Lemma \ref{cor:recursion} are explicitly given as
\abovedisplayskip0.3em
\belowdisplayskip0.3em
\begin{gather}
\label{explicit-formula}
w_a
\,=\,N\sum_{k=0}^{a}\alpha^{a-k}\left(\nicefrac{N}{2}-\alpha\right)^k
-\nicefrac{N}{2}\sum_{k=0}^{a-1}\alpha^{a-1-k}\left(\nicefrac{N}{2}-\alpha\right)^k
\end{gather}
\end{lemma}
%&=&N\beta^n-\frac{N\delta}{2}\sum_{k=0}^{n-1}\alpha^{n-1-k}\beta^k\\
\begin{proof}
For $a=0, 1$ the claim follows immediately. Note that the recursion formula \eqref{recursion} has the general solution $w_a=A\alpha^a+B\beta^a$ with boundary conditions $A+B=N$ and $A\alpha+B\beta=\nicefrac{N}{2}(N-1)$, see e.g. \cite[Theorem 33.10]{Lidl-Pilz}.
Hence
\abovedisplayskip0.3em
\belowdisplayskip0.3em
\begin{gather*}
A\,=\,{\scriptstyle \frac{1}{\alpha-\beta}}\left(\nicefrac{N}{2}(N-1)-N\beta\right)\,=\,{\scriptstyle\frac{1}{\alpha-\beta}}\left(N\alpha-\nicefrac{N}{2}\right)
\end{gather*}
and therefore
\abovedisplayskip0.3em
\belowdisplayskip0.3em
\begin{gather}
\label{formulaw}
w_a\,=\,N(\alpha-\half)\frac{\alpha^a-\beta^a}{\alpha-\beta}+N\beta^a\,=\,
N(\alpha-\half)\sum_{k=0}^{a-1}\alpha^{a-1-k}\beta^{k}+N\beta^a.
\end{gather}
The lemma follows then by substituting $\beta=\nicefrac{N}{2}-\alpha$.
\end{proof}

For convenience we give a direct proof of the following result (which could alternatively be deduced from Lemma \ref{lem:uadm} using \cite[Corollary 3.9]{AMR}).
\begin{prop}
The $w_a$ from \eqref{explicit-formula} are admissible, i.e. satisfy \eqref{admissibility}.
\end{prop}

\begin{proof}
Set $Q(a)=N\sum_{k=0}^{a}\alpha^{a-k}\left(\nicefrac{N}{2}-\alpha\right)^k$. Then $w_a=Q(a)-\half Q(a-1)$ and the admissibility condition \eqref{admissibility} is for $m=2a+1$ equivalent to
\abovedisplayskip0.3em
\belowdisplayskip0.3em
\begin{gather}
\label{Yformel}
\begin{aligned}
0\,=&\,\, 2Q(m)-\half Q(m-2)+R\\
&-\sum_{b=1}^{m-1}(-1)^{b-1}\left(
Q(b-1)Q(m-b)+\nicefrac{1}{4}Q(b-2)Q(m-b-1)\right)
\end{aligned}
\end{gather}
where $R=\sum_{b=1}^{m-1}(-1)^{b-1}(Q(b-2)Q(m-b)+Q(b-1)Q(m-b-1))=0$, since $Q(-1)=0$ by definition and then
\abovedisplayskip0.3em
\belowdisplayskip0.3em
\begin{gather*}
\begin{aligned}
R\,=&\,\sum_{b=0}^{m-1}(-1)^{b}Q(b-1)Q(m-b-1)+\sum_{b=1}^{m}(-1)^{b-1}Q(b-1)Q(m-b-1)\\
=&\,(-1)^{m-1}Q(m-1)Q(-1)+Q(-1)Q(m-1)=0.
\end{aligned}
\end{gather*}

By the right hand side of \eqref{Yformel}, it is enough to show that $S(t)=2Q(t)-\sum_{b=1}^t (-1)^{b-1}Q(b-1)Q(t-b)=0$ for $t=m, m-1$. For this we consider $S(t)$ as a polynomial in $N$ and show that all the coefficients vanish. First note that
\abovedisplayskip0.3em
\belowdisplayskip0.3em
\begin{gather*}
2Q(t)=2N\sum_{k=0}^t\alpha^{t-k}\sum_{r=0}^k\binom{k}{r}\nicefrac{1}{2^r}N^r(-\alpha)^{k-r}
=\sum_{k=0}^t\sum_{r=0}^k\binom{k}{r}\nicefrac{1}{2^{r-1}}(-1)^{k-r}\alpha^{t-r}N^{r+1}.
\end{gather*}
Hence the coefficient in front of $N^{s+1}$ equals
\abovedisplayskip0.3em
\belowdisplayskip0.3em
\begin{gather}
\label{c}
c_{s+1}\,=\,\nicefrac{1}{2^{s-1}}(-1)^s\alpha^{t-s}\sum_{k=0}^t(-1)^{k}\binom{k}{s}.
\end{gather}
On the other hand
\abovedisplayskip0.3em
\belowdisplayskip0.3em
\begin{gather*}
\sum_{b=1}^t (-1)^{b-1}Q(b-1)Q(t-b)
\,=\, N^2\sum_{b=1}^t (-1)^{b-1}\sum_{r=0}^{b-1}\sum_{j=0}^{t-b}\alpha^{b-1-r+t-b-j}
\left(\nicefrac{N}{2}-\alpha\right)^{r+j},
\end{gather*}
hence the coefficient in front of $N^{s+1}$ equals
\abovedisplayskip0.3em
\belowdisplayskip0.3em
\begin{gather}
\label{d}
d_{s+1}\,=\,\nicefrac{1}{2^{s-1}}(-1)^s\alpha^{t-s}\sum_{b=1}^t\sum_{r=0}^{b-1}\sum_{j=0}^{t-b}
(-1)^{b+r+j}\binom{r+j}{s-1}.
\end{gather}
Clearly $c_0=0=d_0$ and then $c_{s+1}=d_{s+1}$ for all $s\geq 0$ by Lemma \ref{lem:stupid} below.
\end{proof}

\begin{lemma}
\label{lem:stupid}
Let $m\geq 0$ be odd and $s\geq 1$. Then
\abovedisplayskip0.3em
\belowdisplayskip0.3em
\begin{gather}
\label{eq:binomial}
\sum_{k=0}^m (-1)^{k}\binom{k}{s}\,\,=\,\,\sum_{b=1}^m\sum_{r=0}^{b-1}
\sum_{j=0}^{m-b}(-1)^{b+r+j}\binom{r+j}{s-1}.
\end{gather}
\end{lemma}

\begin{proof}
For $m=1$ the statement is clear. Let $L(m)$ and $R(m)$ be the left and right side of \eqref{eq:binomial}. We assume $L(m)=R(m)$ and want to deduce $L(m+2)=R(m+2)$ for which it is enough to show $R(m+2)-R(m)=L(m+2)-L(m)$. Since $m$ is odd, the latter is equivalent to verifying
\begin{gather*}
\binom{m+1}{s}-\binom{m+2}{s}\,\,=\,\,R(m+2)-R(m)
\end{gather*}
Now by definition $R(m+2)-R(m)$ equals
\begin{align*}
&\sum_{r=0}^{m+1}\sum_{j=0}^0(-1)^{m+r+j}\binom{r+j}{s-1}
+\sum_{r=0}^m\sum_{j=0}^1(-1)^{m+r+j+1}\binom{r+j}{s-1}\\
&+\sum_{b=1}^m\sum_{r=0}^{b-1}(-1)^{m+r+1}\binom{r+m+1-b}{s-1}
+\sum_{b=1}^m\sum_{r=0}^{b-1}(-1)^{m+r}\binom{r+m+2-b}{s-1}\\
=&(-1)^{2m+1}\binom{m+1}{s-1}+\sum_{r=0}^{m}(-1)^{m+r+2}\binom{r+1}{s-1}
+\sum_{r=0}^{m-1}(-1)^{m+r+1}\binom{r+1}{s-1}\\
&+\sum_{r=0}^0(-1)^m\binom{m+1}{s-1}+\sum_{b=1}^{m-1}(-1)^{m+b}\binom{m+1}{s-1}\\
=&-\binom{m+1}{s-1}+\binom{m+1}{s-1}-\binom{m+1}{s-1}-
\binom{m+1}{s-1}
=\binom{m+1}{s}-\binom{m+2}{s}.
\end{align*}
\end{proof}

\begin{lemma}
\label{lem:uadm}
The sequence $w_a$, $a\geq 0$ from Lemma \ref{lem:omegan} is $(\alpha,\beta)$-admissible.
\end{lemma}

\begin{proof}
Substituting $N$ in \eqref{formulaw} we obtain $w_a\,\,=\,\,{\scriptstyle \frac{\alpha+\beta}{\alpha-\beta}}(2\alpha^{a+1}-2\beta^{a+1}-\alpha^a+\beta^a)$ which is easy to see to agree with the formula for $w_a$ in Example~\ref{uadmissibility}.
\end{proof}

\begin{corollary}
\label{dimcycl}
The level $l=2$ cyclotomic quotients $\VWd(\alpha,\beta)$ are of dimension $2^d(2d-1)!!$ with basis $\mathbb{B}(\VWd(\alpha,\beta))$ given by the regular monomials from Proposition \ref{VWbasis}.
\end{corollary}

\section{Isomorphism theorem and special projective functors}
\subsection{The Isomorphism Theorem}
\label{section:isotheorem_cyclotomic}
\begin{theorem}[The Isomorphism Theorem]
\label{iso}
If $n \geq 2d$ and $\delta\in\mZ$ then the map $\Psi_{M^\pp(\de)}^{d,n}$ from Theorem~\ref{thm:actionofvw} induces an isomorphism of algebras
\abovedisplayskip0.3em
\belowdisplayskip0.3em
\begin{gather*}
\Psi(\de)\,:\,\VWd(\alpha,\beta) \,\,\longrightarrow \,\, \END_\mg(\MdV)^{\op{opp}}.
\end{gather*}
%The isomorphism is compatible with the quasi hereditary structure on %both algebras.
\end{theorem}
\begin{proof}
By definition it is an algebra homomorphism. It is injective by Proposition~\ref{prop:injectivity} below and surjective since the dimensions agree, Corollary~\ref{tableaux}.
\end{proof}

\begin{prop}
\label{prop:injectivity}
If $n \geq 2d$ and $\delta\in\mZ$ then $\Psi_{M^\pp(\de)}^{d,n}$ from Theorem~\ref{thm:actionofvw} induces an injective map of algebras
\abovedisplayskip0.3em
\belowdisplayskip0.3em
\begin{gather*}
\Psi(\de)\,:\,\VWd(\alpha,\beta) \,\, \longrightarrow \,\, \END_\mg(\MdV)^{\op{opp}}.
\end{gather*}
\end{prop}
\begin{proof}
By Corollary \ref{dimcycl} it is enough to show that the regular monomials $y_1^{\gamma_1}y_2^{\gamma_2}\cdots y_d^{\gamma_d}b y_1^{\eta_1}y_2^{\eta_2}\cdots y_d^{\eta_d}$ for $b\in\mathbb{B}(\Br_d)$ and $0\leq \gamma_i,\eta_i<2$ are mapped to linearly independent  morphisms. It might help to see basis elements diagrammatically by drawing the Brauer diagram with small decorations indicating the $y_i$'s, e.g. 
\begin{gather}
\label{basiscycdia}
\begin{tikzpicture}[anchorbase,thick,>=angle 90]
\begin{scope}[xshift=8cm]
\draw (0,0) -- +(0,1);\fill (-.1,.75) rectangle +(.2,.1);
\draw (.6,0) -- +(.6,1);
\draw (1.2,0) -- +(-.6,1);
\draw (1.8,0) to [out=90,in=-180] +(.9,.5) to [out=0,in=90] +(.9,-.5);\fill (1.72,.1) rectangle +(.2,.1);
\draw (1.8,1) to [out=-90,in=-180] +(.6,-.4) to [out=0,in=-90] +(.6,.4);
\draw (2.4,0) -- +(0,1);
\draw (3,0) -- +(.6,1);
\draw (4.2,0) -- +(0,1);
\draw (4.8,0) to [out=90,in=-180] +(.3,.3) to [out=0,in=90] +(.3,-.3);
\draw (4.8,1) to [out=-90,in=-180] +(.3,-.3) to [out=0,in=-90] +(.3,.3);\fill (5.3,.85) rectangle +(.2,.1);
\draw (6,0) -- +(0,1);\fill (5.9,.75) rectangle +(.2,.1);
\end{scope}
\end{tikzpicture}
\end{gather}
stands for $y_4by_1y_{10}y_{11}\in\mathbb{B}(\VWd(\alpha,\beta))$ where $b$ corresponds to the Brauer diagram in $\mathbb{B}(\Br_d)$  without the decorations. Sliding a decoration through an arc produces a linear combination of basis vectors according to Definition~\ref{def:VW}, (VW.7), (VW.8).

By Remark \ref{omega_sigma_tau} and \eqref{DefOmega} we have $y_{i}=\Omega_{0,i}+x$, where $x$ is some Brauer algebra element.
Let now $m\in\Mde$ be a highest weight vector. Then $X_\gamma m=0$ for $X_\gamma \in \mathfrak{n}^+$ and furthermore $X_{-(\epsilon_i + \epsilon_j)}m=0$ for all $i>j$ because of our specific choice of $\de$. (Note that $E(\de)$ is one dimensional in the description of $\Mde$ from \eqref{parabolic-Verma}.) Applying \eqref{DefOmega} we obtain for any $a\in \mathbb{I}^+$
\begin{gather*}
\Omega_{0,1}m\otimes v_a\,=\,{\scriptstyle \frac{\delta}{2}}m\otimes v_a, \quad\Omega_{0,1}m\otimes v_{-a}\,=\,{\scriptstyle \frac{\delta}{2}}m\otimes v_a+\sum_{i\not=a}a_i X_{-(\epsilon_i+\epsilon_j)}m\otimes v_i
\end{gather*}
for some $a_i \in \mC^*$. The $pm\otimes v_{i_1}\otimes v_{i_2}\otimes\cdots \otimes v_{i_{d-1}}\otimes v_{i_d}$, where $p$ runs through all monomials in the variables $X_{-(\epsilon_i+\epsilon_j)}$ for $i,j\in \mathbb{I}^+$ with $i>j$ form a basis $\mathbb{B}(\MdV)$ of $\MdV$. (Note that the $X_{-(\epsilon_i+\epsilon_j)}$'s commute due to the structure of the root system).
In particular, $y_r m\otimes v_{a}$ equals
\abovedisplayskip0.45em
\belowdisplayskip0.45em
\begin{gather}
\label{yaction}
\begin{cases}
\sum_{i\not=i_r} a_{i_r} X_{-(\epsilon_i+\epsilon_{i_r})}v_{i_1}\otimes\cdots\otimes v_{i_{r-1}}\otimes v_i\otimes v_{i_{r+1}}\otimes v_{i_d}
+(\dagger)&\text{if $-a\in \mathbb{I}^+$},\\
(\dagger)&\text{if $\phantom{-}a\in \mathbb{I}^+$},
\end{cases}
\end{gather}
where $(\dagger)$ stands for some linear combination of basis vectors where the above mentioned monomials have degree zero.

Now consider a standard basis vector $\gamma\in\mathbb{B}(\VWd(\alpha,\beta))$. Take its diagram \eqref{basiscycdia} and label the vertices with numbers from $\{i,-i\mid 1\leq i\leq 2d\}$ as follows: Label all vertices at the top from left to right by $d+1$ to $2d$ and at the bottom from left to right by $-1$ to $-d$. In case there is a vertical arc without decoration connecting two vertices replace the label at the top endpoint with the label at the bottom endpoint.
If there is a horizontal arc without decoration change the label at the right endpoint to the negative of the left endpoint. In the resulting diagram the labels, say $a$ and $b$, at the endpoints of arcs satisfy the following: $|a|\not=|b|$ and negative at the bottom and positive at the top if the arc is decorated; $a=b<0$ if the arc is undecorated and vertical; and $a=-b$ if the arc is undecorated and horizontal. Moreover, no number appears more than twice. In the example \eqref{basiscycdia} we get
$12,-3,-2,15,-5,-15,-6,-8,20,21,22$ at the top and $-1,-2,-3,-4,-5,-6,-7,-8,-9,9,-11$ at the bottom.

Let now $S=\sum_{\gamma\in\mathbb{B}(\VWd(\alpha,\beta))}r_\gamma \Psi(\de)(\gamma)=0$ with $r_\gamma\in\mC$. Pick $\gamma\in \mathbb{B}(\VWd(\alpha,\beta))$ corresponding to a diagram with all arcs decorated.
and let $(-a_i, b_i)$, $1\leq i\leq d$, $a_i,b_i>0$ be the pairs of labels attached to each arc.
Then, by \eqref{yaction} the coefficient of $X_{-(\epsilon_{a_1}+\epsilon_{b_1})}\cdots X_{-(\epsilon_{a_d}+\epsilon_{b_d})} m\otimes v_{1}\otimes\cdots \otimes v_{d}$ when expressing $S m\otimes v_{-1}\otimes\cdots \otimes v_{-d}$ in our basis is precisely $r_\gamma$, hence $r_\gamma=0$. Repeating this argument gives $c_\gamma=0$ for all diagrams with all strands decorated. Next pick $\gamma^\prime\in \mathbb{B}(\VWd(\alpha,\beta))$ which corresponds to a diagram with all arcs except one decorated and let $j_1,\ldots j_d$ and $j_1',\ldots j_d'$ be the associated labels at the bottom respectively top of the diagram read from left to right. Let $(-a_i, b_i)$, $1\leq i\leq d-1$, $a_i,b_i>0$ be the pairs of labels attached to each arc. By \eqref{yaction}, the coefficient of $X_{-(\epsilon_{a_1}+\epsilon_{b_1})}\cdots X_{-(\epsilon_{a_{d-1}}+\epsilon_{b_{d-1}})} m\otimes v_{i_1}\otimes\cdots v_{i_d}$  when expressing $Sm\otimes v_{j_1}\otimes\cdots v_{j_d}$ in our basis is then precisely $r_{\gamma^\prime}$. Hence $r_{\gamma^\prime}=0$ for all diagrams with only one undecorated arc. Proceeding like this gives finally $r_\gamma=0$ for all $\gamma$. Hence linear independence follows.
\end{proof}

\subsection{Bitableaux and Verma paths}
Based on \cite{AMR} we introduce a labelling set for a basis in the cyclotomic quotients and connect it with the counting of Verma modules from Proposition~\ref{Verma} via the diagrammatic weights.

A \emph{partition} of $d$ is a sequence $\la=(\la_1,\la_2,\ldots)$ of weakly decreasing non-negative integers with size $|\la|=\sum_{i\geq 1}{\la_i}=d$. We identify a partition with its Young diagram containing $\la_i$ boxes in the $i$-th row and call its top left vertex the origin. The set of all partitions will be denoted by $\cP^1$, respectively $\cP^1_d$ if we fix the size to be $d$. A \emph{bipartition} of $d$ is an ordered pair $\bla=(\la^{(1)},\la^{(2)})$ of partitions with size $|\lambda|=|(\la^{(1)},\la^{(2)})|=|\la^{(1)}|+|\la^{(2)}|=d$. We denote the set of bipartitions by $\cP^2$, respectively $\cP^2_d$ if we fix the size to be $d$.

An \emph{\updb} of length $d$ (or \emph{\updbd}) is a sequence $\left(\bla(0)=(\emptyset,\emptyset), \bla(1),\cdots, \bla(d) \right)$ of bipartitions such that consecutive elements, $\la(i+1)$ and $\la(i)$, differ by removing or adding a single box. We denote by 
\[
\abovedisplayskip0.45em
\belowdisplayskip0.45em
\cT^2\supset \cT^2_d\supset\cT^2_d(\bla)\quad\text{and}\quad\cT^1\supset \cT^1_d\supset\cT^1_d(\bla)
\]
the set of all \updb x, respectively those with a fixed length and those with fixed length and final bitableaux $\lambda$ and similarly for tableaux. From \cite{AMR} it follows in particular that for $l=1,2$
\abovedisplayskip0.3em
\belowdisplayskip0.3em
\begin{gather}
\label{cardinality}
\sum_\la|\cT^l_d(\la)|^2 \,\,=\,\, l^d(2d-1)!!.
\end{gather}
Recall the diagrammatic weights from Section \ref{section:diagrammatics}. We will abuse notation and denote the diagrammatic weight $a_\lambda$ by $\lambda$ as well for $\lambda \in \La$.  The diagrammatic weight $\de$ is then (The first and third are in $\mX_n$, the second in $\mX_n^\hint$).
\abovedisplayskip0.3em
\belowdisplayskip0.3em
\begin{gather*}
\begin{cases}
\begin{picture}(-60,33)
\put(2,20){\line(1,0){223}}
\put(10.7,17.3){$\circ$}
\put(10.7,5.3){$0$}
\put(33.7,17.3){$\circ$}
\put(56.7,17.3){$\circ$}
\put(79.7,17.3){$\circ$}
\put(74.7,5.3){$\frac{\delta-2}{2}$}
%\put(-41.9,2){\line(0,1){6}}
\put(101.3,20.1){$\scriptstyle\down$}
\put(101.3,5.3){$\frac{\delta}{2}$}
\put(124.3,20.1){$\scriptstyle\down$}
\put(147.3,20.1){$\scriptstyle\down$}
\put(170.3,20.1){$\scriptstyle\down$}
\put(193.3,20.1){$\scriptstyle\down$}
\put(183.3,5.3){$\frac{\delta+N-2}{2}$}
\put(216.3,17.3){$\circ$}
\put(233,17.5){$\cdots\quad$ if $\delta>0$ even,}
%\put(-46,11){$\underbrace{\phantom{hellow worl}}_{\delta}$}
%\put(23,11){$\underbrace{\phantom{hellow worl}}_{n}$}
\end{picture}
\phantom{xxxxxxxxxxxxxxxxxxxxxxxxxxxxxxxxxxxxxxxxxxxxxxxxxxxxxx}\\
\begin{picture}(-60,33)
\put(2,20){\line(1,0){223}}
\put(10.7,17.3){$\circ$}
\put(10.7,5.3){$\frac{1}{2}$}
\put(33.7,17.3){$\circ$}
\put(56.7,17.3){$\circ$}
\put(79.7,17.3){$\circ$}
\put(74.7,5.3){$\frac{\delta-2}{2}$}
%\put(-41.9,2){\line(0,1){6}}
\put(101.3,20.1){$\scriptstyle\down$}
\put(101.3,5.3){$\frac{\delta}{2}$}
\put(124.3,20.1){$\scriptstyle\down$}
\put(147.3,20.1){$\scriptstyle\down$}
\put(170.3,20.1){$\scriptstyle\down$}
\put(193.3,20.1){$\scriptstyle\down$}
\put(183.3,5.3){$\frac{\delta+N-2}{2}$}
\put(216.3,17.3){$\circ$}
\put(233,17.5){$\cdots\quad$ if $\delta>0$ odd,}
%\put(-46,11){$\underbrace{\phantom{hellow worl}}_{\delta}$}
%\put(23,11){$\underbrace{\phantom{hellow worl}}_{n}$}
\end{picture}\\
\begin{picture}(-60,33)
\put(2,20){\line(1,0){223}}
\put(10.7,17.3){$\diamondb$}
\put(10.7,5.3){$0$}
\put(33.7,20.1){$\scriptstyle\down$}
\put(56.7,20.1){$\scriptstyle\down$}
\put(79.7,20.1){$\scriptstyle\down$}
%\put(-41.9,2){\line(0,1){6}}
\put(101.3,20.1){$\scriptstyle\down$}
\put(124.3,20.1){$\scriptstyle\down$}
\put(147.3,20.1){$\scriptstyle\down$}
\put(170.3,20.1){$\scriptstyle\down$}
\put(193.3,20.1){$\scriptstyle\down$}
\put(193.3,5.3){$n$}
\put(216.3,17.3){$\circ$}
\put(233,17.5){$\cdots\quad$ if $\delta=0$.}
\end{picture}
\end{cases}
\end{gather*}
In light of Corollary~\ref{tableaux} we stick here to the case $\delta\geq 0$. 

Suppose we are given a $d$-admissible weight sequence as in \eqref{Vermaseq}
\abovedisplayskip0.3em
\belowdisplayskip0.3em
\begin{gather}
\label{wseq}
  \de\,=\,\lambda^1 \rightarrow \lambda^2 \rightarrow \ldots \rightarrow \lambda^d.
\end{gather}
We are going to represent this diagrammatically by what we call a Verma path of length $d$ (depending on the parameter $\delta\in\mZ_{\geq 0}$), similar to the diagrams in \cite{BSIII}, \cite{BS_walled_Brauer}. Note that, since the weights $\la^i$ and $\la^{i+1}$ differ only by some $\pm \epsilon_j$ and let $r_i=|\lambda^i_j|$ and $s_i=|\lambda^i_j \pm 1|$, the associated weight diagrams differ precisely at the coordinates $r_i$ and $s_i$. 
\begin{definition} The {\it Verma path} associated with a weight sequence \eqref{wseq} is given as follows: 
first draw the corresponding sequence of weight diagrams from bottom to top. 
\abovedisplayskip0.45em
\belowdisplayskip0.3em
\begin{gather}
\label{fig:Vermapath}
\begin{gathered}
\begin{tikzpicture}[scale=.67,>=angle 90,c/.style={insert path={circle[radius=2pt]}},d/.style={fill=black,shape=diamond,text height=-10pt,draw}]
\def\dist{2.2cm}
\draw (0,0) +(-3pt,3pt) -- +(0,0pt) -- +(3pt,3pt);\draw (1,0) [c];
\draw (0,0) to [in=-90,out=90] +(1,1);
\draw +(0,1) [c];\draw (1,1) +(-3pt,3pt) -- +(0,0) -- +(3pt,3pt);
\begin{scope}[xshift=\dist]
\draw (0,0) +(-3pt,-3pt) -- +(0,0pt) -- +(3pt,-3pt);\draw (1,0) [c];
\draw (0,0) to [in=-90,out=90] +(1,1);
\draw +(0,1) [c];\draw (1,1) +(-3pt,-3pt) -- +(0,0) -- +(3pt,-3pt);
\end{scope}
\begin{scope}[xshift=2*\dist]
\draw (0,0) [c];\draw (1,0) +(-3pt,3pt) -- +(0,0) -- +(3pt,3pt);
\draw (1,0) to [in=-90,out=90] +(-1,1);
\draw (0,1) +(-3pt,3pt) -- +(0,0) -- +(3pt,3pt);\draw +(1,1) [c];
\end{scope}
\begin{scope}[xshift=3*\dist]
\draw (0,0) [c];\draw (1,0) +(-3pt,-3pt) -- +(0,0) -- +(3pt,-3pt);
\draw (1,0) to [in=-90,out=90] +(-1,1);
\draw (0,1) +(-3pt,-3pt) -- +(0,0) -- +(3pt,-3pt);\draw +(1,1) [c];
\end{scope}
\begin{scope}[xshift=4*\dist]
\draw (0,0) +(-3pt,-3pt) -- +(3pt,3pt) +(-3pt,3pt) -- +(3pt,-3pt);
\draw (1,0) +(-3pt,3pt) -- +(0,0) -- +(3pt,3pt);
\draw (1,0) to [in=-90,out=90] +(-1,1);
\draw (0,1) +(-3pt,3pt) -- +(0,0) -- +(3pt,3pt);
\draw (1,1) +(-3pt,-3pt) -- +(3pt,3pt) +(-3pt,3pt) -- +(3pt,-3pt);
\end{scope}
\begin{scope}[xshift=5*\dist]
\draw (0,0) +(-3pt,-3pt) -- +(3pt,3pt) +(-3pt,3pt) -- +(3pt,-3pt);\draw (1,0) +(-3pt,-3pt) -- +(0,0) -- +(3pt,-3pt);
\draw (1,0) to [in=-90,out=90] +(-1,1);
\draw (0,1) +(-3pt,-3pt) -- +(0,0) -- +(3pt,-3pt);\draw (1,1) +(-3pt,-3pt) -- +(3pt,3pt) +(-3pt,3pt) -- +(3pt,-3pt);
\end{scope}
\begin{scope}[xshift=6*\dist]
\draw (0,0) +(-3pt,3pt) -- +(0,0pt) -- +(3pt,3pt);
\draw (1,0) +(-3pt,-3pt) -- +(3pt,3pt) +(-3pt,3pt) -- +(3pt,-3pt);
\draw (0,0) to [in=-90,out=90] +(1,1);
\draw (0,1) +(-3pt,-3pt) -- +(3pt,3pt) +(-3pt,3pt) -- +(3pt,-3pt);\draw (1,1) +(-3pt,3pt) -- +(0,0) -- +(3pt,3pt);
\end{scope}
\begin{scope}[xshift=7*\dist]
\draw (0,0) +(-3pt,-3pt) -- +(0,0pt) -- +(3pt,-3pt);\draw (1,0) +(-3pt,-3pt) -- +(3pt,3pt) +(-3pt,3pt) -- +(3pt,-3pt);
\draw (0,0) to [in=-90,out=90] +(1,1);
\draw (0,1) +(-3pt,-3pt) -- +(3pt,3pt) +(-3pt,3pt) -- +(3pt,-3pt);\draw (1,1) +(-3pt,-3pt) -- +(0,0) -- +(3pt,-3pt);
\end{scope}
\begin{scope}[yshift=-2cm]
\draw (0,0) +(-3pt,3pt) -- +(0,0pt) -- +(3pt,3pt);\draw (1,0) +(-3pt,-3pt) -- +(0,0) -- +(3pt,-3pt);
\draw (0,0) to [out=90,in=-180] +(.5,.5) to [out=0,in=90] +(.5,-.5);
\draw +(0,1) [c];\draw (1,1) +(-3pt,-3pt) -- +(3pt,3pt) +(-3pt,3pt) -- +(3pt,-3pt);
\begin{scope}[xshift=\dist]
\draw (0,0) +(-3pt,3pt) -- +(0,0pt) -- +(3pt,3pt);\draw (1,0) +(-3pt,-3pt) -- +(0,0) -- +(3pt,-3pt);
\draw (0,0) to [out=90,in=-180] +(.5,.5) to [out=0,in=90] +(.5,-.5);
\draw (0,1) +(-3pt,-3pt) -- +(3pt,3pt) +(-3pt,3pt) -- +(3pt,-3pt);\draw +(1,1) [c];
\end{scope}
\begin{scope}[xshift=2*\dist]
\draw (0,0) +(-3pt,-3pt) -- +(0,0pt) -- +(3pt,-3pt);\draw (1,0) +(-3pt,3pt) -- +(0,0) -- +(3pt,3pt);
\draw (0,0) to [out=90,in=-180] +(.5,.5) to [out=0,in=90] +(.5,-.5);
\draw +(0,1) [c];\draw (1,1) +(-3pt,-3pt) -- +(3pt,3pt) +(-3pt,3pt) -- +(3pt,-3pt);
\end{scope}
\begin{scope}[xshift=3*\dist]
\draw (0,0) +(-3pt,-3pt) -- +(0,0pt) -- +(3pt,-3pt);\draw (1,0) +(-3pt,3pt) -- +(0,0) -- +(3pt,3pt);
\draw (0,0) to [out=90,in=-180] +(.5,.5) to [out=0,in=90] +(.5,-.5);
\draw (0,1) +(-3pt,-3pt) -- +(3pt,3pt) +(-3pt,3pt) -- +(3pt,-3pt);\draw +(1,1) [c];
\end{scope}
\begin{scope}[xshift=4*\dist]
\draw (0,0) +(-3pt,-3pt) -- +(3pt,3pt) +(-3pt,3pt) -- +(3pt,-3pt);\draw +(1,0) [c];
\draw (0,1) to [out=-90,in=-180] +(.5,-.5) to [out=0,in=-90] +(.5,.5);
\draw (0,1) +(-3pt,-3pt) -- +(0,0pt) -- +(3pt,-3pt);\draw (1,1) +(-3pt,3pt) -- +(0,0) -- +(3pt,3pt);
\end{scope}
\begin{scope}[xshift=5*\dist]
\draw +(0,0) [c];\draw (1,0) +(-3pt,-3pt) -- +(3pt,3pt) +(-3pt,3pt) -- +(3pt,-3pt);
\draw (0,1) to [out=-90,in=-180] +(.5,-.5) to [out=0,in=-90] +(.5,.5);
\draw (0,1) +(-3pt,-3pt) -- +(0,0) -- +(3pt,-3pt);\draw (1,1) +(-3pt,3pt) -- +(0,0pt) -- +(3pt,3pt);
\end{scope}
\begin{scope}[xshift=6*\dist]
\draw (0,0) +(-3pt,-3pt) -- +(3pt,3pt) +(-3pt,3pt) -- +(3pt,-3pt);\draw +(1,0) [c];
\draw (0,1) to [out=-90,in=-180] +(.5,-.5) to [out=0,in=-90] +(.5,.5);
\draw (0,1) +(-3pt,3pt) -- +(0,0pt) -- +(3pt,3pt);\draw (1,1) +(-3pt,-3pt) -- +(0,0) -- +(3pt,-3pt);
\end{scope}
\begin{scope}[xshift=7*\dist]
\draw +(0,0) [c];\draw (1,0) +(-3pt,-3pt) -- +(3pt,3pt) +(-3pt,3pt) -- +(3pt,-3pt);
\draw (0,1) to [out=-90,in=-180] +(.5,-.5) to [out=0,in=-90] +(.5,.5);
\draw (0,1) +(-3pt,3pt) -- +(0,0) -- +(3pt,3pt);\draw (1,1) +(-3pt,-3pt) -- +(0,0pt) -- +(3pt,-3pt);
\end{scope}
\end{scope}
\begin{scope}[yshift=-4cm]
\begin{scope}[xshift=0]
\node[d] at (0,0) {};
\draw (1,0) [c];
\draw (0,.1) to [in=-90,out=90] +(1,.9);
\draw +(0,1) [c];
\draw (1,1) +(-3pt,3pt) -- +(0,0) -- +(3pt,3pt);
\node at (0,-.5) {$\scriptstyle 0$};
\end{scope}
\begin{scope}[xshift=1*\dist]
\node[d] at (0,0) {};\draw (1,0) [c];
\draw (0,.1) to [in=-90,out=90] +(1,.9);
%\fill (.5,.55) circle(2pt);
\draw +(0,1) [c];\draw (1,1) +(-3pt,-3pt) -- +(0,0) -- +(3pt,-3pt);
\node at (0,-.5) {$\scriptstyle 0$};
\end{scope}
\begin{scope}[xshift=2*\dist]
\draw +(0,0) [c];\draw (1,0) +(-3pt,-3pt) -- +(0,0) -- +(3pt,-3pt);
\draw (1,0) to [in=-90,out=90] +(-1,.9);
%\fill (.5,.45) circle(2pt);
\node[d] at (0,1) {};\draw (1,1) [c];
\node at (0,-.5) {$\scriptstyle 0$};
\end{scope}
\begin{scope}[xshift=3*\dist]
\draw +(0,0) [c];\draw (1,0) +(-3pt,-3pt) -- +(0,0) -- +(3pt,-3pt);
\draw (1,0) to [in=-90,out=90] +(-1,.9);
\node[d] at (0,1) {};\draw (1,1) [c];
\node at (0,-.5) {$\scriptstyle 0$};
\end{scope}
\begin{scope}[xshift=4*\dist]
\node[d] at (0,0) {};\draw (1,0) +(-3pt,-3pt) -- +(0,0) -- +(3pt,-3pt);
\draw (0,.1) to [out=90,in=-180] +(.5,.4) to [out=0,in=90] +(.5,-.5);
\draw +(0,1) [c];\draw (1,1) +(-3pt,-3pt) -- +(3pt,3pt) +(-3pt,3pt) -- +(3pt,-3pt);
\node at (0,-.5) {$\scriptstyle 0$};
\end{scope}
\begin{scope}[xshift=5*\dist]
\node[d] at (0,0) {};\draw (1,0) +(-3pt,3pt) -- +(0,0) -- +(3pt,3pt);
\draw (0,.1) to [out=90,in=-180] +(.5,.4) to [out=0,in=90] +(.5,-.5);
%\fill (.5,.5) circle(2pt);
\draw +(0,1) [c];\draw (1,1) +(-3pt,-3pt) -- +(3pt,3pt) +(-3pt,3pt) -- +(3pt,-3pt);
\node at (0,-.5) {$\scriptstyle 0$};
\end{scope}
\begin{scope}[xshift=6*\dist]
\draw +(0,0) [c];\draw (1,0) +(-3pt,-3pt) -- +(3pt,3pt) +(-3pt,3pt) -- +(3pt,-3pt);
\draw (0,.9) to [out=-90,in=-180] +(.5,-.4) to [out=0,in=-90] +(.5,.5);
%\fill (.5,.5) circle(2pt);
\node[d] at (0,1) {};\draw (1,1) +(-3pt,3pt) -- +(0,0) -- +(3pt,3pt);
\node at (0,-.5) {$\scriptstyle 0$};
\end{scope}
\begin{scope}[xshift=7*\dist]
\draw +(0,0) [c];\draw (1,0) +(-3pt,-3pt) -- +(3pt,3pt) +(-3pt,3pt) -- +(3pt,-3pt);
\draw (0,.9) to [out=-90,in=-180] +(.5,-.4) to [out=0,in=-90] +(.5,.5);
\node[d] at (0,1) {};\draw (1,1) +(-3pt,-3pt) -- +(0,0) -- +(3pt,-3pt);
\node at (0,-.5) {$\scriptstyle 0$};
\end{scope}
\end{scope}
\end{tikzpicture}\\
\begin{tikzpicture}[scale=.7,>=angle 90,c/.style={insert path={circle[radius=2pt]}},d/.style={shape=diamond,text height=-10pt,draw}]
\def\dist{2.3cm}
\begin{scope}[yshift=-6cm]
\begin{scope}[xshift=3.3*\dist]
\draw (0,0) +(-3pt,3pt) -- +(0,0) -- +(3pt,3pt);
\draw (0,1) +(-3pt,-3pt) -- +(0,0) -- +(3pt,-3pt);
\draw [-] (0,0)--(0,1);\fill (0,.5) circle(2pt) node at (0,-.5) {$\nicefrac{1}{2}$};
\end{scope}
\begin{scope}[xshift=4.3*\dist]
\draw (0,1) +(-3pt,3pt) -- +(0,0) -- +(3pt,3pt);
\draw (0,0) +(-3pt,-3pt) -- +(0,0) -- +(3pt,-3pt);
\draw [-] (0,0)--(0,1);\fill (0,.5) circle(2pt) node at (0,-.5) {$\nicefrac{1}{2}$};
\end{scope}
\end{scope}
\end{tikzpicture}
\end{gathered}
\end{gather}
Then for each $i=1,\dots,d-1$, we insert vertical line segments connecting
all coordinates strictly smaller or strictly larger than $r_i$ and $s_i$ that are labelled $\down$ or $\up$ in
$\lambda^{i}$ and $\lambda^{i+1}$ and connect the remaining coordinates $r_i$ and $s_i$ of $\lambda^{i}$ and $\lambda^{i+1}$ as in the appropriate one from \eqref{fig:Vermapath}.
\end{definition}

Note that in the third row of \eqref{fig:Vermapath} any strand might need to be decorated with a $\bullet$ as in the fourth row. To determine this, replace $\diamondb$ by either $\up$ or $\down$ such that the diagram has an even number of $\down$'s in that row. If the strand is not oriented as in the first or second row, then a decoration $\bullet$ must be added.

\begin{ex}
The first eight pictures in \eqref{exVermapath} display the Verma paths for $d=2$ and $\delta>1$, representing the eight Verma modules appearing in a Verma filtration of $\MdV$. In case $\delta=1$, the first picture has to be replaced by the last one. Note the $\bullet$ in the last Verma path that indicates that the symbol $\diamondb$ at position $0$ in the second row would need to be substituted with a $\up$ to obtain an even number of $\down$'s.
\abovedisplayskip0.45em
\belowdisplayskip0.3em
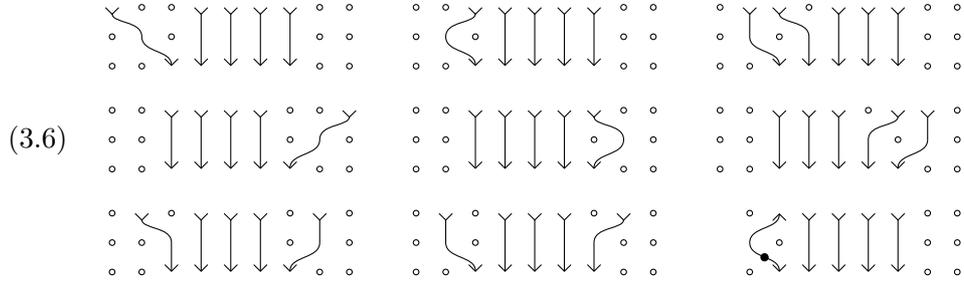
\begin{gather}
\label{exVermapath}
\begin{gathered}
\begin{tikzpicture}[anchorbase,scale=0.39,>=angle 90,c/.style={insert path={circle[radius=2.7pt]}}]
\foreach \x in {0,1,7,8} \draw (\x,0) [c];
\foreach \x in {0,2,7,8} \draw (\x,1) [c];
\foreach \x in {1,2,7,8} \draw (\x,2) [c];
\draw[<-<] (2,0) to [in=-90,out=90] +(-1,1) to [out=90,in=-90] +(-1,1);
\draw[<-<] (3,0) to [in=-90,out=90] +(0,1) to [out=90,in=-90] +(0,1);
\draw[<-<] (4,0) to [in=-90,out=90] +(0,1) to [out=90,in=-90] +(0,1);
\draw[<-<] (5,0) to [in=-90,out=90] +(0,1) to [out=90,in=-90] +(0,1);
\draw[<-<] (6,0) to [in=-90,out=90] +(0,1) to [out=90,in=-90] +(0,1);
\begin{scope}[xshift=10.25cm]
\foreach \x in {0,1,7,8} \draw (\x,0) [c];
\foreach \x in {0,2,7,8} \draw (\x,1) [c];
\foreach \x in {0,1,7,8} \draw (\x,2) [c];
\draw[<-<] (2,0) to [in=-90,out=90] +(-1,1) to [out=90,in=-90] +(1,1);
\draw[<-<] (3,0) to [in=-90,out=90] +(0,1) to [out=90,in=-90] +(0,1);
\draw[<-<] (4,0) to [in=-90,out=90] +(0,1) to [out=90,in=-90] +(0,1);
\draw[<-<] (5,0) to [in=-90,out=90] +(0,1) to [out=90,in=-90] +(0,1);
\draw[<-<] (6,0) to [in=-90,out=90] +(0,1) to [out=90,in=-90] +(0,1);
\end{scope}
\begin{scope}[xshift=20.5cm]
\foreach \x in {0,1,7,8} \draw (\x,0) [c];
\foreach \x in {0,2,7,8} \draw (\x,1) [c];
\foreach \x in {0,3,7,8} \draw (\x,2) [c];
\draw[<-<] (2,0) to [in=-90,out=90] +(-1,1) to [out=90,in=-90] +(0,1);
\draw[<-<] (3,0) to [in=-90,out=90] +(0,1) to [out=90,in=-90] +(-1,1);
\draw[<-<] (4,0) to [in=-90,out=90] +(0,1) to [out=90,in=-90] +(0,1);
\draw[<-<] (5,0) to [in=-90,out=90] +(0,1) to [out=90,in=-90] +(0,1);
\draw[<-<] (6,0) to [in=-90,out=90] +(0,1) to [out=90,in=-90] +(0,1);
\end{scope}
\begin{scope}[yshift=-3.5cm]
\begin{scope}[]
\foreach \x in {0,1,7,8} \draw (\x,0) [c];
\foreach \x in {0,1,6,8} \draw (\x,1) [c];
\foreach \x in {0,1,6,7} \draw (\x,2) [c];
\draw[<-<] (2,0) to [in=-90,out=90] +(0,1) to [out=90,in=-90] +(0,1);
\draw[<-<] (3,0) to [in=-90,out=90] +(0,1) to [out=90,in=-90] +(0,1);
\draw[<-<] (4,0) to [in=-90,out=90] +(0,1) to [out=90,in=-90] +(0,1);
\draw[<-<] (5,0) to [in=-90,out=90] +(0,1) to [out=90,in=-90] +(0,1);
\draw[<-<] (6,0) to [in=-90,out=90] +(1,1) to [out=90,in=-90] +(1,1);
\end{scope}
\begin{scope}[xshift=10.25cm]
\foreach \x in {0,1,7,8} \draw (\x,0) [c];
\foreach \x in {0,1,6,8} \draw (\x,1) [c];
\foreach \x in {0,1,7,8} \draw (\x,2) [c];
\draw[<-<] (2,0) to [in=-90,out=90] +(0,1) to [out=90,in=-90] +(0,1);
\draw[<-<] (3,0) to [in=-90,out=90] +(0,1) to [out=90,in=-90] +(0,1);
\draw[<-<] (4,0) to [in=-90,out=90] +(0,1) to [out=90,in=-90] +(0,1);
\draw[<-<] (5,0) to [in=-90,out=90] +(0,1) to [out=90,in=-90] +(0,1);
\draw[<-<] (6,0) to [in=-90,out=90] +(1,1) to [out=90,in=-90] +(-1,1);
\end{scope}
\begin{scope}[xshift=20.5cm]
\foreach \x in {0,1,7,8} \draw (\x,0) [c];
\foreach \x in {0,1,6,8} \draw (\x,1) [c];
\foreach \x in {0,1,5,8} \draw (\x,2) [c];
\draw[<-<] (2,0) to [in=-90,out=90] +(0,1) to [out=90,in=-90] +(0,1);
\draw[<-<] (3,0) to [in=-90,out=90] +(0,1) to [out=90,in=-90] +(0,1);
\draw[<-<] (4,0) to [in=-90,out=90] +(0,1) to [out=90,in=-90] +(0,1);
\draw[<-<] (5,0) to [in=-90,out=90] +(0,1) to [out=90,in=-90] +(1,1);
\draw[<-<] (6,0) to [in=-90,out=90] +(1,1) to [out=90,in=-90] +(0,1);
\end{scope}
\end{scope}
\begin{scope}[yshift=-7cm]
\begin{scope}[]
\foreach \x in {0,1,7,8} \draw (\x,0) [c];
\foreach \x in {0,1,6,8} \draw (\x,1) [c];
\foreach \x in {0,2,6,8} \draw (\x,2) [c];
\draw[<-<] (2,0) to [in=-90,out=90] +(0,1) to [out=90,in=-90] +(-1,1);
\draw[<-<] (3,0) to [in=-90,out=90] +(0,1) to [out=90,in=-90] +(0,1);
\draw[<-<] (4,0) to [in=-90,out=90] +(0,1) to [out=90,in=-90] +(0,1);
\draw[<-<] (5,0) to [in=-90,out=90] +(0,1) to [out=90,in=-90] +(0,1);
\draw[<-<] (6,0) to [in=-90,out=90] +(1,1) to [out=90,in=-90] +(0,1);
\end{scope}
\begin{scope}[xshift=10.25cm]
\foreach \x in {0,1,7,8} \draw (\x,0) [c];
\foreach \x in {0,2,7,8} \draw (\x,1) [c];
\foreach \x in {0,2,6,8} \draw (\x,2) [c];
\draw[<-<] (2,0) to [in=-90,out=90] +(-1,1) to [out=90,in=-90] +(0,1);
\draw[<-<] (3,0) to [in=-90,out=90] +(0,1) to [out=90,in=-90] +(0,1);
\draw[<-<] (4,0) to [in=-90,out=90] +(0,1) to [out=90,in=-90] +(0,1);
\draw[<-<] (5,0) to [in=-90,out=90] +(0,1) to [out=90,in=-90] +(0,1);
\draw[<-<] (6,0) to [in=-90,out=90] +(0,1) to [out=90,in=-90] +(1,1);
\end{scope}
\begin{scope}[xshift=20.5cm]
%\draw (1,-1) -- (1,3);
\foreach \x in {1,7,8} \draw (\x,0) [c];
\foreach \x in {2,7,8} \draw (\x,1) [c];
\foreach \x in {1,7,8} \draw (\x,2) [c];
\draw[<->] (2,0) to [in=-90,out=90] +(-1,1) to [out=90,in=-90] +(1,1);
\draw[<-<] (3,0) to [in=-90,out=90] +(0,1) to [out=90,in=-90] +(0,1);
\draw[<-<] (4,0) to [in=-90,out=90] +(0,1) to [out=90,in=-90] +(0,1);
\draw[<-<] (5,0) to [in=-90,out=90] +(0,1) to [out=90,in=-90] +(0,1);
\draw[<-<] (6,0) to [in=-90,out=90] +(0,1) to [out=90,in=-90] +(0,1);
\fill (1.5,0.5) circle(4pt);
\end{scope}
\end{scope}
\end{tikzpicture}
\end{gathered}
\end{gather}
\end{ex}

\begin{definition} 
\label{def:associated bipartition}
To a weight $\la$ appearing in a Verma path of length $d\leq n$ we assign a bipartition $\varphi(\la)=(\la^{(1)},\la^{(2)})$ with $\la^{(1)}_i=-m_i$ and $\la^{(2)}_i=m_{s+i}$, where
$s$ is such that $\la-\de=\sum_{i=1}^n m_i\epsilon_i$ with $m_i\leq 0$ for $1\leq i\leq s$ and $m_i>0$ for $i>s$.  
\end{definition}
It is convenient to draw here the bipartition $\varphi(\la)\in\cT^2(d)$, $d\leq n$, by first transposing both diagrams, then rotating the second by $180°$ and then arranging them in the plane such that their origins are positioned at the coordinates $(0,0)$ and $(0,n)$ respectively. The $i$-th column contains then $m_i$ boxes arranged below the $x$-axis if $m_i<0$ and above if $m_i>0$. Let $(\la^t,{}^t\mu)$ denote the resulting diagram. For example, the bipartition $(3,2,1,1), (2,2,1)$ is displayed as follows:
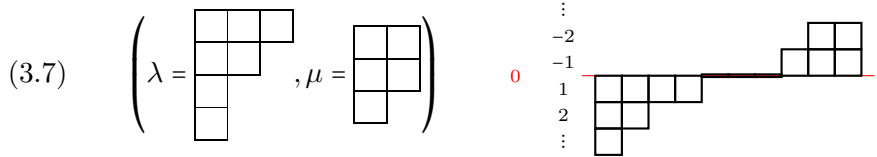
\abovedisplayskip0.3em
\belowdisplayskip0.3em
\begin{gather}
\label{displaybipartition}
\left(\la=\yng(3,2,1,1), \mu=\yng(2,2,1)\right)
\qquad
\begin{tikzpicture}[anchorbase,scale=0.7, thick]
\draw[thin,red] (-.25,0)  -- +(5.5,0);
\node at (-.6,1.25) {\tiny $\vdots$};
\node at (-.6,.75) {\tiny $-2$};
\node at (-.6,.25) {\tiny $-1$};
\node at (-.6,-.25) {\tiny $1$};
\node at (-.6,-.75) {\tiny $2$};
\node at (-.6,-1.25) {\tiny $\vdots$};
\node[red] at (-1.5,0) {\tiny $0$};
\draw rectangle (.5,-.5);
\draw (.5,0) rectangle +(.5,-.5);
\draw (0,-.5) rectangle +(.5,-.5);
\draw (0,-1) rectangle +(.5,-.5);
\draw (.5,-.5) rectangle +(.5,-.5);
\draw (1,0) rectangle +(.5,-.5);
\draw (1.5,0) rectangle +(.5,-.5);
\draw (2,-.025) rectangle +(.5,.05);
\draw (2.5,-.025) rectangle +(.5,.05);
\draw (3,-.025) rectangle +(.5,.05);
\draw (3.5,0) rectangle +(.5,.5);
\draw (4,0) rectangle +(.5,.5);
\draw (4,.5) rectangle +(.5,.5);
\draw (4.5,0) rectangle +(.5,.5);
\draw (4.5,.5) rectangle +(.5,.5);
\end{tikzpicture}
\end{gather}

\begin{definition}
We label columns from left to right by $0,1, 2,\dots$ and the rows increasing from top to bottom such that the ${\la}^t$ has rows $1, 2,\ldots$ and ${}^t{\mu}$  has rows $-1, -2, \cdots$ (with a "ground state line" counted $0$, horizontally between the two diagrams). The \emph{content} $c(b)$ of a box $b$  is its column number minus its row number, $c(b)=\op{col}(b)-\op{row}(b)$ and the content with charge $\nicefrac{\delta}{2}$ is defined as $c_\delta(b)=\nicefrac{\delta}{2}+c(b)$. 
 \end{definition}
For \eqref{displaybipartition} we have the contents:
{\abovedisplayskip-1em
\belowdisplayskip0.3em
\begin{gather}
\label{coolpicture}
\begin{tikzpicture}[anchorbase,scale=1.9, thick]
\colorlet{lightgray}{black!15}
\draw[thin,red] (-.25,0)  -- +(5.5,0);
\draw rectangle (.5,-.5);
\draw (.5,0) rectangle +(.5,-.5);
\draw (0,-.5) rectangle +(.5,-.5);
\fill[fill=lightgray, draw=black] (0,-1) rectangle +(.5,-.5);
\fill[fill=lightgray, draw=black] (.5,-.5) rectangle +(.5,-.5);
\fill[fill=lightgray, draw=black] (1,0) rectangle +(.5,-.5);
\fill[fill=lightgray, draw=black] (1.5,0) rectangle +(.5,-.5);
\node at (.25,-.25) {$\scriptstyle \frac{\delta}{2}\scriptscriptstyle-1$};
\node at (.25,-.75) {$\scriptstyle \frac{\delta}{2}\scriptscriptstyle-2$};
\node at (.25,-1.25) {$\scriptstyle \frac{\delta}{2}\scriptscriptstyle-3$};
\node at (.75,-.25) {$\scriptstyle \frac{\delta}{2}$};
\node at (.75,-.75) {$\scriptstyle \frac{\delta}{2}\scriptscriptstyle-1$};
\node at (1.25,-.25) {$\scriptstyle \frac{\delta}{2}\scriptscriptstyle+1$};
\node at (1.75,-.25) {$\scriptstyle \frac{\delta}{2}\scriptscriptstyle+2$};
\draw (2,-.025) rectangle +(.5,.05);
\draw (2.5,-.025) rectangle +(.5,.05);
\draw (3,-.025) rectangle +(.5,.05);
\fill[fill=lightgray, draw=black] (3.5,0) rectangle +(.5,.5);
\draw (4,0) rectangle +(.5,.5);
\fill[fill=lightgray, draw=black] (4,.5) rectangle +(.5,.5);
\draw (4.5,0) rectangle +(.5,.5);
\fill[fill=lightgray, draw=black] (4.5,.5) rectangle +(.5,.5);
\node at (3.75,.25) {$\scriptstyle \frac{\delta}{2}\scriptscriptstyle+n-2$};
\node at (4.25,.75) {$\scriptstyle \frac{\delta}{2}\scriptscriptstyle+n$};
\node at (4.25,.25) {$\scriptstyle \frac{\delta}{2}\scriptscriptstyle+n-1$};
\node at (4.75,.75) {$\scriptstyle \frac{\delta}{2}\scriptscriptstyle+n+1$};
\node at (4.75,.25) {$\scriptstyle \frac{\delta}{2}\scriptscriptstyle+n$};
\end{tikzpicture}
\end{gather}
}

Then we can read off the corresponding diagrammatic weight $\la\in\La$ by setting
 $I(\la)=\{\op{cont}_\delta(\la,\mu)\}=\{c_\delta(b_i)\mid {1\leq i\leq n}\}$, where the $b_i$ are the boundary boxes of the columns. In the above example (with the relevant boxes shaded grey)  we obtain $I(\la)=\frac{\delta}{2}+\{-3,-1,1,2,\ldots,n-4, n-2,n,n+1\}$. In particular, $\de$ corresponds to the bipartition $(\emptyset,\emptyset)$. In this way, any $\mathbf{t}\in\mathcal{T}_d$ defines a sequence $\op{cont}(\mathbf{t})$ of diagrammatic weights.

\begin{prop}
\label{surjectivity}
Let $n, \delta\in\mathbb{Z}_{\geq 0}$. Assume $n\geq \delta$. Then the assignment ${\bf t}\mapsto \op{cont}(\bf{t})$ defines a bijection, with inverse map $\la\mapsto\varphi(\la)$,
\abovedisplayskip0.3em
\belowdisplayskip0.3em
\begin{gather*}
\begin{array}{cl}
\mathcal{T}_d & =\{\text{up-down bitableaux of length $d$}\}\\
{\updownarrow} & \\
\vpath_d(\delta) & =\{\text{Verma paths of length $d$ starting at $\delta$}\}
\end{array}
\end{gather*}
inducing also a bijection between $\bigcup_{0\leq k\leq \lfloor \frac{d}{2}\rfloor}\cP^2_{d-2k}$ and the set of weights appearing at the end of a path in $\vpath_d(\delta)$.
\end{prop}

\begin{proof}
The empty bipartition $(\emptyset, \emptyset)$ corresponds to the weight $\de+\rho$.  Adding (resp. removing) a box in the $j$-th column of the first partition corresponds to subtracting $\epsilon_j$ from (resp. adding $\epsilon_j$ to) the weight. On the other hand,  adding (resp. removing) a box in the $(n-j)$-th column of the second partition corresponds to adding $\epsilon_j$ to (resp. subtracting $\epsilon_{n-j}$ from) the weight. The result is then an allowed weight if and only if adding (resp. removing) the box produces a new bipartition. Then the claim follows from the definitions.
\end{proof}

For $M^\pp(\la)\in\Op(n)$ we denote by $L(\la)$ its irreducible quotient and by $P^\pp(\la)\in\Op(n)$ its projective cover. For a module $M$ in category $\Op(n)$ with Verma flag (i.e with a filtration with subquotients isomorphic to parabolic Verma modules) we denote by $\nu_M=(M : M^\pp(\nu))$ the multiplicity of $M^\pp(\nu)$ in such a flag. By BGG reciprocity we have $(P^\pp(\mu) : M^\pp(\nu))=[M^\pp(\nu):L(\mu)]$ where the latter denotes the Jordan-H\"older multiplicity of $L(\mu)$ in $M^\pp(\nu)$.

\begin{prop}
\label{dimformula}
The following holds in $\Op(n)$:
\begin{enumerate}
\item $\op{dim}\HOM_\mg(P^\pp(\la),P^\pp(\mu))=\sum_{\nu\in\La}\nu_{P^\pp(\la)}\nu_{P^\pp(\mu)}$, the number of Verma modules which appear at the same time in $P^\pp(\la)$ and in $P^\pp(\mu)$.
\item  If $P$ is a projective object then $\op{dim}\END_\mg(P)=\sum_{\nu\in\La} (\nu_P)^2.$
\end{enumerate}
\end{prop}
\begin{proof}
Let $\La_{\lambda,\mu}=\left\{\nu\in\La\mid \nu_{P^\pp(\la)}\not=0,\, \nu_{P^\pp(\mu)}\not=0\right\}$. Then 
\abovedisplayskip0.3em
\belowdisplayskip0.3em
\begin{gather*}
\begin{aligned}
\op{dim}\HOM_\mg(P^\pp(\la),P^\pp(\mu))=&\,[P^\pp(\mu) : L(\la)]\\
=\sum_{\nu \in \La}(P^\pp(\mu) : M^\pp(\nu))[M^\pp(\nu):L(\la)] 
=&\sum_{\nu\in\La}\nu_{P^\pp(\mu)}\nu_{P^\pp(\la)}=\hspace{-6pt}\sum_{\nu\in \La_{\lambda,\mu}}1\cdot1,
\end{aligned}
\end{gather*}
where for the last equality we used that the multiplicities are at most $1$. This follows for regular weights for instance from \cite[Theorem 2.1]{LS} and for singular weights from Lemma \ref{singO} below. Part (1) follows.
Recall that $P$ decomposes into a direct sum of indecomposable projective modules $P=\bigoplus_\mu a_\mu P^\pp(\mu)$, where $a_\mu$ is the multiplicity of $P^\pp(\mu)$. Then
\abovedisplayskip0.3em
\belowdisplayskip0.3em
%\begin{gather*}
\begin{eqnarray*}
\op{dim}\END_\mg(P)
&=&\hspace{-6pt}\sum_{\la,\mu\in\La}a_\la a_\mu \op{dim}\HOM_\mg(P^\pp(\la),P^\pp(\mu))\\
&=&\hspace{-6pt}\sum_{\la,\mu\in\La}a_\la a_\mu \sum_{\nu\in\La}\nu_{P^\pp(\mu)}\nu_{P^\pp(\la)}
=\hspace{-4pt}\sum_{\la,\mu\in\La}a_\la a_\mu \sum_{\nu\in\La}\nu_{P^\pp(\mu)}\nu_{P^\pp(\la)}.
%=\sum_{\nu\in\La}\left(\sum_{\la,\mu\in\La}a_\la a_\mu \nu_{P^\pp(\mu)}\nu_{P^\pp(\la)}\right)\\
%=&\sum_{\nu\in\La}\left(\sum_{\la\in\La}a_\la \nu_{P^\pp(\mu)}\sum_{\mu\in\La} a_\mu\nu_{P^\pp(\la)}\right)
%=\sum_{\nu\in\La}\nu_{P}\nu_P=\sum_\nu (\nu_P)^2.
\end{eqnarray*}
The proposition follows.
\end{proof}

The main result of this section is the following:
\begin{corollary}
\label{tableaux}
We have $\op{dim}\END_\mg(\MdV)^{\op{opp}}=2^d(2d-1)!!$. Moreover, $\END_\mg(\MdV)^{\op{opp}}\cong\END_\mg({M^\pp({-\underline{\delta}})\otimes V^{\otimes d}})$ as algebras.
\end{corollary}

\begin{proof}
Let first $\delta\geq0$. Since then $\Mde$ is projective in $\Op(n)$, so is $\MdV$ and therefore it has a Verma filtration. By Proposition~\ref{Verma} and the paragraph afterwards $\nu_{\MdV}$ equals the number of elements in $\vpath_d(\delta)$ ending at $\nu$. Then the statement follows from  Propositions~\ref{surjectivity} and~\ref{dimformula} and the equality \eqref{cardinality}. If $\delta<0$ then $M^\pp({-\underline{\delta}})$ is a tilting module and  the algebra isomorphism follows from Ringel self-duality \cite[Proposition 4.4]{MSSerre} of $\Op(n)$. 
\end{proof}

We want to stress that the reader familiar with branching graphs or cellular bases should view our Verma paths as a generalization of tableaux adapted to the case where the algebras are not semisimple.  

\section{Special projective functors and cup diagram combinatorics}
In analogy to \cite{BSIII} we introduce special translation functors which will later categorify the actions of the coideal subalgebras $\cH$ and $\cH^\hint$ in Section \ref{section:coideal}.

For this let $\Ga$ be a diagrammatic block, i.e. given by a block sequence and a parity for the number of $\down$'s and $\times$'s. Furthermore fix $i \in \mZ_{\geq 0}+\half$ if $\Ga$ is supported on the integers and $i \in \mZ_{\geq 0}$ if $\Ga$ is supported on half-integers.

\textbf{Case $i \geq 1$:} We  denote by $\Ga_{i,+}$ and $\Ga_{i,-}$ the blocks, whose corresponding block diagram differs from $\Ga$ diagrammatically only at the vertices $i-\half$ and $i+\half$ as displayed in the following list (note that it has the same type of integrality and the same parity):
\abovedisplayskip0.3em
\belowdisplayskip0.3em
\begin{gather}
\label{deftable1}
\begin{array}{c||c|c|c|c|c|c|c|c|}
\Gamma_{i,+}&\cross\circ&&\bullet\bullet&&&\bullet\circ&
\cross\bullet\\
\Gamma_{i,-}&\circ\cross&\bullet\bullet&&
\circ\bullet&\bullet\cross&&\\
\hline
\Gamma&\bullet\bullet&\cross\circ&\circ\cross&\bullet\circ&\cross\bullet&
\circ\bullet&\bullet\cross
\end{array}
\end{gather}

\textbf{Case $i = \half$:}
We  denote by $\Ga_{\half,+}$ and $\Ga_{\half,-}$ the blocks or in one instance the sum of two blocks, whose corresponding block diagram differs from $\Ga$ diagrammatically only at the vertices $0$ and $1$ as displayed in the following list (note that it has the same type of integrality):
\abovedisplayskip0.3em
\belowdisplayskip0.3em
\begin{gather}
\label{deftable2}
\begin{array}{c||c|c|c|c|}
\Gamma_{\half,+}& & & \diamondb \bullet & \diamondb \circ\\
\Gamma_{\half,-}&\circ\times&\circ\bullet + \overline{\circ \bullet}&&\\
\hline
\Gamma&\diamondb\bullet & \diamondb \circ & \circ \times & \circ\bullet
\end{array}
\end{gather}
Here the entry $\circ\bullet + \overline{\circ \bullet}$ means that we take the direct sum of the two blocks with the same underlying block diagram but the two possible choices of parity.

\textbf{Case $i = 0$:} Let $\Ga_0$ be given by the same block diagram as $\Ga$ but with the opposite parity. 

\subsection{Special projective functors}
\begin{definition} \label{def:functors}
Denote the projections onto $\Op_\Ga(n)$ by
\abovedisplayskip0.3em
\belowdisplayskip0.3em
\[
\pr^1_\Ga:\Op_1(n) \longrightarrow \Op_\Ga(n) \qquad \pr^\hint_\Ga:\Op_\hint(n) \longrightarrow \Op_\Ga(n).
\]
We define \emph{special projective functors} for $i \in \mZ_{\geq 0} + \half$ by
\abovedisplayskip0.3em
\belowdisplayskip0.3em
\begin{gather*}
\begin{aligned}
\cF_{i,-} &=& {\bigoplus}_{\Ga} \left(
\pr^1_{\Ga_{i,-}} \circ (? \otimes V) \circ \pr^1_\Ga \right) \,&:& \Op_1(n) \,\longrightarrow\, \Op_1(n),\\
\cF_{i,+} &=& {\bigoplus}_{\Ga} \left(
\pr^1_{\Ga_{i,+}} \circ (? \otimes V) \circ \pr^1_\Ga \right) \,&:& \Op_1(n) \,\longrightarrow\, \Op_1(n),
\end{aligned}
\end{gather*}
and for $i \in \mZ_{\geq 1}$ by
\abovedisplayskip0.3em
\belowdisplayskip0.3em
\begin{gather*}
\begin{aligned}
\cF_{i,-} &=& {\bigoplus}_{\Ga} \left(
\pr^\hint_{\Ga_{i,-}} \circ (? \otimes V) \circ \pr^\hint_\Ga \right) \,&:& \Op_\hint(n) \,\longrightarrow\, \Op_\hint(n),\\
\cF_{i,+} &=& {\bigoplus}_{\Ga} \left(
\pr^\hint_{\Ga_{i,+}} \circ (? \otimes V) \circ \pr^\hint_\Ga \right) \,&:& \Op_\hint(n) \,\longrightarrow\, \Op_\hint(n),\\
\cF_{0} &=& {\bigoplus}_{\Ga} \left(
\pr^\hint_{\Ga_{0}} \circ (? \otimes V) \circ \pr^\hint_\Ga \right) \,&:& \Op_\hint(n) \,\longrightarrow\, \Op_\hint(n).
\end{aligned}
\end{gather*}
In each case the direct sums are over all blocks $\Ga$, where $\Ga_{i,-}$, resp. $\Ga_{i,+}$, resp. $\Ga_0$ are defined.
\end{definition}

Note that $\cF_{i,-}$ and $\cF_{i,+}$ are biadjoint and $\cF_{0}$ is selfadjoint, since $V$ is self-dual. Hence they are exact and send projective modules to projectives modules. The following Proposition (see Section \ref{section:proof-of-a} for the proof), describes the acton of special projective functors on indecomposable projective, Verma and simple modules. In Section \ref{section:proj-diagrams} we give the diagrammatic interpretation. 

The symbols $\langle i\rangle$ in Proposition \ref{a} refer to a grading shift which should, for now, be ignored. It only makes sense in the graded setup of Proposition~\ref{lem:gradeda}. 

\begin{prop}\label{a}
There are isomorphisms of functors
\abovedisplayskip0.3em
\belowdisplayskip0.3em
\begin{gather*}
\begin{aligned}
\cF = (? \otimes V) &\cong& \bigoplus_{i\,\in\, \mZ_{\geq 0}+\half} \hspace{-.8em}(\cF_{i,-}\oplus \cF_{i,+})&:\Op_1(n)\longrightarrow\Op_1(n), &\text{ and}& \\
\cF^\hint = (? \otimes V) &\cong& \bigoplus_{i\,\in\, \mZ_{> 0}}(\cF_{i,-}\oplus \cF_{i,+}) \oplus \cF_0 &: \Op_\hint(n)\longrightarrow\Op_\hint(n).&&
\end{aligned}
\end{gather*}

Let $\la \in \mX_n$ and $i \in \mZ_{\geq 1}+\half$ or $\la \in \mX_n^\hint$ and $i \in \mZ_{\geq 1}$. For $x,y \in \{\circ,\up,\down,\cross\}$, let $\la_{xy}$ denote the diagrammatic weight obtained by relabelling $\lambda_{(i-\half)}$ to $x$ and $\lambda_{(i+\half)}$ to $y$. Then we have
%%%%%%%%%%%%%%%%%%%%%%%%%%%%%%%%%%%%
\begin{enumerate}[(a)]
\item\label{HIM1}
$\la = \la_{\scriptscriptstyle \down\circ}:$ 
$\cF_{i,-} P^\pp(\la) \cong P^\pp(\la_{\scriptscriptstyle  \circ \down}), \,\,\,
\cF_{i,-} M^\pp(\la) \cong M^\pp(\la_{\scriptscriptstyle  \circ\down}), \,\,\,
\cF_{i,-} L(\la) \cong L(\la_{\scriptscriptstyle \circ\down})$.
\item
\label{HIM2}
$\la = \la_{\scriptscriptstyle \up\circ}:$
$\cF_{i,-} P^\pp(\la) \cong P^\pp(\la_{\scriptscriptstyle \circ\up}), \,\,\,
\cF_{i,-} M^\pp(\la) \cong M^\pp(\la_{\scriptscriptstyle \circ\up}), \,\,\,
\cF_{i,-} L(\la) \cong L(\la_{\scriptscriptstyle \circ\up})$.
\item
\label{HIM3}
$\la = \la_{\scriptscriptstyle\cross\down}:$
$\cF_{i,-} P^\pp(\la) \cong P^\pp(\la_{\scriptscriptstyle\down\cross}), \,\,
\cF_{i,-} M^\pp(\la) \cong M^\pp(\la_{\scriptscriptstyle\down\cross}), \,\,
\cF_{i,-} L(\la) \cong L(\la_{\scriptscriptstyle\down\cross})$.
\item
\label{HIM4}
$\la = \la_{\scriptscriptstyle\cross\up}:$
$\cF_{i,-} P^\pp(\la) \cong P^\pp(\la_{\scriptscriptstyle\up\cross}), \,\,
\cF_{i,-} M^\pp(\la) \cong M^\pp(\la_{\scriptscriptstyle\up\cross}), \,\,
\cF_{i,-} L(\la) \cong L(\la_{\scriptscriptstyle\up\cross})$.
\item
\label{HIM5}
$\la = \la_{\scriptscriptstyle\down\up}:$
$\cF_{i,-} P^\pp(\la) \cong
P^\pp(\la_{\scriptscriptstyle\circ\cross})\oplus P^\pp(\la_{\scriptscriptstyle\circ\cross})\langle 2\rangle, \,\,\,
\cF_{i,-} M^\pp(\la) \cong M^\pp(\la_{\scriptscriptstyle\circ\cross}),$\\
$\phantom{\la = \la_{\scriptscriptstyle\down\up}:\,\,}\cF_{i,-} L(\la) \cong L(\la_{\scriptscriptstyle\circ\cross})$.

\item
\label{HIM6}
$\la = \la_{{\scriptscriptstyle\up\down}}:$
$\cF_{i,-} M^\pp(\la) \cong M^\pp(\la_{\scriptscriptstyle\circ\cross})\left\langle 1 \right\rangle$
and $\cF_{i,-} L(\la) = \{0\}$.

\item
\label{HIM7}
$\la = \la_{{\scriptscriptstyle\cross}\circ}:$
\begin{enumerate}[(1)]
\item\label{HIM7a}
$\cF_{i,-} P^\pp(\la) \cong P^\pp(\la_{\scriptscriptstyle\down\up})\langle -1\rangle$;
\item\label{HIM7b}
there is a short exact sequence
\abovedisplayskip0.2em
\belowdisplayskip0.2em
\[
0 \rightarrow M^\pp(\la_{\scriptscriptstyle\up\down}) \rightarrow \cF_{i,-} M^\pp(\la)
\rightarrow M^\pp(\la_{\scriptscriptstyle\down\up})\left\langle-1\right\rangle \rightarrow 0;
\]
\item\label{HIM7c}
$[\cF_{i,-} L(\la):L(\la_{\scriptscriptstyle\down\up})\langle \pm1\rangle] = 1$ and 
all other composition
factors are of the form $L(\mu)$ for $\mu$
with
$\mu = \mu_{\scriptscriptstyle\down\down}$,
$\mu = \mu_{\scriptscriptstyle\up\up}$ or
$\mu = \mu_{\scriptscriptstyle\up\down}$;
\item\label{HIM7d}
$\cF_{i,-} L(\la)$ has irreducible
socle and head isomorphic to $L(\la_{\scriptscriptstyle\down\up})$.
\end{enumerate}
\item
\label{HIM8}
$\la = \la_{{\scriptscriptstyle\down\down}}:$ 
$\cF_{i,-} M^\pp(\la) = \cF_{i,-} L(\la) = \{0\}$.
\item
\label{HIM9}
$\la = \la_{{\scriptscriptstyle\up\up}}:$
$\cF_{i,-} M^\pp(\la) = \cF_{i,-} L(\la) = \{0\}$.
\item
\label{HIM10}
For the statements with $\cF_{i,+}$,
interchange all occurrences of $\circ$ and $\cross$.
\end{enumerate}
%%%%%%%%%%%%%%%%%%%%%%%%%%%%%%%%%%%%%%%%%%%%%%%%%%%%%%%

%%%%%%%%%%%%%%%%%%%%%%%%%%%%%%%%%%%%%%%%%%%%%%%%%%%%%%%%

Let $\la \in \mX_n$. For $x,y \in \{\circ,\up,\down,\cross,\diamondb\}$ let $\la_{xy}$ denote the diagrammatic weight obtained by relabelling $\lambda_0$ to $x$ and $\lambda_1$ to $y$. Then we have

\begin{enumerate}[(a)]
\setcounter{enumi}{10}
\item
\label{HIM11}
$\la = \la_{\scriptscriptstyle\diamondb\down}:$ 
$\cF_{{\scriptscriptstyle \half},-} P^\pp(\la) \cong P^\pp(\la_{\scriptscriptstyle\circ\cross})\left\langle 1 \right\rangle, \,\,\,
\cF_{{\scriptscriptstyle \half},-} M^\pp(\la) \cong
M^\pp(\la_{\scriptscriptstyle\circ\cross}) \left\langle 1 \right\rangle,$\\
$\phantom{\la = \la_{\scriptscriptstyle\diamondb\down}:\,\,} \cF_{{\scriptscriptstyle \half},-} L(\la) \cong L(\la_{\scriptscriptstyle\circ\cross})\left\langle 1 \right\rangle$.
\item\label{HIM12}
$\la = \la_{\scriptscriptstyle\diamondb\up}:$ 
$\cF_{{\scriptscriptstyle \half},-} P^\pp(\la) \cong P^\pp(\la_{\scriptscriptstyle\circ\cross}), \,\,\,
\cF_{{\scriptscriptstyle \half},-} M^\pp(\la) \cong
M^\pp(\la_{\scriptscriptstyle\circ\cross})$,\\
$\phantom{\la = \la_{\scriptscriptstyle\diamondb\up}:\,\,}\cF_{{\scriptscriptstyle \half},-} L(\la) \cong L(\la_{\scriptscriptstyle\circ\cross})$.
\item\label{HIM13}
$\la = \la_{\scriptscriptstyle\diamondb\circ}:$ 
$\cF_{{\scriptscriptstyle \half},-} P^\pp(\la) \cong P^\pp(\la_{\scriptscriptstyle\circ\up})\oplus P^\pp(\la_{\scriptscriptstyle\circ\down})$,\\
$\phantom{\la = \la_{\scriptscriptstyle\diamondb\circ}:\,\,}\cF_{{\scriptscriptstyle \half},-} M^\pp(\la) \cong
M^\pp(\la_{\scriptscriptstyle\circ\up})\oplus M^\pp(\la_{\scriptscriptstyle\circ\down})$,\\
$\phantom{\la = \la_{\scriptscriptstyle\diamondb\circ}:\,\,}\cF_{{\scriptscriptstyle \half},-} L(\la) \cong L(\la_{\circ\up})\oplus L(\la_{\circ\down})$.
\item
\label{HIM14}
$\la = \la_{\scriptscriptstyle \circ\down}:$ 
$\cF_{{\scriptscriptstyle \half},+} P^\pp(\la) \cong P^\pp(\la_{\scriptscriptstyle \diamondb\circ}), \,\,
\cF_{{\scriptscriptstyle \half},+} M^\pp(\la) \cong
M^\pp(\la_{\scriptscriptstyle \diamondb\circ}), \,\,
\cF_{{\scriptscriptstyle \half},+} L(\la)=\{0\}$.
\item
\label{HIM15}
$\la = \la_{\scriptscriptstyle \circ\up}:$ 
$\cF_{{\scriptscriptstyle \half},+} P^\pp(\la) \cong P^\pp(\la_{\scriptscriptstyle \diamondb\circ}), \,\,
\cF_{{\scriptscriptstyle \half},+} M^\pp(\la) \cong
M^\pp(\la_{\scriptscriptstyle \diamondb\circ}), \,\,
\cF_{{\scriptscriptstyle \half},+} L(\la)=\{0\}$.
\item
\label{HIM16}
$\la = \la_{\scriptscriptstyle \circ\cross}:$ 
\begin{enumerate}[(1)]
\item\label{HIM16a}
$\cF_{{\scriptscriptstyle \half},+} P^\pp(\la) \cong P^\pp(\la_{\scriptscriptstyle \diamondb\up})\left\langle -1 \right\rangle$;
\item\label{HIM16b}
there is a short exact sequence
\abovedisplayskip0.2em
\belowdisplayskip0.2em
\begin{gather*}
\quad\quad 0 \rightarrow M^\pp(\la_{\scriptscriptstyle\diamondb\down}) \rightarrow \cF_{{\scriptscriptstyle \half},+} M^\pp(\la)
\rightarrow M^\pp(\la_{\scriptscriptstyle\diamondb\up})\left\langle-1\right\rangle \rightarrow 0;
\end{gather*}
\item\label{HIM16c}
$[\cF_{{\scriptscriptstyle \half},+} L(\la) : L(\lambda_{\scriptscriptstyle \diamondb\up})\left\langle \pm1 \right\rangle] = 1$
and all other composition factors are of the form $L(\mu)$ for $\mu$
with
$\mu = \mu_{\scriptscriptstyle\diamondb\down}$,
\end{enumerate}
\item
\label{HIM17}
$\la = \la_{\scriptscriptstyle \diamondb\cross}:$ 
$\cF_{{\scriptscriptstyle \half},-} P^\pp(\la)\,=\, \cF_{{\scriptscriptstyle \half},-} M^\pp(\la)\,=\,
\cF_{{\scriptscriptstyle \half},-} L(\la) =\{0\}$.

\item
\label{HIM18}
$\la = \la_{ \circ\circ}:$ 
$\cF_{{\scriptscriptstyle \half},-} P^\pp(\la)\,=\,
\cF_{{\scriptscriptstyle \half},-} M^\pp(\la)\,=\,
\cF_{{\scriptscriptstyle \half},-} L(\la) =\{0\}$.
\end{enumerate}

Let $\la \in \mX_n^\hint$. For symbols $y \in \{\circ,\up,\down,\cross\}$ let $\la_{y}$ denote the diagrammatic weight obtained by relabelling $\lambda_\half$ to $y$. Then we have
\begin{enumerate}[(a)]
\setcounter{enumi}{18}
\item\label{HIM19}
$\la = \la_{\scriptscriptstyle \up}:$
$\cF_{0} P^\pp(\la) \cong P^\pp(\la_{\scriptscriptstyle \down}), \,\,\,
\cF_{0} M^\pp(\la) \cong
M^\pp(\la_{\scriptscriptstyle \down}), \,\,\,
\cF_{0} L(\la) \cong L(\la_{\scriptscriptstyle \down})$.
\item\label{HIM20}
$\la = \la_{\scriptscriptstyle \down}:$
$\cF_{0} P^\pp(\la) \cong P^\pp(\la_{\scriptscriptstyle \up}), \,\,\,
\cF_{0} M^\pp(\la) \cong
M^\pp(\la_{\scriptscriptstyle \up}), \,\,\,
\cF_{0} L(\la) \cong L(\la_{\scriptscriptstyle \up})$.
\end{enumerate}

%\begin{enumerate}[(xiv)]
%\item
%If $\la = \la_{\up}$ then
%$\cF_{0} P(\la) \cong P(\la_\down)$,
%$\cF_{0} M^\pp(\la) \cong
%M^\pp(\la_{\down})$,
%$\cF_{0} L(\la) \cong L(\la_\down)$.
%\item[\rm(xiv)]
%If $\la = \la_{\down}$ then
%$\cF_{0} P(\la) \cong P(\la_\up)$,
%$\cF_{0} M^\pp(\la) \cong
%M^\pp(\la_{\up})$,
%$\cF_{0} L(\la) \cong L(\la_\up)$.
%\end{enumerate}
\end{prop}

\begin{remark} \label{rem:transl}
Note that each of the $\cF_{i,+}$, $\cF_{i,-}$, $F_{0}$ is a translation functor in the sense of \cite{Hbook}. It is either an equivalence of categories, a translation functor to the wall, a translation functor out of the wall or, in the special case of $\cF_{\half,-}$ and the choice of a specific block diagram a direct sum of two translation functors out of the wall. In contrast to the type $ \mathrm{A}$ situation from \cite{BSIII} not all translations to walls appear however as summands.
\end{remark}

\subsection{Projective modules and cup diagrams} \label{section:proj-diagrams}
In this section we indicate how to make explicit calculations related to projective modules using decorated cup diagrams following \cite{LS}. We summarise how to associate a {\it (decorated) cup diagram} $\underline \la$ to a diagrammatic weight $\la$, for details see \cite[Definition 3.5]{ES_diagrams}.

Let $\lambda_i$ for $i \in \mathbb{Z} \cup (\mathbb{Z}+\half)$ be the symbol at position $i$ on the real positive line. If $\lambda_0 = \diamondb$ treat this as either $\up$ if the number of $\down$ is even and as $\down$ otherwise. Successively connect neighboured symbols on the line labelled $\down\up$ by arcs (ignoring already joint symbols and symbols $\circ$ or $\times$). If none of these are left, successively connect from left to right pairs of two neighboured $\up$'s by an arc decorated with a $\bullet$. If a single $\up$ remains, attach a vertical ray decorated with a $\bullet$ to it and finally attach a vertical ray to all remaining $\down$. Finally forget all symbols except $\circ$ and $\times$. The result is denoted $\underline{\lambda}$. Drawing a weight $\mu$ on the real line above $\underline{\lambda}$ we say that the result is oriented if the labels $\circ$ and $\cross$ in $\mu$ and $\underline{\la}$ match up, undecorated rays are labelled with $\down$, decorated rays with $\up$, undecorated cups with two distinct symbols, and decorated cups with two equal symbols. The multiplicities $\mu_{P(\la)}$ are then easily calculated:
\begin{prop}For $\lambda,\mu \in \Lambda$ it holds
\label{singO}
\abovedisplayskip0.2em
\belowdisplayskip0.2em
\begin{gather*}
\mu_{P(\la)}\,=\,\begin{cases} 1& \mbox{if } \mu\underline{\la}\mbox{ is oriented,}\\
0 & \mbox{if $\mu\underline{\la}$}\mbox{ is not oriented.}
\end{cases}
\end{gather*}
\end{prop}
\begin{proof}
For regular weights, i.e. weights which do not contain $\cross$'s or $\circ$'s this was proved is \cite[Theorem 2.1]{LS}. (Lemma 4.3 therein and the paragraph afterwards makes the translation into our setup.)
For singular weights observe that diagrammatically the numbers stay the same if we remove the $\cross$'s and $\circ$'s. This corresponds Lie theoretically to the Enright-Shelton equivalence \cite{EnSh} between singular blocks and regular blocks for smaller rank Lie algebras, and hence the multiplicities agree.
\end{proof}

The action of the special projective functors on projectives can be calculated easily in terms of yet another diagram calculus inside the Temperley-Lieb category of type $ \mathrm{D}$, \cite{Green}. Let $\cC$ be the free abelian group generated by $\underline{\la}$, $\la\in\La$, and ${\cF}^{\op{diag}}:\cC \rightarrow \cC$ the $\mathbb{Z}$-linear endomorphism defined as follows: Given a basis vector $\underline{\lambda}$ pick a diagram \scalebox{1.3}{$\square$} from $\eqref{tangles}$ which can be put on top of $\underline{\lambda}$ such that the all symbols $\circ$, $\times$, and $\diamondb$ match. Note that whether a strand involving $\diamondb$ has to carry a $\bullet$ or not depends on how $\diamondb$ needs to be interpreted to have an even number of $\down$ symbols.
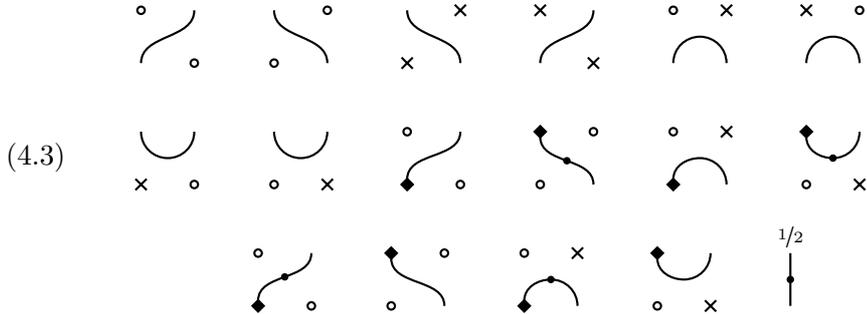
\abovedisplayskip0.2em
\belowdisplayskip0.2em
\begin{gather}
\label{tangles}
\begin{tikzpicture}[anchorbase,scale=.7,thick,>=angle 60,c/.style={insert path={circle[radius=2pt]}},d/.style={fill=black,shape=diamond,text height=-10pt,draw}]
\def\dist{2.5cm}
\draw (1,0) [c];
\draw (0,0) to [in=-90,out=90] +(1,1);
\draw +(0,1) [c];
\begin{scope}[xshift=\dist]
\draw (0,0) [c];
\draw (1,0) to [in=-90,out=90] +(-1,1);
\draw +(1,1) [c];
\end{scope}
\begin{scope}[xshift=2*\dist]
\draw (0,0) +(-3pt,-3pt) -- +(3pt,3pt) +(-3pt,3pt) -- +(3pt,-3pt);
\draw (1,0) to [in=-90,out=90] +(-1,1);
\draw (1,1) +(-3pt,-3pt) -- +(3pt,3pt) +(-3pt,3pt) -- +(3pt,-3pt);
\end{scope}
\begin{scope}[xshift=3*\dist]
\draw (1,0) +(-3pt,-3pt) -- +(3pt,3pt) +(-3pt,3pt) -- +(3pt,-3pt);
\draw (0,0) to [in=-90,out=90] +(1,1);
\draw (0,1) +(-3pt,-3pt) -- +(3pt,3pt) +(-3pt,3pt) -- +(3pt,-3pt);
\end{scope}
\begin{scope}[xshift=4*\dist]
\draw (0,0) to [out=90,in=-180] +(.5,.5) to [out=0,in=90] +(.5,-.5);
\draw +(0,1) [c];\draw (1,1) +(-3pt,-3pt) -- +(3pt,3pt) +(-3pt,3pt) -- +(3pt,-3pt);
\end{scope}
\begin{scope}[xshift=5*\dist]
\draw (0,0) to [out=90,in=-180] +(.5,.5) to [out=0,in=90] +(.5,-.5);
\draw (0,1) +(-3pt,-3pt) -- +(3pt,3pt) +(-3pt,3pt) -- +(3pt,-3pt);\draw +(1,1) [c];
\end{scope}
\begin{scope}[yshift=-2.3cm]
\begin{scope}[xshift=0]
\draw (0,0) +(-3pt,-3pt) -- +(3pt,3pt) +(-3pt,3pt) -- +(3pt,-3pt);\draw +(1,0) [c];
\draw (0,1) to [out=-90,in=-180] +(.5,-.5) to [out=0,in=-90] +(.5,.5);
\end{scope}
\begin{scope}[xshift=1*\dist]
\draw +(0,0) [c];\draw (1,0) +(-3pt,-3pt) -- +(3pt,3pt) +(-3pt,3pt) -- +(3pt,-3pt);
\draw (0,1) to [out=-90,in=-180] +(.5,-.5) to [out=0,in=-90] +(.5,.5);
\end{scope}
\begin{scope}[xshift=2*\dist]
\node[d] at (0,0) {};\draw (1,0) [c];
\draw (0,.1) to [in=-90,out=90] +(1,.9);
\draw +(0,1) [c];
\end{scope}
\begin{scope}[xshift=3*\dist]
\draw +(0,0) [c];
\draw (1,0) to [in=-90,out=90] +(-1,.9);\fill (.5,.45) circle(2pt);
\node[d] at (0,1) {};\draw (1,1) [c];
\end{scope}
\begin{scope}[xshift=4*\dist]
\node[d] at (0,0) {};
\draw (0,.1) to [out=90,in=-180] +(.5,.4) to [out=0,in=90] +(.5,-.5);
\draw +(0,1) [c];\draw (1,1) +(-3pt,-3pt) -- +(3pt,3pt) +(-3pt,3pt) -- +(3pt,-3pt);
\end{scope}
\begin{scope}[xshift=5*\dist]
\draw +(0,0) [c];\draw (1,0) +(-3pt,-3pt) -- +(3pt,3pt) +(-3pt,3pt) -- +(3pt,-3pt);
\draw (0,.9) to [out=-90,in=-180] +(.5,-.4) to [out=0,in=-90] +(.5,.5);\fill (.5,.5) circle(2pt);
\node[d] at (0,1) {};
\end{scope}
\end{scope}
\begin{scope}[yshift=-4.6cm,xshift=-2.8cm]
\begin{scope}[xshift=2*\dist]
\node[d] at (0,0) {};\draw (1,0) [c];
\draw (0,.1) to [in=-90,out=90] +(1,.9);\fill (.5,.55) circle(2pt);
\draw (0,1) [c];
\end{scope}
\begin{scope}[xshift=3*\dist]
\draw +(0,0) [c];
\draw (1,0) to [in=-90,out=90] +(-1,.9);
\node[d] at (0,1) {};\draw (1,1) [c];
\end{scope}
\begin{scope}[xshift=4*\dist]
\node[d] at (0,0) {};
\draw (0,.1) to [out=90,in=-180] +(.5,.4) to [out=0,in=90] +(.5,-.5);\fill (.5,.5) circle(2pt);
\draw +(0,1) [c];\draw (1,1) +(-3pt,-3pt) -- +(3pt,3pt) +(-3pt,3pt) -- +(3pt,-3pt);
\end{scope}
\begin{scope}[xshift=5*\dist]
\draw +(0,0) [c];\draw (1,0) +(-3pt,-3pt) -- +(3pt,3pt) +(-3pt,3pt) -- +(3pt,-3pt);
\draw (0,.9) to [out=-90,in=-180] +(.5,-.4) to [out=0,in=-90] +(.5,.5);
\node[d] at (0,1) {};
\end{scope}
\begin{scope}[xshift=6*\dist]
\draw (0,0)--(0,1);\fill (0,.5) circle(2pt) node at (0,1.3) {$\scriptstyle \half$};
\end{scope}
\end{scope}
\end{tikzpicture}
\end{gather}
The result is (up to isotopy) a digram with possibly some, say $c(\square,\lambda)$, internal circles and possibly more than one $\bullet$ per strand. In case a circle has an odd number of $\bullet$ symbols on it, we set ${\cF}^{\op{diag}}_\square (\underline{\lambda})=0$. Otherwise let ${\cF}^{\op{diag}}_\square (\underline{\lambda})$ be the decorated cup diagram obtained  by cancelling pairs of $\bullet$'s on the same line and removing internal circles and define
\[
{\cF}^{\op{diag}}(\underline{\lambda}) = \sum_\square (q+q^{-1})^{c(\square,\lambda)} {\cF}^{\op{diag}}_\square(\underline{\lambda}),
\]
where the sum runs over all possible diagrams \scalebox{1.3}{$\square$}.

\begin{theorem}[Diagrammatic translation of projectives]
\label{combinatorics}
Let
\abovedisplayskip0.2em
\belowdisplayskip0.2em
\[
\Phi: K_0(\Op(n))\,\,\,\xrightarrow{\phantom{x}\textstyle \sim\phantom{x}} \cC,
\]
be the isomorphism of abelian groups, determined by $\Phi([P(\lambda)])=\underline{\la}$. Then ${\cF}^{\op{diag}}(\Phi([P]))=\Phi([P\otimes V])$ for any projective module $P\in\Op(n)$. 
\end{theorem}
\begin{proof}
In the case of columns 4-7 in \eqref{deftable1}, the group homomorphism $K_0(\cF_{i,-}):K_0(\Op_\Ga(n))\rightarrow K_0(\Op_{\Ga_{i,-}}(n))$ induced by $\cF_{i,-}$ is an isomorphism when expressed in the basis of the isomorphism classes of Verma modules; it has an inverse, the morphism induced by $\cF_{i,+}$. They both agree with the morphism induced from the restriction of $\cF$ to these blocks. Hence indecomposable projectives are sent to the corresponding indecomposable projectives. The corresponding statement on the diagram side is obvious, involving the first four diagrams from \eqref{tangles} only. In the case of column 2 or 3 in \eqref{deftable1} we can apply to a cup diagram $\underline{\la}$ exactly one of the first two diagrams from the second line of \eqref{tangles}. In each case the result is $\underline{\la}$ with an extra cup and we are done by Proposition~\ref{a} (vii). Column 1 is the most involved situation. In this case $\underline{\la}$ looks locally as one of the 18 diagrams in \eqref{cool2} drawn in thick black lines or the same picture with $\circ$ and $\cross$ swapped. Together with the dashed lines they indicate the result after applying the diagrammatic functor (the first two in the second line give zero).
\begin{gather*}
\begin{tikzpicture}[thick]
\begin{scope}[yshift=-3cm]
\draw node at (.25,-.3) {$P(\la_{\up\down})$};
\draw (0,0) -- +(0,.7);\fill (0,.35) circle(2pt);
\draw (.5,0) -- +(0,.7);
\draw[dashed,red] (0,.7) to [out=90,in=180] +(.25,.4) to [out=0,in=90] +(.25,-.4);
\draw[red] (0,1.3) circle(2pt);\draw[red] (.5,1.3) +(-2pt,-2pt) -- +(2pt,2pt) +(-2pt,2pt) -- +(2pt,-2pt);
\draw node at (.25,1.7) {$P(\la_{\circ\cross})$};
\draw node at (2,-.3) {$P(\la_{\up\down\up})$};
\draw (1.5,0) -- +(0,.7);\fill (1.5,.35) circle(2pt);
\draw (2,.7) to [out=-90,in=-180] +(.25,-.4) to [out=0,in=-90] +(.25,.4);
\draw[dashed,red] (1.5,.7) to [out=90,in=180] +(.25,.4) to [out=0,in=90] +(.25,-.4);
\draw[red] (1.5,1.3) circle(2pt);\draw[red] (2,1.3) +(-2pt,-2pt) -- +(2pt,2pt) +(-2pt,2pt) -- +(2pt,-2pt);
\draw[red] (2.5,.7) -- +(0,.6);
\draw node at (2,1.7) {$P(\la_{\circ\cross\up})$};
\draw node at (3.8,-.3) {$P(\la_{\down\down\up})$};
\draw (3.3,0) -- +(0,.7);
\draw (3.8,.7) to [out=-90,in=-180] +(.25,-.4) to [out=0,in=-90] +(.25,.4);
\draw[dashed,red] (3.3,.7) to [out=90,in=180] +(.25,.4) to [out=0,in=90] +(.25,-.4);
\draw[red] (3.3,1.3) circle(2pt);\draw[red] (3.8,1.3) +(-2pt,-2pt) -- +(2pt,2pt) +(-2pt,2pt) -- +(2pt,-2pt);
\draw[red] (4.3,.7) -- +(0,.6);
\draw node at (3.8,1.7) {$P(\la_{\circ\cross\up})$};
\draw node at (5.5,-.3) {$P(\la_{\down\up\down})$};
\draw (5,.7) to [out=-90,in=-180] +(.25,-.4) to [out=0,in=-90] +(.25,.4);
\draw (6,0) -- +(0,.7);
\draw[dashed,red] (5.5,.7) to [out=90,in=180] +(.25,.4) to [out=0,in=90] +(.25,-.4);
\draw[red] (5.5,1.3) circle(2pt);\draw[red] (6,1.3) +(-2pt,-2pt) -- +(2pt,2pt) +(-2pt,2pt) -- +(2pt,-2pt);
\draw[red] (5,.7) -- +(0,.6);
\draw node at (5.5,1.7) {$P(\la_{\down\circ\cross})$};
\draw node at (7.3,-.3) {$P(\la_{\up\up\down})$};
\draw (6.8,.7) to [out=-90,in=-180] +(.25,-.4) to [out=0,in=-90] +(.25,.4);\fill (7.05,.31) circle(2pt);
\draw (7.8,0) -- +(0,.7);
\draw[dashed,red] (7.3,.7) to [out=90,in=180] +(.25,.4) to [out=0,in=90] +(.25,-.4);
\draw[red] (7.3,1.3) circle(2pt);\draw[red] (7.8,1.3) +(-2pt,-2pt) -- +(2pt,2pt) +(-2pt,2pt) -- +(2pt,-2pt);
\draw[red] (6.8,.7) -- +(0,.6);
\draw node at (7.3,1.7) {$P(\la_{\up\circ\cross})$};
\draw node at (9.1,-.3) {$P(\la_{\down\up\up})$};
\draw (8.6,.7) to [out=-90,in=-180] +(.25,-.4) to [out=0,in=-90] +(.25,.4);
\draw (9.6,0) -- +(0,.7);\fill (9.6,.35) circle(2pt);
\draw[dashed,red] (9.1,.7) to [out=90,in=180] +(.25,.4) to [out=0,in=90] +(.25,-.4);
\draw[red] (9.1,1.3) circle(2pt);\draw[red] (9.6,1.3) +(-2pt,-2pt) -- +(2pt,2pt) +(-2pt,2pt) -- +(2pt,-2pt);
\draw[red] (8.6,.7) -- +(0,.6);
\draw node at (9.1,1.7) {$P(\la_{\down\circ\cross})$};
\draw node at (10.9,-.3) {$P(\la_{\up\up\up})$};
\draw (10.4,.7) to [out=-90,in=-180] +(.25,-.4) to [out=0,in=-90] +(.25,.4);\fill (10.65,.31) circle(2pt);
\draw (11.4,0) -- +(0,.7);\fill (11.4,.35) circle(2pt);
\draw[dashed,red] (10.9,.7) to [out=90,in=180] +(.25,.4) to [out=0,in=90] +(.25,-.4);
\draw[red] (10.9,1.3) circle(2pt);\draw[red] (11.4,1.3) +(-2pt,-2pt) -- +(2pt,2pt) +(-2pt,2pt) -- +(2pt,-2pt);
\draw[red] (10.4,.7) -- +(0,.6);
\draw node at (10.9,1.7) {$P(\la_{\down\circ\cross})$};
\end{scope}
\begin{scope}[yshift=-6cm]
\draw node at (.15,-.3) {$P(\la_{\down\down})$};
\draw (0,0) -- +(0,.7);
\draw (.5,0) -- +(0,.7);
\draw[dashed,red] (0,.7) to [out=90,in=180] +(.25,.4) to [out=0,in=90] +(.25,-.4);
\draw[red] (0,1.3) circle(2pt);\draw[red] (.5,1.3) +(-2pt,-2pt) -- +(2pt,2pt) +(-2pt,2pt) -- +(2pt,-2pt);
\draw node at (.25,1.7) {$0$};
\draw node at (1.45,-.3) {$P(\la_{\up\up})$};
.5\draw (1.3,.7) to [out=-90,in=-180] +(.25,-.4) to [out=0,in=-90] +(.25,.4);\fill (1.55,.31) circle(2pt);
\draw[dashed,red] (1.3,.7) to [out=90,in=180] +(.25,.4) to [out=0,in=90] +(.25,-.4);
\draw[red] (1.3,1.3) circle(2pt);\draw[red] (1.8,1.3) +(-2pt,-2pt) -- +(2pt,2pt) +(-2pt,2pt) -- +(2pt,-2pt);
\draw node at (1.55,1.7) {$0$};
\end{scope}
\begin{scope}[xshift=2.5cm,yshift=-6cm]
\draw node at (.25,-.3) {$P(\la_{\down\up})$};
\draw (0,.7) to [out=-90,in=-180] +(.25,-.4) to [out=0,in=-90] +(.25,.4);
\draw[dashed,red] (0,.7) to [out=90,in=180] +(.25,.4) to [out=0,in=90] +(.25,-.4);
\draw[red] (0,1.3) circle(2pt);\draw[red] (.5,1.3) +(-2pt,-2pt) -- +(2pt,2pt) +(-2pt,2pt) -- +(2pt,-2pt);
\draw node at (.25,1.7) {$P(\circ\cross)^{\oplus 2}$};
\draw node at (1.9,-.3) {$P(\la_{\down\up\down\up})$};
\draw (1.2,.7) to [out=-90,in=-180] +(.25,-.4) to [out=0,in=-90] +(.25,.4);
\draw[dashed,red] (1.7,.7) to [out=90,in=180] +(.25,.4) to [out=0,in=90] +(.25,-.4);
\draw[red] (1.7,1.3) circle(2pt);\draw[red] (2.2,1.3) +(-2pt,-2pt) -- +(2pt,2pt) +(-2pt,2pt) -- +(2pt,-2pt);
\draw (2.2,.7) to [out=-90,in=-180] +(.25,-.4) to [out=0,in=-90] +(.25,.4);
\draw[red] (1.2,.7) -- +(0,.6);
\draw[red] (2.7,.7) -- +(0,.6);
\draw node at (1.9,1.7) {$P(\la_{\down\circ\cross\up})$};
\draw node at (4.1,-.3) {$P(\la_{\up\up\down\up})$};
\draw (3.4,.7) to [out=-90,in=-180] +(.25,-.4) to [out=0,in=-90] +(.25,.4);\fill (3.65,.31) circle(2pt);
\draw[dashed,red] (3.9,.7) to [out=90,in=180] +(.25,.4) to [out=0,in=90] +(.25,-.4);
\draw[red] (3.9,1.3) circle(2pt);\draw[red] (4.4,1.3) +(-2pt,-2pt) -- +(2pt,2pt) +(-2pt,2pt) -- +(2pt,-2pt);
\draw (4.4,.7) to [out=-90,in=-180] +(.25,-.4) to [out=0,in=-90] +(.25,.4);
\draw[red] (3.4,.7) -- +(0,.6);
\draw[red] (4.9,.7) -- +(0,.6);
\draw node at (4.1,1.7) {$P(\la_{\up\circ\cross\up})$};
\draw node at (6.3,-.3) {$P(\la_{\down\up\up\up})$};
\draw (5.6,.7) to [out=-90,in=-180] +(.25,-.4) to [out=0,in=-90] +(.25,.4);
\draw[dashed,red] (6.1,.7) to [out=90,in=180] +(.25,.4) to [out=0,in=90] +(.25,-.4);
\draw[red] (6.1,1.3) circle(2pt);\draw[red] (6.6,1.3) +(-2pt,-2pt) -- +(2pt,2pt) +(-2pt,2pt) -- +(2pt,-2pt);
\draw (6.6,.7) to [out=-90,in=-180] +(.25,-.4) to [out=0,in=-90] +(.25,.4);\fill (6.85,.31) circle(2pt);
\draw[red] (5.6,.7) -- +(0,.6);
\draw[red] (7.1,.7) -- +(0,.6);
\draw node at (6.3,1.7) {$P(\la_{\up\circ\cross\up})$};
\draw node at (8.5,-.3) {$P(\la_{\up\up\up\up})$};
\draw (7.8,.7) to [out=-90,in=-180] +(.25,-.4) to [out=0,in=-90] +(.25,.4);\fill (8.05,.31) circle(2pt);
\draw[dashed,red] (8.3,.7) to [out=90,in=180] +(.25,.4) to [out=0,in=90] +(.25,-.4);
\draw[red] (8.3,1.3) circle(2pt);\draw[red] (8.8,1.3) +(-2pt,-2pt) -- +(2pt,2pt) +(-2pt,2pt) -- +(2pt,-2pt);
\draw (8.8,.7) to [out=-90,in=-180] +(.25,-.4) to [out=0,in=-90] +(.25,.4);\fill (9.05,.31) circle(2pt);
\draw[red] (7.8,.7) -- +(0,.6);
\draw[red] (9.3,.7) -- +(0,.6);
\draw node at (8.5,1.7) {$P(\la_{\down\circ\cross\up})$};
\end{scope}
\end{tikzpicture}
\end{gather*}

\begin{gather}
\label{cool2}
\begin{tikzpicture}[anchorbase,thick]
\begin{scope}[xshift=1.5cm,yshift=-9cm]
\draw node at (.7,-.45) {$P(\la_{\down\down\up\up})$};
\draw[dashed,red] (0,.7) to [out=90,in=180] +(.25,.4) to [out=0,in=90] +(.25,-.4);
\draw[red] (0,1.3) circle(2pt);\draw[red] (.5,1.3) +(-2pt,-2pt) -- +(2pt,2pt) +(-2pt,2pt) -- +(2pt,-2pt);
\draw (0,.7) to [out=-90,in=-180] +(.75,-.8) to [out=0,in=-90] +(.75,.8);
\draw (.5,.7) to [out=-90,in=-180] +(.25,-.4) to [out=0,in=-90] +(.25,.4);
\draw[red] (1,.7) -- +(0,.6);
\draw[red] (1.5,.7) -- +(0,.6);
\draw node at (.7,1.7) {$P(\la_{\circ\cross\up\up})$};
\draw node at (3,-.45) {$P(\la_{\down\down\up\up})$};
\draw[dashed,red] (3.3,.7) to [out=90,in=180] +(.25,.4) to [out=0,in=90] +(.25,-.4);
\draw[red] (3.3,1.3) circle(2pt);\draw[red] (3.8,1.3) +(-2pt,-2pt) -- +(2pt,2pt) +(-2pt,2pt) -- +(2pt,-2pt);
\draw (2.3,.7) to [out=-90,in=-180] +(.75,-.8) to [out=0,in=-90] +(.75,.8);
\draw (2.8,.7) to [out=-90,in=-180] +(.25,-.4) to [out=0,in=-90] +(.25,.4);
\draw[red] (2.3,.7) -- +(0,.6);
\draw[red] (2.8,.7) -- +(0,.6);
\draw node at (3.3,1.7) {$P(\la_{\down\down\circ\cross})$};
\draw node at (5.6,-.45) {$P(\la_{\up\down\up\up})$};
\draw[dashed,red] (4.9,.7) to [out=90,in=180] +(.25,.4) to [out=0,in=90] +(.25,-.4);
\draw[red] (4.9,1.3) circle(2pt);\draw[red] (5.4,1.3) +(-2pt,-2pt) -- +(2pt,2pt) +(-2pt,2pt) -- +(2pt,-2pt);
\draw (4.9,.7) to [out=-90,in=-180] +(.75,-.8) to [out=0,in=-90] +(.75,.8);\fill (5.65,-.1) circle(2pt);
\draw (5.4,.7) to [out=-90,in=-180] +(.25,-.4) to [out=0,in=-90] +(.25,.4);
\draw[red] (5.9,.7) -- +(0,.6);
\draw[red] (6.4,.7) -- +(0,.6);
\draw node at (5.6,1.7) {$P(\la_{\circ\cross\up\up})$};
\draw node at (8.2,-.45) {$P(\la_{\up\down\up\up})$};
\draw[dashed,red] (8.5,.7) to [out=90,in=180] +(.25,.4) to [out=0,in=90] +(.25,-.4);
\draw[red] (8.5,1.3) circle(2pt);\draw[red] (9,1.3) +(-2pt,-2pt) -- +(2pt,2pt) +(-2pt,2pt) -- +(2pt,-2pt);
\draw (7.5,.7) to [out=-90,in=-180] +(.75,-.8) to [out=0,in=-90] +(.75,.8);\fill (8.25,-.1) circle(2pt);
\draw (8,.7) to [out=-90,in=-180] +(.25,-.4) to [out=0,in=-90] +(.25,.4);
\draw[red] (7.5,.7) -- +(0,.6);
\draw[red] (8,.7) -- +(0,.6);
\draw node at (8.2,1.7) {$P(\la_{\up\down\circ\cross})$};
\end{scope}
\end{tikzpicture}
\end{gather}
In the first seven non-zero pictures, we have only one possible orientation for the displayed resulting diagram. On the other hand,  there are at most two orientations, for each choice of a diagram, of which exactly one corresponds to a Verma module which is not annihilated by $\cF$ with the image in our chosen blocks except of the last and penultimate diagram where none of the two is annihilated. Similar arguments can be used for the functors $\cF_{\half,+}$, $\cF_{\half,-}$, and $\cF_0$. Comparing with Proposition~\ref{a} using Proposition~\ref{singO}, we see that ${\cF}^{\op{diag}}(\Phi([P]))=\Phi([P\otimes V])$ in the
basis of Verma modules.
\end{proof}

As a special case we obtain the following:

\begin{corollary}
\label{special-case}
Assume we are in case \eqref{HIM6} of Proposition~\ref{a} and assume there is no $\down$ to the left of our fixed $\up\down$ pair then  $\cF_{i,-} P^\pp(\la) \cong P^\pp(\la_{\circ{\scriptscriptstyle\times}})$.
\end{corollary}

\begin{proof}
By assumption and construction, the cup diagram $\la$ looks locally at our two fixed vertices as displayed in one of the pictures in the top line of the following diagram. Applying ${\cF}^{\op{diag}}$ gives for each diagram precisely one new diagram in the required block as displayed. Hence the statement follows by applying $\Phi^{-1}$ and the definition of $\cF_{i,-}$.
\end{proof}

\section{The cyclotomic quotient \texorpdfstring{$\VWd(\alpha,\beta)$}{of VW}} \label{section:cyclotomic_quotients}
\textbf{For this whole section we will always assume $n\geq 2d$.} This is crucial for all statements involving the idempotent $z_d$ from Definition~\ref{def:brauer_idempotent}.

The action of the commuting $y_i$'s decomposes $\MdV$ and $\VWd(\alpha,\beta)$ into generalized eigenspaces. Fix the sets $J=\nicefrac{\delta}{2}+\mZ$ and $J_{<}=\nicefrac{\delta}{2}+\mZ_{<(n-1)}$.

\begin{definition}
We identify $\VWd(\alpha,\beta)$ with $\mathrm{End}_\mg(\MdV)$ via Theorem~\ref{iso}. Define the orthogonal {\em weight idempotents} $\ei$, $\bi \in J^d$ characterized by the property that
\abovedisplayskip0.2em
\belowdisplayskip0.2em
\[
\ei (y_r-i_r)^m = (y_r-i_r)^m \ei = 0
\]
for each $1\leq r\leq d$ and $m\gg 0$ (for instance $m\geq d$ is enough).
\end{definition}
Then we have the generalized eigenspace decompositions
\abovedisplayskip0.2em
\belowdisplayskip0.2em
\begin{gather*}
\begin{gathered}
\MdV={\sum}_{\bi\in J^d}\MdV\ei,\\
\VWd(\alpha,\beta)={\sum}_{\bi\in J^d}\VWd(\alpha,\beta)\ei \quad \text{and} \quad 
\VWd(\alpha,\beta)={\sum}_{\bi\in J^d}\ei\VWd(\alpha,\beta)
\end{gathered}
\end{gather*}
as right respectively left modules. The quasi-idempotents $e_k$ only acts between certain generalized eigenspaces:
\begin{lemma}
\label{eigenvaluesfore}
Let $1\leq k\leq d-1$. If $p=\sum_{\bi\in J^d, i_{k+1}=-i_k}\ei$ then we have $e_k=pe_k=e_kp=pe_kp$. In particular $\ei e_k=0=e_k\ei$ if $i_{k+1}\not=-i_k$.
\end{lemma}

\begin{proof}
Let $m\in \VWd(\alpha,\beta)$ or $\MdV$ be contained in the generalized eigenspace for $\ei$, hence $m\ei\not=0$. We claim that if $me_k\not=0$ then $i_{r+1}+i_r=0$. First note that for  $(y_k+y_{k+1}-i_{k+1}-i_k)^rm=\sum_{a=0}^r(y_k-i_{k})^a(y_{k+1}-i_{k+1})^{r-a}m=0$ for $r\gg 0$. Hence $m$ is in the $\mu={i_{k+1}+i_k}$-generalized eigenspace for $y_{k+1}+y_k$. Then $0=m(y_k+y_{k+1}-\mu)^{r}e_k=me_k(-\mu)^k$ by \eqref{8a}. Hence $\mu=0$ and the claim follows. Since the modules have a generalized eigenspace decomposition $pe_k=0=e_kp$. The rest follows analogously.
\end{proof}

\subsection{The Brauer algebra idempotent}
In this section we establish the technical tools which allow us in \cite{ES_Brauer} to identify the Brauer algebra $B_d(\delta)$ as an {\it idempotent truncation} of $\VWd(\alpha,\beta)$ and hence connect it with Lie theory.
\begin{definition} \label{def:brauer_idempotent}
The {\it Brauer algebra idempotent} is defined as
\abovedisplayskip0.2em
\belowdisplayskip0.2em
\begin{gather*}
z_d\,\,=\sum_{\bi\in (J{_<})^d} \ei.
\end{gather*}
\end{definition}
Multiplication by $\ei$ projects any $\VWd(\alpha,\beta)$-module
onto its {\em $\bi$-weight space}, that is, the simultaneous
generalised eigenspace for the commuting operators
$y_1,\dots,y_r$ with respective eigenvalues $i_1,\dots,i_r$.
By \eqref{coolpicture}, Proposition~\ref{surjectivity} and our assumption on $n$ being large, the element $z_d$ is just the projection onto the blocks containing Verma modules indexed by bipartitions of the form $(\la,\emptyset)$ where $\la$ is an \upd. (The corresponding weights are precisely those $\la$ which satisfy $I(\la)\subset [-(\nicefrac{\delta}{2}+n-1),\nicefrac{\delta}{2}+n-1]$.)

\begin{prop} \label{prop:dimension_brauer}
The algebra $z_d\VWd(\alpha,\beta)z_d$ has dimension $(2d-1)!!$.
\end{prop}
\begin{proof}
Using Proposition~\ref{surjectivity} and that the algebra has a basis given by Verma paths that correspond to bipartitions of the form $(\la,\emptyset)$ gives the result.
\end{proof}

\begin{theorem}[Semisimplicity theorem]
\label{thm:semisimplicity}\hfill
\begin{enumerate}
\item The algebra $\VWd(\alpha,\beta)$ is semisimple if and only if $|\delta|\geq d-1$.
\item The algebra $z_d\VWd(\alpha,\beta)z_d$ is semisimple if and only if
    $\delta\not=0$ and $|\delta|\geq d-1$ or $\delta=0$ and $d=1,3,5$.
\end{enumerate}
\end{theorem}
\begin{proof}
We first assume $\delta \geq 0$. Note that $\END_\mg(\MdV)$ is semisimple if and only if $\MdV$ is a direct sum of projective Verma modules. Or equivalently if all occurring Verma modules have highest weight $\la$ such that $\la+\rho$ is dominant. This is equivalent to the statement that the corresponding diagrammatic weights do not contain two $\up$'s, an ordered pair $\down\up$ or a pair $\diamondb\up$.

Starting from the diagrammatic weight $\de$ we need precisely $\delta$ steps to create an $\up$ left to a $\down$, hence $\delta+2$ steps to create a $\down$, $\up$; we also need $2(\nicefrac{\delta}{2}+1)$ steps to create a pair of two $\up$'s and $\delta+2$ steps to create $\diamondb$, $\up$. Hence both algebras are semisimple at least if $d< \delta+2$. Now in case (1), we always can add $\epsilon_n$ without changing those pairs, an hence the algebra is not semisimple for all $d\geq\delta+2$. In case (2) with $\delta\not= 0,1$ we can always find some $\circ$ at some position $<n$ and hence add or subtract some appropriate $\epsilon_j$ without changing the pair. That means the truncated algebra is not semisimple for $d\geq\delta+2$. For $\delta=1$ we need $3$ steps to create a weight starting with $\diamondb\up\circ\down$. Hence the algebra is not semisimple and stays not semisimple, since we can always repeatedly swap the $\circ$ with the $\down$ without changing the $\diamondb\up$ pair. Hence (2) holds for any $\delta>0$. In case $\delta=0$ one can calculate directly that it is semisimple in the cases $d=1, d=3, d=5$, but not in the cases $d=2,4,6$ and for $=7$ we obtain a weight starting with $\diamondb\up\up\down$. Since we can change the last two symbols $\up\down$ into $\circ\cross$ and back again without loosing the  $\diamondb\up$-pair, the algebra stays not semisimple for $d\geq 7$ and $\delta=0$. 

For the case $\delta < 0$ we use Corollary \ref{tableaux}. The theorem follows.
\end{proof}

\subsection{The basic algebra} %underlying $\END_\mg(\MdV)$
Our next goal is to determine for $\delta\geq 0$ the projective modules appearing in $\MdV$ and hence the basic algebra underlying $\END_\mg(\MdV)$ for $\delta\in\mZ$. Our main tool here is the notion of $\dht$ which measures the minimal length of a Verma path needed to create a Verma module of a given weight.

\begin{definition}
Let $\delta, d\in\mZ$, $\delta\geq 0$, $d\geq 1$ and $\mu\in\La$. The $\delta$-height of $\mu$ is defined as $\dht(\mu)=\sum_{i=1}^{n}|(\de+\rho)_i-(\mu+\rho)_i|=\sum_{i=1}^{n}|\de_i-\mu_i|$.
\end{definition}

Note that, by Proposition~\ref{Verma}, the highest weights $\la$ occurring in a Verma filtration of $\MdV$ satisfy $\dht(\la)\leq d$ and $\dht(\la)$ is precisely the number of boxes in $\varphi(\la)$, see \eqref{displaybipartition}, \eqref{coolpicture}.
The proof of the following two Propositions is given in Section~\ref{whichandsat}
\begin{prop}
\label{whichproj}
Let $\delta, d\in\mZ$, $\delta\geq 0$, $d\geq 1$.
Then 
\abovedisplayskip0.3em
\belowdisplayskip0.3em
\[
\MdV\cong {\bigoplus}_{\mu\in\La} a_\mu P(\mu)
\]
with multiplicities $a_\mu$ and $a_\mu\not=0$ if and only if  $\dht\leq d$ and $d-\dht$ is even.
\end{prop}

\begin{prop}
\label{saturated}
Let $\delta, d\in\mZ$, $\delta\geq 0$, $d\geq 1$. The set
\abovedisplayskip0.3em
\belowdisplayskip0.3em
\begin{gather}
\label{S}
\mathcal{S}=\mathcal{S}(\delta,d)=\{\mu\in\Lambda\mid \dht(\mu)\leq d, \dht(\mu)\equiv d\mod 2\}
\end{gather}
is saturated in the following sense: if $\mu\in \mathcal{S}$ and $\nu\in\La$ such that $\mu$ and $\nu$ are in the same block and $\nu\geq\la$ then $\nu\in\mathcal{S}$.
\end{prop}

\subsection{Quasi-hereditaryness}
The category $\Op(n)$ is a highest weight category in the sense of \cite{CPSI}, i.e. it is a category $\mathcal{C}$ with a poset $(\cS,\leq)$ satisfying
\begin{itemize}
\item
$\mathcal C$ is a $\mC$-linear Artinian category equipped
with a duality, that is, a contravariant involutive  equivalence of categories $\circledast:\mathcal C \rightarrow
\mathcal C$;
\item
for each $\la \in \cS$ there is a given object $L(\la)
\cong{L(\la)}^\circledast \in\mathcal C$ such that $\{L(\la) | \la \in \cS\}$ is a complete set of representatives for
the isomorphism classes of irreducible objects in $\mathcal C$;
\item
each $L(\la)$ has a projective cover $P(\la)\in\mathcal C$ such that
all composition factors of $P(\la)$ are isomorphic to $L(\mu)$'s
for $\mu \in \cS$;
\item writing $V(\la)$ for the largest quotient of
$P(\la)$ with the property that all composition factors of
its radical are isomorphic to
$L(\mu)$'s for $\mu > \la$,
the object $P(\la)$ has a
filtration with $V(\la)$ at the top and all
other factors isomorphic to $V(\nu)$'s for
$\nu < \la$.
\end{itemize}
The poset $\cS$ is hereby the set $\Lambda$ of highest weights for irreducible modules with the reversed Bruhat ordering $\leq_{\op{rBruhat}}$; $V(\la)$ is the parabolic Verma module of highest weight $\la$ and the duality is the standard duality from $\cO^\pp(n)$, \cite{Hbook}.

The category $\VWd(\alpha,\beta)^{\op{opp}}-\op{mod}$ of finite dimensional (right) $\VWd(\alpha,\beta)$-modules is also a highest weight category, by \cite[Corollary 8.6]{AMR}, with poset the set of bipartitions of $d-2t$ for $1\leq t\leq \lfloor \frac{d}{2}\rfloor$ equipped with the dominance ordering defined below, $V(\la)$ the cell module from \cite[Section 6]{AMR}, and $\circledast$ the ordinary vector space duality using Remark \ref{Wop}.

The category of finite dimensional $\Br_d(\delta)$-modules with $\delta\not=0$ is by \cite[Theorem 1.3]{Koenig-Xi}, \cite[Corollary 2.3]{CDVMII} a highest weight category, with poset the set of partitions of $d-2t$ for $1\leq t\leq \lfloor \frac{d}{2}\rfloor$ with the dominance ordering, $V(\la)$ the cell module from \cite[Lemma 2.4]{CDVMII} and $\circledast$, the ordinary vector space duality using the analogue of Remark \ref{Wop} for the Brauer algebra. In case $\delta=0$ it is a highest weight category if and only if $d$ odd.

For the following result recall the bijection from Proposition~\ref{surjectivity}.
\begin{theorem}[Highest weight structure]
For fixed $\delta$, $d$ and $\cS=\cS(\delta,d)$ as in \eqref{S} let $\Op(n,\delta,d)$ be the full subategory of $\cO^\pp(n)$ given by all modules with composition factors $L(\mu)$ with $\mu\in\cS$. Then
\begin{enumerate}
\item $\Op(n,\delta,d)$ inherits the structure of a highest weight category from $\Op(n)$ with poset $(\cS,\leq_{\op{rBruhat}})$ and the $V(\la)$ for $\la\in\cS$. The duality on $\cO^\pp(n)$ restricts to the duality on the subcategory.
\item The functor $\cE=\Hom_\mg(\MdV,?)$ defines an equivalence
\abovedisplayskip0.3em
\belowdisplayskip0.3em
\begin{gather*}
\cE:\quad\Op(n,\delta,d)\,\,\,\xrightarrow{\phantom{x}\textstyle \sim\phantom{x}} \VWd(\alpha,\beta)^{\op{opp}}-\op{mod}
\end{gather*}
of highest weight categories, i.e. $\cE P(\la)\cong P(\varphi(\la))$, $\cE V(\la)\cong V(\varphi(\la))$, $\cE L(\la)\cong L(\varphi(\la))$ for any $\la\in\cS(\delta,\la)$ and $\cE$ interchanges the dualities.
\end{enumerate}
\end{theorem}
\begin{proof}
The first statement follows directly from Propositions~\ref{whichproj} and \ref{saturated}, see \cite[Appendix]{Donkin}. The functor $\cE$ is an equivalence of categories, hence induces a highest weight structure on $\VWd(\alpha,\beta)^{\op{opp}}-\op{mod}$ and it is enough to show that it agrees with the one defined in \cite{AMR}. Using Remark~\ref{Wop} we can find an isomorphism $\VWd(\alpha,\beta)^{\op{opp}}\cong \VWd(\alpha,\beta)$ which we fix. Theorem~\ref{iso} induces an isomorphism $\END_{\mg}(\MdV)^{\op{opp}}\cong \END_{\mg}(\MdV)$.
Under the identifications $\cO^\pp(n,\delta,d)=\END_{\mg}(\MdV)^{\op{opp}}-\op{mod}=\END_{\mg}(\MdV)-\op{mod}$, the duality on $\cO^\pp(n,\delta,d)$ corresponds to the ordinary duality $\HOM_{\END_{\mg}(\MdV)}(?,\mC)$ and hence our equivalence intertwines the dualities.
Since $\varphi$ induces a bijection between the labelling sets of the two highest weight structures, it is enough to show that one ordering refines the other which is done in Lemma~\ref{refinedordering} below.
\end{proof}

\begin{remark}
Due to the fact that a Verma path only include diagrammatic weights with the same integrality as $\de$, $\Op(n,\delta,d)$ is either a subcategory of $\Op_1(n)$ or $\Op_\hint(n)$ depending on whether $\delta$ is even or odd.
\end{remark}

\begin{definition}
The \emph{dominance ordering} on $\cP^1$ is defined by
\begin{eqnarray*}
\la\unrhd\mu\text{ if $|\la|>|\mu|$ or $|\la|=|\mu|$} &\text{and}& \sum_{j=1}^k\la_j\geq \sum_{j=1}^k\mu_j\text{ for any $k\geq 1$};
\end{eqnarray*}
 and on $\cP^2$ by $\bla=(\la^{1},\la^{2})\unrhd\bmu=(\mu^{1},\mu^{2})$
if 
\abovedisplayskip0.3em
\belowdisplayskip0.3em
\begin{gather*}
|\lambda|>|\mu|\quad\text{or}\quad
\left\lbrace\begin{array}{c}
|\lambda|=|\mu|\, , \,\,
\la^{(1)}\unrhd\mu^{(1)}\,\, \text{and} \\
|\la^{(1)}|+\sum_{j=1}^k\la_j^{(2)}\geq |\mu^{(1)}|+\sum_{j=1}^k\mu_j^{(2)}\,\, \text{for all} \,\, k\geq 0
\end{array}\right\rbrace.
\end{gather*}
\end{definition}
This order refines the reversed partial ordering on $\cS$ via the bijection $\varphi$:

\begin{lemma}
\label{refinedordering}
For any $\la,\mu\in\cS(\delta,d)$ we have, with $\varphi$ as in Definition \ref{def:associated bipartition},
\abovedisplayskip0.3em
\belowdisplayskip0.3em
\begin{gather*}
\la\leq_{\op{rBruhat}}\mu \,\, \Rightarrow\,\, \varphi(\la)\unrhd\varphi(\mu).
\end{gather*}
\end{lemma}

\begin{proof}
Recall that the Bruhat ordering is generated by basic swaps, turning two symbols $\up\down$ (in this order with only $\circ$'s and $\cross$'s between them) into a $\down\up$ or two $\up$'s into two $\down$'s making the weight smaller. In the first case, the $\up$ moves to the right and a $\down$ moves to the left. The $\up$ increases the number of boxes in the first partition, whereas the $\down$ increases the number of boxes in the first partition or decreases the number of boxes in the second partition. In any case however, the total number of boxes removed is at most the total number of boxes added. Similarly, in case two if two $\up$'s turn into two $\down$'s, the total number of boxes increases or stays the same, but the number of boxes in the first partition increases. In both cases, the newly created bipartition is larger in the dominance order.
\end{proof}

Due to our assumptions on $n \geq 2d$ and thus the simple structure of the idempotent $z_d$, the category $z_d\END_\mg(\MdV)^{\op{opp}}z_d-\op{mod}$ is equivalent to a quotient category of $\cO^\pp(n,\delta,d)$.

For $\cO^\pp_1(n)$, resp. $\cO^\pp_\hint(n)$, we denote by $\mathrm{pr}_\delta$ the projection onto those blocks of diagrammatic weights $\lambda$ satisfying $\lambda_i=\circ$ for $i \geq \nicefrac{\delta}{2}+n$. For even $\delta$ we put $\widetilde{\cF}=\mathrm{pr}_\delta \, \cF \, \mathrm{pr}_\delta$, otherwise $\widetilde{\cF}=\mathrm{pr}_\delta \, \cF^\hint \, \mathrm{pr}_\delta$. Let $P_d = \widetilde{\cF}^d(M^\pp(\underline{\delta}))$. Then
\abovedisplayskip0.3em
\belowdisplayskip0.3em
\[
\END_{\mg}(P_d)= z_d\END_\mg(\MdV)z_d \cong z_d\VWd(\alpha,\beta)^{\op{opp}}z_d,
\]
with the latter isomorphism due to Theorem \ref{iso}. We have the quotient functor
\abovedisplayskip0.45em
\belowdisplayskip0.45em
\begin{gather*}
\Hom_\mg(P_d,?)\,:\,\cO(n,\delta,d)\,\,\longrightarrow \,\, z_d\VWd(\alpha,\beta)z_d-\op{mod}.
\end{gather*}
We will denote the corresponding quotient category by $\cO^\pp(n,\delta,d)'$. The proof of the following result is standard and can be found Section~\ref{sec:proofOprime}.

\begin{theorem}
\label{Oprime}
The highest weight structure on $\cO^\pp(n,\delta,d)$ induces a highest weight structure on $\cO^\pp(n,\delta,d)'$ iff $\delta\not=0$ or $\delta=0$ and $d$ odd or $d=0$.
\end{theorem}

\begin{remark}
The truncated algebra $z_d\VWd(\alpha,\beta)z_d$ inherits always a cellular algebra structure in the sense of \cite{Graham-Lehrer} from $\VWd(\alpha,\beta)$, see \cite{Koenig-Xi}. By Theorem~\ref{Oprime},  $z_d\VWd(\alpha,\beta)z_d$  is quasi-hereditary if and only if  $\delta\not=0$ or $\delta=0$ and $d$ odd.
\end{remark}

\section{Koszulity and Gradings}
\label{sec:Koszul}
We now deduce the existence of a hidden grading on the cyclotomic VW-algebras. Recall that any block $\mathcal{B}$ of ordinary or parabolic category $\cO$ for $\mg$ is equivalent to the category of modules over the endomorphism algebra $A=\END_\mg(P_\mathcal{B})$ of a minimal projective generator $P_\mathcal{B}$. By \cite{BGS}, $A$ has a natural Koszul grading which we fix. The category of finite dimensional graded $A$-modules is called the {\it graded version} of $\cB$ and denoted $\hat{\cB}$. In case of $\cO^\pp(n)$, the algebra $A$ is the Khovanov algebra from \cite{ES_diagrams} attached to the block.

\begin{theorem}[Koszulity] Let $d,\delta\in\mZ_{\geq 0}$.
\begin{enumerate}
\item The algebra $\VWd(\alpha,\beta)$ is Morita equivalent to a Koszul algebra.
\item The algebra $z_d\VWd(\alpha,\beta)z_d$ is Morita equivalent to a Koszul algebra if and only if $\delta\not=0$ or $\delta=0$ and $d$ odd or $d=0$.
\end{enumerate}
\end{theorem}

\begin{proof}
The category $\Op(n)$ is a highest weight category and each block is equivalent to a category of finite dimensional modules over a Koszul algebra $A=\END_\mg(P)$, \cite{BGS}. Then $A$ is standard Koszul by \cite[Corollary 3.8]{ADL}. Since the set $\cS$ is saturated, the algebra $\VWd(\alpha,\beta)$ is again standard Koszul by \cite[Proposition 1.11]{ADL}. Since it is standard Koszul and quasi-hereditary it is Koszul by \cite[Theorem 1]{ADL}. The second statement follows by the same arguments using Theorem~\ref{Oprime}.
\end{proof}

Any indecomposable projective, Verma or simple module in ${\cB}$ as above has a graded lift in $\hat{\cB}$ by \cite[Theorem 2.1]{StDuke} and this lift is unique up to isomorphism and overall shift in the grading, \cite[Lemma 1.5]{Stroppel}. We choose a lift such that the head is concentrated in degree zero and call it {\it the standard lift}. 

\begin{remark}
The Koszul grading on the basic algebra underlying $\VWd(\alpha,\beta)$ or $z_d\VWd(\alpha,\beta)z_d$ is unique up to isomorphism \cite[Proposition 2.5.1]{BGS}. Then, by \cite{SS},  Proposition~\ref{singO} can be refined to graded decomposition numbers given by parabolic Kazhdan-Lusztig polynomials of type $(\mathrm{D}_n, \mathrm{A}_{n-1})$, see \cite{LS} for explicit formulas and \cite[Section 2.1]{ES_diagrams} for an explicit description of the underlying basic Koszul algebra.
\end{remark}

\begin{prop}
\label{lem:gradeda}
The special projective functors from Definition~\ref{def:functors} lift to
\abovedisplayskip0.3em
\belowdisplayskip0.3em
\begin{gather*}
\begin{aligned}
\hat{\cF} &\cong& \bigoplus_{i\in \mZ_{\geq 0}+\half}(\hat{\cF}_{i,-}\oplus \hat{\cF}_{i,+})&:&\hat{\cO}^\pp_1(n) \rightarrow \hat{\cO}^\pp_1(n)& \quad\text{ and}\\
\hat{\cF}^\hint &\cong& \bigoplus_{i\in \mZ_{> 0}}(\hat{\cF}_{i,-}\oplus \hat{\cF}_{i,+}) \oplus \hat{\cF}_0&:&\hat{\cO}^\pp_\hint(n)\rightarrow\hat{\cO}^\pp_\hint(n)&
\end{aligned}
\end{gather*}
satisfying the properties of Proposition~\ref{a} by replacing the $\cF$ by $\hat{\cF}$ and including the degree shifts $\langle i\rangle$.
\end{prop}

\begin{proof}
By Remark~\ref{rem:transl} each summand is a translation functor to or out of a wall or an equivalence of categories, hence the existence of graded lifts is given by \cite[Theorems 8.1 and 8.2]{Stroppel}. (More precisely we first consider the graded version of the non-parabolic category $\mathcal{O}$ where the graded lifts of the translation functors exist and then  observe that they induce also graded lifts for the parabolic version as in \cite[Section 2.1]{StDuke}. To deduce the formulas note that our normalization is slightly different here, since we shift the translation functor to the wall in \cite{Stroppel} by $\langle 1\rangle$ and the translation functor out of the wall by $\langle -1\rangle$. Alternatively one could also mimic the arguments in \cite{BSIII} for the Khovanov algebra of type $ \mathrm{D}$ from \cite{ES_diagrams}.
\end{proof}

Given a graded abelian category $\cA$, the Grothendieck group is a $\mZ[q,q^{-1}]$-module where $q^r$ acts by shifting the grading up by $r$. The $\mQ(q)$-vector space obtained by extension of scalars is denoted $K_0(\cA)$. In particular, we have $K_0\left(\hat{\cO}^\pp(n)\right)$ with the induced action of the graded special projective functors. We describe this now in detail and generalize it to more general parabolics.

\section{Coideal subalgebras and categorification theorems}
\label{section:coideal}

\subsection{Quantum groups}
Our next goal is to introduce certain quantum symmetric pairs that is coideals of quantum groups. From now we work over the ground field $\mQ(q)$ of rational functions. The following constructions fall again into two families, the even/half-integer family and the odd/integer family. To be able to treat these cases simultaneously as much as possible we introduce the following indexing sets:
\abovedisplayskip0.3em
\belowdisplayskip0.3em
\begin{gather*}
\begin{aligned}
\mathrm{I}=&\,\mZ + \half, &\mathrm{I}^{++}=&\,\{i \in \mathrm{I} \mid i \geq \nicefrac{3}{2}\}, & \mathrm{J}=&\,\mZ, & \mathrm{J}^+=&\,\{j \in \mathrm{J} \mid j \geq 1\},\\
\mathrm{I}^\hint=&\,\mZ, &\mathrm{I}^{\hint,++}=&\,\{i \in \mathrm{I}^\hint \mid i \geq 1\}, & \mathrm{J}^\hint=&\,\mZ + \half, & \mathrm{J}^{\hint,+}=&\,\{j \in \mathrm{J}^\hint \mid j \geq \half\}.
\end{aligned}
\end{gather*}

%\begin{array}{llcll}
%\mathrm{I}=\mZ + \half, & \mathrm{I}^{++}=\{i \in \mathrm{I} \mid i \geq \nicefrac{3}{2}\},& \quad \quad &J=\mZ, & J^+=\{j \in J \mid j \geq 1\},\\
%\mathrm{I}^\hint=\mZ, & \mathrm{I}^{\hint,++}=\{i \in \mathrm{I} \mid i \geq 1\}, & \quad \quad &J^\hint=\mZ + \half, & J^{\hint,+}=\{j \in J \mid j \geq \half\}.
%\end{array}

\begin{definition}
Let $\cU=\mathcal{U}_q(\mathfrak{gl}_{\mZ})$ be the \emph{(generic) quantized universal enveloping algebra} for $\mathfrak{gl}_{\mZ}$. That means $\cU$ is the associative unital algebra over  $\mQ(q)$ in an indeterminate $q$ with generators $D_{j}^{\pm 1}$ for  $j \in \mathrm{J}$ and $E_{i}$, $F_{i}$ for $i \in \mathrm{I}$ subject to the following relations:
\begin{enumerate}[(i)]
\item all $D_{j}^{\pm 1}$ commute and $D_{j} D_{j}^{-1} = 1$ for $j \in \mathrm{J}$,
\item For $i \in \mathrm{I}$ and $j \in \mathrm{J}$:
$$D_{j}E_{i}D_{j}^{-1} = \left\lbrace \begin{array}{ll}
q E_{i} & \text{for } j=i-\half \\
q^{-1} E_{i} & \text{for } j=i+\half \\
E_{i} & \text{otherwise}
\end{array} \right. \qquad D_{j}F_{i}D_{j}^{-1} = \left\lbrace \begin{array}{ll}
q F_{i} & \text{for } j=i+\half \\
q^{-1} F_{i} & \text{for } j=i-\half \\
F_{i} & \text{otherwise}
\end{array} \right.$$

\item $E_{i}$ and $F_{i'}$ commute unless $i=i'$ and in this case
$$E_{i}F_{i}-F_{i}E_{i} = \frac{D_{i-\half}D_{i+\half}^{-1}-D_{i+\half}D_{i-\half}^{-1}}{q-q^{-1}}.$$
\item If $| i-i' | =1$ then
$$ E_{i}^2E_{i'} - (q + q^{-1})E_{i}E_{i'}E_{i} + E_{i'} E_{i}^2 = 0$$
in all other cases $E_{i}E_{i'} = E_{i'}E_{i}$.
\item Analogously, if $| i-i' | =1$ then
$$ F_{i}^2F_{i'} - (q + q^{-1})F_{i}F_{i'}F_{i} + F_{i'} F_{i}^2 = 0$$
in all other cases $F_{i}F_{i'} = F_{i'}F_{i}$.
\end{enumerate}

The \emph{quantized enveloping algebra} $\cU^\hint=\mathcal{U}_q(\mathfrak{gl}_{\mZ^\hint})$ is defined in exactly the same way by replacing $\mathrm{I}$ with $\mathrm{I}^\hint$ and $\mathrm{J}$ with $\mathrm{J}^\hint$ in the definitions.
\end{definition}

In $\cU$, respectively $\cU^\hint$, let $K_{i} = D_{i-\half}D_{i+\half}^{-1}$ for $i \in \mathrm{I}$, respectively $i \in \mathrm{I}^\hint$. From this definition it follows immediately that
$$K_{i} E_{i} K_{i}^{-1} = q^2 E_{i} \qquad \text{ and }\quad K_{i} F_{i} K_{i}^{-1} = q^{-2} F_{i}.$$

Both $\cU$ and $\cU^\hint$ are Hopf algebras with comultiplication $\Delta$ defined by
\abovedisplayskip0.3em
\belowdisplayskip0.3em
\begin{gather*}
\Delta(E_{i}) = K_{i} \otimes E_{i} + E_{i} \otimes 1, \,\,\, \Delta(F_{i}) = F_{i} \otimes K_{i}^{-1} + 1  \otimes F_{i}, \,\,\, \Delta(D_{j}^{\pm 1}) = D_{j}^{\pm 1} \otimes D_{j}^{\pm 1}.
\end{gather*}

\begin{definition}
\label{Def1}
Let $\mV$ be the vector space on basis $\{v_l\mid l \in \mathrm{J}\}$. The \emph{natural representation} of $\cU$ is the representation on $\mV$ given by the rules
\[
E_{i} v_{l} = \delta_{i+\half, l} v_{l-1} \qquad F_{i} v_l = \delta_{i-\half,l} v_{l+1} \qquad D_{j}^{\pm 1} v_l = q^{\pm \delta_{j,l}} v_l
\]
for $i \in \mathrm{I}$ and $j \in \mathrm{J}$. The \emph{natural representation} $\mV^\hint$ of $\cU^\hint$ is defined analogously using $\mathrm{J}^\hint$ as the indexing set for the basis; all formulas stay the same.
\end{definition}

\begin{remark}
The action on $\mV$, respectively on  $\mV^\hint$, can be illustrated as follows, showing how the
$F_{i}$'s and $E_{i}$'s move between weight spaces and
on which weight space $D_{j}$ acts by multiplication with $q$:
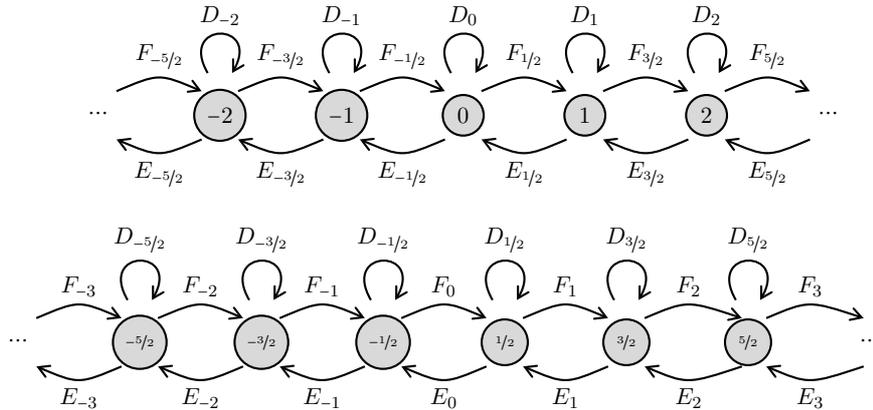
\begin{center}
\begin{tikzpicture}[thick,>=angle 60,scale=0.8, transform shape]
\node (-3) at (0,2) {$\cdots$};

\node[draw,shape=circle, fill=gray!30] (-2) at (2,2) {$-2$};
\draw [->] (0.3,2.4) .. controls (1,2.7) .. (1.7, 2.4) node[pos=.5,above] {$F_{-\nicefrac{5}{2}}$};
\draw [<-] (0.3,1.6) .. controls (1,1.3) .. (1.7, 1.6) node[pos=.5,below] {$E_{-\nicefrac{5}{2}}$};
\path (1.8,2.7) edge[out=120, in=60, looseness=0.5, loop, distance=1cm, ->] node[pos=.5,above] {$D_{-2}$} (2.2,2.7);

\node[draw,shape=circle, fill=gray!30] (dotsl) at (4,2) {$-1$};
\draw [->] (2.3,2.4) .. controls (3,2.7) .. (3.7, 2.4) node[pos=.5,above] {$F_{-\nicefrac{3}{2}}$};
\draw [<-] (2.3,1.6) .. controls (3,1.3) .. (3.7, 1.6) node[pos=.5,below] {$E_{-\nicefrac{3}{2}}$};

\path (3.8,2.7) edge[out=120, in=60, looseness=0.5, loop, distance=1cm, ->] node[pos=.5,above] {$D_{-1}$} (4.2,2.7);

\node[draw,shape=circle, fill=gray!30] (0) at (6,2) {$0$};
\draw [->] (4.3,2.4) .. controls (5,2.7) .. (5.7, 2.4) node[pos=.5,above] {$F_{-\nicefrac{1}{2}}$};
\draw [<-] (4.3,1.6) .. controls (5,1.3) .. (5.7, 1.6) node[pos=.5,below] {$E_{-\nicefrac{1}{2}}$};

\path (5.8,2.7) edge[out=120, in=60, looseness=0.5, loop, distance=1cm, ->] node[pos=.5,above] {$D_{0}$} (6.2,2.7);

\node[draw,shape=circle, fill=gray!30] (1) at (8,2) {$1$};
\draw [->] (6.3,2.4) .. controls (7,2.7) .. (7.7, 2.4) node[pos=.5,above] {$F_{\nicefrac{1}{2}}$};
\draw [<-] (6.3,1.6) .. controls (7,1.3) .. (7.7, 1.6) node[pos=.5,below] {$E_{\nicefrac{1}{2}}$};

\path (7.8,2.7) edge[out=120, in=60, looseness=0.5, loop, distance=1cm, ->] node[pos=.5,above] {$D_{1}$} (8.2,2.7);

\node[draw,shape=circle, fill=gray!30] (dotsr) at (10,2) {$2$};
\draw [->] (8.3,2.4) .. controls (9,2.7) .. (9.7, 2.4) node[pos=.5,above] {$F_{\nicefrac{3}{2}}$};
\draw [<-] (8.3,1.6) .. controls (9,1.3) .. (9.7, 1.6) node[pos=.5,below] {$E_{\nicefrac{3}{2}}$};

\path (9.8,2.7) edge[out=120, in=60, looseness=0.5, loop, distance=1cm, ->] node[pos=.5,above] {$D_{2}$} (10.2,2.7);

\node (n-1) at (12,2) {$\cdots$};
\draw [->] (10.3,2.4) .. controls (11,2.7) .. (11.7, 2.4) node[pos=.5,above] {$F_{\nicefrac{5}{2}}$};
\draw [<-] (10.3,1.6) .. controls (11,1.3) .. (11.7, 1.6) node[pos=.5,below] {$E_{\nicefrac{5}{2}}$};

\end{tikzpicture}
\end{center}

\begin{tikzpicture}[thick,>=angle 60,scale=0.8, transform shape]
\node (-4) at (0,2) {$\cdots$};
\node[draw,shape=circle, fill=gray!30] (-3) at (2,2) {\tiny $-\nicefrac{5}{2}$};
\draw [->] (0.3,2.4) .. controls (1,2.7) .. (1.7, 2.4) node[pos=.5,above] {$F_{-3}$};
\draw [<-] (0.3,1.6) .. controls (1,1.3) .. (1.7, 1.6) node[pos=.5,below] {$E_{-3}$};

\path (1.8,2.7) edge[out=120, in=60, looseness=0.5, loop, distance=1cm, ->] node[pos=.5,above] {$D_{-\nicefrac{5}{2}}$} (2.2,2.7);

\node[draw,shape=circle, fill=gray!30] (dotsl) at (4,2) {\tiny $-\nicefrac{3}{2}$};
\draw [->] (2.3,2.4) .. controls (3,2.7) .. (3.7, 2.4) node[pos=.5,above] {$F_{-2}$};
\draw [<-] (2.3,1.6) .. controls (3,1.3) .. (3.7, 1.6) node[pos=.5,below] {$E_{-2}$};

\path (3.8,2.7) edge[out=120, in=60, looseness=0.5, loop, distance=1cm, ->] node[pos=.5,above] {$D_{-\nicefrac{3}{2}}$} (4.2,2.7);

\node[draw,shape=circle, fill=gray!30] (-1) at (6,2) {\tiny $-\nicefrac{1}{2}$};
\draw [->] (4.3,2.4) .. controls (5,2.7) .. (5.7, 2.4) node[pos=.5,above] {$F_{-1}$};
\draw [<-] (4.3,1.6) .. controls (5,1.3) .. (5.7, 1.6) node[pos=.5,below] {$E_{-1}$};

\path (5.8,2.7) edge[out=120, in=60, looseness=0.5, loop, distance=1cm, ->] node[pos=.5,above] {$D_{-\nicefrac{1}{2}}$} (6.2,2.7);

\node[draw,shape=circle, fill=gray!30] (1) at (8,2) {\tiny $\nicefrac{1}{2}$};
\draw [->] (6.3,2.4) .. controls (7,2.7) .. (7.7, 2.4) node[pos=.5,above] {$F_{0}$};
\draw [<-] (6.3,1.6) .. controls (7,1.3) .. (7.7, 1.6) node[pos=.5,below] {$E_{0}$};

\path (7.8,2.7) edge[out=120, in=60, looseness=0.5, loop, distance=1cm, ->] node[pos=.5,above] {$D_{\nicefrac{1}{2}}$} (8.2,2.7);

\node[draw,shape=circle, fill=gray!30] (dotsr) at (10,2) {\tiny $\nicefrac{3}{2}$};
\draw [->] (8.3,2.4) .. controls (9,2.7) .. (9.7, 2.4) node[pos=.5,above] {$F_{1}$};
\draw [<-] (8.3,1.6) .. controls (9,1.3) .. (9.7, 1.6) node[pos=.5,below] {$E_{1}$};

\path (9.8,2.7) edge[out=120, in=60, looseness=0.5, loop, distance=1cm, ->] node[pos=.5,above] {$D_{\nicefrac{3}{2}}$} (10.2,2.7);

\node[draw,shape=circle, fill=gray!30] (n-1) at (12,2) {\tiny $\nicefrac{5}{2}$};
\draw [->] (10.3,2.4) .. controls (11,2.7) .. (11.9, 2.4) node[pos=.5,above] {$F_{2}$};
\draw [<-] (10.3,1.6) .. controls (11,1.3) .. (11.9, 1.6) node[pos=.5,below] {$E_{2}$};

\path (11.8,2.7) edge[out=120, in=60, looseness=0.5, loop, distance=1cm, ->] node[pos=.5,above] {$D_{\nicefrac{5}{2}}$} (12.2,2.7);

\node (n) at (14,2) {$\cdots$};
\draw [->] (12.3,2.4) .. controls (13,2.7) .. (13.9, 2.4) node[pos=.5,above] {$F_{3}$};
\draw [<-] (12.3,1.6) .. controls (13,1.3) .. (13.9, 1.6) node[pos=.5,below] {$E_{3}$};

\end{tikzpicture}
\end{remark}

The modules we are interested in are the quantum analogues of the fundamental representations of $\mathfrak{gl}_\mZ$.

\begin{definition}
The $k$-th \emph{quantum exterior power} $\bigwedge_q^k \mV$ is the $\cU$-submodule of $\bigotimes^k \mV$ with basis given by the vectors
\abovedisplayskip0.3em
\belowdisplayskip0.3em
\begin{gather}
\label{bvs}
\begin{aligned}
v_{\bf i}=v_{i_1} \wedge\cdots \wedge v_{i_k} =&\,
\sum_{w \in S_k} (-q)^{\ell(w)} v_{i_{w(1)}} \otimes\cdots\otimes
v_{i_{w(k)}}
\end{aligned}
\end{gather}
for all $i_1 < \cdots < i_k$ from the index set $\mathrm{J}$. Here, $\ell(w)$ denotes the usual length of a permutation $w$ in the symmetric group $S_k$.

The corresponding module $\bigwedge_q^k \mV^\hint$ for $\cU^\hint$ is defined in the same way by again just changing the indexing set for the basis.
\end{definition}

\begin{definition}
\label{Def2}
The \emph{quantized universal enveloping algebra} $\mathcal{U}_q(\mathfrak{gl}_\mN)$ is defined as the associative unital algebra over the field of rational functions $\mQ(q)$ with the sets of generators
\[
\{ E_i \mid i \in \mathrm{I}^{++} \} \cup \{F_i \mid i \in \mathrm{I}^{++} \} \cup \{D_j^{\pm 1} \mid j \in \mathrm{J}^+\},
\]
which satisfy the same relations as the generators of $\cU$ by identifying $X_i \in U_q(\mathfrak{gl}_\mN)$ with $X_{i}$ in $\cU$ for $X \in \{E,F,D\}$.
\end{definition}

\begin{definition}
\label{Def3}
Let $\mW$ denote the natural representation of $\mathcal{U}_q(\mathfrak{gl}_\mN)$ with basis $\{ w_j \mid j \in \mathrm{J}^+\}$ and consider the $\mathcal{U}_q(\mathfrak{gl}_\mN)$-module $\bigwedge_q^r \mW \otimes \bigwedge_q^{k-r} \mW$, with monomial basis
\abovedisplayskip0.3em
\belowdisplayskip0.3em
\begin{gather}\label{mb1}
\left\lbrace
(w_{i_1}\wedge\cdots\wedge w_{i_r}) \otimes
(w_{j_1}\wedge\cdots\wedge w_{j_{k-r}}) \, \bigg| \,
\begin{array}{l}
i_1,\dots,i_r,j_1,\dots,j_{k-r} \in \mathrm{J}^+,\\
i_1<\cdots < i_r , \, j_1 < \cdots < j_{k-r}
\end{array}\right\rbrace.
\end{gather}
As before we can use $\mathrm{I}^{\hint,++}$ and $\mathrm{J}^{\hint,+}$ instead to define $\mathcal{U}_q(\mathfrak{gl}_\mN)$ with corresponding natural representation $\mW^\hint$.
\end{definition}

\begin{remark}
Although the definition of $\mathcal{U}_q(\mathfrak{gl}_\mN)$ gives a natural embedding of Hopf algebras into both, $\cU$ and $\cU^\hint$, we will be interested in another copy of $\mathcal{U}_q(\mathfrak{gl}_\mN)$ inside $\cU$ and $\cU^\hint$ given by certain involution invariant generators.
\end{remark}

\subsection{Coideal subalgebra}

We are interested in the quantization of the fixed point Lie algebra $\mathfrak{g}^\theta$ for a certain involution $\theta$. This involution does not lift to an embedding of $\mathcal{U}_q(\mathfrak{g}^\theta)$ into $\mathcal{U}_q(\mathfrak{g})$ in any obvious way. A way around this is to study coideal subalgebras (instead of Hopf subalgebras) of $\mathcal{U}_q(\mathfrak{g})$ which specialize to $\mathfrak{g}^\theta$. This was studied in \cite{Letzter} and \cite{LetzterII}, see also \cite{Kolb}. 

\begin{definition}\label{def:modified_generators}
We fix the following elements in $\cU$
\abovedisplayskip0.3em
\belowdisplayskip0.2em
\begin{gather*}
B_i = E_i K_{-i}^{-1} + F_{-i} \,\, (\text{for } i \in \mathrm{I} \setminus \{\pm \half\}),\\
B_{\half} = E_{\half} K_{-\half}^{-1} + F_{-\half},\qquad\text{and}\qquad
B_{-\half}=q E_{-\half} K_{\half}^{-1} + F_{\half}
\end{gather*}
and in $\cU^\hint$ we fix
\abovedisplayskip0.3em
\belowdisplayskip0.3em
\begin{gather*}
B_i = E_i K_{-i}^{-1} + F_{-i} \,\, (\text{for } i \in \mathrm{I}^\hint \setminus \{0\})\quad  \text{and} \quad B_0 = q^{-1}E_0K_0^{-1} + F_0.
\end{gather*}
\end{definition}

\begin{definition}
Let $\mathcal R$ denote the commutative subalgebra of $\cU$ generated by the elements of the form $(D_{j}D_{-j})^{\pm 1}$ for all $j \in \mathrm{J}$. Analogously we have $\cR^\hint$ the commutative subalgebra of $\cU^\hint$ generated by all $(D_{j}D_{-j})^{\pm 1}$ for all $j \in \mathrm{J}^\hint$.
\end{definition}

\begin{definition}
Let $\mathcal{H}$ be the subalgebra of $\cU$ generated by the $B_{i}$ for $i \in \mathrm{I}$ and the subalgebra $\mathcal{R}$. In analogy, let $\mathcal{H}^\hint$ be the subalgebra of $\cU^\hint$ generated by the $B_{i}$ for $i \in \mathrm{I}^\hint$ and the subalgebra $\mathcal{R}^\hint$. We call $\cH$ and $\cH^\hint$ the \emph{coideal subalgebras} of $\cU$ respectively $\cU^\hint$.
\end{definition}

The following lemma is straight-forward and justifies the name coideal.
\begin{lemma}
We have $\Delta(\cH) \subset \cH \otimes \cU$ and $\Delta(\cH^\hint) \subset \cH^\hint \otimes \cU^\hint$, hence $\cH \subset \cU$ and $\cH^\hint \subset \cU^\hint$ are right coideal subalgebras of $\cU$ respectively $\cU^\hint$.
\end{lemma}

Although $\cH$ and $\cH^\hint$ are not quantum groups, they each contain a copy of the quantum group $\cU_q(\mathfrak{gl}_\mN)$:

\begin{lemma} \label{lem:gln_inclusion}
The assignment
\abovedisplayskip0.3em
\belowdisplayskip0.3em
\[
E_i \mapsto B_{i}, \qquad F_i \mapsto B_{-i}, \qquad D_j^{\pm 1} \mapsto \left(D_{j}D_{-j}\right)^{\pm 1},
\]
for $i \in \mathrm{I}^{++}$ and $j \in \mathrm{J}^+$ extends to an injective map of algebras from $\mathcal{U}_q(\mathfrak{gl}_\mN)$ to $\cH$. We denote the image of this embedding by $\check{\cU}$.
\end{lemma}
\begin{proof}
Obviously the images of the $\check{D}_j^{\pm 1}$'s are pairwise inverse elements that form a commutative subalgebra. Furthermore one checks that
\abovedisplayskip0.3em
\belowdisplayskip0.3em
\[
D_{j}D_{-j} B_{i} (D_{j}D_{-j})^{-1}= \begin{cases}
q B_{i}& \text{if } j=i-\half, \\
q^{-1}B_{i} & \text{if } j=i+\half, \\
B_{i} & \text{otherwise}. \\
\end{cases}
\]
and that the analogous statement for $B_{-i}$ with inverse $q$'s holds as well.
Computing the commutator we obtain that $[B_{i},B_{-i}]$ equals
\abovedisplayskip0.3em
\belowdisplayskip0.3em
\[
[E_{i}K_{-i}^{-1},F_{i}] + [F_{-i},E_{-i}K_i^{-1}]=\frac{K_iK_{-i}^{-1} - K_{-i}K_i^{-1}}{q-q^{-1}},
\]
while $[B_{i},B_{-j}]=0$ if $i \neq j$. It remains to verify the quantum Serre relations, which is a long, but straight-forward calculation and therefore omitted.
\end{proof}

\begin{remark}
\label{lem:gln_inclusion2}
By using the indexing set $\mathrm{I}^{\hint,++}$ and $\mathrm{J}^{\hint,+}$ for $\mathcal{U}_q(\mathfrak{gl}_\mN)$ we obtain the same embedding for $\cH^\hint$ and denote its image by $\check{\cU}^\hint$.
\end{remark}

\subsection{Quantized fixed points subalgebras}

If we truncate our indexing set at $m$, we obtain the coideal subalgebra of type $\op{(AIII)}$ discussed in \cite[Section 7]{Letzter} specializing to the Lie subalgebra $\mathfrak{gl}_m \times \mathfrak{gl}_m \subset \mathfrak{gl}_{2m}$ (resp. $\mathfrak{gl}_m \times \mathfrak{gl}_{m+1}$) of fixed points under a certain involution $\Theta$ discussed in detail in Section~\ref{subsection:fixedpoints}. The above embeddings from Lemma~\ref{lem:gln_inclusion} and Remark~\ref{lem:gln_inclusion2} quantize what we call the diagonally embedded $\mathfrak{gl}_m$ there. To make the precise connection to \cite{Letzter} and \cite{Kolb} we consider the following quantization of $\Theta$.

\begin{lemma}
\label{lem:theta}
The assignments
\abovedisplayskip0.3em
\belowdisplayskip0.3em
\begin{gather*}
E_{i}  \mapsto q K_{-i}F_{-i}, \quad F_{i}  \mapsto  q^{-1} E_{-i}K_{-i}^{-1},\quad D_{j}  \mapsto  D_{-j}^{-1}, \quad K_{i}  \mapsto  K_{-i}, \quad q  \mapsto  q^{-1}
\end{gather*}
define a $q$-antilinear involution $\theta$ on both $\cU$ and $\cU^\hint$.
\end{lemma}
\begin{proof}
This is a straightforward calculation left to the reader.
\end{proof}

To match \cite{Letzter}, we enlarge the quantum group $\cU$ by including elements $\dd_j^{\pm 1}$ for $j\in \mathrm{J}$, satisfying $\dd_j^2=D_j$ and similarly $\kk_{i}$ for $i \in \mathrm{I}$. The relations for these are the same as for their squares, with powers of $q$ divided by two. Hence, for this, we use $\mQ(q^\half)$ as a base field. We call this algebra $\widetilde{\cU}$ and construct $\widetilde{\cU}^\hint$ analogously and fix the obvious extension of $\theta$ to these algebras.

\begin{prop}
The involution $\theta:\widetilde{\cU} \rightarrow \widetilde{\cU}$ fixes the elements 
\begin{gather*}
\widetilde{B}_{i} = q^{-\half} B_i \, \kk_{-i} \kk_i^{-1} \,\, (\text{for } i \in \mathrm{I} \setminus \{\pm\half\}),\\
\widetilde{B}_{\half} = q^{-\half} B_\half \,  \kk_{-\half} \kk_\half^{-1}, \quad \text{and} \quad \widetilde{B}_{-\half} = q B_{-\half} \, \kk_{\half} \kk_{-\half}^{-1}.
\end{gather*}
The involution $\theta:\widetilde{\cU}^\hint \rightarrow \widetilde{\cU}^\hint$ fixes the elements
\begin{gather*}
\widetilde{B}_{i} = q^{-\half} B_i \, \kk_{-i} \kk_i^{-1} \,\, (\text{for } i \in \mathrm{I}^\hint \setminus \{0\}) \quad \text{and} \quad \widetilde{B}_{0} = B_0.
\end{gather*}
\end{prop}
\begin{proof}
This is just a straight-forward calculation. For example one checks $ \theta \left(q^{-\half} E_{i}\kk_{i}^{-1}\kk_{-i}^{-1} \right) = q^{\nicefrac{3}{2}} K_{-i}F_{-i}\kk_{-i}^{-1}\kk_{i}^{-1} = q^{-\half}F_{-i}\kk_{-i}\kk_{i}^{-1}.$
\end{proof}

\begin{remark}
Note that elements $\kk_{-i} \kk_{i}^{-1}$ are in subalgebras $\widetilde{\mathcal R}$ respectively $\widetilde{\mathcal R}^\hint$ generated by the $(\dd_{i}\dd_{-1})^{\pm 1}$. Thus together with these subalgebras, the $B_i$'s and $\widetilde{B}_i$'s generate the same subalgebra of $\widetilde{\cU}$ respectively $\widetilde{\cU}^\hint$.
\end{remark}

The coideal subalgebras have a Serre type presentation very similar to the quantum groups except that the Serre relations involving the generators $B_{\pm\half}$ and $B_0$ are slightly modified (see \cite[Theorem 7.1]{Letzter} for a general statement):

\begin{prop} \label{prop:relations_coideal}
The coideal subalgebra $\cH$ is (as abstract algebra) isomorphic to the $\mQ(q)$-algebra $\check{\cH}$ with generators
$$ \{\check{E}_i, \, \check{F}_i \mid i \in \mathrm{I}^{++}\} \cup \{\check{B}_+, \check{B}_-\} \cup \{\check{D}_j^{\pm 1} \mid j \in \mathrm{J}^+\} \cup \{\check{D}_0^{\pm 1}\},$$
subject to the following relations
\begin{enumerate}
\item The $\check{D}_i^{\pm 1}$ generate a subalgebra isomorphic to $\cR$.
\item The $\check{E}_i$, $\check{F}_i$, and $\check{D}_j^{\pm 1}$ for $i \in \mathrm{I}^{++}$ and $j \in \mathrm{J}^+$ generate a subalgebra isomorphic to $\cU_q(\mathfrak{gl}_\mN)$.
\item $\check{D}_0$ commutes with $\check{E}_i$ and $\check{F}_i$ for all $i \in \mathrm{I}^{++}$, while $\check{B}_+$ and $\check{B}^-$ commute with all other generators except for the following relations
\begin{enumerate}
\item Commutation relations:
\abovedisplayskip0.3em
\belowdisplayskip0.3em
\begin{gather*}
\begin{aligned}
\check{D}_1 \check{B}_+ \check{D}_1^{-1} =&\, q^{-1} \check{B}_+, & \check{D}_1 \check{B}_- \check{D}_1^{-1} =&\, q \check{B}_-,\\
\check{D}_0 \check{B}_+ \check{D}_0^{-1} =&\, q^2 \check{B}_+, & \check{D}_0 \check{B}_- \check{D}_0^{-1} =&\, q^{-2} \check{B}_-,
\end{aligned}
\end{gather*}
\item Quantum Serre relations:
\abovedisplayskip0.3em
\belowdisplayskip0.3em
\begin{gather*}
\begin{aligned}
\check{B}_+^2 \check{E}_{\nicefrac{3}{2}} - (q+q^{-1})\check{B}_+\check{E}_{\nicefrac{3}{2}}\check{B}_+ + \check{E}_{\nicefrac{3}{2}} \check{B}_+^2 =&\, 0,\\
\check{B}_-^2 \check{F}_{\nicefrac{3}{2}} - (q+q^{-1})\check{B}_-\check{F}_{\nicefrac{3}{2}}\check{B}_- + \check{F}_{\nicefrac{3}{2}} \check{B}_-^2 =&\, 0,\\
\check{E}_{\nicefrac{3}{2}}^2 \check{B}_+ - (q+q^{-1})\check{E}_{\nicefrac{3}{2}}\check{B}_+\check{E}_{\nicefrac{3}{2}} + \check{B}_+ \check{E}_{\nicefrac{3}{2}}^2 =&\, 0,\\
\check{F}_{\nicefrac{3}{2}}^2 \check{B}_- - (q+q^{-1})\check{F}_{\nicefrac{3}{2}}\check{B}_-\check{F}_{\nicefrac{3}{2}} + \check{B}_- \check{F}_{\nicefrac{3}{2}}^2 =&\, 0,\\
\end{aligned}
\end{gather*}
\item Modified quantum Serre relations:
\abovedisplayskip0.3em
\belowdisplayskip0.3em
\begin{gather*}
\begin{aligned}
\check{B}_+^2 \check{B}_- -& (q+q^{-1})\check{B}_+\check{B}_- \check{B}_+ + \check{B}_- \check{B}_+^2 \\
=&\, -(q+q^{-1})\check{B}_+(q\check{D}_1^{-1}\check{D}_0 + q^{-1}\check{D}_1\check{D}_0^{-1}),\\
\check{B}_-^2 \check{B}_+ -& (q+q^{-1})\check{B}_-\check{B}_+ \check{B}_- + \check{B}_+ \check{B}_-^2 \\
=&\, -(q+q^{-1})\check{B}_-(q^2\check{D}_1\check{D}_0^{-1} + q^{-2}\check{D}_1^{-1}\check{D}_0).
\end{aligned}
\end{gather*}
\end{enumerate}
\end{enumerate}
\end{prop}
\begin{proof}
Up to the renormalization analogous to  passing from $B_i$ to $\widetilde{B}_i$ these are the relations from \cite{Letzter} for the coideal subalgebra of type $\op{(AIII})$. The isomorphism is then given by
$ \check{E}_i \mapsto B_{i}$, $\check{F}_i \mapsto B_{-i}$, $\check{D}_j^{\pm 1} \mapsto \left(D_{j}D_{-j}\right)^{\pm 1},$ and $ \check{B}_+ \mapsto B_{\half}$, $\check{B}_- \mapsto B_{-\half}$, $\check{D}_0 \mapsto D_0^2.$
\end{proof}

An analogous statement also holds for $\cH^\hint$.

\begin{prop}
The coideal subalgebra $\cH^\hint$ is isomorphic to the $\mQ(q)$-algebra $\check{\cH}^\hint$ with generators
$$ \{\check{E}_i, \, \check{F}_i \mid i \in \mathrm{I}^{\hint,++}\} \cup \{\check{B}\} \cup \{\check{D}_j^{\pm 1} \mid j \in \mathrm{J}^{\hint,+}\},$$
subject to the following relations
\begin{enumerate}
\item The $\check{D}_i^{\pm 1}$ generate a subalgebra isomorphic to $\cR^\hint$.
\item The $\check{E}_i$, $\check{F}_i$, and $\check{D}_j^{\pm 1}$ for $i \in \mathrm{I}^{\hint,++}$ and $j \in \mathrm{J}^{\hint,+}$ generate a subalgebra isomorphic to $\cU_q(\mathfrak{gl}_\mN)$.
\item The generator $\check{B}$ commutes with all other generators except for the following relations:
\abovedisplayskip0.3em
\belowdisplayskip0.3em
\begin{gather*}
\begin{aligned}
\check{E}_1^2 \check{B} - (q+q^{-1})\check{E}_1\check{B}\check{E}_1 + \check{B} \check{E}_1^2 =&\, 0,\\
\check{F}_1^2 \check{B} - (q+q^{-1})\check{F}_1\check{B}\check{F}_1 + \check{B} \check{F}_1^2 =&\, 0,\\
\check{B}^2 \check{E}_1 - (q+q^{-1})\check{B}\check{E}_1\check{B} + \check{E}_1 \check{B}^2 =&\, \check{E}_1,\\
\check{B}^2 \check{F}_1 - (q+q^{-1})\check{B}\check{F}_1\check{B} + \check{F}_1 \check{B}^2 =&\, \check{F}_1.
\end{aligned}
\end{gather*}
\end{enumerate}
\end{prop}

\section{Categorification and Bar involution on \texorpdfstring{$\bigwedge^n_q \mV$}{exterior powers}}
We start by describing explicitly one of the easiest representations of $\cH$ and $\cH^\hint$, obtained by restricting representations of $\cU$ and $\cU^\hint$. In general the representation theory of $\cH$, respectively $\cH^\hint$, is not yet understood. 
\subsection{Action on \texorpdfstring{$\bigwedge_q^n \mV$}{exterior powers} and categorification theorem}
 The first step is to decompose specific $\cU$-modules with respect to $\check{\cU}$ as defined in Lemma~\ref{lem:gln_inclusion}.

For a sequence $\textbf{i} = (i_1 < \ldots < i_r)$ of strictly decreasing integers we denote by $\textbf{i}_+$, $\textbf{i}_-$, $\textbf{i}_0$ respectively the subsequences of strictly positive numbers, of strictly negative numbers and of entries equal to $0$, and set $-\textbf{i}= (-i_r < \ldots < -i_1)$. Recall the notation from Definitions~\ref{Def1}-\ref{Def3}.
\begin{prop} \label{lem:decomp_module}
There is an isomorphism of $\mathcal{U}_q(\mathfrak{gl}_\mN)$-modules
\abovedisplayskip0.3em
\belowdisplayskip0.3em
\begin{gather*}
\begin{aligned}
\Phi:\,\, {\bigwedge^n}_q \mV \longrightarrow &\,\, \bigoplus_{r=0}^n \left({\bigwedge^r}_q \mW \otimes {\bigwedge^{n-r}}_q \mW \right) \oplus \bigoplus_{r=0}^{n-1} \left({\bigwedge^r}_q \mW \otimes \mQ(q) \otimes {\bigwedge^{n-1-r}}_{\!\!\!\!q} \mW \right) \\
v_{\textbf{i}} \longmapsto& \,\,
\begin{cases}
w_{-\textbf{i}_-} \otimes w_{\textbf{i}_0} \otimes w_{\textbf{i}_+} & \text{if } \textbf{i}_0 \neq \emptyset,\\
w_{-\textbf{i}_-} \otimes w_{\textbf{i}_+} & \text{otherwise,}
\end{cases}
\end{aligned}
\end{gather*}
where $\mathcal{U}_q(\mathfrak{gl}_\mN)$ acts via the identification with $\check{\cU}$ from the left, with $\mQ(q)$ being the trivial representation with basis $\{w_0\}$. Similarly, there is an isomorphism of $\mathcal{U}_q(\mathfrak{gl}_\mN)$-modules
\begin{gather*}
\Phi^\hint :\,\, {\bigwedge^n}_q \mV^\hint \longrightarrow \bigoplus_{r=0}^n \left({\bigwedge^r}_q \mW^\hint \otimes {\bigwedge^{n-r}}_{\!\!q} \mW^\hint \right),\quad\Phi^\hint(v_{\textbf{i}}) = w_{-\textbf{i}_-} \otimes w_{\textbf{i}_+},
\end{gather*}
where $\mathcal{U}_q(\mathfrak{gl}_\mN)$ acts on the left via the isomorphism with $\check{\cU}^\hint$.
\end{prop}
\begin{proof}
We will omit the proof since the second isomorphism is a special case of Proposition~ \ref{lem:decomp_module_general} below, and the first completely analogous.
\end{proof}

The following more general decomposition will be used in Section \ref{section:howe}.

\begin{prop} \label{lem:decomp_module_general}
Let $r \in \mZ_{\geq 1}$ and $k_1,\ldots,k_r \in \mZ_{\geq 1}$. We have an isomorphism 
\abovedisplayskip0.3em
\belowdisplayskip0.3em
\begin{gather*}
\begin{aligned}
\Phi^\hint :\, {\bigwedge^{k_1}}_q \mV^\hint \otimes \ldots \otimes {\bigwedge^{k_r}}_q \mV^\hint \rightarrow&\,\,\hspace{-.8em}\bigoplus_{0 \leq s_i \leq k_i} \hspace{-.1em}\left({\bigwedge^{k_r-s_r}}_{\!\!\!q} \hspace{-.1em} \mW^\hint \otimes \ldots \otimes {\bigwedge^{k_1-s_1}}_{\!\!\!q} \hspace{-.1em} \mW^\hint \otimes {\bigwedge^{s_1}}_q \mW^\hint \otimes \ldots \otimes {\bigwedge^{s_r}}_q \mW^\hint \right),\\
v_{\textbf{i}_1} \otimes \ldots \otimes v_{\textbf{i}_r}\mapsto& \,\,w_{-\textbf{i}_{r,-}} \otimes \ldots \otimes w_{-\textbf{i}_{1,-}} \otimes w_{\textbf{i}_{1,+}} \otimes \ldots \otimes w_{\textbf{i}_{r,+}}.
\end{aligned}
\end{gather*}
 of $\mathcal{U}_q(\mathfrak{gl}_\mN)$-modules, where $\mathcal{U}_q(\mathfrak{gl}_\mN)$ acts on the left via the isomorphism with $\check{\cU}^\hint$ followed by the comultiplication of $\cU^\hint$.
\end{prop}
\begin{proof}
The map is an isomorphism of vector spaces since it sends a basis to a basis and both spaces have the same dimension. We need to verify that it is $\cU_q(\mathfrak{gl}_\mN)$-linear.
We will only check the linearity on generators $E_a$ for $a \in \mathrm{I}^{\hint,++}$, since the arguments for the others are parallel. On the left hand side, $E_a$ acts via the isomorphism, which maps $E_a$ to $B_a$ followed by successive application of the comultiplication. Hence, $E_a$ acts by multiplication with
\abovedisplayskip0.3em
\belowdisplayskip0.3em
\begin{gather*}
\sum_{l=0}^{r-1} (K_aK_{-a}^{-1})^{\otimes l} \otimes E_a K_{-a}^{-1} \otimes (K_{-a}^{-1})^{\otimes (r-1-l)}+\, \sum_{l=0}^{r-1} 1^{\otimes l} \otimes F_{-a} \otimes (K_{-a}^{-1})^{\otimes (r-1-l)}.
\end{gather*}
If we denote by $d(\textbf{i},a)$ and $d(\textbf{i},-a)$ the integers such that
$ K_a v_{\textbf{i}} = q^{d(\textbf{i},a)} v_{\textbf{i}}$,  $K_{-a} v_{\textbf{i}} = q^{-d(\textbf{i},-a)} v_{\textbf{i}}.$
for a vector $v=v_{\textbf{i}}=v_{\textbf{i}_1} \otimes \ldots \otimes v_{\textbf{i}_r}$, then we get
\abovedisplayskip0.3em
\belowdisplayskip0.3em
\begin{gather*}
\begin{aligned}
E_a.v =&\,
\sum_{l=0}^{r-1} 1^{\otimes l} \otimes E_a \otimes 1^{\otimes (r-1-l)} . v \cdot q^{d(\textbf{i}_1,-a)+ \ldots + d(\textbf{i}_r,-a)} q^{d(\textbf{i}_{1},a)+\ldots +d(\textbf{i}_{l},a)} \\
+&\, \sum_{l=0}^{r-1} 1^{\otimes l} \otimes F_{-a} \otimes 1^{\otimes (r-1-l)} . v \cdot q^{d(\textbf{i}_{l+2},-a)+ \ldots + d(\textbf{i}_r,-a)}.
\end{aligned}
\end{gather*}
On the other hand let $w=\Phi^\hint(v)=w_{-\textbf{i}_{r,-}} \otimes \ldots \otimes w_{-\textbf{i}_{1,-}} \otimes w_{\textbf{i}_{1,+}} \otimes \ldots \otimes w_{\textbf{i}_{r,+}}$.
Since $K_a w_{\textbf{i}_+} = q^{d(\textbf{i},a)} w_{\textbf{i}_+}$, $K_{a} w_{-\textbf{i}_i} = q^{d(\textbf{i},-a)} w_{-\textbf{i}_-}$ for $K_a \in \cU_q(\mathfrak{gl}_\mZ)$ and any $\textbf{i}$, we obtain
\abovedisplayskip0.3em
\belowdisplayskip0.3em
\begin{gather*}
\begin{aligned}
E_a.w =&\,
\sum_{l=0}^{r-1} 1^{\otimes l} \otimes E_a \otimes 1^{\otimes (2r-1-l)} . w \cdot q^{d(\textbf{i}_r,-a)+ \ldots + d(\textbf{i}_{r-l+1},-a)} \\
+&\, \sum_{l=r}^{2r-1} 1^{\otimes l} \otimes E_{a} \otimes 1^{\otimes (2r-1-l)} . w \cdot q^{d(\textbf{i}_{r},-a)+\ldots +d(\textbf{i}_{1},-a)}
q^{d(\textbf{i}_1,a)+ \ldots + d(\textbf{i}_{l-r},a)}\\
=&\,
\sum_{l=0}^{r-1} 1^{\otimes (r-1-l)} \otimes E_a \otimes 1^{\otimes (r+l)} . w \cdot q^{d(\textbf{i}_r,-a)+ \ldots + d(\textbf{i}_{l+2},-a)} \\
+&\, \sum_{l=0}^{r-1} 1^{\otimes (r+l)} \otimes E_{a} \otimes 1^{\otimes (r-1-l)} . w \cdot q^{d(\textbf{i}_{r},-a)+\ldots +d(\textbf{i}_{1},-a)}
q^{d(\textbf{i}_1,a)+ \ldots + d(\textbf{i}_{l},a)}.
\end{aligned}
\end{gather*}
In the last step we just reordered the first sum and shifted the indexing of the second. In particular, $\Phi^\hint$ commutes with the action of $E_a$.
Similar calculations for the other generators imply that $\Phi^\hint$ is a $\cU_q(\mathfrak{gl}_\mN)$-module homomorphism.
\end{proof}

Consider first again the $\cU_q(\mathfrak{gl}_\mN)$-modules $\bigwedge^n_q \mV$ and $\bigwedge^n_q \mV^\hint$. 
To extend the $\cU_q(\mathfrak{gl}_\mN)$-action explicitly to the action of the coideal subalgebra we pass to weight diagrams as indexing set of our basis. To a standard basis vector $v_{\underline{i}}$ in either we assign a diagrammatic weight in $\mX_n$ respectively $\mX_n^\hint$ by defining the sets
$\mathrm{P}_{\down}(\underline{i}) = \{i_j \mid i_j > 0 \}$, $\mathrm{P}_\up(\underline{i})=\{-i_j \mid i_j < 0\}$ and $\mathrm{P}_\diamondb(\underline{i})=\{i_j \mid i_j = 0\}$. Then the maps
\abovedisplayskip0.1em
\belowdisplayskip0.3em
\begin{gather}
\label{wtmap}
\wt :\,\,{\bigwedge^n}_q \mV \,\longrightarrow \, \left\langle \mX_n \right\rangle_{\mathbb{Q}(q)} \quad  \text{and} \quad
\wt^\hint :\,\, {\bigwedge^n}_q \mV^\hint \,\longrightarrow\, \left\langle \mX_n^\hint \right\rangle_{\mathbb{Q}(q)}
\end{gather}
are defined on basis vectors by setting $\wt(v_{\underline{i}})_l$ respectively $\wt^\hint(v_{\underline{i}})_l$ to be
\abovedisplayskip0.3em
\belowdisplayskip0.3em
\begin{gather*}
\begin{array}{clclcl}
\down & \text{if } l \in \mathrm{P}_\down(\underline{i}) \setminus \mathrm{P}_\up(\underline{i}),&
\up & \text{if } l \in \mathrm{P}_\up(\underline{i}) \setminus \mathrm{P}_\down(\underline{i}),&\diamondb & \text{if } l \in \mathrm{P}_\diamondb(\underline{i}),\\
\times & \text{if } l \in \mathrm{P}_\up(\underline{i}) \cap \mathrm{P}_\down(\underline{i}), &
\circ & \text{if } l \not\in \mathrm{P}_\up(\underline{i}) \cup \mathrm{P}_\down(\underline{i}) \cup \mathrm{P}_\diamondb(\underline{i}).&
\end{array}
\end{gather*}
These are isomorphism of vector spaces and we denote the basis vector corresponding to a diagrammatic weight $\lambda$ by $v_\lambda$.

The proofs of the following Lemmas are straightforward and thus omitted.

\begin{lemma} \label{lem:action_comm_generators}
Let $\lambda$ be a diagrammatic weight, then
\abovedisplayskip0.3em
\belowdisplayskip0.3em
\[
D_jD_{-j} \lambda = \left\lbrace \begin{array}{rl}
\lambda & \text{ if } \lambda_j = \circ,\\
q \lambda & \text{ if } \lambda_j \in \{ \down,\up\},\\
q^2 \lambda & \text{ if } \lambda_j = \times.\\
\end{array} \right.
\]
\end{lemma}

\begin{lemma} \label{lem:action_generators}
Let $\lambda \in \mX_n$ and $i \in \mathrm{I}^{++}$. For $x,y \in \{\circ,\up,\down,\times\}$ we write $\lambda_{xy}$ for the diagrammatic weight obtained by relabelling $\lambda_{(i-\half)}$ to $x$ and $\lambda_{(i+\half)}$ to $y$.\\[0.2cm]
\begin{minipage}{6.5cm}
\noindent Consider the action of $B_i$:
\begin{enumerate}[(a)]
\item $\lambda = \lambda_{\scriptscriptstyle\circ\down}:$ $B_i.\lambda = \lambda_{\scriptscriptstyle\down\circ}$.
\item $\lambda = \lambda_{\scriptscriptstyle\circ\up}:$ $B_i.\lambda = \lambda_{\scriptscriptstyle\up\circ}$.
\item $\lambda = \lambda_{\scriptscriptstyle\up\cross}:$ $B_i.\lambda = \lambda_{\scriptscriptstyle\cross\up}$.
\item $\lambda = \lambda_{\scriptscriptstyle\down\cross}:$ $B_i.\lambda = \lambda_{\scriptscriptstyle\cross\down}$.
\item $\lambda = \lambda_{\scriptscriptstyle\up\down}:$ $B_i.\lambda = q \lambda_{\scriptscriptstyle\cross\circ}$.
\item $\lambda = \lambda_{\scriptscriptstyle\down\up}:$ $B_i.\lambda = \lambda_{\scriptscriptstyle\cross\circ}$.
\item $\lambda = \lambda_{\scriptscriptstyle\circ\cross}:$ $B_i.\lambda = q^{-1}\lambda_{\scriptscriptstyle\down\up} + \lambda_{\scriptscriptstyle\up\down}$.
\item Otherwise $B_i.\lambda=0$.
\end{enumerate}
\end{minipage}
\begin{minipage}{5.5cm}
Consider the action of $B_{-i}$:
\begin{enumerate}[(a)]
\setcounter{enumi}{8}
\item $\lambda = \lambda_{\scriptscriptstyle\up\circ}:$ $B_{-i}.\lambda = \lambda_{\scriptscriptstyle\circ\up}$.
\item $\lambda = \lambda_{\scriptscriptstyle\down\circ}:$ $B_{-i}.\lambda = \lambda_{\scriptscriptstyle\circ\down}$.
\item $\lambda = \lambda_{\scriptscriptstyle\cross\down}:$ $B_{-i}.\lambda = \lambda_{\scriptscriptstyle\down\cross}$.
\item $\lambda = \lambda_{\scriptscriptstyle\cross\up}:$ $B_{-i}.\lambda = \lambda_{\scriptscriptstyle\up\cross}$.
\item $\lambda = \lambda_{\scriptscriptstyle\up\down}:$ $B_{-i}.\lambda = q \lambda_{\scriptscriptstyle\circ\cross}$.
\item $\lambda = \lambda_{\scriptscriptstyle\down\up}:$ $B_{-i}.\lambda = \lambda_{\scriptscriptstyle\circ\cross}$.
\item $\lambda = \lambda_{\scriptscriptstyle\cross\circ}:$ $B_{-i}.\lambda = q^{-1}\lambda_{\scriptscriptstyle\down\up} + \lambda_{\scriptscriptstyle\up\down}$.
\item Otherwise $B_{-i}.\lambda=0$.
\end{enumerate}
\end{minipage}

The statements also holds if we replace $\mX_n$ by $\mX^\hint_n$ and $\mathrm{I}^{++}$ by $\mathrm{I}^{\hint,++}$.
\end{lemma}

With Lemmas \ref{lem:action_comm_generators} and \ref{lem:action_generators} we know the action of the generators $\check{E}_i$, $\check{F}_i$, and $\check{D}_i^{\pm 1}$ of $\cH$ and $\cH^\hint$ for $i \in \mathrm{I}^{++}$ and $i \in \mathrm{I}^{\hint,++}$. We consider the action of the special generators $B_0$, $B_{-\half}$.

\begin{lemma} \label{lem:specgen}
Let $\lambda \in \mX_n$. For $x,y \in \{\circ,\up,\down,\times,\diamondb\}$ we write $\lambda_{xy}$ for the diagrammatic weight obtained by relabelling $\lambda_0$ to $x$ and $\lambda_1$ to $y$.\\[.2cm]
\begin{minipage}{6.5cm}
\noindent Consider the action of $B_\half$:
\begin{enumerate}[(a)]
\item $\lambda = \lambda_{\scriptscriptstyle\circ\down}:$ $B_{\half}.\lambda = \lambda_{\scriptscriptstyle\diamondb\circ}$.
\item $\lambda = \lambda_{\scriptscriptstyle\circ\up}:$ $B_{\half}.\lambda = \lambda_{\scriptscriptstyle\diamondb\circ}$.
\item $\lambda = \lambda_{\scriptscriptstyle\circ\cross}:$ $B_{\half}.\lambda = q^{-1}\lambda_{\scriptscriptstyle\diamondb\up} + \lambda_{\scriptscriptstyle\diamondb\down}$.
\item Otherwise $B_{\half}.\lambda=0$.
\end{enumerate}
\end{minipage}
\begin{minipage}{5.5cm}
Consider the action of $B_{-\half}$:
\begin{enumerate}[(a)]
\setcounter{enumi}{4}
\item $\lambda = \lambda_{\scriptscriptstyle\diamondb\circ}:$ $B_{-\half}.\lambda = \lambda_{\scriptscriptstyle\circ\up} + \lambda_{\scriptscriptstyle\circ\down}$.
\item $\lambda = \lambda_{\scriptscriptstyle\diamondb\down}:$ $B_{-\half}.\lambda = q \lambda_{\scriptscriptstyle\circ\cross}$.
\item $\lambda = \lambda_{\scriptscriptstyle\diamondb\up}:$ $B_{-\half}.\lambda = \lambda_{\scriptscriptstyle\circ\cross}$.
\item Otherwise $B_{-\half}.\lambda=0$.
\end{enumerate}
\end{minipage}\\[0.1cm]
\noindent The element $B_0$ acts by zero on $\lambda \in \mX_n^\hint$ unless $\lambda_{\half} \in \{\up,\down\}$. If $\lambda=\lambda_{\scriptscriptstyle\down}$ then $B_0.\lambda = \lambda_{\scriptscriptstyle\up}$ and if $\lambda=\lambda_{\scriptscriptstyle\up}$ then $B_0.\lambda = \lambda_{\scriptscriptstyle\down}$.
\end{lemma}

Identifying now $\left\langle \mX_n \right\rangle_{\mathbb{Q}(q)}$ with $K_0\left(\hat{\cO}^\mathfrak{p}_1(n)\right)$ and $\left\langle \mX_n^\hint \right\rangle_{\mathbb{Q}(q)}$ with $K_0\left(\hat{\cO}^\mathfrak{p}_\hint(n)\right)$, as in \eqref{eq:identify}, the induced action of $B_i$ can be compared with the induced action of the of the special projective functors to obtain the following result.

\begin{theorem}[Categorification Theorem] \label{prop:generators_and_functors}
The action of $B_{\pm i}$ and $B_{\pm\half}$ on $K_0\left(\hat{\cO}^\mathfrak{p}_1(n)\right)$ coincides with the action of $[\cF_{i,\pm}]$ and $[\cF_{\half,\pm}]$ respectively. 

The action of $B_{\pm i}$ and $B_{0}$ on $K_0\left(\hat{\cO}^\mathfrak{p}_\hint(n)\right)$ coincides with the action of $[\cF_{i,\pm}]$ and $[\cF_{0}]$ respectively.
\end{theorem}
\begin{proof}
This follows by comparing the action as described in Lemma ~\ref{lem:action_generators} and Lemma \ref{lem:specgen} with the description of the action of the special translation functors $\cF_{i,\pm}$ on Verma modules in Lemma \ref{a}. 
\end{proof}

\subsection{Bar involution and dual canonical basis}

We start by defining an involution on $\cH$ and $\cH^\hint$. It generalizes Lusztig's bar involution and should be compared with the bar involution in \cite{BW}. Both are special cases of  \cite{BalKolb1}.

\begin{prop}
There is a unique $\mQ(q)$-antilinear bar-involution $-:\cH \rightarrow \cH$ such that
$\overline{B_i} = B_i$ and $\overline{D_jD_{-j}}=D_j^{-1}D_{-j}^{-1}$; similar for $\cH^\hint$.
\end{prop}
\begin{proof}
Using Proposition \ref{prop:relations_coideal} one checks easily that such an assignment extends to an algebra involution.
\end{proof}

\begin{definition}
By a {\em compatible bar-involution} on an $\cH$-module $M$ we mean an anti-linear involution $\overline{\phantom{x}}:M \rightarrow M$ such that $\overline{uv} = \overline{u}\,\overline{v}$ for each $u \in \cH, v \in M$ and analogously for a $\cH^\hint$-module.
\end{definition}
Note that as a $\cU$-module $\mV$ has one-dimensional weight spaces with all standard basis vectors being weight vectors. The next lemma shows that the module $\bigwedge^n_q \mV$ possesses a compatible bar-involution. As our identification of the action of $\cH$ and the translation functors already indicates, we will need to focus on vectors $w_\lambda$ that correspond to a diagrammatic weight $\lambda$. Recall that for two diagrammatic weights in the same block $\Lambda_\Theta^{{\epsilon}}$ we have the Bruhat order which gives a unique minimal element in each block. As usual, we start with the case of diagrammatic weights supported on the integers.

\begin{prop}\label{lem:bar_unique}
There is a unique compatible bar-involution $w\mapsto\overline{w}$ on the $\cH$-module $\bigwedge^n_q \mV$ such that $\overline{w_\lambda} = w_\lambda$ for each diagrammatic weight
$\lambda \in \mX$ that is minimal in its block with respect to the reversed Bruhat order. For $\lambda \in \mX$,
\abovedisplayskip0.3em
\belowdisplayskip0.3em
\[
\overline{w_\lambda} \,\in\, w_\lambda + \sum_{\mu < \lambda} q^{-1} \mZ[q^{-1}] w_\mu.
\]
\end{prop}
\begin{proof}
We prove uniqueness, existence will follow from Theorem~\ref{prop:bar_exists}. For  $\lambda \in \mX$, denote by $l(\lambda)$ its height in the reversed Bruhat order, i.e. the minimal number of changes that need to made to obtain the minimal element as described in \cite[Lemma 2.3]{ES_diagrams}. It suffices to show that $\overline{w_\lambda}$ is uniquely determined by the claim. For $\lambda$ minimal $\overline{w_\lambda}=w_\lambda$, hence it is determined. Fix $\lambda$ non-minimal and assume that $\overline{w_\mu}$ is uniquely determined for all $\mu$ with $l(\mu)<l(\lambda)$. Note that this includes diagrammatic weights in other blocks.

\textit{Assume there exists $i<j$ such that $\lambda_i= \up$, $\lambda_{j}=\down$ and $\lambda_{s} \in \{\circ,\times\}$ for $i<s<j$.} Let $\lambda' \in \mX$ coincide with $\lambda$ except for $\lambda_{i+1}'=\down$ and $\lambda_s'=\lambda_{s-1}$ for $i+1<s \leq j$, i.e. we move the $\down$ at $j$ to the right of $\up$ at $i$.
By Lemma \ref{lem:action_generators} there exists $B \in \cH$ such that $Bw_{\lambda'}=w_\lambda$.
Furthermore let $\mu \in \mX$ coincide with $\lambda'$ except that positions $i$ and $i+1$ are swapped, then $l(\mu)<l(\lambda')=l(\lambda)$. Then, again by Lemma \ref{lem:action_generators}, we have $B_{-i-\half}B_{i+\half}.w_\mu = q^{-1}w_\mu + w_{\lambda'}$ and thus
$\overline{w_\lambda} = BB_{-i-\half}B_{i+\half}.\overline{w_\mu} - q B\overline{w_\mu}$.
Hence $\overline{w_\lambda}$ is uniquely determined.

\textit{Assume no pair $i<j$ as above exists.} Then choose $i<j$ minimal such that $\lambda_i \in \{\down,\diamondb\}$ and $\lambda_{j}=\down$. As before, assume $j=i+1$ since there are no $\up$'s between $i$ and $j$ by assumption. If $i=0$, let $\mu \in \mW$ coincide with $\lambda$ except $\mu_1=\up$. Then, by Lemma \ref{lem:specgen}, $B_{\half}B_{-\half}(w_\mu) = q^{-1}w_\mu + w_\lambda$, implying $\overline{w_\lambda} = B_{\half}B_{-\half}(\overline{w_\mu}) - q \overline{w_\mu}$.

If $i\neq 0$, obtain $\eta \in \mX$ by moving the symbols at $i$ and $i+1$ to $0$ and $1$. Then $\overline{w_\eta} = B_{\half}B_{-\half}(\overline{w_\mu}) - q \overline{w_\mu}$,
where $\mu \in \mW$ coincides with $\eta$ except that $\mu_1=\up$.
We can now use the element $B' \in \cH$ that first moves the symbols at position $1$ back to position $i+1$ and afterwards the symbol at position $0$ back to position $i$ and obtain
$B'\overline{w_\eta} = B'B_{\half}B_{-\half}(\overline{w_\mu}) - q B'\overline{w_\mu}$,
since $l(\mu)<l(\lambda)$ the right hand side is known, but $B'\overline{w_\eta} = \overline{w_\lambda} + \overline{w_{\lambda'}}$, where $\lambda'$ coincides with $\lambda$ except $\lambda_i'=\up$. But $l(\lambda')=l(\lambda)$ and includes a pair $\up \down$, so is already determined by the previous arguments. Hence $\overline{w_\lambda}$ is uniquely determined.
\end{proof}

\begin{prop}\label{lem:bar_unique_2}
There is a unique compatible bar-involution $w\mapsto \overline{w}$ on the $\cH^\hint$-module $\bigwedge^n_q \mV^\hint$ such that $\overline{w_\lambda} = w_\lambda$ for each diagrammatic weight
$\lambda \in \mX^\hint$ that is minimal in its block with respect to the Bruhat order.
Moreover, for $\lambda \in \mX^\hint$,
\abovedisplayskip0.3em
\belowdisplayskip0.3em
\[
\overline{w_\lambda} \, \in\, w_\lambda + \sum_{\mu < \lambda} q^{-1} \mZ[q^{-1}] w_\mu.
\]
\end{prop}
\begin{proof}
The uniqueness of a bar-involution on the $\cH^\hint$-module $\bigwedge^n_q \mV^\hint$ follows as in  Proposition~\ref{lem:bar_unique}, although the arguments are slightly simpler because the special generator $B_0$ behaves nicer. The details are left to the reader.
\end{proof}

For the existence of a bar-involution with these properties we use that $\bigwedge^n_q \mV \cong K_0\left(\hat{\cO}_1^\pp(n)\right)$ and that there is an exact graded duality functor  $\textbf{d}$,  see \cite[6.1.2]{Stroppel}, on $\hat{\cO}_1^\pp(n)$ satisfying these properties. (Note the typo in the definition of $M^*$ in \cite{Stroppel}, which should just denote the vector space dual.) 

\begin{theorem} [Categorified bar involution]
\label{prop:bar_exists}
The graded duality $\textbf{d}$ on $\hat{\cO}_1^\pp(n)$, sending a module $M$ to $M^\ostar$, induces a compatible bar-involution on ${\bigwedge^n}_q \mV$, satisfying the properties of Proposition~ \ref{lem:bar_unique}. The same also holds for the graded duality on $\hat{\cO}_\hint^\pp(n)$ and the induced involution on ${\bigwedge^n}_{\!\!\!q} \mV^\hint$.
\end{theorem}
\begin{proof}
Let $M^\pp(\lambda)$ be a parabolic Verma module corresponding to a diagrammatic weight being minimal in its block. Then $M^\pp(\lambda)=L(\lambda)$ and hence 
\abovedisplayskip0.3em
\belowdisplayskip0.3em
\[
\textbf{d}M^\pp(\lambda)=\textbf{d}L(\lambda)=L(\lambda)=M^\pp(\lambda).
\]
If on the other hand $M^\pp(\lambda)$ is a Verma module corresponding to an arbitrary diagrammatic weight, we have that $\textbf{d}M^\pp(\lambda)=M^\pp(\lambda)^\ostar=\nabla^\pp (\lambda)$, the dual parabolic Verma module by definition. It is known, see \cite{Stroppel}  that
\abovedisplayskip0.3em
\belowdisplayskip0.3em
\[
[\nabla^\pp(\lambda)] \,\in\, [M^\pp(\lambda)] + {\sum}_{\mu<\lambda} q^{-1} \mZ[q^{-1}] [M^\pp(\mu)].
\]
It remains to show that $\textbf{d}$ is compatible with the action of $\cH$. Since the simple modules form a basis of the Grothendieck group it is enough to check the compatibility on those, i.e. we need to show that for a functor $\hat{\cF}$ corresponding to a generator of $\cH$ respectively $\cH^\hint$ and any simple module $L(\lambda)$ we have
\abovedisplayskip0.3em
\belowdisplayskip0.3em
\begin{gather}
\label{transcommwithd}
[\hat{\cF} L(\lambda)]\, =\, [\hat{\cF} \textbf{d} L(\lambda)] = [\textbf{d}\hat{\cF} L(\lambda)],
\end{gather}
where the first equality is due to the fact that simple modules are self-dual. Thus we need to show that $[\hat{\cF} L(\lambda)]$ is self-dual as well. We will focus on $\hat{\cF}_{i,+}$ and leave the other cases to the reader. Using Proposition~\ref{a} we know that
\abovedisplayskip0.3em
\belowdisplayskip0.3em
\begin{gather*}
\begin{aligned}
\hat{\cF}_{i,+} L(\lambda_{\circ\down}) = L(\lambda_{\down \circ}) ,&
\qquad \hat{\cF}_{i,+} L(\lambda_{\up\times}) = L(\lambda_{\times\up}), &&\quad 
\hat{\cF}_{i,+} L(\lambda_{\down\up}) = L(\lambda_{\times\circ}),\\
\hat{\cF}_{i,+} L(\lambda_{\circ\up}) = 0, &\qquad \hat{\cF}_{i,+} L(\lambda_{\down\times}) = 0, &&\quad\hat{\cF}_{i,+} L(\lambda_{\up\down}) = 0.
\end{aligned}
\end{gather*}
All of these are obviously self-dual and we are left with the case where $\hat{\cF}_{i,+}$ is applied to  $L(\lambda)= L(\lambda_{\circ \times})$. The resulting module $N$ is by \eqref{transcommwithd} self-dual if we forget the grading thanks to  Remark~\ref{rem:transl} and the fact that translation functors commute with duality \cite[Proposition 7.1]{Hbook}. By Propositon~\ref{a}\eqref{HIM7d}, $N$ is indecomposable with simple socle and simple head. In particular, a graded lift is unique up to an overall shift and the grading filtration agrees with the socle and the radical filtration by  \cite[Propositions 2.5.1 and 2.4.1]{BGS}.   Then by Propositon~\ref{a}\eqref{HIM7c}, we obtain $[\hat{\cF}_{i,+} L(\lambda_{\circ \times})] = q[L(\lambda_{\down \up})]+[Q]+q^{-1}[L(\lambda_{\down \up})],$ with $q$ semisimple and concentrated in degree zero, and so $\hat{\cF}_{i,+} L(\lambda_{\circ \times})$ is self-dual.  In fact one can even show combinatorially that 
\abovedisplayskip0.3em
\belowdisplayskip0.3em
\[
[\hat{\cF}_{i,+} L(\lambda_{\circ \times})] = q[L(\lambda_{\down \up})]+[L(\lambda_{\up \down})]+q^{-1}[L(\lambda_{\down \up})].
\]
For this let $Y$ be a set of diagrammatic weights and $r_\mu$, $s_\mu$ integers such that
$$[L(\lambda_{\circ \times})] = \sum_{\mu \in Y} (-1)^{r_\mu} q^{s_\mu} [M(\mu)].$$
Then for all these $\mu$ we have $\mu = \mu_{\circ \times}$. Applying Lemma \ref{lem:action_generators} we obtain
\abovedisplayskip0.3em
\belowdisplayskip0.3em
\begin{gather*}
\begin{aligned}
&[\hat{\cF}_{i,+} L(\lambda_{\scriptstyle \circ \times})] \,=\, \sum_{\mu \in Y} (-1)^{r_\mu} q^{s_\mu-1} [M^\pp(\mu_{\scriptstyle \down \up})] + (-1)^{r_\mu} q^{s_\mu} [M^\pp(\mu_{\scriptstyle \up\down})]\\
=&\, {\scriptstyle (q^{-1}+q) \sum_{\mu \in Y} {\scriptstyle (-1)}^{r_\mu} q^{s_\mu} [M^\pp(\mu_{\scriptstyle \down \up})]} +\hspace{-.3em}  \sum_{\mu \in Y} {\scriptstyle (-1)^{r_\mu}} \left(\scriptstyle q^{s_\mu} [M^\pp(\mu_{\scriptstyle \up\down})] - q^{s_\mu+1} [M^\pp(\mu_{\scriptstyle \down\up})] \right)
\end{aligned}
\end{gather*}

Expressing $[L(\lambda_{\down \up})]$, and $[L(\lambda_{\up \down})]$ in terms of Verma modules using the type $D$ analogue of \cite[Lemma 5.2]{BSII} from \cite{Tim}, \cite{SS} one checks that this is equal to $(q+q^{-1})[L(\lambda_{\down \up})] + [L(\lambda_{\up \down})]$.
\end{proof}

\begin{remark}
Note that the isomorphism classes of simple modules concentrated in degree zero form a {\it dual canonical basis} of $\bigwedge^n_q \mV \cong K_0\left(\hat{\cO}_1^\pp(n)\right)$, that is they are bar-invariant and the transformation matrix to the standard basis is lower diagonal with $1$'s in the diagonal and elements in $\mZ[q]$ below the diagonal. These entries are the dual parabolic Kazhdan-Lusztig polynomials of type $(\mathrm{D}_n,\mathrm{A}_{n-1})$ which are studied and described in detail in \cite{Tim} and \cite{SS}.
\end{remark}

\section{Skew Howe duality: classical case} \label{section:howe}
In the following we want to look at more general versions of parabolic category $\cO$ of type $ \mathrm{D}$ and correspondingly categorifications of tensor products as in Proposition~\ref{lem:decomp_module_general} for $\cU_q(\mathfrak{gl}_\mZ)$ and $\cH^\hint$. Since the quantum parameter in these cases is very technical we will work first in the classical case and consider categorifications of $\mathfrak{gl}_r \times \mathfrak{gl}_r$-modules, instead of $\cH^\hint$-modules. For simplicitly we stick hereby to the case of  $\mathfrak{gl}_r \times \mathfrak{gl}_r$ instead of  $\mathfrak{gl}_{r+1} \times \mathfrak{gl}_r$, but this is just to simplify the treatment.

\subsection{The module \texorpdfstring{$\bigwedge(n,m,r)$}{exterior powers} in the classical case}
For a positive integer $t$ we denote by $\mW_t^\hint$ the vector space of dimension $t$ with basis $v_{\half},\ldots, v_{t-\half}$, with basis elements labelled by half-integers. Similarly we denote by $\mV_{2k}^\hint$ the vector space of dimension $2t$ with basis $v_{-t+\half},\ldots, v_{t-\half}$. Finally denote by $\mM$ the 2-dimensional vector space with basis $v_+$ and $v_-$. All of these will be considered as complex vector spaces.

In the following fix $m$ and $r$ positive integers.  We want to consider module structures on the vector space
\abovedisplayskip0.3em
\belowdisplayskip0.3em
\[
\bigwedge(n,m,r)=\bigwedge^n \left(\mW_m^\hint \otimes \mM \otimes \mW_r^\hint \right).
\]
Since one can identify $\mW_m^\hint \otimes \mM \simeq \mV_{2m}^\hint$ and $\mM \otimes \mW_r^\hint \simeq \mV_{2r}^\hint$ we can view $\bigwedge(n,m,r)$ naturally both as a $(\mathfrak{gl}_{2m}, \mathfrak{gl}_r)$-bimodule and as a $(\mathfrak{gl}_m,\mathfrak{gl}_{2r})$-bimodule.

\subsection{Fixed point subalgebras} \label{subsection:fixedpoints}
We first identify the fixed point subalgebra of $\mathfrak{gl}_{2m}$ which is the non-quantized analogue of $\cH^\hint$. The classical analogue of the involution from Lemma~\ref{lem:theta} is $\Theta:\mathfrak{gl}_{2m}\rightarrow \mathfrak{gl}_{2m}$ defined as $X\mapsto \widetilde{\mathbf{J}}X\widetilde{\mathbf{J}}$, where $\widetilde{\mathbf{J}} = \left(\begin{array}{cc} 0 & \mathbf{J} \\ \mathbf{J} & 0 \end{array}\right)$ with $\mathbf{J}$ as in Section \ref{section:VW}. i.e has $1$'s on the anti-diagonal and zeros elsewhere. Then
\abovedisplayskip0.3em
\belowdisplayskip0.3em
\[
\mathfrak{g}^\Theta=\left\lbrace X \in \mathfrak{gl}_{2m} \mid \widetilde{\mathbf{J}} X \widetilde{\mathbf{J}} = X \right\rbrace.
\]
We have $\mathfrak{g}^\Theta\cong \mathfrak{gl}_m \times \mathfrak{gl}_m$, since it is the image of the involution
\abovedisplayskip0.3em
\belowdisplayskip0.3em
\[
T_A : \mathfrak{gl}_{2m} \longrightarrow \mathfrak{gl}_{2m}, \, X \mapsto \mathbf{A}X\mathbf{A}^{-1} \text{ with } \mathbf{A} = \left( \begin{array}{cc} \mathbf{1} & \mathbf{J} \\ \mathbf{J} & -\mathbf{1} \end{array}\right),
\]
where $\mathbf{1}$ is the $m \times m$ identity matrix. Restricted to $\mathfrak{gl}_m \times \mathfrak{gl}_m$, let  $\mathfrak{gl}^+_m$ and $\mathfrak{gl}^-_m$ denote the images of the first respectively second factor. One easily verifies, using
$ \widetilde{\mathbf{J}}\mathbf{A} = \mathbf{A} \left( \begin{array}{cc}\mathbf{1} & 0 \\ 0 & -\mathbf{1} \end{array}\right) \text{ and } \mathbf{A} \widetilde{\mathbf{J}} = \left( \begin{array}{cc}\mathbf{1} & 0 \\ 0 & -\mathbf{1} \end{array}\right) \mathbf{J},
$
that
\abovedisplayskip0.3em
\belowdisplayskip0.3em
\begin{gather*}
\mathfrak{gl}^+_m =\, \left\lbrace \hspace{-.5em}\begin{array}{c|c} \left( \begin{array}{cc} X & X\mathbf{J} \\ \mathbf{J}X & \mathbf{J}X\mathbf{J} \end{array} \right) & X \in \mathfrak{gl}_m \end{array} \hspace{-.5em} \right\rbrace \text{ and } \mathfrak{gl}^-_m =\, \left\lbrace \hspace{-.5em} \begin{array}{c|c} \left( \begin{array}{cc} X & -X\mathbf{J} \\ -\mathbf{J}X & \mathbf{J}X\mathbf{J} \end{array} \right) & X \in \mathfrak{gl}_m \end{array} \hspace{-.5em} \right\rbrace.
\end{gather*}
Fix the identification $\mW_m^\hint \otimes \mM \simeq \mV_{2m}^\hint$, $v_i \otimes v_+ \mapsto v_i$ and $v_i \otimes v_- \mapsto v_{-i}$ as vector spaces. Twisting with $T_A$ gives the decomposition
\abovedisplayskip0.3em
\belowdisplayskip0.3em
\begin{gather}
\label{C2m_decomposition}
\mV_{2m}^\hint = \mW_m^\hint \otimes \left\langle v_+ + v_- \right\rangle \oplus \mW_m^\hint \otimes \left\langle v_+ - v_- \right\rangle,
\end{gather}
as a $\mg^\Theta=\mathfrak{gl}^+_m \times \mathfrak{gl}^-_m$-module. One should note that $\mathfrak{gl}^+_m$ acts as zero on the second summand, while $\mathfrak{gl}^-_m$ acts as zero on the first summand. The {\it diagonally embedded} $\mathfrak{gl}_m$ in $\mathfrak{gl}_m\times \mathfrak{gl}_m$ with its Chevalley generators normalized by $\frac{1}{2}$ maps via $T_A$ to an isomorphic subalgebra of $\mathfrak{gl}_{2m}$ with generators $E_i=E_{i+1,i} + E_{-i-1,-i}$ and $F_i=E_{i,i+1} + E_{-i,-i-1}$ for $0 < i < m$. These are exactly the specialization of the generators $B_i$ and $B_{-i}$ from Definition \ref{def:modified_generators}. Together with the specialization of $\cR$ they generate the \emph{diagonally embedded} $\mathfrak{gl}_m$ in  $\mathfrak{gl}^+_m \times \mathfrak{gl}^-_m$. The special generator $B_0$ specializes to the element $G=E_{1,-1} + E_{-1,1}$ and generates together with the diagonal embedded Lie algebra the whole fixed point Lie algebra  $\mathfrak{g}^\Theta$.

\subsection{Classical skew Howe duality}
All the identifications made above can of course also be done for $\mathfrak{gl}_{2r}$ and we will use the same notations for them. Using the decomposition from \eqref{C2m_decomposition} we obtain
\abovedisplayskip0.3em
\belowdisplayskip0.3em
\begin{gather}
\label{dec}
\begin{gathered}
\begin{aligned}
\bigwedge(n,m,r)=& \bigwedge^n \left( \mW_m^\hint \otimes \left\langle v_+ \hspace{-.2em}+\hspace{-.1em} v_- \right\rangle \otimes \mW_r^\hint \oplus \mW_m^\hint \otimes \left\langle v_+ \hspace{-.2em}-\hspace{-.1em} v_- \right\rangle \otimes \mW_r^\hint \right)\\
\simeq& \bigoplus_{i=0}^n \bigwedge^i \left( \mW_m^\hint \otimes \left\langle v_+ \hspace{-.2em}+\hspace{-.1em} v_- \right\rangle \otimes \mW_r^\hint \right) \otimes \bigwedge^{n-i} \left( \mW_m^\hint \otimes \left\langle v_+ \hspace{-.2em}-\hspace{-.1em} v_- \right\rangle \otimes \mW_r^\hint \right) \\
\simeq& \bigoplus_{i=0}^n \bigwedge^i \left( \mW_m^\hint \otimes \left\langle v_+ \hspace{-.2em}+\hspace{-.1em} v_- \right\rangle \otimes \mW_r^\hint \right) \boxtimes \bigwedge^{n-i} \left( \mW_m^\hint \otimes \left\langle v_+ \hspace{-.2em}-\hspace{-.1em} v_- \right\rangle \otimes \mW_r^\hint \right),
\end{aligned}
\end{gathered}
\end{gather}
where $\boxtimes$ means that we regard the first factor as a $(\mathfrak{gl}^+_m,\mathfrak{gl}^+_r)$-bimodule and the second as a $(\mathfrak{gl}^-_m,\mathfrak{gl}^-_r$)-bimodule and take the outer tensor product of these two bimodules in each summand. From the usual skew Howe duality of $\mathfrak{gl}_m$ and $\mathfrak{gl}_r$ on $\bigwedge^n(\mW_m^\hint \otimes \mW_r^\hint)$ we can deduce that on each summand the $\mathfrak{gl}^+_m \times \mathfrak{gl}^-_m$-endomorphisms are generated by $\mathfrak{gl}^+_r \times \mathfrak{gl}^-_r$ and that there are no such morphisms between different summands. We obtain

\begin{theorem}[Classical skew Howe duality] \label{thm:skewhowe}
The actions of $U(\mathfrak{gl}^+_m \times \mathfrak{gl}^-_m)$ and $U(\mathfrak{gl}^+_r \times \mathfrak{gl}^-_r)$ on $\bigwedge(n,m,r)$ commute and generate each others commutant.
\end{theorem}

\subsection{Categorified classical skew Howe duality}

In the following we will need certain sets of non-negative integers.
\begin{definition}
For a positive integer $t$ we denote by $C(n,t)$ the set of $t$-tuples $\underline{k}=(k_1,\ldots,k_{t})$ of non-negative integers that sum up to $n$.  It will be convenient to also allow indexing the entries of $\underline{k} \in C(n,t)$ by half-integers $\half,\ldots,t-\half$, i.e. $k_{i-\half} = k_i$ for $1 \leq i \leq t$. 

Similarly we denote by $C(n)$ the set of $\mathbb{Z}_{\geq 1}$-tuples $\underline{k}=(k_1,k_2,\ldots)$ of non-negative integers that sum up to $n$. As before, we also allow indexing the entries by half-integers.
\end{definition}

To categorify the bimodule $\bigwedge(n,m,r)$ we first use classical skew Howe duality for the pair $(\mathfrak{gl}_{2m},\mathfrak{gl}_r)$ to obtain a decomposition as $\mathfrak{gl}_{2m}$-module
\abovedisplayskip0.3em
\belowdisplayskip0.3em
\begin{gather}
\label{Defwedge}
\bigwedge(n,m,r) \,\cong \,\bigoplus_{\underline{k} \in C(n,r)} \bigwedge^{\underline{k}} \mV_m^\hint,
\end{gather}
where $\bigwedge^{\underline{k}}\mV_m^\hint = \bigwedge^{k_1} \mV_m^\hint \otimes \ldots \otimes \bigwedge^{k_{r}} \mV_m^\hint$. Each summand is also a module for the subalgebra $\mathfrak{gl}^+_m \times \mathfrak{gl}^-_m$. To identify each summand with the Grothendieck group of certain blocks of parabolic category $\cO$ we denote by $\mathfrak{q}_{\underline{k}}$, for  $\underline{k} \in C(n,r)$, the standard parabolic of $\mathfrak{so}_{2n}$ corresponding to the root subsystem where we omit the simple roots $\alpha_0, \alpha_{k_1}, \alpha_{k_1+k_2}, \alpha_{k_1+k_2+k_3}, \ldots, \alpha_{k_1 + \ldots + k_{r-1}}.$ Thus in $\mathfrak{gl}_n$ this corresponds to a parabolic subalgebra with diagonal blocks of sizes $k_1,\ldots,k_r$. Our original parabolic is then $\pp=\mathfrak{q}_{(n,0^{r-1})}$. 

As in Section \ref{section:lie_basics} we have the corresponding parabolic category $\cO^{\mathfrak{q}_{\underline{k}}}(n)$ inside $\cO(n)$ containing the simple modules whose highest weights are $\mathfrak{q}_{\underline{k}}$-dominant.

\begin{definition}
An integral weight $\lambda$ is \emph{$\mathfrak{q}_{\underline{k}}$-dominant} if it is of the form
$$\lambda = (\lambda_1 \leq \ldots \leq \lambda_{k_1},\lambda_{k_1+1} \leq \ldots \leq \lambda_{k_1+k_2},\ldots,\lambda_{k_1+\ldots+k_{r-1}+1} \leq \ldots \leq \lambda_n).$$
Let $\La_{\underline{k}}$ denote the set of all $\mathfrak{q}_{\underline{k}}$-dominant weights, decomposing as usual into $\La_{\underline{k}} = \La_{\underline{k}}^1 \cup \La_{\underline{k}}^\hint$, depending on whether the $\lambda_i$'s are integers or half-integers.
\end{definition}

\begin{remark}
Note that two $\mathfrak{q}_{\underline{k}}$-dominant weights give rise to parabolic Verma modules with the same central character iff they are in the same Weyl group orbit with respect to the dot-action. After adding $\rho$ this corresponds to usual Weyl group orbits. Then two weights lie in the same orbit iff the multiplicity for each half-integer $\half,\ldots,m-\half$, up to sign, agrees in the two weights and the multiplicities of negative entries have the same parity.
\end{remark}

Consider the Serre subcategory $\cO^{\mathfrak{q}_{\underline{k}}}_{\leq m}(n)$ of $\cO_\hint(n)$ generated by simple modules with highest weight in
\abovedisplayskip0.3em
\belowdisplayskip0.3em
\[
\La_{\underline{k}}^{\leq m} = \left\lbrace \lambda \in \La_{\underline{k}}^\hint \, \scalebox{1.3}{$\mid$} \, |\lambda_i+\rho_i| \leq m-\half \right\rbrace.
\]
This is obviously a finite set and stable under the dot-action of the Weyl group, thus $\cO^{\mathfrak{q}_{\underline{k}}}_{\leq m}(n)$ is a finite sum of blocks in $\cO^{\mathfrak{q}_{\underline{k}}}_{\leq m}(n)$. Generalizing \eqref{wtmap} and \eqref{eq:identify}, we associate to $\lambda \in \La_{\underline{k}}^{\leq m}$ a vector $v_\lambda \in \bigwedge^{\underline{k}}\mV_m^\hint$ as follows. Take $\lambda'=\lambda+\rho$, consider the sequence $\underline{i}^j_\lambda = (\lambda_{k_1+\ldots+k_j}'>\lambda_{k_1+\ldots+k_j-1}'>\ldots >\lambda_{k_1+\ldots+k_{j-1}+1}'),$
and then define the corresponding standard basis vector as
$v_\lambda = v_{\underline{i}^1_\lambda} \otimes \ldots \otimes v_{\underline{i}^r_\lambda}$. Then
 there is an isomorphism of vector spaces
\abovedisplayskip0.3em
\belowdisplayskip0.3em
\begin{gather*}
\kw_{\underline{k}}: \quad K_0\left(\cO^{\mathfrak{q}_{\underline{k}}}_{\leq m}(n)\right) \longrightarrow \bigwedge^{\underline{k}}\mV_m^\hint, \quad [M^{\mathfrak{q}_{\underline{k}}}(\lambda)]\mapsto v_\lambda.
\end{gather*}
Taking (direct) sums we obtain two $\mQ$-isomorphisms $\kw_m$ and $\kw_r$:
\abovedisplayskip0.3em
\belowdisplayskip0.3em
\begin{gather}
\label{Gammam}
\quad K_0 \left( \bigoplus_{\underline{k} \in C(n,r)} \cO^{\mathfrak{q}_{\underline{k}}}_{\leq m}(n) \right)\stackrel{\kw_m} \longrightarrow \bigwedge(n,m,r)\stackrel{\kw_r}{\longleftarrow} K_0 \left( \bigoplus_{\underline{k} \in C(n,m)} \cO^{\mathfrak{q}_{\underline{k}}}_{\leq r}(n) \right).
\end{gather}

Next, we decompose the categories $\cO^{\mathfrak{q}_{\underline{k}}}_{\leq m}(n)$. Although we will refer to this as a {\it block decomposition} and to the summands as {\it blocks}, this is not completely correct as the subcategories might decompose further, but it is a decomposition by central character.

\begin{definition} \label{def:blocks_howe}
Fix $(\underline{k},\underline{d},\epsilon)$ with $\underline{k} \in C(n,r)$, $\underline{d} \in C(n)$ and ${\epsilon} \in \mathbb{Z}/2\mathbb{Z}$. Then $\cO_{(\underline{k},\underline{\mu}, {\epsilon})}(n)$ denotes the block of $\cO^{\mathfrak{q}_{\underline{k}}}(n)$ containing those parabolic Verma modules of highest weights $\lambda$ satisfying
\abovedisplayskip0.3em
\belowdisplayskip0.3em
\[
d_i = \# \left\lbrace 
j \,\scalebox{1.3}{$\mid$} \mid \lambda_j + \rho_j \mid = i \right\rbrace, \text{ for } i \in \mathbb{Z}_{\geq 0} + \half, \quad \text{and} \quad \epsilon = \overline{\# \left\lbrace
j \,\scalebox{1.3}{$\mid$}\, \lambda_j + \rho_i < 0 \right\rbrace}+1.
\]
Such a triple $(\underline{k},\underline{d}, {\epsilon})$ is called a \emph{block triple}. We denote by ${\rm Bl}(n,r)$ the set of all such triples. For a block triple $\kappa$ we denote by $\mathrm{pr}_\kappa$ the projection onto $\cO_{\kappa}(n)$ in any sum of blocks.

By abuse of notations we can use $\underline{d} \in C(n,m)$, which will result in a block contained in $\cO^{\mathfrak{q}_{\underline{k}}}_{\leq m}(n)$. This subset of ${\rm Bl}(n,r)$ will be denoted by ${\rm Bl}(n,r,m)$
\end{definition}

\begin{remark}
Note that Definition~\ref{def:blocks_howe} is only indexing blocks with weights supported on the half integers. The notations for blocks with weights supported on the integers will be given in Definition~\ref{def:blocks_general} and is more technical.

Note that, while for a fixed $\underline{k}$ a block triple $\kappa=(\underline{k},\underline{d},\epsilon)$ uniquely determines a block $\cO_{\kappa}(n)$ of $\cO^{\mathfrak{q}_{\underline{k}}}(n)$, there might be more than one choice of $\underline{k}$ to obtain the same parabolic subalgebra. This also means that $\mathrm{pr}_\kappa$ only projects onto a block labelled by $\kappa$, not some other $\kappa'$ even if they are the same category.
This is important for our categorification of skew Howe duality.
\end{remark}

\begin{definition} 
\label{defsd}
Let $\bd \in C(n)$ and $i \in \mathbb{Z}_{\geq 1}$. Then we define
\abovedisplayskip0.3em
\belowdisplayskip0.3em
\begin{gather*}
\begin{aligned}
\bd^{+i} =&\, ({\scriptstyle \ldots},d_{i-\half},1,d_{i+\half}-1,d_{i+\nicefrac{3}{2}},{\scriptstyle \ldots}), && {}_{+i}\bd =&\, ({\scriptstyle \ldots},d_{i-\half}+1,d_{i+\half}-1,d_{i+\nicefrac{3}{2}},{\scriptstyle \ldots}), \\
\bd^{-i} =&\, ({\scriptstyle \ldots},d_{i-\half}-1,1,d_{i+\half},d_{i+\nicefrac{3}{2}},{\scriptstyle \ldots}), && {}_{-i}\bd =&\, ({\scriptstyle \ldots},d_{i-\half}-1,d_{i+\half}+1,d_{i+\nicefrac{3}{2}},{\scriptstyle \ldots}).
\end{aligned}
\end{gather*}
Furthermore for a given block triple $\kappa = (\underline{k},\underline{d}, {\epsilon})$ we define block triples:
\[
{}_{+i}\kappa \,=\, (\underline{k},{}_{+i}\underline{d}, {\epsilon}), \,\,\, {}_{-i}\kappa \,=\, (\underline{k},{}_{-i}\underline{d}, {\epsilon}), \text{ for }i \in \mathbb{Z}_{\geq 1}, \,\, \text{ and } \kappa_0 \,=\, (\underline{k},\underline{d},\epsilon+1).\]
\end{definition}

\begin{remark}\label{maxparab}
In case that $\mathfrak{q}_{\underline{k}}=\pp$, passing from $\bd$ to ${}_{+i}\bd$ can be interpreted in the weight combinatorics and corresponds to moving a symbol $\up$ or $\down$ from position $i+ \half$ to position $i-\half$, see $\Gamma_{i,+}$ in \eqref{deftable1}. The symbol $\times$ counts as a union of $\up$ and $\down$ for this. Similarly for the passage from $\bd$ to ${}_{-i}\bd$. 
\end{remark}

We define now special translation functors generalizing Definition \ref{def:functors}.

\begin{definition}
\label{def:specproj}
We have the \emph{special projective functors} $\cF_{i,+}$, $\cF_{i,-}$ for $i \in \mZ_{> 0}$ and $\cF_0$ defined by
\abovedisplayskip0.3em
\belowdisplayskip0.3em
\begin{gather*}
\begin{aligned}
\cF_{i,+} &=& \bigoplus_{\kappa}
\pr_{{}_{+i}\kappa} \;(? \otimes V)\; \pr_\kappa &:& \bigoplus_{\kappa} \cO_\kappa(n) \rightarrow \bigoplus_{\kappa} \cO_\kappa(n),\\
\cF_{i,-} &=& \bigoplus_{\kappa}
\pr_{{}_{+i}\kappa} \; (? \otimes V) \; \pr_\kappa &:& \bigoplus_{\kappa} \cO_\kappa(n) \rightarrow \bigoplus_{\kappa} \cO_\kappa(n),\\
\cF_{0} &=& \bigoplus_{\kappa}
\pr_{\kappa_{0}} \;(? \otimes V) \; \pr_\kappa &:& \bigoplus_{\kappa} \cO_\kappa(n) \rightarrow \bigoplus_{\kappa} \cO_\kappa(n),
\end{aligned}
\end{gather*}
with all sums running over $\kappa \in {\rm Bl}(n,r)$.
\end{definition}

\begin{theorem}[Categorification of  $\bigwedge(n,m,r)$]\label{prop:catskewclass}
Under the identification 
\abovedisplayskip0.3em
\belowdisplayskip0.3em
\[
\kw_m:\,\, K_0\left({\bigoplus}_{\kappa \in {\rm Bl}(n,r,m)} \cO_\kappa(n)\right) \,\longrightarrow\, \bigwedge(n,m,r)
\]
from \eqref{Gammam}, the action of $[\cF_{i,-}]$,  $[\cF_{i,+}]$, $[\cF_{0}]$ on the left coincide with the action of $F_i$, $E_i$, $G$ respectively on the right, for $1 \leq i \leq m-1$.  Hence this gives a categorification of the action of $\mathfrak{gl}^+_m \times \mathfrak{gl}^-_m$ on $\bigwedge(n,m,r)$.
\end{theorem}
\begin{proof}
We show the claim for $\cF_{i,+}$ and leave the rest to the reader. Let $\kappa = (\underline{k},\underline{\mu},\epsilon)$ be a block triple and $M^{\mathfrak{q}_{\underline{k}}}(\lambda)$ be a parabolic Verma module in the corresponding block. Then 
\abovedisplayskip0.3em
\belowdisplayskip0.3em
\[
[M^{\mathfrak{q}_{\underline{k}}}(\lambda) \otimes V] = \sum_{l: \lambda + \epsilon_l \in \Lambda_{\underline{k}}} [M^{\mathfrak{q}_{\underline{k}}}(\lambda + \epsilon_l)] + \sum_{l: \lambda - \epsilon_l \in \Lambda_{\underline{k}}} [M^{\mathfrak{q}_{\underline{k}}}(\lambda - \epsilon_l)].
\]
Checking now when a parabolic Verma module of the form $M^{\mathfrak{q}_{\underline{k}}}(\lambda - \epsilon_l)$ lies in the block corresponding to $\kappa_{i,+}$ we see that this is the case if and only if $
\lambda_l + \rho_l = i + \half$, and similarly with $M^{\mathfrak{q}_{\underline{k}}}(\lambda + \epsilon_l)$. Under our identification this coincides exactly with the way $E_i$ acts on the exterior power \eqref{Defwedge}.
\end{proof}

\begin{remark}
\label{analogy}
Using $\kw_r $,  we can also identify $K_0\left(\bigoplus_{\kappa \in {\rm Bl}(n,m,r)} \cO_\kappa(n)\right)$ with $\bigwedge(n,m,r)$ and use analogous functors to obtain an action of $\mathfrak{gl}^+_r \times \mathfrak{gl}^-_r$.  
\end{remark}

\section{Skew Howe duality: quantum case} \label{section:howe-quantum}
We will now go back to the quantum case, hence the ground field is again $\mQ(q)$, but will use often the same notations as in the last section. In particular we consider truncations of the spaces $\mW$ and $\mV$ we had in Section \ref{section:coideal}.

\subsection{The module \texorpdfstring{$\bigwedge_q(n,m,r)$}{exterior powers} in the quantum case}
In general it is not obvious how to lift exterior powers to the quantum setting, see \cite{BZ} for a detailed treatment of this problem. In particular a quantization of \eqref{dec} is non-trivial, since we had to deal with exterior powers of triple tensor products.  Here we take an easier route and quantize our space $\bigwedge(n,m,r)$ as
\abovedisplayskip0.3em
\belowdisplayskip0.3em
\begin{gather*}
{\bigwedge}_q(n,m,r) \,=\, \bigoplus_{\underline{k} \in C(n,r)} {\bigwedge^{\underline{k}}}_q \mV_m^\hint,
\end{gather*}
and define on this space explicitly two commuting coideal actions. We start with a part of it which is already interesting on its own.

\subsection{Schur--Weyl duality and its categorification}
\label{sectionSW}
We restrict to the half-integer case and consider $\underline{k}=(1,1,\ldots, 1)$ and $n=r$, hence $\bigwedge^{\underline{k}}_q\mV_m^\hint={\mV_m^\hint}^{\otimes n}.$ Recall the Hecke algebras with unequal parameters from \cite[2.1]{GJbook}.
\begin{definition} 
The {\it Hecke algebra}  $\mathbb{H}_n(\mathrm{D})$ corresponding to the Weyl group $W_n=\langle s_{i}\mid 0\leq i\leq n-1\rangle$ of $\mathfrak{so}_{2n}$ is the unital $\mathbb{Q}(q)$-algebra with generators $H_i$, for $0\leq i\leq n-1$, subject to the  quadratic relation $H_i^2=1+(q^{-1}-q)H_i$, and the type $ \mathrm{D}$ braid relations, for $1\leq i,j\leq n-1$:
\abovedisplayskip0.3em
\belowdisplayskip0.3em
\[
H_iH_j=H_jH_i \text{ if } s_{i}s_{j}=s_{j}s_{i} \quad \text{and} \quad H_iH_jH_i=H_jH_iH_j \text{ if } s_{i}s_{j}s_{i}=s_{j}s_{i}s_{j}.
\]
Let  $\mathbb{H}_n(\mathrm{B})$  be the {\it $2$-parameter Hecke algebra specialised at $(1,q)$} corresponding to the Weyl group of $\mathfrak{so}_{2n+1}$, that is the algebra with generators $H_i$, for $0\leq i\leq n-1$, subject to the same defining relations not involving $H_0$ and the relations $H_0^2=1,\,\,\, H_0H_1H_0H_1=H_1H_0H_1H_0, \text{ and } H_0H_i=H_iH_0 \text{ for } i > 1$.
\end{definition}

The following straight-forward fact realizes $\mathbb{H}_n(\mathrm{D})$ as a subalgebra of $\mathbb{H}_n(\mathrm{B})$.
\begin{lemma}\label{twoHeckes}
There is an inclusion of $\mathbb{Q}(q)$-algebras $\mathbb{H}_n(\mathrm{D})\hookrightarrow \mathbb{H}_n(\mathrm{B})$ given on generators by $H_i\mapsto H_i$ for $i\not=0$ and $H_0\mapsto H_0H_1H_0$.
\end{lemma}

\begin{lemma} \label{lem:Hecke}
There is a right action of the Hecke algebra $\mathbb{H}_n(\mathrm{D})$ on ${\mV_m^\hint}^{\otimes n}$ commuting with the action of $\cH^\hint$. The generator $H_i$ for $1\leq i\leq n-1$ acts on the $i$-th and $i+1$-st tensor factor by the formula
\abovedisplayskip0.3em
\belowdisplayskip0.3em
\begin{gather}
v_a\otimes v_b.H_i \,=\,
\begin{cases}
v_b\otimes v_a&\text{if }a<b,\\
v_b\otimes v_a+(q^{-1}-q)v_a\otimes v_b&\text{if }a>b,\\
q^{-1}v_b\otimes v_a&\text{if }a=b.
\end{cases}
\end{gather}
and $H_0$ acts on the first and second tensor factor by
\abovedisplayskip0.3em
\belowdisplayskip0.3em
\begin{gather}
v_a\otimes v_b.H_0\,=\,
\begin{cases}
v_{-b}\otimes v_{-a}&\text{if } a+b>0,\\
v_{-b}\otimes v_{-a}+(q^{-1}-q)v_a\otimes v_b&\text{if } a+b < 0,\\
q^{-1}v_{-b}\otimes v_{-a}&\text{if }a=-b.
\end{cases}
\end{gather}
It extends to an action of $\mathbb{H}_n(\mathrm{B})$ by setting $v_a\otimes v_b.H_0=v_{-a}\otimes v_b$ for $H_0 \in \mathbb{H}_n(\mathrm{B})$.
\end{lemma}
\begin{proof}
One directly verifies that this gives a well-defined action of both $\mathbb{H}_n(\mathrm{D})$ and $\mathbb{H}_n(\mathrm{B})$. To see that they commute with the action of $\cH^\hint$, note that the generators $H_i$ commute with the action of all of $\cU^\hint$. Thus it is left to show that $H_0$ for $\mathbb{H}_n(\mathrm{B})$ commutes with the action of $\cH^\hint$, which is straight-forward.
\end{proof}

\begin{theorem}[Schur-Weyl duality]
\label{prop:Schur_Weyl}
The actions of $\cH^\hint$ and $\mathbb{H}_n(\mathrm{B})$ on ${\mV_m^\hint}^{\otimes n}$ are each others centralizers.
\end{theorem}
\begin{proof}
Since both actions commute, the image in $\mathrm{End}_{\mathbb{Q}(q)}({\mV_m^\hint}^{\otimes n})$ is contained in the respective centralizer algebra, and it suffices to compare the dimensions over $\mathbb{Q}(q)$.   For this we apply a general deformation argument reducing it to the $q=1$ case.  We abbreviate $K=\mathbb{Q}(q)$, $M={\mV_m^\hint}^{\otimes n}$  and consider the Laurent polynomial rings  $\cA=\mathbb{Q}[q,q^{-1}]$, which acts on $\mathbb{C}$ via $q \mapsto 1$. Let $\cH^\hint_{\cA}$, $\mathbb{H}_n(\mathrm{B})_{\cA}$, and $M_{\cA}$ be the respective integral forms. The integral versions 
\abovedisplayskip0.3em
\belowdisplayskip0.3em
\begin{gather*}
\cH_{\cA}^\hint \,\, \xrightarrow{\phantom{x} \alpha \phantom{x}} \,\, \op{End}_{\cA}(M_{\cA}) \,\, \xleftarrow{\phantom{x} \beta \phantom{x}} \,\,  \mathbb{H}_r(\mathrm{B})_{\cA}
\end{gather*}
of the actions still commute and induce, via base change $_-\otimes_{\cA} K$, the actions $\alpha_K, \beta_K$, and, via $_-\otimes_{\cA} \mathbb{C}$, the classical action $\alpha_\mathbb{C}, \beta_\mathbb{C}$. In the latter case it follows from the classical Schur--Weyl duality for wreath products, see e.g. \cite[Theorem 9]{MSSchurWeyl} that the two actions centralize each other, that is $C(\op{im}\alpha_\mC)=\op{im}\beta_\mC$ and $C(\op{im}\beta_\mC)=\op{im}\alpha_\mC$ where $C$ denotes the centralizer. To obtain the lemma it suffices then to show
\abovedisplayskip0.3em
\belowdisplayskip0.3em
\begin{gather}
\label{claim1}
\op{dim}_KC(\op{im}\alpha_K)\leq\op{dim}_K\op{im}\beta_K \quad \text{and} \quad \op{dim}_K C(\op{im}\beta_K)\leq\op{dim}_K\op{im}\alpha_K.
\end{gather}
We prove the first part, the second is done analogous. Since $M_{\cA}$ is a free $\cA$-module $\op{End}_{\cA}(M_{\cA})$ is free over $\cA$. Moreover,  $\op{im}\alpha_\cA$ is torsionfree, hence free since $\cA$ is PID, and thus $ \op{rk}_\cA\op{im}\alpha_\cA=\op{dim}_K\op{im}\alpha_K$. On the other hand, $\op{rk}_\cA\op{im}\alpha_\cA\geq\op{dim}_\mathbb{C}\op{im}\alpha_\mathbb{C}$ and therefore 
\abovedisplayskip0.3em
\belowdisplayskip0.3em
\begin{gather*}
\op{dim}_K\op{im}\alpha_K \,\, \geq \,\, \op{dim}_\mathbb{C}\op{im}\alpha_\mathbb{C}.
\end{gather*}
This implies $\op{dim}_KC(\op{im}\alpha_K)\leq \op{dim}_\mathbb{C}C(\op{im}\alpha_\mathbb{C})$ by definition of $C$. On the other hand we have $\op{rk}_{\cA} C(\op{im}\alpha_\cA)=\op{dim}_KC(\op{im}\alpha_K)$, again by definition of $C$. Thus altogether $\op{dim}_KC(\op{im}\alpha_K)\leq \op{dim}_\mathbb{C}C(\op{im}\alpha_\mathbb{C})=\op{dim}_\mathbb{C}\op{im}\beta_\mathbb{C}\leq \op{dim}_K\op{im}\beta_\mathbb{K}$. This proves Claim \eqref{claim1} and we are done.
\end{proof}

\begin{remark} The above Lemma gives a simultaneous treatment of the type $\mathrm{D}$ and specialised type $\mathrm{B}$ Hecke algebras. Similar commuting actions where already defined for affine Hecke algebras in \cite{JM}, \cite{CGM}  quantising the original construction of \cite{EFM}.  A Schur--Weyl duality statement analogous to Theorem~\ref{prop:Schur_Weyl} for the ordinary type $B$ Hecke algebra was recently also obtained in \cite[Section 5.1]{BW}, \cite{Watanabe} and implicitly in \cite{Li}.  
\end{remark}

Consider the graded version $\hat{\cO}^{\mathfrak{q}_{\underline{k}}}_{\leq m}(n)$ of ${\cO}^{\mathfrak{q}_{\underline{k}}}_{\leq m}(n)$ as described in Section~\ref{sec:Koszul}. We again have the corresponding identification of $\mathbb{Q}(q)$-vector spaces
\abovedisplayskip0.3em
\belowdisplayskip0.3em
\begin{gather}
\label{identgrad}
\hat{\kw}_{\underline{k}}\,:\,K_0\left(\hat{\cO}^{\mathfrak{q}_{\underline{k}}}_{\leq m}(n)\right) \,\longrightarrow \, {\bigwedge^{\underline{k}}}_q\mV_m^\hint.
\end{gather}
by sending the class $[\hat{M}^{\mathfrak{q}_{\underline{k}}}(\la)]$ of the standard graded lift of the parabolic Verma module $M^{\mathfrak{q}_{\underline{k}}}(\la)$ to $v_\la$ with head concentrated in degree zero. (The multiplication with $q$ corresponds to the grading shift $\langle 1\rangle$).

Recall the right exact {\it twisting endofunctors $T_i=T_{s_{\alpha_i}}$}, from \cite{AL}, \cite{AndersenStroppel}, of $\cO(n)$ attached to $s_{\alpha_i}$. Their left derived functors define auto-equivalences of the derived category. The Hecke algebra action has a categorification given by graded lifts of these derived functors:

\begin{theorem}[Categorified Schur--Weyl duality]
\label{prop:twisting}
Let $\underline{k}=(1,1,\ldots ,1)$. Then the derived functor $\mathcal{L}T_i$, $0\leq i\leq n-1$ has a graded lift,
\abovedisplayskip0.3em
\belowdisplayskip0.3em
\begin{gather*}
\mathcal{L}\hat{T}_i\,:\, D^b(\hat{\cO}^{\mathfrak{q}_{\underline{k}}}_{\leq m}(n))\,\longrightarrow \,D^b(\hat{\cO}^{\mathfrak{q}_{\underline{k}}}_{\leq m}(n)),
\end{gather*}
an automorphism of the bounded derived category $D^b(\hat{\cO}^{\mathfrak{q}_{\underline{k}}}_{\leq m}(n))$, such that the induced action of $\hat{\mathcal{L}}T_i$ on $K_0(D^b(\hat{\cO}^{\mathfrak{q}_{\underline{k}}}_{\leq m}(n)))=K_0(\hat{\cO}^{\mathfrak{q}_{\underline{k}}}_{\leq m}(n))$ agrees with the action of the Hecke algebra generator $H_i\in\mathbb{H}_n(\mathrm{D})$ under the identification $\hat{\kw}_{\underline{k}}$. 

For the extension to $\mathbb{H}_n(\mathrm{B})$,  the element $H_0\in\mathbb{H}_n(\mathrm{B})$ acts by the corresponding equivalence changing the parity of the block.  
\end{theorem}
\begin{proof}
 Consider first $\mathbb{H}_n(\mathrm{D})$. If we forget the grading this is \cite[2.1]{AL}. The graded version is \cite[Propositions 5.1 and 5.2]{FKS}, see also \cite[Proposition 7]{MOSpair}. The extension to $\mathbb{H}_n(\mathrm{B})$ is then straight-forward, since the additional functor is an equivalence of categories just changing the parity of the block.
\end{proof}

\begin{remark}
The analogous statement of Theorem~\ref{prop:twisting} using category $\cO$ for type $B$ decategorifies by definition to the construction in \cite{Li}, \cite{Li2}. The twisting functors naturally commute, in the sense of \cite{AS}, with the translation functors, decategorifying to commuting actions on the Grothendieck group.
\end{remark}

\subsection{Soergel bimodules and graded lifts of special projective functors}
We now extend the categorification from Theorem~\ref{prop:catskewclass} to the quantum case. We first have to give an alternative description of the special functors.

To also include blocks containing parabolic Verma modules with highest weights supported on the integers, we have to introduce a more technical version of block triples from Definition~\ref{def:blocks_howe}.

\begin{definition} \label{def:blocks_general}
Fix $(\underline{k},\underline{d},\epsilon)$ with $\underline{k} \in C(n,r)$, $\underline{d} \in C(n)$ and ${\epsilon} \in \{ 0,1,01,\emptyset\}$, such that $\epsilon \in \{0,1\}$ if $d_1=0$, $\epsilon = \emptyset$ if $d_1=1$ and $\epsilon=01$ otherwise. 

Then $\cO_{(\underline{k},\underline{d}, {\epsilon})}(n)$ denotes the block of $\cO^{\mathfrak{q}_{\underline{k}}}(n)$ containing those parabolic Verma modules of highest weights $\lambda$ satisfying
\abovedisplayskip0.3em
\belowdisplayskip0.3em
\[
d_i = \# \left\lbrace 
j \,\scalebox{1.3}{$\mid$} \mid \lambda_j + \rho_j \mid = i-1 \right\rbrace, \text{ for } i \in \mathbb{Z}_{\geq 1}
\]
and if $d_1=0$, then $\epsilon = 1$ if the parity of negative entries in $\lambda+\rho$ is even and $\epsilon=0$ otherwise. Again we call such triples \emph{block triple} (supported on integers). For a block triple $\kappa$ we denote by $\mathrm{pr}_\kappa$ the projection onto $\cO_{\kappa}(n)$ in any sum of blocks.

We allow $\underline{d} \in C(n,m+1)$, which results in a block contained in $\cO^{\mathfrak{q}_{\underline{k}}}_{\leq m}(n)$. 
\end{definition}

\begin{definition}
To a block triple $\kappa = (\underline{k},\underline{d},\epsilon)$ we associate a {\it corresponding parabolic Weyl group} as the subgroup of the Weyl group $W_n$ of the form
\begin{gather*}
W_{\kappa}\,=\,W_{d_1'}^\epsilon \times W_{d_2'} \times \ldots \times W_{d_l'},
\end{gather*}
where $d_1',\ldots,d_l'$ are the non-zero entries in $d$ in order and where $W_{d_1'}^\epsilon$ is generated by $s_2,\ldots,s_{d_1'-1}$ together with $s_1$ (respectively $s_0$) if $\epsilon = 1$ (respectively $\epsilon=0$), respectively together with  $s_0,s_1$ if $\epsilon = 01$; and $W_{d_k'}$ is generated by $s_{d_1'+\ldots+d_{k-1}'+1},\ldots,s_{d_1'+\ldots+d_{k}'-1}$ for $k > 1$. 
\end{definition}

Recall from \cite{BGGclass}, \cite{Hbook} the definition and classification of translation on and out of walls. Let $\cO_\kappa(n)$ and $\cO_{\kappa'}(n)$ be blocks in $\cO(n)$ for $\kappa =( (1,\ldots 1),\bd,\epsilon)$ and $\kappa^\prime=((1,\ldots 1),\bd^\prime,\epsilon)$.

\begin{definition}
Then we denote by $\Theta_\kappa^{\kappa'}: \cO_\kappa\rightarrow\cO_{\kappa'}$ the {\it translation functor off the walls} respectively {\it onto the walls} in case  $W_{\kappa'}\subseteq W_{\kappa}$ respectively $W_{\kappa}\subseteq W_{\kappa'}$. (The doubled usage of the letter $\Theta$ for totally different objects is hopefully not causing confusions.)
\end{definition}

\begin{lemma}\label{lemma:Fastransl}
The special projective functors from Definition~\ref{def:specproj} are the following compositions of translation functors (see Section~\ref{boring} for a proof): 
\abovedisplayskip0.3em
\belowdisplayskip0.3em
\begin{gather*}
\pr_{{}_{+i}\kappa}\;{\cF_{i,+}}\;\pr_\kappa\,\cong\,\Theta^{{}_{+i}\kappa}_{\kappa^{+i}}\,\Theta_{\kappa}^{\kappa^{+i}}\quad
\text{and} \quad\pr_{{}_{-i}\kappa}\;{\cF_{i,-}}\;\pr_\kappa \,\cong\,
\Theta^{{}_{-i}\kappa}_{\kappa^{-i}}\;\Theta_{\kappa}^{\kappa^{-i}}.
\end{gather*}
Moreover, $\cF_{0}$ is an equivalence of categories changing the parity in the block triple.
\end{lemma}

\begin{definition} Given a block triple $\kappa$, we let $R=S(\mathfrak{h})$ and $R^\kappa=R^{W_{\kappa}}$ the {\it associated ring of invariants} given by the invariants of $W_{\kappa}$ viewed as a subgroup of $W_n$ acting in the obvious way on $R$. 
\end{definition}

To each functor $\Theta^{{}_{+i}\kappa}_{\kappa^{+i}}\;\Theta_{\kappa}^{\kappa^{+i}}$ respectively $\Theta^{{}_{-i}\kappa}_{\kappa^{-i}}\;\Theta_{\kappa}^{\kappa^{-i}}$ we associate a {\it singular Soergel bimodule}  in the sense of \cite{StHC}, \cite{Geordie} that is a certain $(R^{{}_{+i}\kappa},R^{\kappa})$ respectively $(R^{{}_{-i}\kappa},R^{\kappa})$-bimodule, as follows.

\begin{definition}\label{Soergelbimodules}
Assume we are given a block triple $\kappa = (\underline{k},\underline{d},\epsilon)$.

\emph{If $i > \half$:} We have the block triples $\kappa^{+i}=(\underline{k},\underline{d}^{+i},\epsilon)$, $\kappa^{-i}=(\underline{k},\underline{d}^{-i},\epsilon)$, ${}_{+i}\kappa = (\underline{k},{}_{+i}\underline{d},\epsilon)$, and ${}_{-i}\kappa = (\underline{k},{}_{-i}\underline{d},\epsilon)$ and take  the $(R^{{}_{+i}\kappa} ,R^{\kappa})$-bimodule
\abovedisplayskip0.3em
\belowdisplayskip0.3em
\begin{gather*}
R(i, +,\kappa)\,=\,\begin{cases}
\{0\}&\text{if } d_{i+\half} = 0,\\
R^{\kappa^{+i}} &\text{if } d_{i+\half}>0,
\end{cases}
\end{gather*}
and the $(R^{{}_{-i}\kappa} ,R^{\kappa})$-bimodule
\abovedisplayskip0.3em
\belowdisplayskip0.3em
\begin{gather*}
R(i, -,\kappa)\,=\,\begin{cases}
\{0\}&\text{if } d_{i-\half} = 0,\\
R^{\kappa^{-i}} &\text{if } d_{i-\half}>0.
\end{cases}
\end{gather*}

\emph{If $i = \half$:} We set $\kappa^{+i}=(\underline{k},\underline{d}^{+i},\epsilon)$ as above, but ${}_{+i}\kappa = (\underline{k},{}_{+i}\underline{d},\epsilon^\prime)$ where $\epsilon^\prime=\emptyset$ if $\epsilon \in \{0,1\}$ and $\epsilon^\prime=01$ if $\epsilon \in \{01,\emptyset\}$. The $(R^{{}_{+i}\kappa} ,R^{\kappa})$-bimodule is 
\abovedisplayskip0.3em
\belowdisplayskip0.3em
\begin{gather*}
R(i, +,\kappa)\,=\,\begin{cases}
\{0\}&\text{if } d_{i+\half} = 0,\\
R^{\kappa^{+i}} &\text{if } d_{i+\half}>0.
\end{cases}
\end{gather*}

\emph{If $i = -\half$ and $d_1>1$:} Then define $\kappa^{-i}=(\underline{k},\underline{d}^{-i},\epsilon^\prime)$ and ${}_{-i}\kappa = (\underline{k},{}_{-i}\underline{d},\epsilon^\prime)$, with $\epsilon^\prime=\emptyset$ if $d_1=1$ and $\epsilon^\prime=01$ otherwise, and the $(R^{{}_{-i}\kappa} ,R^{\kappa})$-bimodule
\abovedisplayskip0.3em
\belowdisplayskip0.3em
\begin{gather*}
R(i, -,\kappa)\,=\, R^{\kappa^{-i}}.
\end{gather*}

\emph{If $i = -\half$ and $d_1 \leq 1$:} Then define $\kappa^{-i}_0=(\underline{k},\underline{d}^{-i},\epsilon_0)$, $\kappa^{-i}_1=(\underline{k},\underline{d}^{-i},\epsilon_1)$, ${}_{-i}\kappa_0 = (\underline{k},{}_{-i}\underline{d},\epsilon_0)$, and ${}_{-i}\kappa_1 = (\underline{k},{}_{-i}\underline{d},\epsilon_1)$, with $\epsilon_0=0$ and $\epsilon_1=1$. Then we take the $(R^{{}_{-i}\kappa_0} ,R^{\kappa})$-bimodule
\abovedisplayskip0.3em
\belowdisplayskip0.3em
\begin{gather*}
R(i, -,\kappa)_0\,=\,\begin{cases}
\{0\}&\text{if } d_{i-\half} = 0,\\
R^{\kappa^{-i}_0} &\text{if } d_{i-\half}>0,
\end{cases}
\end{gather*}
and the $(R^{{}_{-i}\kappa_1} ,R^{\kappa})$-bimodule
\abovedisplayskip0.3em
\belowdisplayskip0.3em
\begin{gather*}
R(i, -,\kappa)_1\,=\,\begin{cases}
\{0\}&\text{if } d_{i-\half} = 0,\\
R^{\kappa^{-i}_1} &\text{if } d_{i-\half}>0.
\end{cases}
\end{gather*}
\emph{If $i = 0$:} We have $\kappa^{+i}=(\underline{k},\underline{d}^{+i},\epsilon)$ as above, but ${}_{+i}\kappa = (\underline{k},\underline{d},\epsilon^\prime)$ where $\epsilon^\prime=\epsilon+1$. The $(R^{{}_{+i}\kappa} ,R^{\kappa})$-bimodule is given by
\abovedisplayskip0.3em
\belowdisplayskip0.3em
\begin{gather*}
R(0,\kappa)\,=\,\begin{cases}
\{0\}&\text{if } d_{\half} = 0,\\
R^{\kappa^{+i}} &\text{if } d_{\half}>0.
\end{cases}
\end{gather*}
\end{definition}

Let $\op{Coinv}^\kappa=R^{W_{\kappa}}/\op{Inv}$ be the {\it algebra of partial coinvariants}, where  $\op{Inv}$ is the ideal generated by all $W_n$-invariant polynomials without constant term. Then we have, \cite[1.2]{Sperv},  Soergel's functor $\mV_{\it{Soergel}}$ ``into combinatorics'' 
\abovedisplayskip0.3em
\belowdisplayskip0.3em
\begin{gather}
\label{Soergel}
\mV_{\it{Soergel}}:\quad{\cO}_\kappa(n) \,\, \longrightarrow \,\, \op{Coinv}^\kappa-\mathrm{mod}.
\end{gather}
Each of the above Soergel bimodules $X$ defines an exact functor $X\otimes_R  ?$ between the corresponding categories of finitely generated modules of the form $\op{Coinv}^\kappa-\mathrm{mod}$.
Note that each Soergel bimodule from Definition~\ref{Soergelbimodules} is an $(R_3,R_1)$-bimodule of the form $R_2$, that is free from each side, where the $R_i$ are certain rings of invariants in $R$ such that $R_3\subseteq R_2\supseteq R_1$. In particular, the functors $X\otimes_R  ?$ are always a composition of an induction and a restriction functor.  By \cite[Theorem 10]{Sperv}, these then  lift to functors between the corresponding  categories  ${\cO}_\kappa(n)$, since $\mV_{\it{Soergel}}$ is a well-behaved quotient functor. These lifts can be chosen to be exactly the corresponding functor {$\Theta^{{}_{+i}\kappa}_{\kappa^{+i}}\;\Theta_{\kappa}^{\kappa^{+i}}$ respectively $\Theta^{{}_{-i}\kappa}_{\kappa^{-i}}\;\Theta_{\kappa}^{\kappa^{-i}}$ we started with.

\subsection{Graded translation functors}
Using the functor \eqref{Soergel} and the grading on $R$, \cite[3.2]{Stroppel} shows that translation functors  $\Theta_\kappa^{\kappa'}: \cO_\kappa\rightarrow\cO_{\kappa'}$ off or onto the walls have graded lifts, that means can be lifted to graded functors
\begin{eqnarray}
\label{gradedTransl}
\hat{\Theta}_\kappa^{\kappa'}: \hat{\cO}_\kappa(n)\longrightarrow\hat{\cO}_{\kappa'}(n)
\end{eqnarray}
(Strictly speaking, the arguments in \cite{Stroppel}  about the existence of graded lifts are only formulated for the principal block, but generalize directly to arbitrary integral blocks in the obvious way by taking invariants of the coinvariants.)
Translation to the wall sends Verma modules to Verma modules \cite[7.6]{Hbook}, translation off walls sends indecomposable projectives to indecomposable projectives \cite[adjunction formulas 7.2 and 7.7]{Hbook}. Thus by the classification theorem of translation functors from \cite{BGGclass}, see \cite[Theorem 3.1]{StDuke} for an explicit formulation, these functors are indecomposable,  hence a graded lift is unique up to isomorphism and overall shift in the grading, \cite[Lemma 1.5]{Stroppel}. For the functors $\hat{\Theta}_\kappa^{\kappa'}$, graded lifts \eqref{gradedTransl} therefore exists and are unique up to isomorphism and overall grading shift. We {\it choose} such graded lifts such that the graded version of the Verma module with maximal weight in the orbit is sent to a projective module with head concentrated in degree zero.  

\begin{definition}
Define graded lifts of the functors from  Lemma~\ref{lemma:Fastransl} as
\abovedisplayskip0.3em
\belowdisplayskip0.3em
\begin{gather}
\label{Fhatdef}
\begin{gathered}
\pr_{{}_{+i}\kappa}\;{\cF_{i,+}}\;\pr_\kappa\,\cong\,\Theta^{{}_{+i}\kappa}_{\kappa^{+i}}\,\Theta_{\kappa}^{\kappa^{+i}}\langle d_{i}-d_{i+1}+1\rangle\\
\pr_{{}_{-i}\kappa}\;{\cF_{i,-}}\;\pr_\kappa \,\cong\,
\Theta^{{}_{-i}\kappa}_{\kappa^{-i}}\;\Theta_{\kappa}^{\kappa^{-i}}\langle d_{i+1}-d_{i}+1\rangle.
\end{gathered}
\end{gather}
where $\langle r\rangle$, $r\in \mZ$, means the grading is shifted by adding $r$ to the degree. 
Let $\hat\cF_{0}$ be the graded equivalence which changes the parity in the block triple (if possible) and is zero otherwise. Finally for $i\in \mZ_{>0}$ we set 
\begin{gather}
\label{gradedBs}
\hat{\mathbb{B}}_{\pm i}=\bigoplus_{\kappa}  \pr_{{}_{\pm
i}\kappa}\;{\hcF_{i,\pm}}\;\pr_{\kappa}\quad \text{and}\quad
\hat{\mathbb{B}}_{0}=\bigoplus_{\kappa}\pr_{\kappa_{0}}\;
{\hcF_{0}}\;\pr_{\kappa}.
\end{gather}
\end{definition}

These functors induce functors on $\hat{\cO}^{\mathfrak{q}_{\underline{k}}}_{\leq m}(n)$ and then via the identification \eqref{identgrad} $\mathbb{Q}(q)$-linear endomorphisms on the $\mathbb{Q}(q)$-vector space $\bigwedge_q(n,m,r)$. Our main result here is the following categorification theorem.

\begin{theorem}[Categorified quantum wedge products]
\label{thm:main}
When identifying
\begin{eqnarray}
\label{identification}
K_0\left(\bigoplus_{\kappa \in {\rm Bl}(n,r,m)} \hat{\cO}_\kappa(n)\right)&=&{\bigwedge}_q(n,m,r),
\end{eqnarray}
as $\mathbb{Q}(q)$-vector spaces, the action induced from the exact functor $\hat{\mathbb{B}}_{i}$ coincides with the action of $B_i\in\cH^\hint$ for any $i\in\mZ$.
\end{theorem}

We start with the following main insight:

\begin{prop}
\label{lem:easy}
Theorem~\ref{thm:main} is true when restricted to ${\mV_m^\hint}^{\otimes r}$.
\end{prop}

\begin{proof}
It is enough to verify that the action of $[\hat{\mathbb{B}}_{i}]$ agrees with the action of $B_i\in\cH^\hint$ on the standard basis corresponding to isomorphism classes of graded lifts of Verma modules. Recall the commuting action of the Hecke algebra from  Lemma~\ref{lem:Hecke} and note that twisting functors commute with translation functors, \cite[Theorem 3.2]{AndersenStroppel}. Hence by Theorem~\ref{prop:twisting} it is enough to compare the action for Verma modules whose highest weight $\la$ is maximal in their dot-orbit. Moreover by linearity we can restrict ourselves to the case of standard graded lifts $\hat{M}(\la)$. We follow the proof of Lemma~\ref{lem:boring}. Consider first the functors $\hat{\mathbb{B}}_{i}$ for $i>0$. The equivalence $\eta$ lifts to the graded setting and just renames the highest weight of $\hat{M}(\la)$. The translation out of the wall produces an indecomposable projective module $P$, \cite[Section 7.11]{Hbook} with a graded Verma flag. The Verma subquotients are, up to some grading shifts, the $\hat{M}(\nu-\rho)$ with the same weights $\nu$ as before. Amongst them let $\nu(r)$ be the weight where the $r$-th $i+2$ from the left has been changed. By our normalization, $\hat{M}(\nu(r_{\op{max}})-\rho)$ with $r=r_{\op{max}}$ maximal occurs without grading shift and hence $P\cong \hat{P}(\nu(r_{\op{max}})-\rho)$. By \cite[Theorem 3.11.4]{BGS}, the graded multiplicities are encoded in parabolic Kazhdan-Lusztig polynomials. Our case corresponds to the explicit formula \cite[Proposition 2.9]{SoergelKL} and we obtain that $P$ has a filtration with subquotients $\hat{M}(\nu(r)-\rho)\langle r_{\op{max}-r}\rangle$. 

In our example we get
\begin{eqnarray*}
[\hat{M}(\half(3,3,3,7,7,5,9,9)-\rho)]&+&q[\hat{M}(\half(3,3,3,7,5,7,9,9)-\rho)] \\
&+& q^2[\hat{M}(\half(3,3,3,5,7,7,9,9)-\rho)]
\end{eqnarray*}
in the Grothendieck group. Now translation to the wall just changes the numbers $i+1$ appearing in the weights to $i$ without any extra grading shift by \cite[Theorem 8.1]{Stroppel}. The last step is a graded equivalence renaming the weights. In our example we obtain
\begin{eqnarray*}
[\hat{M}(\half(3,3,3,5,5,3,7,7)-\rho)]&+&q[\hat{M}(\half(3,3,3,5,3,5,7,7)-\rho)]\\
&+& q^2[\hat{M}(\half(3,3,3,3,5,5,7,7)-\rho)].
\end{eqnarray*}
In any case, the result agrees with the action of $B_i$ up to an overall grading shift by  $\langle \mu_{i}-\mu_{i+1}+1\rangle$.
In case $i=0$ the arguments are completely parallel except of the grading shift at the end. As above, the lowest grading shift appearing for Verma modules is zero.
Expressing $B_0v_\la$ in the standard basis, the smallest $q$ power appearing equals $(-1)+n_1-n_{-1}+(n_{-1}-n_1+1)=0$, where $n_j$ is the number of $j$'s appearing in the weight $\la$. Hence the claim of the Lemma follows for $B_0$. For $B_{-i}$, $i>0$ the arguments are analogous.
\end{proof}

\begin{proof}[Proof of Theorem~\ref{thm:main}]
Let $i>0$. Recall from \eqref{bvs} that the action of $B_i$ on some standard basis vector $\mathbf{v}=v_{\textbf{i}_1} \otimes \ldots \otimes v_{\textbf{i}_r}\in \bigwedge_q^{k_1} \mV^\hint \otimes \ldots \otimes \bigwedge_q^{k_r} \mV^\hint$ from \eqref{mb1} can be obtained by viewing $\mathbf{v}$ as a vector inside $\bigotimes^{k_1} \mV^\hint \otimes \ldots \otimes \bigotimes^{k_r} \mV^\hint$ and compute the action there. Let $[\hat{M}^{\mathfrak{q}_{\underline{k}}}(\mu)]$ be the class of the standard graded lift $\hat{M}^{\mathfrak{q}_{\underline{k}}}(\mu)$ of the parabolic Verma module attached to $\mathbf{v}$. Assume first that the stabilizer of $\mu$ is trivial under the dot-action and write $\mu=x\cdot\la$ with $\la$ maximal in the same dot-orbit and $x\in W_n$. Recall that by definition the Levi subalgebra of $\mathfrak{q}_{\underline{k}}$ is of type $ \mathrm{A}$ with Weyl group isomorphic to the product $S_{\underline{k}}$ of symmetric groups. Then
\begin{eqnarray}
\label{alternate}
[\hat{M}^{\mathfrak{q}_{\underline{k}}}(\mu)]=[\hat{M}^{\mathfrak{q}_{\underline{k}}}(x\cdot\la)]=\sum_{y\in S_{\underline{k}}}(-q)^{\ell(w)}[\hat{M}(yx\cdot\la)]
 \end{eqnarray}
 by the formula \cite[Proposition 3.4]{SoergelKL} for parabolic Kazhdan-Lusztig polynomials together with \cite[Corollary 2.5]{StDuke}. This however fits precisely with the formula \eqref{bvs} and the claim follows. Assume now that the stabilizer $W(\mu)$ of $\mu$ is not trivial under the dot-action. Then choose $\la'$ an integral weight, maximal in its dot-orbit and with $W(\la')$ trivial. Let $\kappa$ respectively  $\kappa'$ be the non-parabolic block triples corresponding to $\la'$ and $\la$. Then $\hat{\Theta}_\kappa^{\kappa'}\hat{M}(\la')\cong\hat{M}(\la)$ by definition. We claim that
 \begin{eqnarray}
 \label{auf}
   \hat{\Theta}_\kappa^{\kappa'}\hat{M}(x\cdot\la')&\cong&\hat{M}(x\cdot\la)
   \end{eqnarray}
   for any shortest coset representative $x\in W_n/W(\la)$. Since this is true if we forget the grading and ${M}(x\cdot\la)$ is indecomposable we only have to figure out possible overall grading shifts, \cite[Lemma 1.5]{Stroppel}. If $x$ is the identity we are done by convention. Otherwise we find a simple reflection $s_{\alpha_i}$ such that $\ell(s_{\alpha_i}x)<\ell(x)$ and  use again that twisting functors, even graded, commute with translation functors. From Lemma~\ref{lem:Hecke} and Theorem~\ref{prop:twisting} the result follows then by induction on $\ell(x)$.
   This implies $\hat{\Theta}_\kappa^{\kappa'}\hat{M}^{\mathfrak{q}_{\underline{k}}}(x\cdot\la')
  \cong\hat{M}^{\mathfrak{q}_{\underline{k}}}(x\cdot\la)$,
  since this is again true up to a possible overall shift in the grading which is however zero by \eqref{auf}.
  Moreover, we can choose $\la'$ such that $yx$ is of minimal length in its coset for any $y\in S_{\underline{k}}$. (In practice we consider $\la+\rho$ and replace its entries by a strictly increasing sequence of positive integers. We do this by replacing first the lowest numbers from the left to right, then replace the second lowest from the left to right etc.)
  Now we apply $[\hat{\Theta}_\kappa^{\kappa'}]$ to \eqref{alternate} and obtain with \eqref{auf}
  \begin{eqnarray}
  \label{parabs}
  [\hat{M}^{\mathfrak{q}_{\underline{k}}}(x\cdot\la)]=\sum_{y\in S_{\underline{k}}}(-q)^{\ell(w)}[\hat{M}(yx\cdot\la)]
   \end{eqnarray}
   The claim follows then from \eqref{bvs} and Proposition~\ref{lem:easy}.
  \end{proof}
  
\subsection{(Categorified) Bar involution}
Recall, from \eqref{identification}, the identification $K_0\left(\bigoplus_{\kappa \in {\rm Bl}(n,r,m)} \hat{\cO}_\kappa(n)\right)={\bigwedge}_q(n,m,r)$, then  the following generalization of Proposition~\ref{lem:bar_unique_2} holds. 
\begin{theorem}\label{thm:bar_unique}
There is a unique compatible bar-involution $w\mapsto \overline{w}$ on the $\cH^\hint$-module ${\bigwedge}_q(n,m,r)$ which fixes the standard basis vectors corresponding to classes of graded lifts of simple Verma modules concentrated in degree zero. 
\end{theorem}
\begin{proof}
The existence of an involution fixing the classes formulated in the proposition is again clear from the existence of an exact graded duality functor $\textbf{d}$, \cite[6.1.2]{Stroppel}. To see that it is a compatible involution, it is enough to show that $\textbf{d}F$ and $F\textbf{d}$ induce the same map on the Grothendieck group where $F$ is any of the functors from \eqref{gradedBs}. This is clear in the nongraded setting, since the duality commutes with the translation functors, in particular with the functors ${\Theta}_\kappa^{\kappa'}$. Since these functors are indecomposable and gradable, we know that $\hat{\Theta}_\kappa^{\kappa'}$ and then also the functors in \eqref{Fhatdef} commute (even as functors) with $\textbf{d}$ up to an overall grading shift, \cite[Lemma 1.5]{Stroppel}. 

To check the grading shift it is enough to verify it on the basis given by the classes of simple modules concentrated in degree zero. The result of applying $F$ will be a tilting module and it suffices to see that its grading is symmetric around zero. This can be verified keeping in mind the shifts in  \eqref{Fhatdef} by an explicit direct standard calculation using Definition~\ref{Soergelbimodules} and their adjunction properties,  see e.g. \cite[Lemma 7.2.1]{Geordie} which we omit. 

In contrast to  the situation of Proposition~\ref{prop:bar_exists} we could not find a manageable direct combinatorial proof of the uniqueness property. Instead we therefore use a representation theoretic argument involving the basis of the classes of graded lifts of tilting modules (which exist by \cite[Corollary 5]{MOSpair}. Assume $\tau$ is another compatible bar involution as in the proposition. Note that the graded lifts of simple Verma modules concentrated in degree zero are graded self-dual titling modules, and each block contains such a module, see \cite[Chapter 11]{Hbook}. 

Assume first that we have a regular block triple $\kappa$ such that there exists integers $a<b$ with $d_j=1$ for $a\leq j\leq b$ and $d_j=0$ otherwise, and let $T$ be such a corresponding simple tilting module with graded lift $\hat{T}$ concentrated in degree zero. Then $\pr_{{}_{+i}\kappa}\;{\cF_{i,+}}\;\pr_\kappa$ for  $a\leq i<b$ is a translation functor to one wall, and  $\pr_\kappa\;{\cF_{i,-}}\;\pr_{{}_{+i}\kappa}$ translates back out again, in particular their composition is the translation functor
through this wall and sends therefore $T$ to a tilting module in the same block $\hat{\cO}_\kappa(n)$. If $\hat{F}$ is the composition of the graded lifts \eqref{gradedBs} of such endofunctors, the compatibility implies $\tau([\hat{F}\hat{T}])=[\hat{F}](\tau([\hat{T}])=[\hat{F}\hat{T}]=[\textbf{d}\hat{F}\hat{T}]$. By the classification theorem of tilting modules we can chose such functors $\hat{F}$ such that the elements $[\hat{F}\hat{T}]$ generate the Grothendieck group of $\hat{\cO}_\kappa(n)$ and hence the two involutions agree there.

Assume now that we have an arbitrary regular block triple $\kappa'$. Then we can find a composition of functors of the form \eqref{gradedBs} to pass from some $\kappa$ as above to $\kappa'$, passing only regular block triples on the way, and hence defining an equivalence of categories from $\hat{\cO}_\kappa(n)$ to $\hat{\cO}_{\kappa'}(n)$, and hence the two involutions also agree on these Grothendieck groups. 

Finally any other block triple $\kappa'$ can be obtained from a regular one by again  functors of the form \eqref{gradedBs}. Although they are not equivalences, they can be chosen to induce surjections on the Grothendieck groups, and again the compatibility implies that the two involutions agree. All the constructions can be done inside the blocks from ${\rm Bl}(n,r,m)$ and directly also transfer to the parabolic setting as long as there exists a regular block triple $\kappa$. In case there is no such regukar block triple, we should chose instead a block with minimal possible singularity and then proceed as above.  
\end{proof}

Theorem~\ref{thm:fourth2} in the introduction is proved exactly as Theorem~\ref{thm:bar_unique}.

\subsection{Koszul duality}
For the block triple $\Gamma = (\underline{k},\underline{\mu},\epsilon)
\in {\rm Bl}(n,r,m)$ we associate the \emph{transposed block triple}
$\Gamma^\vee = (\underline{\mu},\underline{k}, \overline{n}\epsilon) \in
{\rm Bl}(n,m,r)$. In other words, the parabolic type is flipped with the
type of the singularity and the parity is kept if $n$ is even and
swapped if $n$ is odd.

\begin{theorem}\label{thm:skewKoszul}
Let $\Gamma \in {\rm Bl}(n,r,m)$ then $\hat{\cO}_\Gamma(n)$ is Koszul
dual to $\hat{\cO}_{\Gamma^\vee}(n)$.
\end{theorem}

\begin{proof}
We first recall the description of the Koszul dual block in general. Assume we have a block $\cO^I_J$  of category
$\cO$ for any complex semisimple Lie algebra, where the parabolic is
given by a subset $I$ of the simple roots and the stabilizer of the
highest weights of the parabolic Verma modules are generated by simple
reflections given by a subset $J$ of simple roots. Then the Koszul dual
is the block $\cO^J_{I'}$, where $I'=-w_0(I)$ with $w_0$ the longest
element in the Weyl group, \cite[Theorem 1.1]{Backelin}. In case of
type $\mathrm{D}_n$, we have $-w_0(I)=I$ if $n$ is even, whereas if $n$ is odd $-w_0(I)$ is the set
obtained from $I$ by applying the unique non-trivial diagram automorphism
of the Dynkin diagram. Hence the Koszul dual of
$\Gamma=(\underline{k},\underline{\mu},\overline{0})$ is the block
$(\underline{\mu},\underline{k},\overline{0})$ if $n$ is even and equal to
$(\underline{\mu},\underline{k},\overline{1})$ if $n$ is odd; and
the Koszul dual of $\Gamma=(\underline{k},\underline{\mu},\overline{1})$
is the block $(\underline{\mu},\underline{k},\overline{1})$ if $n$ is
even and equal to
$(\underline{\mu},\underline{k},\overline{0})$ if $n$ is odd.
\end{proof}

Koszul duality, \cite{BGS}, \cite{MOS} in the generalization of \cite{Backelin}, defines an equivalence of categories
\begin{eqnarray}
\label{Koszulduality}
D^b(\hat{\cO}_\Gamma(n))&\cong& D^b(\hat{\cO}_{\Gamma^\vee}(n))
\end{eqnarray}
This duality and its connection to the quantum exterior
powers can be nicely encoded in terms of box diagrams which we introduce next.

\subsection{Combinatorial Koszul duality}
Let $r$, $m$ be positive integers.
\begin{definition}
An $r \times m$ {\it box diagram} $\boxplus$ is a filling of $r$ rows and $m$ columns of boxes such that each box is either empty, filled with one of the symbols $+,-$ or filled with both, i.e. with $\pm$; for an Example see \eqref{ex:box_diagram} below. 
\end{definition}
We denote by $\boxplus_{i,j}$ the box in row $i$ and column $j$. We say $\boxplus$ is {\it of type} $\kappa=(\underline{k},\underline{d},\epsilon)$ if there are a total of $k_i$ symbols in row $i$, a total of $d_j$ symbols in column $j$, and the parity of the number of $-$ is $\epsilon$. Given $\hat{M}^{\mathfrak{q}_{\underline{k}}}(\mu)$, let $\mu+\rho=(\mu_1^\prime,\mu_2^\prime,\ldots,\mu_n^\prime)$. This is a sequence of integers, strictly increasing
in the $r$ parts given by $\underline{k}$. Assign to this weight an $r \times m$ box
diagram $\boxplus=\boxplus(\hat{M}^{\mathfrak{q}_{\underline{k}}}(\mu))$ where $\boxplus_{i,j}$ is filled with:
\begin{itemize}
\item the symbol $+$ if $j-\half$, but not $-(j-\half)$, occurs in the $i$-th part of $\underline{k}$,
\item the symbol $-$ if $-(j-\half)$, but not $j-\half$, occurs,
\item both symbols, i.e. $\pm$, if both occur.
\end{itemize}

\begin{lemma}
\label{bijboxdiagram}
Fix a block $\cO_\kappa(n)$ with block triple $\kappa$. Then $M\mapsto \boxplus(M)$ defines a bijection between the set of parabolic Verma modules in $\cO_\kappa(n)$ and the set of box diagrams of type $\kappa$. 
\end{lemma}

\begin{proof} This follows directly from the definitions, the inverse map is given by reading the rows in $\boxplus$ from top to bottom and ordering the entries for each row in increasing order. 
\end{proof}
We denote by $\boxplus^\vee$ the transposed box diagram of $\boxplus$ with additionally all symbols ($+$ and $-$) swapped and by $\lambda^\mathrm{T}$ the weight giving rise to the box-diagram $\boxplus^\vee$.
corresponding weight.

\begin{ex} \label{ex:box_diagram}
Let $n=9$, $r=5$, $m=4$ and ${\underline{k}}=(2,2,2,2,1)$. The box diagram $\boxplus$ corresponding to 
$w\cdot\mu+\rho=(-\nicefrac{5}{2},\nicefrac{3}{2},-\nicefrac{1}{2},\nicefrac{1}{2},-\nicefrac{7}{2},\nicefrac{3}{2},-\nicefrac{1}{2},\nicefrac{5}{2},\nicefrac{3}{2})$ and its transposed are (with $\circ$ indicating the empty boxes)
\begin{eqnarray}
\boxplus = \young(\circ
+-\circ,\mp\circ\circ\circ,\circ+\circ-,-\circ+\circ,\circ+\circ\circ) \qquad
&& \qquad \boxplus^\vee =
\young(\circ \mp\circ +\circ,-\circ -\circ -,+\circ \circ -\circ ,\circ
\circ +\circ \circ )
\end{eqnarray}
and thus
$w^{-1}w_0\cdot\la+\rho =(-\nicefrac{3}{2},\nicefrac{3}{2},\nicefrac{7}{2},-\nicefrac{9}{2},-\nicefrac{5}{2},-\nicefrac{1}{2},-\nicefrac{7}{2},\nicefrac{1}{2},\nicefrac{5}{2}).$ More generally we have
\end{ex}

\begin{prop}[Combinatorial Koszul duality]
\label{transpose}
Let $\boxplus$ be the box diagram corresponding to ${M}^{\mathfrak{q}_{\underline{k}}}(w\cdot\mu)$ in the block ${\cO}_\Gamma(n)$. Then ${M}^{\mathfrak{q}_{\underline{k}}}(w^{-1}w_0 \cdot\lambda)\in\hat{\cO}_{\Gamma^\vee}(n)$ corresponds to the box diagram $\boxplus^\vee$. Here $\mu$ and $\lambda$ are maximal elements in their respective orbits for the dot action of the Weyl group.
\end{prop}
\begin{proof}
Transposing the box diagram corresponds to passing from $x$ to $x^{-1}$, and swapping the signs corresponds to right multiplication with $w_0$. If $n$ is even, then this procedure keeps the parity of $-$ symbols, whereas if $n$ is odd then the parity gets swapped. The claim follows.
\end{proof}

Encoding basis vectors by box diagrams using \eqref{identification} allows to give explicit formulae for the actions of $\cH^\hint_m$ on $\bigwedge_q(n,m,r)$ using the following notation.

\begin{definition} \label{def:boxdiagram_and_statistics}
We define for any box $\boxplus_{i,j}$ as above the statistics
$$c_{i,j}^+(\boxplus,\rightarrow) = \#\{ l \mid l > j: \boxplus_{i,l}=+ \text{ or } \boxplus_{i,l}=\pm\},$$
i.e., the number of symbols $+$ to the right of the box $\boxplus_{i,j}$; analogously for the arrows $\leftarrow$, $\uparrow$, and $\downarrow$ and for the symbol $-$. Similarly we define
$$c_i^+(\boxplus,\leftrightarrow) = \#\{ l \mid \boxplus_{i,l}=+ \text{ or } \boxplus_{i,l}=\pm\},$$
i.e., the number of symbols $+$ in row $i$; analogously for $\updownarrow$ and the symbol $-$.
\end{definition}

For  an $r \times m$ box diagram $\boxplus$ and  $1 \leq j < m$ set 
\begin{equation}
 \label{boxaction1}
 B_{j}.\boxplus = \sum_{\boxplus'} q^{m(\boxplus')} \boxplus', \quad B_{-j}.\boxplus = \sum_{\boxplus''} q^{m(\boxplus'')} \boxplus'', \quad B_0.\boxplus = \sum_{\boxplus'''} q^{m(\boxplus''')} \boxplus''',
\end{equation}
where the first sum runs over all box diagrams $\boxplus'$ obtained from $\boxplus$ by moving a symbol from column $j+1$ to column $j$. Let $\boxplus_{i,j+1}$ be the box with the symbol that is moved to the left, then the power of $q$ is given by 
\abovedisplayskip0.3em
\belowdisplayskip0.3em
 \begin{eqnarray*}
m(\boxplus')&=&
\begin{cases}
c_{i,j}^+(\boxplus,\uparrow) - c_{i,j+1}^+(\boxplus,\uparrow) + c_j^-(\boxplus,\updownarrow) - c_{j+1}^-(\boxplus,\updownarrow) & \text{for } + \\
c_{i,j}^-(\boxplus,\downarrow) - c_{i,j+1}^-(\boxplus,\downarrow) & \text{for } -.
\end{cases}
\end{eqnarray*}
The second sum runs over all $\boxplus''$ obtained from $\boxplus$ by moving a symbol from column $j$ to column $j+1$ and 
\abovedisplayskip0.3em
\belowdisplayskip0.3em
\begin{eqnarray*}
m(\boxplus'')&=&\begin{cases}
c_{i,j+1}^+(\boxplus,\downarrow) - c_{i,j}^+(\boxplus,\downarrow) & \text{for } + \\
c_{i,j+1}^-(\boxplus,\uparrow) - c_{i,j}^-(\boxplus,\uparrow) + c_{j+1}^+(\boxplus,\updownarrow) - c_{j}^+(\boxplus,\updownarrow)& \text{for } -.
\end{cases}
\end{eqnarray*}
The third sum runs over all $\boxplus'''$ obtained from $\boxplus$ by swapping a symbol in the first column from $-$ to $+$ or vice versa. In this case the power of $q$ is independent of the symbol and only depends on the position
$m(\boxplus''') = c_{i,1}^+(\boxplus,\downarrow) - c_{i,1}^-(\boxplus,\downarrow).$

\subsection{Quantum skew Howe duality}
For the commuting action of $\cH_r^\hint$ on $\bigwedge_q(n,m,r)$ we need special morphisms between the summands of the form ${\bigwedge^{\underline{k}}}_q \mV_m^\hint$. Let $\underline{k} \in C(n,r)$ with $k_{i+1} > 0$ and consider  $\underline{k}^{+i}, \underline{k}^{-i} \in C(n,r+1)$ and ${}_{+i}\underline{k}, {}_{-i}\underline{k}\in C(n,r)$ as in Definition ~\ref{defsd}.

%via
%$$ k_j'=\left\lbrace \begin{array}{ll}
%k_j & \text{if } j \leq i \\
%1 & \text{if } j = i+1 \\
%k_{i+1}-1 & \text{if } j = i+2 \\
%k_{j-1} & \text{if } j > i+2
%\end{array} \right. \qquad
%k_j''=\left\lbrace \begin{array}{ll}
%k_j & \text{if } j < i \\
%k_i+1 & \text{if } j = i \\
%k_{i+1}-1 & \text{if } j = i+1 \\
%k_j & \text{if } j > i+1.
%\end{array} \right.
%$$

Then let $\check{B}_{i,+,\underline{k}}$ and $\check{B}_{i,-,\underline{k}}$ be the following composition of the standard $\cU_q(\mathfrak{gl}_{2m})$-morphisms, see \cite[Section 2.1]{MS} for explicit formulas, multiplied with the indicated power of $-q$,
\begin{eqnarray}
\label{wedgemap1}
\check{B}_{i,\underline{k}}&=&(-q)^{1-k_{i+1}} \left({\bigwedge^{\underline{k}}}_q\mV_m^\hint \longrightarrow {\bigwedge^{\underline{k}^{+i}}}_q\mV_m^\hint \longrightarrow {\bigwedge^{{}_{+i}\underline{k}}}_q\mV_m^\hint \right).\\
\label{wedgemap2}
\check{B}_{-i,\underline{k}}&=& (-q)^{1-k_i} \left({\bigwedge^{\underline{k}}}_q\mV_m^\hint \longrightarrow {\bigwedge^{\underline{k}^{-i}}}_q\mV_m^\hint \longrightarrow {\bigwedge^{{}_{+i}\underline{k}}}_q\mV_m^\hint \right)
\end{eqnarray}
For $\underline{k}$ with $k_{i+1}=0$ we put $\check{B}_{i,\underline{k}}=0$ and finally $\check{B}_i = \bigoplus_{\underline{k} \in C(n,r)} \check{B}_{i,\underline{k}}$.

\begin{remark}
The maps in \eqref{wedgemap1}, \eqref{wedgemap2} are quantized versions of  the canonical inclusion maps $\bigwedge^{l+1} \mV_m^\hint \rightarrow \bigwedge^{l} \mV_m^\hint
\otimes \mV_m^\hint$ and $\bigwedge^{l+1} \mV_m^\hint \rightarrow \mV_m^\hint
\otimes \bigwedge^{l} \mV_m^\hint$, and of the projection maps $\bigwedge^{l} \mV_m^\hint \otimes \mV_m^\hint \longrightarrow
\bigwedge^{l+1} \mV_m^\hint$ and $\mV_m^\hint \otimes \bigwedge^{l}\mV_m^\hint \longrightarrow
\bigwedge^{l+1} \mV_m^\hint$, see e.g. \cite{MS} for more details.  
\end{remark}

By construction the following is automatic.

\begin{lemma}
\label{lem:1}
Let $1 \leq i < r$. The maps $\check{B}_i$ and $\check{B}_{-i}$ are $\cU_q(\mathfrak{gl}_{2m})$-equivariant, hence also $\cH_m^\hint$-equivariant.
\end{lemma}

For $\underline{k} \in C(n,r)$ with $k_1 > 0$ define a composition $\underline{k}' \in C(n,r+1)$ by setting
$ k_1'= 1, \, k_2' = k_1-1 \text{ and } k_i' = k_{i-1} \text{ for } i > 2.$ 

\begin{eqnarray}
\label{wedgemap3}
\check{B}_{0,\underline{k}} &=&(-q)^{1-k_1} \left({\bigwedge^{\underline{k}}}_q\mV_m^\hint \longrightarrow {\bigwedge^{\underline{k}'}}_q\mV_m^\hint \longrightarrow {\bigwedge^{\underline{k}'}}_q\mV_m^\hint \longrightarrow {\bigwedge^{\underline{k}}}_q\mV_m^\hint \right),
\end{eqnarray}
where the first and third map are as before the standard maps, while the one in the middle changes the sign of the index in the very first tensor factor of $\bigwedge^{\underline{k}'}\mV_m^\hint$, i.e. it send a vector $v_{i_1} \otimes \text{(Rest)}$ to $v_{-i_1} \otimes \text{(Rest)}$. As before, for $\underline{k}$ with $k_{1}=0$ we put $\check{B}_{0,\underline{k}}=0$ and finally $\check{B}_{0} = \bigoplus_{\underline{k} \in C(n,r)} \check{B}_{0,\underline{k}}$.

\begin{lemma}
\label{lem:2}
The map $\check{B}_0$ is $\cH_m^\hint$-equivariant.
\end{lemma}
\begin{proof}
We have to show that the action commutes with the action of all  $B_i$. We give the arguments for $i\geq 0$ only, since the others are similar. Fix a box diagram $\boxplus$. Since it is obvious that $\check{B}_0(B_i.\boxplus)$ and  $B_i.\check{B}_0(\boxplus)$ contain the same summands  it remains to compare the powers of $q$ involved. Choose boxes $\boxplus_{j,i+1}$ and $\boxplus_{1,l}$ containing symbols that can be moved resp. flipped by the action of $B_i$ resp. $\check{B}_0$.

{If  $l \notin \{i,i+1\}$ and $j \neq 1$:}
The powers of $q$ are independent of the action of the other operator and hence the results will be the same.

{If $l \in \{i,i+1\}$ and $j \neq 1$:}
The box modified by $\check{B}_0$ is in the same column as the one modified by $B_i$. The formulas for $B_i$ imply that a symbol $+$ means we count the symbols above in column $i$ and $i+1$ either all as $q$ or $q^{-1}$. So it does not matter if we apply $\check{B}_0$ before or after, the power remains the same. In case of a $-$, the symbols above are not counted at all.

{If $l \notin \{i,i+1\}$ and $j = 1$:}
Then the box modified by $\check{B}_0$ is in the first row but far enough away from the one for $\check{B}_0$. In this case the powers of $q$ given by $B_i$ do not involve the box $\boxplus_{1,l}$ at all, and the powers given by $\check{B}_0$ are obviously the same, no matter if we apply $B_i$ before or after.

{If  $l \in \{i,i+1\}$ and $j = 1$:}
The box $\boxplus_{1,l}$ is either directly to the left of $\boxplus_{j,i+1}$ or they coincide. It is easy to see that for $\check{B}_0$ the $q$-powers are the same or the two compositions are zero anyway. Similarly for $B_i$, because the $q$-powers do not depend on the sign if our box is in the first row.
\end{proof}

Using the definition of the box diagrams and of the intertwiner maps \eqref{wedgemap1}- \eqref{wedgemap3} we obtain the following explicit formulas ''dual'' to  \eqref{boxaction1}.

\begin{prop}\label{lem:combinatorial_zuckerman} Let $\boxplus$ be a box diagram of type $(\underline{k},\underline{\mu},\epsilon)$. Then 
\begin{gather}
\label{boxaction2}
\begin{gathered}
\check{B}_i(\boxplus) = \sum_{\boxplus'} (-q)^{m(\boxplus')}\boxplus', \quad \check{B}_{-i}(\boxplus) = \sum_{\boxplus''} (-q)^{m(\boxplus'')}\boxplus'',\\
\check{B}_0(\boxplus) = \sum_{\boxplus'''} (-q)^{m(\boxplus''')}\boxplus''',
\end{gathered}
\end{gather}

for $1-r\leq i\leq r-1$, where the first sum runs over all box diagrams obtained from $\boxplus$ by moving a symbol from row $i+1$ to the one directly above, the second sum over those where a symbol in row $i$ is moved to the box directly below and the third sum over those  where a symbol in the first row is swapped from $-$ to $=$ or vice versa, and the exponents are given below.
\end{prop}

Assume that for $\boxplus'$ (respectively $\boxplus''$) the symbol in $\boxplus_{i+1,j}$ (respectively in $\boxplus_{i,j}$) gets moved, and for $\boxplus'''$ the swapped symbol is in $\boxplus_{1,j}$, then
\begin{eqnarray*}
m(\boxplus')&=& \begin{cases}
c_{i,j}^+(\boxplus,\rightarrow) - c_{i+1,j}^+(\boxplus,\rightarrow) & \text{for } + \\
c_{i+1,j}^-(\boxplus,\rightarrow) - c_{i,j}^-(\boxplus,\rightarrow) + k_i - k_{i+1} + 1& \text{for } -.
\end{cases}\\
m(\boxplus'')&=&\begin{cases}
c_{i,j}^+(\boxplus,\rightarrow) - c_{i+1,j}^+(\boxplus,\rightarrow) + k_{i+1} - k_i + 1& \text{for } + \\
c_{i+1,j}^-(\boxplus,\rightarrow) - c_{i,j}^-(\boxplus,\rightarrow) & \text{for } -.
\end{cases}\\
m(\boxplus''')&=& c_{1,j}^-(\boxplus,\rightarrow) - c_{1,j}^+(\boxplus,\rightarrow).
\end{eqnarray*}

Denote by $\cH^\hint_m$ the coideal subalgebra corresponding to the quantum group $\cU_q(\mathfrak{gl}_{2m})$, i.e. the truncation of $\cH^\hint$, and consider the two actions on $\bigwedge_q(n,m,r)$:

\begin{theorem} [Quantum skew Howe duality]\label{thm:skewhowequantum}
Lemmas \ref{lem:1} and \ref{lem:2} define commuting actions of $\cH^\hint_m$ and $\cH^\hint_r$ on $\bigwedge_q(n,m,r)$. They are in fact each others centralizers.
\end{theorem}
\begin{proof}
We have seen that the actions commute. Then we argue as in the proof of Theorem~\ref{prop:Schur_Weyl} using Theorem~\ref{thm:skewhowe} for the specialization $q=1$.
\end{proof}

\subsection{Categorified quantum skew Howe duality and Koszul duality}
Theorem~\ref{thm:skewKoszul} and then \eqref{Koszulduality} implies that the Koszul duality functor  from \cite{BGS}, \cite{MOS} defines a contravariant equivalence of triangulated categories
\begin{eqnarray}
\label{eq:mapKoszulduality}
\mathcal{K}:\bigoplus_{\underline{k} \in C(n,r)}
D^b\left(\hat{\cO}^{\mathfrak{q}_{\underline{k}}}_{\leq m}(n)\right)
&\longrightarrow&\bigoplus_{\underline{k} \in C(n,m)}
D^b\left(\hat{\cO}^{\mathfrak{q}_{\underline{k}}}_{\leq r}(n)\right).
\end{eqnarray}
sending the standard lift of the parabolic Verma module of highest weight $\nu-\rho$ to the standard lift of the parabolic Verma module of highest weight $\nu^\vee-\rho$, see \cite{Backelin} with Proposition~\ref{transpose}. The grading shift on the left corresponds to a homological and grading shift on the right, $\mathcal{K}\langle j\rangle=\langle j\rangle[j]\mathcal{K}$. Since the functors are triangulated they induce homomorphisms on the Grothendieck group of the categories from \eqref{eq:mapKoszulduality} which we identify with the Grothendieck group of the underlying abelian category.  Note that $\left[{\mB}_{i} \right]$ is $q$-linear, whereas $[\cK]$, $[\cK^{-1}]$  are $q$-antilinear, that means $q$ is sent to $-q$. 

\begin{prop}
\label{prop:comb} On the Grothendieck group we have for $1-r\leq i\leq r-1$
\begin{equation*}
[\cK^{-1}] \circ \left[{\mB}_{i} \right] \circ [\cK] =\check{B}_i.
\end{equation*}
\end{prop}
\begin{proof}
It is enough to compare their values on standard basis vectors, hence on the classes of the standard graded lifts of parabolic Verma modules (concentrated in homological degree zero). For the left hand side of the first equation we take therefore the class of a parabolic Verma module and identify it with its box diagram $\boxplus$. Applying $[\cK^{-1}] \circ \left[{\mB}_{i} \right] \circ [\cK] $ means that we first take the transpose $\boxplus^\vee$, secondly apply $B_i$  thanks to Theorem~\ref{thm:main} and finally substitute in the resulting   $\mZ[q,q^{-1}]$-linear combination of box diagrams $q\rightsquigarrow -q$ in all the coefficients and transpose all the occurring box diagrams.  Comparing the final result with the formulas from Proposition~\ref{lem:combinatorial_zuckerman} gives the claim. The other two equalities are checked analogously.
\end{proof}

Since the functors ${\mB}_{i}$ are exact and $\cK$ is a graded triangulated equivalence with inverse $\cK^{-1}$ we directly obtain:
\begin{corollary}
\label{prop:ZuckeronK}
The endofunctors $\hat{\mathbb{B}}_{i}^\vee=\cK^{-1}\; {\mB}_{i}\; \cK$ with $1-r\leq i\leq r-1$ induce the action of $\check{B}_i$ on the Grothendieck group, that is $\check{B}_i=[\hat{\mathbb{B}}_{i}^\vee]$.
\end{corollary}

\begin{remark} 
The functors $\hat{\mathbb{B}}_{i}^\vee$ have a more direct description as graded lifts of so-called  Zuckermann functors precomposed with inclusion functors between graded version of different  parabolic category $\mathcal{O}$'s). This follows directly from the fact,  \cite[Theorem 35]{MOS}  and its straight- forward generalization to more general parabolics,  that graded lifts of Zuckerman respectively inclusion functors are Koszul dual to the graded graded lifts of translation to respectively translation functors out off the walls.
\end{remark}

From the definitions, Proposition~\ref{prop:comb} and Theorem~\ref{thm:skewhowequantum} we obtain 
\begin{theorem}[Categorified quantum skew Howe duality]
\label{Koszul}
 Fix the equivalence given by Koszul duality,
\begin{eqnarray}
\mathcal{K}:\bigoplus_{\underline{k} \in C(n,r)}
D^b\left(\hat{\cO}^{\mathfrak{q}_{\underline{k}}}_{\leq m}(n)\right)
&\cong&\bigoplus_{\underline{k} \in C(n,m)}
D^b\left(\hat{\cO}^{\mathfrak{q}_{\underline{k}}}_{\leq r}(n)\right).
\end{eqnarray}
Then the coideal algebra actions from Theorem~\ref{thm:main} on the two sides turn into commuting actions on the same space, one given by Theorem~\ref{thm:main}, the other by Proposition~\ref{prop:ZuckeronK}.
\end{theorem}
More concretely this means that on $$K_0\left(\bigoplus_{\underline{k} \in C(n,r)} D^b\left(\hat{\cO}^{\mathfrak{q}_{\underline{k}}}_{\leq m}(n)\right)\right)={\bigwedge}_q(N,m,r),$$ 
the action of $\left[\check{\mB}_{i} \right]$ commutes with the one of $[\mB_{j}]$ for $-r< i < r$ and $-m< j < m$.
\begin{remark}
With some technical arguments generalizing \cite{Khomenko} one can actually show that the family of functors $\check{\mB}_{i}$ commutes naturally (in the sense of \cite[Definition 2]{Khomenko} or \cite[Lemma 2.1, Theorem 3.2]{AndersenStroppel} with the action of the family $\mB_{j}$ for $-r< i < r$ and $-m< j < m$.

Totally analogous (but in fact combinatorially much easier) arguments applied to the Lie algebra $\mathfrak{gl}_n$ instead of $\mathfrak{so}_n$ give rise to the quantum and categorified quantum skew Howe duality of type $ \mathrm{A}$ mentioned in the introduction. The quantum (but non-categorified) version can be found  in  \cite{LZZ}, see also \cite{CKM}, and a partial categorification (with only one of the two commuting actions) can be found in \cite{QR}, \cite{LQR}.

Passing to coideals of other types gives a different phenomenon, since skew Howe dualities could then pair quantum groups with coideals and not as in our case coideals with the same type of coideals. For first results see \cite{SarTu}. 
 \end{remark}
\section{Appendix} \label{section:appendix}

\subsection{Proof of Theorem \ref{thm:actionofvw}} This subsection is devoted to the proof of the auxiliary lemmas for Theorem \ref{thm:actionofvw}. 
\begin{lemma} \label{relation2c}
For $1 \leq i \leq d-1$ and $j \not\in \{i,i+1 \}$ we have $s_iy_j = y_js_i$.
\end{lemma}
\begin{proof}
The case $j < i$ is obvious by definition. For the case $j > i+1$, the claim follow from ($v \in \Md$):
\abovedisplayskip0.2em
\belowdisplayskip0.3em
\begin{gather*}
s_iy_j(v) =\, s_i \left( \left( \sum_{0 \leq k < j} \hspace{-.3em} \Omega_{kj} + {\scriptstyle \frac{2n-1}{2}}\right)v \right) = \left( \sum_{0 \leq k < j} \hspace{-.3em} \Omega_{kj}\right) s_i(v) + {\scriptstyle \frac{2n-1}{2}} s_i(v) = y_js_i(v).
\end{gather*}
\vskip-.2cm
\end{proof}

The following are the Brauer algebra relations, see \cite[Lemma 10.1.5]{GW}:
\begin{lemma} \label{relation6a6b6c6d}
For $1 \leq i \leq d-2$, the relations $e_is_i = e_i = s_ie_i$, $s_ie_{i+1}e_i = s_{i+1}e_i$, $e_{i+1}e_is_{i+1} = e_{i+1}s_i$, and $e_{i+1}e_ie_{i+1} = e_{i+1}$ hold.
\end{lemma}
\begin{proof}
The first equality is obvious by definition of $\tau$. We will prove the second one, the final two are shown analogously. To simplify notations we assume that $s_i$, $e_i$, and $e_{i+1}$ only act on a threefold tensor product, with factors being the
position $i$, $i+1$, and $i+2$ in order. Let $a,b,c \in \mathrm{I}$ and $v=v_a \otimes v_b \otimes v_c$, then the claim follows by comparing
$(s_{i+1}e_i).v = \left\langle v_a,v_b \right\rangle {\sum}_{k \in \mathrm{I}} v_k \otimes v_c \otimes v_{k}$ with $(s_ie_{i+1}e_i).v =(s_ie_{i+1})\left( \left\langle v_a,v_b \right\rangle {\sum}_{k \in \mathrm{I}} v_k \otimes v_{k}^* \otimes v_c \right) =
\left\langle v_a,v_b \right\rangle {\sum}_{k \in \mathrm{I}} v_k \otimes v_{c} \otimes v_{k}^*.
$
\end{proof}

\begin{lemma} \label{relation7}
For $1 \leq i \leq d-1$ we have $s_iy_i - y_{i+1}s_i = e_i - 1$.
\end{lemma}
\begin{proof}
The claim is equivalent to showing $y_i - s_iy_{i+1}s_i = e_i - s_i$ by Lemma\ref{relation6a6b6c6d}. Note that $s_iy_{i+1}s_i$ equals
\abovedisplayskip0.2em
\belowdisplayskip0.3em
\begin{gather*}
s_i \left( \sum_{0 \leq k < i+1} \Omega_{ki} + \left( {\scriptstyle \frac{2n-1}{2}} \Id \right) \right) s_i = \left( \sum_{0 \leq k < i}
\Omega_{ki} + s_i\Omega_{i(i+1)}s_i  + \left( {\scriptstyle \frac{2n-1}{2}} \Id \right) \right)
\end{gather*}
Therefore, $y_i - s_iy_{i+1}s_i=- s_i\Omega_{i(i+1)}s_i=  -\Omega_{i(i+1)} =e_i-s_i$, where the last equality follows from Remark \ref{omega_sigma_tau}.
\end{proof}

The proofs of the following identities rely on the explicit form of $\Omega$ and its action on $V \otimes V$.

\begin{lemma} \label{relation5b}
For $1 \leq i,j\leq d-1$ and $j \not\in \{i,i+1 \}$ we have $e_iy_j = y_je_i$.
\end{lemma}
\begin{proof}
As in Lemma \ref{relation2c} the case $j < i$ is obvious by definition and we are left with the case $j > i+1$. The essential part that we need to
prove is
\abovedisplayskip0.2em
\belowdisplayskip0.3em
\begin{gather} \label{eomega}
\left(\Omega_{ij} + \Omega_{(i+1)j}\right)e_i = e_i \left(\Omega_{ij} + \Omega_{(i+1)j}\right),
\end{gather}
since $e_i$ naturally commutes with all other summands of $y_j$. Again, we pretend that $\Omega_{ij}$, $\Omega_{(i+1)j}$, and $e_i$ act on a threefold tensor product, corresponding to the factors at position $i$, $i+1$ and $j$, in order.

Consider first the left hand side of \eqref{eomega}. Since we first apply $e_i$, the element $e_iy_j$ kills all basis vectors except the ones of the form $x=v_a \otimes v_{a}^* \otimes v_c$ for some $a,c \in \mathrm{I}$, in which case $x.e_i=\sum_{k \in \mathrm{I}} v_k \otimes v_k^* \otimes v_c$. and applying $\Omega_{ij} + \Omega_{(i+1)j}$ gives
\abovedisplayskip0.2em
\belowdisplayskip0.3em
\begin{gather*}
\begin{aligned}
& \sum_{\alpha \in \mathbb{B}(\mathfrak{so}_{2n})} \sum_{k \in \mathrm{I}} \left(X_\alpha v_k \otimes v_k^* + v_k^* \otimes X_\alpha v_k \right) \otimes X_\alpha^* v_c \\
=&\, \sum_{\alpha \in \mathbb{B}(\mathfrak{so}_{2n})} \sum_{k \in \mathrm{I}} \sum_{l \in \mathrm{I}} \left( \left\langle X_\alpha v_k,v_l \right\rangle v_l^* \otimes v_k^* +  v_k^* \otimes \left\langle X_\alpha v_k,v_l \right\rangle v_l^* \right) \otimes X_\alpha^* v_c \\
=&\, \sum_{\alpha \in \mathbb{B}(\mathfrak{so}_{2n})} \sum_{k,l \in \mathrm{I}} \left( \left\langle X_\alpha v_k,v_l \right\rangle v_l^* \otimes v_k^* - \left\langle X_\alpha v_l,v_k \right\rangle v_k^* \otimes v_l^* \right) \otimes X_\alpha^* v_c \,=\, 0
\end{aligned}
\end{gather*}

Thus $\left(\Omega_{ij} + \Omega_{(i+1)j}\right)e_i = 0$ on $M \otimes V^{\otimes d}$.

Consider the right hand side of \eqref{eomega}, let $x=v_a \otimes v_b \otimes v_c$ for $a,b,c \in \mathrm{I}$, then
\abovedisplayskip0.2em
\belowdisplayskip0.3em
\begin{gather*}
\begin{aligned}
x.(\Omega_{ij}+\Omega_{(i+1)j}) =&\, \sum_{\alpha \in \mathbb{B}(\mathfrak{so}_{2n})} \hspace{-.3em}\left(X_\alpha v_a \otimes v_b + v_a \otimes X_\alpha v_b \right) \otimes X_\alpha^* v_c \\
=& \hspace{-.5em}\sum_{\alpha \in \mathbb{B}(\mathfrak{so}_{2n}), k \in \mathrm{I}} \hspace{-.7em} \left( \left\langle X_\alpha v_a,v_k \right\rangle v_k^* \otimes v_b + \left\langle X_\alpha v_b,v_k \right\rangle v_a \otimes v_k^* \right) \otimes X_\alpha^* v_c
\end{aligned}
\end{gather*}

Applying $e_i$ kills, for fixed $\alpha$, all except one summand and we obtain
\abovedisplayskip0.2em
\belowdisplayskip0.3em
\begin{gather*}
\begin{aligned}
& \,{\sum}_{l \in \mathrm{I}} \left(\left\langle X_\alpha v_a,v_b \right\rangle v_l \otimes v_l^* + \left\langle X_\alpha v_b,v_a \right\rangle v_l \otimes v_l^* \right) \otimes X_\alpha^* v_c \\
=& \,{\sum}_{l \in \mathrm{I}} \left(\left\langle X_\alpha v_a,v_b \right\rangle v_l \otimes v_l^* - \left\langle X_\alpha v_a,v_b \right\rangle v_l \otimes v_l^* \right) \otimes X_\alpha^* v_c = 0.
\end{aligned}
\end{gather*}
Thus, $e_i\left(\Omega_{ij} + \Omega_{(i+1)j}\right) = 0$ on $M \otimes V^{\otimes d}$ and the lemma follows.
\end{proof}

\begin{lemma} \label{relation5c}
For $1 \leq i,j \leq d$ we have $y_iy_j = y_jy_i$.
\end{lemma}
\begin{proof}
We show this by using the following statements:
\begin{enumerate}[(i)]
\item $[\Omega_{ij},\Omega_{kl}] = 0$ for pairwise different $i,j,k,l$, \label{yy3}
\item $[\Omega_{ij} + \Omega_{jk},\Omega_{ik}] = 0$ for pairwise different $i,j,k$.\label{yy4}
\end{enumerate}

Equality \eqref{yy3} is obvious. For \eqref{yy4} we consider factors $i$, $j$, and $k$ to obtain
\abovedisplayskip0.2em
\belowdisplayskip0.3em
\begin{gather*}
\begin{aligned}
& [\Omega_{ij}+\Omega_{jk},\Omega_{ik}] = \sum_{\alpha,\beta,\gamma} \hspace{-.3em} [X_\alpha \otimes X_\alpha^* \otimes 1 + 1 \otimes X_\beta \otimes X_\beta^*,X_\gamma \otimes 1 \otimes X_\gamma^*] \\
=& \sum_{\alpha,\beta,\gamma} [X_\alpha,X_\gamma] \otimes X_\alpha^* \otimes X_\gamma^* + X_\gamma \otimes X_\beta \otimes [X_\beta^*,X_\gamma^*] \\
=& \hspace{-.2em}\sum_{\alpha,\beta,\gamma,\mu} \hspace{-.2em} \left(([X_\alpha,X_\gamma],X_\mu^*)X_\mu \otimes X_\alpha^* \otimes X_\gamma^* + X_\gamma \otimes X_\beta \otimes ([X_\beta^*,X_\gamma^*],X_\mu)X_\mu^* \right)\\
\overset{(1)}{=}& \hspace{-.2em}\sum_{\alpha,\beta,\gamma,\mu}\hspace{-.2em}\left( ([X_\alpha,X_\gamma],X_\mu^*)X_\mu \otimes X_\alpha^* \otimes X_\gamma^*+ ([X_\alpha^*,X_\mu^*],X_\gamma) X_\mu \otimes X_\alpha \otimes X_\gamma^* \right)\\
\overset{(2)}{=}& \hspace{-.2em}\sum_{\alpha,\beta,\gamma,\mu}\hspace{-.2em}\left( ([X_\alpha,X_\gamma],X_\mu^*)X_\mu \otimes X_\alpha^* \otimes X_\gamma^*- ([X_\alpha,X_\gamma],X_\mu^*)X_\mu \otimes X_\alpha^* \otimes X_\gamma^* \right)= 0.
\end{aligned}
\end{gather*}
Here equality (1) is relabelling the second summand and extracting a scalar factor, while equality (2) is due to the invariance of the Killing form, i.e.
\abovedisplayskip0.2em
\belowdisplayskip0.3em
\begin{align*}
 \sum_{\alpha \in \mathbb{B}(\mathfrak{so}_{2n})} \hspace{-.8em}([X_\alpha^*,X_\mu^*],X_\gamma)X_\alpha = \hspace{-.4em}\sum_{\alpha \in \mathbb{B}(\mathfrak{so}_{2n})}\hspace{-.8em} (X_\alpha^*,[X_\mu^*,X_\gamma])X_\alpha &= \hspace{-.4em}\sum_{\alpha \in \mathbb{B}(\mathfrak{so}_{2n})}\hspace{-.8em} (X_\alpha,[X_\mu^*,X_\gamma])X_\alpha^*
 \end{align*}
 which however equals $-\sum_{\alpha \in \mathbb{B}(\mathfrak{so}_{2n})} ([X_\alpha,X_\gamma],X_\mu^*)X_\alpha^*$ and so we are done.
\end{proof}

\begin{lemma} \label{relation8a8b}
For $1 \leq 1 < d$ we have $e_i(y_i+y_{i+1}) \, =\, 0\, =\, (y_i+y_{i+1})e_i$.
\end{lemma}
\begin{proof}
We start with the first relation and expand the left hand side.
\abovedisplayskip0.2em
\belowdisplayskip0.3em
\begin{align*}
e_i(y_i + y_{i+1}) &= e_i \left( \sum_{0 \leq k < i} \hspace{-.3em}\left( \Omega_{ki} + \Omega_{k(i+1)}\right) + \Omega_{i(i+1)} + N - 1 \right) \\ 
&= e_i \left(
\sum_{0 \leq k < i} \hspace{-.3em}\left( \Omega_{ki} + \Omega_{k(i+1)}\right) + s_i-e_i + N - 1 \right) = e_i \hspace{-.3em}\sum_{0 \leq k < i} \hspace{-.3em}\left( \Omega_{ki} + \Omega_{k(i+1)}\right)
\end{align*}
We have seen in the proof of Lemma \ref{relation5b} that $ e_i \left( \Omega_{ki} + \Omega_{k(i+1)}\right)$ acts as zero if $k > 0$. Thus we are left
to show that $ e_i \left( \Omega_{0i} + \Omega_{0(i+1)}\right)$ acts as zero. Let $x = m \otimes v_a \otimes v_b$ for $m \in M$ and $a,b \in \mathrm{I}$, again these should be thought of positions $0$, $i$, and $i+1$. Applying $\left( \Omega_{0i} + \Omega_{0(i+1)}\right)$ gives us
\abovedisplayskip0.2em
\belowdisplayskip0.3em
\begin{gather*}
\begin{aligned}
&x.\left( \Omega_{0i} + \Omega_{0(i+1)}\right) \,= \sum_{\alpha \in \mathbb{B}(\mathfrak{so}_{2n})}\hspace{-.8em} X_\alpha m \otimes X_\alpha^*v_a \otimes v_b + X_\alpha m \otimes v_a \otimes X_\alpha^*v_b \\
=&\, \sum_{\alpha \in \mathbb{B}(\mathfrak{so}_{2n})} \sum_{k \in \mathrm{I}} \left\langle X_\alpha^*v_a,v_k\right\rangle X_\alpha m \otimes v_k^* \otimes v_b + \left\langle X_\alpha^*v_b,v_k\right\rangle X_\alpha m \otimes v_a \otimes v_k^*.
\end{aligned}
\end{gather*}
Now applying $e_i$ to each summand for $\alpha \in \mathbb{B}(\mathfrak{so}_{2n})$ we obtain
\abovedisplayskip0.2em
\belowdisplayskip0.3em
\begin{gather*}
\begin{aligned}
& \sum_{l \in \mathrm{I}} \left\langle X_\alpha^*v_a,v_b\right\rangle X_\alpha m \otimes v_l \otimes v_l^* + \left\langle X_\alpha^*v_b,v_a\right\rangle X_\alpha m \otimes v_l \otimes v_l^* \\
=& \sum_{l \in \mathrm{I}} \left\langle X_\alpha^*v_a,v_b\right\rangle X_\alpha m \otimes v_l \otimes v_l^* - \left\langle X_\alpha^*v_a,v_b\right\rangle X_\alpha m \otimes v_l \otimes v_l^* = 0.
\end{aligned}
\end{gather*}
Thus we have $e_i(y_i+y_{i+1}) = 0$ on $M \otimes V^{\otimes d}$.

For the second equality with start with the same $x$ as above and apply $e_i$, then the result will be either zero or equal to $ \sum_{l \in \mathrm{I}} m \otimes v_l \otimes v_l^* = \sum_{l \in \mathrm{I}} m \otimes v_l^* \otimes v_l$.
Applying $\left( \Omega_{0i} + \Omega_{0(i+1)}\right)$ to this, we obtain
\abovedisplayskip0.2em
\belowdisplayskip0.3em
\begin{gather*}
\begin{aligned}
& \sum_{\alpha \in \mathbb{B}(\mathfrak{so}_{2n})} \sum_{l \in \mathrm{I}} X_\alpha m\otimes X_\alpha^*v_l^* \otimes v_l + X_\alpha m \otimes v_l \otimes X_\alpha^* v_l^* \\
=& \sum_{\alpha \in \mathbb{B}(\mathfrak{so}_{2n})} \sum_{k,l \in \mathrm{I}} \left\langle X_\alpha^*v_l^*,v_k^*\right\rangle X_\alpha m\otimes v_k \otimes v_l + \left\langle X_\alpha^*v_l^*,v_k^*\right\rangle X_\alpha m \otimes v_l \otimes v_k \\
=& \sum_{\alpha \in \mathbb{B}(\mathfrak{so}_{2n})} \sum_{k,l \in \mathrm{I}} \left\langle X_\alpha^*v_l^*,v_k^*\right\rangle X_\alpha m\otimes v_k \otimes v_l - \left\langle X_\alpha^*v_k^*,v_l^*\right\rangle X_\alpha m \otimes v_l \otimes v_k = 0.
\end{aligned}
\end{gather*}
For the last equality one just switches $k$ and $l$ in the second sum. Hence we also have $(y_i+y_{i+1})e_i = 0$ on $M \otimes V^{\otimes d}$.
\end{proof}

\subsection{Proof of Proposition~\ref{a}} \label{section:proof-of-a}
\begin{proof}[Proof of Proposition~\ref{a}] 
Let us first ignore the grading. Statements \eqref{HIM1} and  \eqref{HIM2} follow directly from Definition~\ref{def:functors}, since we sum over all possible blocks. The statements about Verma modules follow directly from Proposition~\ref{Verma} and the definition of our functors. We now consider the action on  the modules $L(\la)$. 

For cases \eqref{HIM8} and  \eqref{HIM9} we know that $\cF_{i,-}M^\pp(\la) = \{0\}$, hence $\cF_{i,-}L(\la) = \{0\}$, since $L(\la)$ is a quotient of $M^\pp(\la)$ and $\cF_{i,-}$ is exact.

The proofs of  \eqref{HIM1} - \eqref{HIM4} are not much harder. For example, if
$\la = \la_{{\scriptscriptstyle \down}\circ}$ as in \eqref{HIM1},
then $\cF_{i,-} L(\la)$ is a quotient of
$\cF_{i,-} M^\pp(\la) \cong
M^\pp(\la_{\circ{\scriptscriptstyle \down}})$.
Moreover it is self-dual. This is because
$L(\la)$ is self-dual with respect to the duality on category $\cO$, and $\cF$ commutes with this duality \cite{AS}, and hence so does $\cF_{i,-}$. Thus we either have 
$\cF_{i_-} L(\la) = \{0\}$ or $\cF_{i_-} L(\la) \cong L(\la_{\circ{\scriptscriptstyle \down}})$,
as these are the only self-dual quotients of
$M^\pp(\la_{\circ{\scriptscriptstyle \down}})$.
To rule out the possibility that it is zero,
consider the group homomorphism $K_0(\cF_{i,-}):K_0(\Op_\Ga(n))\rightarrow K_0(\Op_{\Ga_{i,-}}(n))$ induced by $\cF_{i,-}$. Expressed in the basis of classes of parabolic Verma modules, it has an inverse, the morphism induced by $\cF_{i,+}$.
Hence $\cF_{i,-}$ is non-zero on every non-zero module. This proves  \eqref{HIM1}, and the proofs of  \eqref{HIM2} - \eqref{HIM4} are similar.

Next we check  \eqref{HIM5} and  \eqref{HIM6} for $L(\la)$, i.e.
we show that
$\cF_{i,-} L(\la_{\scriptscriptstyle \down\up})
\cong L(\la_{\circ{\scriptscriptstyle\times}})$
and
$\cF_{i,-} L(\la_{\scriptscriptstyle \up\down})
= \{0\}$.
We know that $\cF_{i,-} M^\pp(\la_{\scriptscriptstyle \down\up}) \cong \cF_{i,-}
M^\pp(\la_{\scriptscriptstyle \up\down}) \cong
M^\pp(\la_{\circ{\scriptscriptstyle\times}})$.
By self-duality, we either have
$\cF_{i_-} L(\la_{\scriptscriptstyle \down\up})
\cong L(\la_{\circ{\scriptscriptstyle\times}})$
or $\cF_{i,-} L(\la_{\scriptscriptstyle\down\up}) = \{0\}$, as well as $\cF_{i,-} L(\la_{\scriptscriptstyle \up\down})
\cong L(\la_{\circ{\scriptscriptstyle\times}})$
or $\cF_{i,-} L(\la_{\scriptscriptstyle\up\down}) = \{0\}$.
As $[\cF_{i,-} M^\pp(\la_{\scriptscriptstyle\down\up}) \hspace{-.1em}:\hspace{-.1em}
L(\la_{\circ{\scriptscriptstyle\cross}})] = 1$,
we have a composition factor
$L(\mu)$ of $M^{\pp}(\la_{\scriptscriptstyle\down\up})$ with $[\cF_{i,-} L(\mu)\hspace{-.1em}:\hspace{-.1em} L(\la_{\circ{\scriptscriptstyle\times}})] = 1$.
The facts proved so far imply either that
$\mu = \la_{\scriptscriptstyle\down\up}$
or that
$\mu = \la_{\scriptscriptstyle\up\down}$.
But the latter case cannot occur as
$\la_{\scriptscriptstyle\up\down}$ is strictly bigger than
$\la_{\scriptscriptstyle\down\up}$ in the Bruhat ordering.
Hence
$\mu = \la_{\scriptscriptstyle\down\up}$
and we have proved
$[\cF_{i,-} L(\la_{\scriptscriptstyle\down\up}):L(\la_{\circ{\scriptscriptstyle\times}})]=1$.
This proves \eqref{HIM5}, since $\cF_{i,-} L(\la_{\scriptscriptstyle \down\up})
\cong L(\la_{\circ{\scriptscriptstyle\times}})$.
It remains for  \eqref{HIM6} to show that
$\cF_{i,-} L(\la_{\scriptscriptstyle \up\down}) = \{0\}$.
Suppose for a contradiction that it is non-zero, hence
$\cF_{i,-} L(\la_{\scriptscriptstyle \up\down})
\cong L(\la_{\circ{\scriptscriptstyle\times}})$.
By Proposition~\ref{singO} and BGG-reciprocity we know that
$M^\pp(\la_{\scriptscriptstyle \up\down})$ has both
$L(\la_{\scriptscriptstyle \up\down})$
and $L(\la_{\scriptscriptstyle \down\up})$ as composition factors, so we deduce that
$[\cF_{i,-} M^\pp(\la_{\scriptscriptstyle \up\down})\hspace{-.1em}:\hspace{-.1em}
L(\la_{\circ{\scriptscriptstyle\times}})]
\geq 2$, which is the desired contradiction.

In this paragraph, we check  \eqref{HIM7c}.
Take $\la$ with $\la = \la_{{\scriptscriptstyle \times}\circ}$.
Let $\Ga$ be the block containing $\la$ and $\Ga_{i,-}$, $\Ga_{i,--}$ the domains of the functors $\cF_{i,-}$ respectively $\cF_{i,-}^2$ restricted to $\Ga$. We know $\cF_{i,-}^2 M^\pp(\nu)
\cong
M^\pp(\nu_{\circ{\scriptscriptstyle\times}})
\oplus M^\pp(\nu_{\circ{\scriptscriptstyle\times}})$ for any $\nu \in \Ga$.
Hence $\cF_{i,-}^2$ induces a $\mZ$-module isomorphism
between $[\Op_\Ga(n)]$ and
$2 [\Op_{\Ga_{i,--}}(n)]$.
Thus for non-zero $M \in \Op(n)$,
$\cF_{i,-}^2 M \neq 0$ with class divisible by two in $[\Op_{\Ga_{i,--}}(n)]$.
In particular,  $\cF_{i,-}^2 L(\la)$ is a non-zero self-dual quotient of
$M^\pp(\la_{\circ{\scriptscriptstyle\times}})
\oplus M^\pp(\la_{\circ{\scriptscriptstyle\times}})$
whose class is divisible by two.
This implies
\abovedisplayskip0.2em
\belowdisplayskip0.3em
\begin{gather}\label{drop}
\cF_{i,-}^2 L(\la) \,\cong\, L(\la_{\circ{\scriptscriptstyle\times}})\oplus L(\la_{\circ{\scriptscriptstyle\times}}).
\end{gather}
Now take any $\mu\in\Ga_{i,-}$.
We know already that
$\cF_{i,-} L(\mu) \cong L(\mu_{\circ{\scriptscriptstyle \times}})$
if $\mu = \mu_{\scriptscriptstyle \down\up}$, and
$\cF_{i,-} L(\mu) = \{0\}$ otherwise.
Assuming now that $\mu =
\mu_{\scriptscriptstyle \down\up}$,
we deduce from this that
$[\cF_{i,-} L(\la)\hspace{-.1em}:\hspace{-.1em} L(\mu)] =
[\cF_{i,-}^2 L(\la)\hspace{-.1em}:\hspace{-.1em} L(\mu_{\circ{\scriptscriptstyle \times}})]$.
Using \eqref{drop}, we conclude for $\mu = \mu_{\scriptscriptstyle \down\up}$
that
$[\cF_{i,-} L(\la)\hspace{-.1em}:\hspace{-.1em}L(\mu)] = 0$ unless $\mu =
\la_{\scriptscriptstyle \down\up}$, and
$[\cF_{i,-} L(\la)\hspace{-.1em}:\hspace{-.1em}L(\la_{\scriptscriptstyle \down\up})] = 2$.
It is easy to see that the analogous statements hold for $\cF_{i,+}$.
 
Now we deduce all the statements \eqref{HIM1} - \eqref{HIM9} for $P^\pp(\la)$
by using the fact that $(\cF_{i,-}, \cF_{i,+})$ is an adjoint pair of functors.
We just explain the argument in case of \eqref{HIM5}, since the other cases
are similar (actually, easier).
As $\cF_{i,-}$ sends projectives to projectives, $\cF_{i,-} P^\pp(\la)$ is a direct
sum of indecomposable projectives. To compute the multiplicity of $P^\pp(\mu)$ we calculate
\abovedisplayskip0.2em
\belowdisplayskip0.3em
\[
\HOM_{\mg}(\cF_{i,-} P^\pp(\la), L(\mu))
\cong \HOM_{\mg}(P^\pp(\la), \cF_{i,+} L(\mu)) = [\cF_{i,+} L(\mu): L(\la)].
\] 
By the analogue of \eqref{HIM7c} for $\cF_{i,+}$
this multiplicity is zero unless $\mu = \la_{\circ{\scriptscriptstyle\times}}$,
when it is two.
Hence $\cF_{i,-} P^\pp(\la) \cong P^\pp(\la_{\circ{\scriptscriptstyle\times}}) \oplus
P^\pp(\la_{\circ{\scriptscriptstyle\times}})$.

We still need to verify \eqref{HIM7d}.
By \eqref{HIM7a}, \eqref{HIM7c} and exactness of $\cF_{i,-}$, we get that
$\cF_{i,-} L(\la)$ is a non-zero quotient of
$P(\la_{\scriptscriptstyle \down\up})$,
hence it has irreducible head
isomorphic to
$L(\la_{\scriptscriptstyle\down\up})$.
Since it is self-dual it also has irreducible socle
isomorphic to
$L(\la_{\scriptscriptstyle\down\up})$.

Now consider the case $i=\half$. The statement for the Verma modules is again clear and for the simple modules we argue as above. For the projective modules we consider again each case. Part \eqref{HIM13} is proved as \eqref{HIM1} or \eqref{HIM2}, whereas \eqref{HIM15} is proved as the dual version \eqref{HIM10} of \eqref{HIM3}. Now for \eqref{HIM14} and \eqref{HIM18} each Verma module occurring in a  filtration of the projective module has highest weight $\nu$ and $\nu$ equals $\la$ at the places $0$ and $1$, hence the whole module get killed by the functor.  Case\eqref{HIM16} is proved as case \eqref{HIM14} whereas \eqref{HIM17} is proved as the dual \eqref{HIM10}of \eqref{HIM7}. Hence it remains to show \eqref{HIM11} and \eqref{HIM12} which however can be proved with the arguments from \eqref{HIM5}.  

Finally \eqref{HIM19} and \eqref{HIM20} can be proved as \eqref{HIM1}-\eqref{HIM4}.
\end{proof}

\subsection{Proof of Proposition~\ref{whichproj} and Proposition~\ref{saturated}}
\label{whichandsat}
\begin{proof}[Proof of Proposition~\ref{whichproj}]
Starting with $d=0$ there is only one weight with $\delta$-height zero, namely $\de$ and the statement is clear. For $d=1$ there are two weights, $\de-\epsilon_1$ and $\de+\epsilon_n$. The statement then follows from both parabolic Verma modules being projective. Assume now that the statement is true for $0 \leq d^\prime \leq d-1$. Note that $a_\mu\not=0$ implies that $M^\pp(\mu)$ appears in a Verma flag of $\MdV$, and hence $\dht(\mu)\leq d$. Moreover, the highest weights of all occurring parabolic Verma modules change each step by $\pm\epsilon_j$, changing the $\delta$-height by $\pm 1$, thus $d-\dht(\mu)$ is even.

Conversely, assume $\mu$ is a weight with $d-\dht(\mu)=2k$ for some $k\geq0$. If $k>0$ then by induction hypothesis $P(\mu)$ appears in $\MdVmtwo$. The weight $\mu$ contains at least one neighboured $\up\circ$, $\down\circ$, or $\cross\circ$ pair (since only finitely many vertices are labelled $\up$, $\down$ or $\cross$). Hence we are in \eqref{HIM1}, \eqref{HIM2} or \eqref{HIM7} of Proposition~\ref{a} thus $P(\mu)$ is a summand of $\cF^2 P(\mu)$ and therefore also of $\MdV$.
Now we have to proceed via a case-by case analysis for $k=0$:

\textit{(1) Assume first that there exist $i\in I(\mu)$, $-i\not\in I(\mu)$ with $i=(\mu+\rho)_a>0$ and $\de_a-\mu_a>0$ for some $a$.} (In the diagram picture this means that the $a$-th $\down$ in $\de$ was moved to the left and gives a $\down$ in $\mu$ at position $i$). We assume that $a$ was chosen to be maximal amongst these (that is the rightmost $\down$ which got moved to the left to create some $\down$). Locally at the positions $i$ and $i+1$ the diagram for $\mu$ could look as follows
    \begin{enumerate}[(a)]
    \item $\down\circ$. In this case set $\nu=\circ\down$
    \item $\down\up$. In this case set $\nu=\circ\cross$
    \item $\down\cross$. In this case set $\nu=\cross\down$
    \item $\down\down$. Since $a$ was maximal, this means that the $\down$ at position $i+1$ forces $i+1=(\mu+\rho)_{a+1}$ and at the same time $\de_{a+1}-\mu_{a+1}\leq 0$ which is a contradiction.
    \end{enumerate}
In each of the first three cases $P(\mu)$ appears as a summand in $\cF P(\nu)$ (by Proposition~\ref{a} \eqref{HIM3}, \eqref{HIM7}, and \eqref{HIM10}) and $\dht(\nu)=\dht(\mu)-1$. By hypothesis $P(\nu)$ appears as summand in $\cF^{d-1}M^\pp(\underline{\delta})$ and we are done.

\textit{(2) Assume (1) does not hold and moreover that if $i=(\mu+\rho)_a>0$ for some $a$ and $-i\notin \mathrm{I}$ then $\de_a-\mu_a=0$.} (Diagrammatically this means that $\down's$ did either not move or turned into $\up's$, creating possibly some $\cross$'s). Choose, if it exists, $-j\in \mathrm{I}$ maximal such that $-j<0$. (Diagrammatically we look at the $\up$'s and $\cross$'s and choose the leftmost. It  sits at position $j$).
    Locally at the positions $j-1$ and $j$, for $j>\frac{1}{2}$, the diagram for $\mu$ could look as follows
    \begin{enumerate}[(a)]
    %\item $\up\up$. By the maximality of $-j$ this is not possible.
    \item $\diamondb\up$. In this case set $\nu=\circ\cross$.
    \item $\down\up$. In this case set $\nu=\cross\circ$.
    \item $\circ\up$. In this case set $\nu=\up\circ$.
    \item $\down\cross$. In this case set $\nu=\cross\down$.
    \item $\circ\cross$. In this case set $\nu=\up\down$. The $\down$ in $\de$ at place $j-1$ got moved, hence turned into an $\up$ or created a $\cross$. But then the same happened to all the $\down$'s further to the left. In particular, $\nu$ has no $\down$ left of the $\up$ at position $j$ and hence we can apply Corollary~\ref{special-case} to $P(\nu)$.
    \end{enumerate}
Again, in each of the last four cases $P(\mu)$ appears as a summand in $\cF P(\nu)$ by Proposition~\ref{a} and by hypothesis $P(\nu)$ appears as summand in $\cF^{d-1}M^\pp(\underline{\delta})$. In case no $\up$ or $\cross$ exists we either have $\mu=\de$ and there is nothing to do or $\mu$ has a $\circ\down$ or $\down\circ$ pair which we swap to create a new weight $\nu$ and argue as above using Proposition~\ref{a} \eqref{HIM1}, \eqref{HIM3} or \eqref{HIM10}.

\textit{(3) Finally, assume that neither (1) nor (2) holds.}
The arguments are completely analogous to (2) except for the case
\begin{enumerate}[(e')]
\item $\circ\cross$. In case there is no $\down$ to the left of the cross we set $\nu=\up\down$ and argue as before using Corollary~\ref{special-case}. Otherwise such a $\down$ did not get moved or got moved to the right, since we excluded case (1). Since it is separated from our $\cross$ at position $j$ by a $\circ$, our $\cross$ is created from a $\down$ which got moved to the right. Hence changing $\circ\cross$ to $\down\up$ decreases the height by $1$. Then we can argue again using Proposition~\ref{a}, this time part \eqref{HIM5}.
\end{enumerate}
\end{proof}

\begin{proof}[Proof of Proposition~\ref{saturated}]
Assume $\mu$ is in the same block as $\nu$ and $\mu<\nu$. Then $\nu$ is obtained from $\mu$ by applying a finite sequence of changes from $\up\up$ to $\down\down$ or from $\down\up$ to $\up\down$ at positions only separated by $\circ$'s and $\cross$'s. Hence it is enough to consider these basic changes.

\textit{Changing from $\down\up$ to $\up\down$:} Let $0\leq i\leq i+j$ be the positions of the two symbols. In the $\epsilon$-basis $\mu+\rho$ and $\nu+\rho$ are then of the form
\abovedisplayskip0.2em
\belowdisplayskip0.3em
\begin{gather*}
(\overbrace{-x_1,\ldots -x_a},-(i+j),\overbrace{-y_1,\ldots -y_b},\overbrace{z_1,\ldots, z_c}, \overbrace{u_1,\ldots u_r}, i, \overbrace{v_1,\ldots, v_s},\overbrace{w_1,\ldots w_t})\\
(\overbrace{-x_1,\ldots -x_a},\overbrace{-y_1,\ldots -y_b},-i \overbrace{z_1,\ldots, z_c}, \overbrace{u_1,\ldots u_r}, \overbrace{v_1,\ldots, v_b}, i+j, \overbrace{w_1,\ldots w_t})
\end{gather*}
where $a,b,c,r,t\geq 0$ with the $x_k>0$ indicating the $\up$'s to the right of position $i+j$, the $y_k>0$ indicating the $\up$'s as part of the crosses between positions $i$ and $i+j$, the $z_k>0$ indicating the $\up$'s to the left of position $i$, the $u_k>0$ indicating the $\down$'s to the left of position $i$, the $v_k>0$ indicating the $\down$'s as part of the crosses between positions $i$ and $i+j$ and the $w_k>0$ indicating the $\down$'s to the right of position $i+j$.
Hence to determine $\dht(\mu)-\dht(\nu)$ we can without loss of generality assume that the $x$, $z$, $u$ and $w$'s are zero. Recall that $\underline{\delta}+\rho$ is an increasing sequence of consecutive numbers, so we are left with
\abovedisplayskip0.2em
\belowdisplayskip0.3em
\begin{gather*}
\begin{aligned}
&\dht(\mu)-\dht(\nu)\\
&=(|m+i+j|+{\sum}_{k=1}^b|m+k+y_k|-{\sum}_{k=1}^b|{m+k-1+y_k}|-|m+b+i|)\\
&+ (|p-i|+{\sum}_{k=1}^b|p+k-v_k|-{\sum}_{k=1}^b|{p+k-1-v_k}|-|p+b-i-j|)
\end{aligned}
\end{gather*}
for some non-negative $m$, $p$ which are (half-)integers $i$, $j$ in case $i$,$j$ are (half-)integers. The absolute values in the first summand can be removed and hence gives $j$ in total. Set $v_0=i$ and $v_{b+1}=
i+j$, and choose $-1\leq k_0\leq b+1$ such that $p+k-v_k>0$ for $0\leq k\leq k_0$ and $p+k-v_{k}\leq 0$ for $b+1\geq k> k_0$.
\abovedisplayskip0.2em
\belowdisplayskip0.3em
\begin{gather*}
\begin{aligned}
&\dht(\mu)-\dht(\nu)\\
&=j+|p-v_0|-|p+b-v_{b+1}|+\sum_{1\leq k\leq k_0,b}((p+k-v_k)-(p+k-1-v_k))\\
&+\sum_{k_0+1\leq k\leq b}((v_k-p-k)-(v_k-p-k+1))\\
&=j+|p-v_0|-|p+b-v_{b+1}|+k_0-(b-k_0)=\begin{cases}
     2j&\text{if $k_0=b+1$,}\\
     2(p-i+k_0)&\text{if $0\leq k_0\leq b$,}\\
     0&\text{if $k_0=-1$.}
     \end{cases}
\end{aligned}
\end{gather*}
In any case $\dht(\mu)\geq\dht(\nu)$ and their difference is even, since $j$, $k_0$, $p-i\in\mathbb{Z}$.

\textit{Change from $\up\up$ to $\down\down$:} It is easy to check that the difference of the heights is even. We claim that if $\mu$ and $\nu$ are weights such that $\nu$ is obtained from $\mu$ just by changing an $\up$ at some position $i$ which has no other $\up$ or $\down$ to its  left into a $\down$ then $\dht(\nu)\leq \dht(\mu)$. Together with the previous paragraph this claim proves the lemma. In the $\epsilon$-basis, the weights $\mu+\rho$ and $\nu+\rho$ are of the form
\abovedisplayskip0.2em
\belowdisplayskip0.3em
\begin{gather*}
(\overbrace{-x_1,\ldots -x_a},\overbrace{-y_1,\ldots -y_b},\overbrace{z_1,\ldots, z_b}, i, \overbrace{w_1,\ldots w_c})\\
(\overbrace{-x_1,\ldots -x_a},-i,\overbrace{-y_1,\ldots -y_b}, \overbrace{z_1,\ldots, z_b}, \overbrace{w_1,\ldots w_c})
\end{gather*}
where $a,b,c\geq 0$ with the $x_k>0$ indicating the $\up$'s to the right of position $i$, the $y_k>0$ indicating the $\up$'s as part of the crosses to the left of $i$, the $z_k>0$ indicating the $\down$'s as part of the crosses to the left of $i$ and the $w_k>0$ indicating the $\down$'s to the right of position $i$. Again, the $x$'s and $w$'s are irrelevant for the difference of the heights. Note that the number of $y$'s equals the number of $z$'s. With the arguments from above we get
\abovedisplayskip0.2em
\belowdisplayskip0.3em
\begin{gather*}
\begin{aligned}
\dht(\mu)-\dht(\nu)
&=(|m+i|+{\sum}_{k=1}^b(m+k+y_k)-{\sum}_{k=1}^b({m+k-1+y_k}))\\
&+({\sum}_{k=1}^b(|p+k-z_k|-{\sum}_{k=1}^b|{p+k-1-z_k}|)-|p+b-i|)
\end{aligned}
\end{gather*}
with $p=m+b+1$. The first sum equals $b$, the second $\geq -b$ and hence $\dht(\mu)-\dht(\nu)\geq m+i-|m+b+1-i|$. This is obviously positive if $m+b+1-i<0$ and equals $-2b+2i-1$ otherwise. On the other hand, $b<i$ and our claim follows.
\end{proof}

\subsection{Proof of Theorem~\ref{Oprime}} 
\label{sec:proofOprime}
\begin{proof}[Proof of Theorem~\ref{Oprime}]
We assume here $\delta \geq 0$, the case $\delta<0$ is again obtained via Corollary~\ref{tableaux}. Let $\cS'$ denote the set $\la\in\cS(\delta,d)$ which appear as endpoints of $\lambda \in \vpath_d(\delta)$ such that any two consecutive steps of $\lambda$ only differ by $\pm \epsilon_j$ for $j<n$.

Via Proposition~\ref{surjectivity}, such a path correspond to an \updbd x with second component being the empty partition. Thus, ignoring the second component, we identify such paths with \updd x. To show that $z_d\VWd(\alpha,\beta)z_d$ is quasi-hereditary it suffices to show that all indecomposable projective modules indexed by weights in $\cS'$ are summands in $P_d$.

For $d=0$ there is nothing to check and for $d=1$ the partition with one box corresponds to $P(\la_{\ldots\circ\up\circ\up})=M^\pp(\de-\epsilon_1)$.
 
Now assume the claim is true for $d-2$. Hence all partitions in $\cS^\prime$ with $d-2k$ $(k \geq 1)$ boxes arise as labels of indecomposable projectives in $P_{d-2} = \widetilde{\cF}^{d-2}M^\pp(\underline{\delta})$. Let now $\la$ be such a partition of $d-2k$ boxes. If $\delta>0$ then the associated weight diagram has at least one $\circ$ and hence there is $i < \nicefrac{\delta}{2}+n$ such that at position $i$ and $i+1$, $\lambda$ looks locally like
\abovedisplayskip0.2em
\belowdisplayskip0.3em
\begin{gather} \label{eq:lambda-mu}
\begin{array}{c||c|c|c|c|c|c|c|c|c}
\mu&\up\circ &\down\circ &\down\up &\circ\up &\circ\down &\down\up\\
\la&\circ\up &\circ\down &\circ\cross &\up\circ &\down\circ &\cross\circ
\end{array}
\end{gather}
Then there is a weight $\mu$ only differing from $\la$ at the positions $i$, $i+1$, as shown in \eqref{eq:lambda-mu}.  Then $\widetilde{\cF}P(\la)$ contains $P(\mu)$ as direct summand and $\widetilde{\cF}^2P(\lambda)$ contains $P(\la)$ again as summand. Hence $P(\la)$ appears as a summand in $P_d$. Assume now, that $\la$ is a partition $d$ and $\delta>0$. If $d>0$ then there is at least one box which can be removed from the partition, hence there exists $i < \nicefrac{\delta}{2}+n$ such that locally at positions $i$ and $i+1$, $\la$ looks as follows (with special case of $i=\half$ in the last column)
\abovedisplayskip0.2em
\belowdisplayskip0.3em
\begin{gather*}
\begin{array}{c||c|c|c|c|c|c|c|c|c|c}
\la&\circ\up &\down\up &\circ\cross &\down\cross &\down\circ &\cross\circ&\down\up&\cross\up&\diamondb\circ&\up\\
\mu&\up\circ &\cross\circ &\up\down &\cross\down &\circ\down &\up\down &\circ\cross&\up\cross&\circ\down&\down
\end{array}
\end{gather*}
In this case, we use Lemma \ref{combinatorics}, \eqref{cool2} and  Corollary~\ref{special-case} to obtain a weight $\mu$ whose partition has $d-1$ boxes such that $\widetilde{\cF}P(\mu)$ contains $P(\la)$ as direct summand and we are done by induction.

It remains to consider the case $\delta=0$. For $d=0$, $d=1$ the assertion is clear. For $d=2$ we have $P_2=P(a)\oplus P(b)$, where $a$ and $b$ correspond to the partitions $(2)$ and $(1,1)$ respectively and $\END_{\mg}(P_2)\cong \mC[x]/(x^{2})\oplus \mC$ which has infinite global dimension and hence is not quasi-hereditary. In case $\delta=0$ and $d$ even then every weight contains at east one $\circ$ and we can argue as in the case $\delta>0$. For $d$ odd observe that the occurring weights $\la$ always contain at least one $\up$ or $\cross$, and $P_d$ contains a summand isomorphic to $P_2$. Hence the algebra is not quasi-hereditary.
\end{proof}

\subsection{Proof of Lemma~\ref{lemma:Fastransl}}
\label{boring}
\begin{proof}[Proof of Lemma~\ref{lemma:Fastransl}]
Recall from \cite{BGGclass} that projective functors, i.e. direct summands of endofunctors on $\cO(n)$ given by taking the tensor product with finite dimensional modules, are determined up to isomorphism already by their value on the Verma modules whose weight is maximal in their dot-orbit of $W_n$. Hence to prove the lemma it is enough to compare the value on these Verma modules. Since such a Verma module is a projective object it is enough to compare their values in the Grothendieck group. We explain this for $\mathbb{B}_{i}$ and $\mathbb{B}_{0}$, the case $\mathbb{B}_{-i}$ is similar. Consider a block $\kappa = ((1,1,\ldots,1),\underline{d}, {\epsilon})$ with corresponding maximal weight $\la$ such that $$\la+\rho=(\underbrace{\half,\ldots,\half}_{d_1},\ldots,\underbrace{n-\half,\ldots,n-\half}_{d_n})$$
in case ${\epsilon}=0$ and with the first number negated if ${\epsilon}=1$. Examples are
\abovedisplayskip0.3em
\belowdisplayskip0.3em
\[
(\nicefrac{3}{2},\nicefrac{3}{2},\nicefrac{3}{2},\nicefrac{5}{2},\nicefrac{5}{2},\nicefrac{5}{2},\nicefrac{7}{2},\nicefrac{7}{2}) \quad \text{respectively} \quad (-\nicefrac{3}{2},\nicefrac{3}{2},\nicefrac{3}{2},\nicefrac{5}{2},\nicefrac{5}{2},\nicefrac{5}{2},\nicefrac{7}{2},\nicefrac{7}{2})
.\]

Now fix first $i>0$ and consider $\mathbb{B}_{i}$. Applying the translation out of the wall $\Theta_{\kappa}^{\kappa^{+i}}$ to $M^\pp(\lambda)$ yields a module with a Verma flag having subquotients with precisely those $M^\pp(\nu)$ such that $\nu+\rho$ is obtained from $\lambda+\rho$ by first changing all entries with absolute value strictly larger than $i+\half$ by $1$ if positive and by $-1$ if negative and then doing the same to all but one entry with absolute value equal to $i+\half$. In the running example one obtains the following $\rho$-shifted highest weights of parabolic Verma modules appearing in $\Theta_{\bd}^{\bd^{+i}} M^\pp(\lambda)$
\abovedisplayskip0.3em
\belowdisplayskip0.3em
\begin{gather*}
(\nicefrac{3}{2},\nicefrac{3}{2},\nicefrac{3}{2},\nicefrac{7}{2},\nicefrac{7}{2},{\color{red}
\nicefrac{5}{2}},\nicefrac{9}{2},\nicefrac{9}{2}), \quad (\nicefrac{3}{2},\nicefrac{3}{2},\nicefrac{3}{2},\nicefrac{7}{2},{\color{red}\nicefrac{5}{2}},\nicefrac{7}{2},\nicefrac{9}{2},\nicefrac{9}{2}), \\
(\nicefrac{3}{2},\nicefrac{3}{2},\nicefrac{3}{2},{\color{red}\nicefrac{5}{2}},\nicefrac{7}{2},\nicefrac{7}{2},\nicefrac{9}{2},\nicefrac{9}{2})
\end{gather*}
in case $\epsilon=0$ and with a sign switch for the first entry if $\epsilon=1$.

Applying then the translation functor $\Theta^{{}_{+i}\kappa}_{\kappa^{+i}}$ to a parabolic Verma module appearing in the flag, we get a parabolic Verma module with a highest weight $\mu$ that is obtained from $\nu$ by changing the unique entry with absolute value $i+\half$ to absolute value $i-\half$, while keeping the sign. And then changing all entries with absolute value strictly bigger than $i+\half$ by $-1$ if positive and by $1$ if negative. In the running example we thus obtain the parabolic Verma modules having highest $\rho$-shifted weights
\abovedisplayskip0.3em
\belowdisplayskip0.3em
\begin{gather*}
(\nicefrac{3}{2},\nicefrac{3}{2},\nicefrac{3}{2},\nicefrac{5}{2},\nicefrac{5}{2},{\color{red}
\nicefrac{3}{2}},\nicefrac{7}{2},\nicefrac{7}{2}), \quad (\nicefrac{3}{2},\nicefrac{3}{2},\nicefrac{3}{2},\nicefrac{5}{2},{\color{red}\nicefrac{3}{2}},\nicefrac{5}{2},\nicefrac{7}{2},\nicefrac{7}{2}), \\
(\nicefrac{3}{2},\nicefrac{3}{2},\nicefrac{3}{2},{\color{red}\nicefrac{3}{2}},\nicefrac{5}{2},\nicefrac{5}{2},\nicefrac{7}{2},\nicefrac{7}{2})
\end{gather*}
for $\epsilon=0$ and with a sign switch for the first entry if $\epsilon=1$.

Comparing with Theorem~\ref{prop:catskewclass} we see this agrees with the values for $\mathcal{F}_{i,\pm}$, hence $\mathcal{F}_{i,+}\cong \mathbb{B}_{i}$. 

In case $i=0$ we obtain all highest $\rho$-shifted weights obtained from $\lambda$ by changing the sign of one entry with absolute value equal to $\half$. 
For instance $(\nicefrac{1}{2},\nicefrac{1}{2},\nicefrac{1}{2},\nicefrac{3}{2},\nicefrac{3}{2})$ gives $(-\nicefrac{1}{2},\nicefrac{1}{2},\nicefrac{1}{2},\nicefrac{3}{2},\nicefrac{3}{2})$, $(\nicefrac{1}{2},-\nicefrac{1}{2},\nicefrac{1}{2},\nicefrac{3}{2},\nicefrac{3}{2})$, and $(\nicefrac{1}{2},\nicefrac{1}{2},-\nicefrac{1}{2},\nicefrac{3}{2},\nicefrac{3}{2})$. Thus, by comparison with Theorem~\ref{prop:catskewclass}, $\mathbb{B}_{0}\cong\mathcal{F}_{0}$.
\end{proof}

\bibliographystyle{abbrv}
\bibliography{Brauer-VW}

\def\cprime{$'$}
\begin{thebibliography}{10}

\bibitem{ADL}
I.~{\'A}goston, V.~Dlab, and E.~Luk{\'a}cs.
\newblock Quasi-hereditary extension algebras.
\newblock {\em Algebr. Represent. Theory}, 6(1):97--117, 2003.

\bibitem{AL}
H.~H. Andersen and N.~Lauritzen.
\newblock Twisted {V}erma modules.
\newblock In {\em Studies in memory of {I}ssai {S}chur ({C}hevaleret/{R}ehovot,
  2000)}, volume 210 of {\em Progr. Math.}, pages 1--26. Birkh\"auser Boston,
  Boston, MA, 2003.

\bibitem{AndersenStroppel}
H.~H. Andersen and C.~Stroppel.
\newblock Twisting functors on {$\mathcal O$}.
\newblock {\em Represent. Theory}, 7:681--699 (electronic), 2003.

\bibitem{AST}
H.~H. Andersen, C.~Stroppel, and D.~Tubbenhauer.
\newblock Semisimplicity of {H}ecke and (walled) {B}rauer algebras.
\newblock {\em J. Aust. Math. Soc.}, 103(1):1--44, 2017.

\bibitem{AS}
T.~Arakawa and T.~Suzuki.
\newblock Duality between {$\mathfrak{sl}_n(\mathbb{C})$} and the degenerate
  affine {H}ecke algebra.
\newblock {\em J. Algebra}, 209(1):288--304, 1998.

\bibitem{AMR}
S.~Ariki, A.~Mathas, and H.~Rui.
\newblock Cyclotomic {N}azarov-{W}enzl algebras.
\newblock {\em Nagoya Math. J.}, 182:47--134, 2006.

\bibitem{Backelin}
E.~Backelin.
\newblock Koszul duality for parabolic and singular category {$\mathcal O$}.
\newblock {\em Represent. Theory}, 3:139--152 (electronic), 1999.

\bibitem{BalKolb1}
M.~Balagovi\'c and S.~Kolb.
\newblock The bar involution for quantum symmetric pairs.
\newblock {\em Represent. Theory}, 19:186--210, 2015.

\bibitem{Li}
H.~{Bao}, J.~{Kujawa}, Y.~{Li}, and W.~{Wang}.
\newblock {Geometric {S}chur duality of classical type}.
\newblock {\em ArXiv e-prints}, Apr. 2014.

\bibitem{BSWW}
H.~{Bao}, P.~{Shan}, W.~{Wang}, and B.~{Webster}.
\newblock {Categorification of quantum symmetric pairs I}.
\newblock {\em ArXiv e-prints}, May 2016.

\bibitem{BW}
H.~{Bao} and W.~{Wang}.
\newblock {A new approach to {K}azhdan-{L}usztig theory of type {B} via quantum
  symmetric pairs}.
\newblock {\em ArXiv e-prints}, Sept. 2013.

\bibitem{BGS}
A.~Beilinson, V.~Ginzburg, and W.~Soergel.
\newblock Koszul duality patterns in representation theory.
\newblock {\em J. Amer. Math. Soc.}, 9(2):473--527, 1996.

\bibitem{BZ}
A.~Berenstein and S.~Zwicknagl.
\newblock Braided symmetric and exterior algebras.
\newblock {\em Trans. Amer. Math. Soc.}, 360(7):3429--3472, 2008.

\bibitem{BGGclass}
J.~N. Bernstein and S.~I. Gel{\cprime}fand.
\newblock Tensor products of finite- and infinite-dimensional representations
  of semisimple {L}ie algebras.
\newblock {\em Compositio Math.}, 41(2):245--285, 1980.

\bibitem{Brauer}
R.~Brauer.
\newblock {O}n algebras which are connected with the semisimple continuous
  groups.
\newblock {\em Ann. of Math. (2)}, 38(4):857--872, 1937.

\bibitem{Brown}
W.~P. Brown.
\newblock The semisimplicity of {$\omega_f^n$}.
\newblock {\em Ann. of Math. (2)}, 63:324--335, 1956.

\bibitem{BK1}
J.~Brundan and A.~Kleshchev.
\newblock Schur-{W}eyl duality for higher levels.
\newblock {\em Selecta Math. (N.S.)}, 14(1):1--57, 2008.

\bibitem{BKKL}
J.~Brundan and A.~Kleshchev.
\newblock Blocks of cyclotomic {H}ecke algebras and {K}hovanov-{L}auda
  algebras.
\newblock {\em Invent. Math.}, 178:451--484, 2009.

\bibitem{BSII}
J.~Brundan and C.~Stroppel.
\newblock Highest weight categories arising from {K}hovanov's diagram algebra.
  {II}. {K}oszulity.
\newblock {\em Transform. Groups}, 15(1):1--45, 2010.

\bibitem{BSIII}
J.~Brundan and C.~Stroppel.
\newblock Highest weight categories arising from {K}hovanov's diagram algebra
  {III}: category {$\mathcal{O}$}.
\newblock {\em Represent. Theory}, 15:170--243, 2011.

\bibitem{BS_walled_Brauer}
J.~Brundan and C.~Stroppel.
\newblock Gradings on walled {B}rauer algebras and {K}hovanov's arc algebra.
\newblock {\em Adv. Math.}, 231(2):709--773, 2012.

\bibitem{CKM}
S.~Cautis, J.~Kamnitzer, and S.~Morrison.
\newblock Webs and quantum skew {H}owe duality.
\newblock {\em Math. Ann.}, 360(1-2):351--390, 2014.

\bibitem{CGM}
H.~Chen, N.~Guay, and X.~Ma.
\newblock Twisted {Y}angians, twisted quantum loop algebras and affine {H}ecke
  algebras of type {$BC$}.
\newblock {\em Trans. Amer. Math. Soc.}, 366(5):2517--2574, 2014.

\bibitem{CPSI}
E.~Cline, B.~Parshall, and L.~Scott.
\newblock Finite-dimensional algebras and highest weight categories.
\newblock {\em J. Reine Angew. Math.}, 391:85--99, 1988.

\bibitem{CDVMII}
A.~Cox, M.~De~Visscher, and P.~Martin.
\newblock The blocks of the {B}rauer algebra in characteristic zero.
\newblock {\em Represent. Theory}, 13:272--308, 2009.

\bibitem{CDVMI}
A.~Cox, M.~De~Visscher, and P.~Martin.
\newblock A geometric characterisation of the blocks of the {B}rauer algebra.
\newblock {\em J. Lond. Math. Soc. (2)}, 80(2):471--494, 2009.

\bibitem{DRVcenter}
Z.~Daugherty, A.~Ram, and R.~Virk.
\newblock Affine and degenerate affine bmw algebras: The center.
\newblock arXiv1105.4207, 2011.
\newblock to appear in Osaka J. Math.

\bibitem{DRV1}
Z.~Daugherty, A.~Ram, and R.~Virk.
\newblock Affine and degenerate affine bmw algebras: Actions on tensor space.
\newblock {\em Selecta Mathematica}, 19(2):611--653, 2013.

\bibitem{Donkin}
S.~Donkin.
\newblock {\em The {$q$}-{S}chur algebra}, volume 253 of {\em London
  Mathematical Society Lecture Note Series}.
\newblock Cambridge University Press, Cambridge, 1998.

\bibitem{ES_springer}
M.~Ehrig and C.~Stroppel.
\newblock 2-row {S}pringer fibres and {K}hovanov diagram algebras for type {D}.
\newblock {\em Canad. J. Math.}, 68(6):1285--1333, 2016.

\bibitem{ES_diagrams}
M.~Ehrig and C.~Stroppel.
\newblock Diagrammatic description for the categories of perverse sheaves on
  isotropic {G}rassmannians.
\newblock {\em Selecta Math. (N.S.)}, 22(3):1455--1536, 2016.

\bibitem{ES_Brauer}
M.~Ehrig and C.~Stroppel.
\newblock Koszul gradings on {B}rauer algebras.
\newblock {\em Int. Math. Res. Not. IMRN}, (13):3970--4011, 2016.

\bibitem{EnSh}
T.~J. Enright and B.~Shelton.
\newblock Categories of highest weight modules: applications to classical
  {H}ermitian symmetric pairs.
\newblock {\em Mem. Amer. Math. Soc.}, 67(367), 1987.

\bibitem{EFM}
P.~Etingof, R.~Freund, and X.~Ma.
\newblock A {L}ie-theoretic construction of some representations of the
  degenerate affine and double affine {H}ecke algebras of type {$BC_n$}.
\newblock {\em Represent. Theory}, 13:33--49, 2009.

\bibitem{Li2}
Z.~{Fan} and Y.~{Li}.
\newblock {Geometric {S}chur Duality of Classical Type, II}.
\newblock {\em ArXiv e-prints}, Aug. 2014.

\bibitem{FKS}
I.~Frenkel, M.~Khovanov, and C.~Stroppel.
\newblock A categorification of finite-dimensional irreducible representations
  of quantum $\mathfrak{sl}_2$ and their tensor products.
\newblock {\em Selecta Math. (N.S.)}, 12(3-4):379--431, 2006.

\bibitem{GJbook}
M.~Geck and N.~Jacon.
\newblock {\em Representations of {H}ecke algebras at roots of unity},
  volume~15 of {\em Algebra and Applications}.
\newblock Springer-Verlag London, Ltd., London, 2011.

\bibitem{GW}
R.~Goodman and N.~R. Wallach.
\newblock {\em Symmetry, representations, and invariants}, volume 255 of {\em
  Graduate Texts in Mathematics}.
\newblock Springer, Dordrecht, 2009.

\bibitem{Graham-Lehrer}
J.~J. Graham and G.~I. Lehrer.
\newblock Cellular algebras.
\newblock {\em Invent. Math.}, 123(1):1--34, 1996.

\bibitem{Green}
R.~M. Green.
\newblock Generalized {T}emperley-{L}ieb algebras and decorated tangles.
\newblock {\em J. Knot Theory Ramifications}, 7(2):155--171, 1998.

\bibitem{Howe}
R.~Howe.
\newblock Perspectives on invariant theory: {S}chur duality, multiplicity-free
  actions and beyond.
\newblock In {\em The {S}chur lectures (1992) ({T}el {A}viv)}, volume~8 of {\em
  Israel Math. Conf. Proc.}, pages 1--182. Bar-Ilan Univ., 1992.

\bibitem{MathasHuQuiv}
J.~Hu and A.~Mathas.
\newblock Quiver {S}chur algebras {I}: linear quivers.
\newblock {\em arXiv:1110.1699}, 2011.

\bibitem{Hbook}
J.~E. Humphreys.
\newblock {\em Representations of semisimple {L}ie algebras in the {BGG}
  category {$\mathcal{O}$}}, volume~94 of {\em Graduate Studies in
  Mathematics}.
\newblock American Mathematical Society, Providence, RI, 2008.

\bibitem{JM}
D.~Jordan and X.~Ma.
\newblock Quantum symmetric pairs and representations of double affine {H}ecke
  algebras of type {$C^\vee C_n$}.
\newblock {\em Selecta Math. (N.S.)}, 17(1):139--181, 2011.

\bibitem{Khomenko}
O.~Khomenko.
\newblock Categories with projective functors.
\newblock {\em Proc. London Math. Soc. (3)}, 90(3):711--737, 2005.

\bibitem{KL}
M.~Khovanov and A.~D. Lauda.
\newblock A diagrammatic approach to categorification of quantum groups. {I}.
\newblock {\em Represent. Theory}, 13:309--347, 2009.

\bibitem{KLII}
M.~Khovanov and A.~D. Lauda.
\newblock A diagrammatic approach to categorification of quantum groups {II},
  2011.

\bibitem{Kolb}
S.~Kolb.
\newblock Quantum symmetric {K}ac-{M}oody pairs.
\newblock {\em Adv. Math.}, 267:395--469, 2014.

\bibitem{Koenig-Xi}
S.~K{\"o}nig and C.~Xi.
\newblock On the structure of cellular algebras.
\newblock In {\em Algebras and modules, {II} ({G}eiranger, 1996)}, volume~24 of
  {\em CMS Conf. Proc.}, pages 365--386. Amer. Math. Soc., Providence, RI,
  1998.

\bibitem{KX}
S.~K\"onig and C.~Xi.
\newblock A characteristic free approach to {B}rauer algebras.
\newblock {\em Trans. Amer. Math. Soc.}, 353(4):1489--1505, 2001.

\bibitem{LQR}
A.~D. Lauda, H.~Queffelec, and D.~E.~V. Rose.
\newblock Khovanov homology is a skew {H}owe 2-representation of categorified
  quantum {$\mathfrak{sl}_m$}.
\newblock {\em Algebr. Geom. Topol.}, 15(5):2517--2608, 2015.

\bibitem{LZZ}
G.~I. Lehrer, H.~Zhang, and R.~B. Zhang.
\newblock A quantum analogue of the first fundamental theorem of classical
  invariant theory.
\newblock {\em Comm. Math. Phys.}, 301(1):131--174, 2011.

\bibitem{LS}
T.~Lejczyk and C.~Stroppel.
\newblock A graphical description of $({D}_n,{A}_{n-1})$ {K}azhdan-{L}usztig
  polynomials.
\newblock {\em Glasg. Math. J.}, 55(2):313--340, 2013.

\bibitem{LetzterII}
G.~Letzter.
\newblock Coideal subalgebras and quantum symmetric pairs.
\newblock In {\em New directions in {H}opf algebras}, volume~43 of {\em Math.
  Sci. Res. Inst. Publ.}, pages 117--165. Cambridge Univ. Press, Cambridge,
  2002.

\bibitem{Letzter}
G.~Letzter.
\newblock Quantum symmetric pairs and their zonal spherical functions.
\newblock {\em Transform. Groups}, 8(3):261--292, 2003.

\bibitem{Lidl-Pilz}
R.~Lidl and G.~Pilz.
\newblock {\em Applied abstract algebra}.
\newblock Undergraduate Texts in Mathematics. Springer-Verlag, New York, second
  edition, 1998.

\bibitem{MW}
M.~{Mackaay} and B.~{Webster}.
\newblock {Categorified skew {H}owe duality and comparison of knot homologies}.
\newblock {\em ArXiv e-prints}, Feb. 2015.

\bibitem{Mazorchuk}
V.~Mazorchuk.
\newblock {\em Lectures on algebraic categorification}.
\newblock QGM Master Class Series. European Mathematical Society (EMS),
  Z\"urich, 2012.

\bibitem{MOSpair}
V.~Mazorchuk and S.~Ovsienko.
\newblock A pairing in homology and the category of linear complexes of tilting
  modules for a quasi-hereditary algebra.
\newblock {\em J. Math. Kyoto Univ.}, 45(4):711--741, 2005.
\newblock With an appendix by Catharina Stroppel.

\bibitem{MOS}
V.~Mazorchuk, S.~Ovsienko, and C.~Stroppel.
\newblock Quadratic duals, {K}oszul dual functors, and applications.
\newblock {\em Trans. Amer. Math. Soc.}, 361(3):1129--1172, 2009.

\bibitem{MSSerre}
V.~Mazorchuk and C.~Stroppel.
\newblock Projective-injective modules, {S}erre functors and symmetric
  algebras.
\newblock {\em J. Reine Angew. Math.}, 616:131--165, 2008.

\bibitem{MS}
V.~Mazorchuk and C.~Stroppel.
\newblock A combinatorial approach to functorial quantum {$\mathfrak{sl}_k$}
  knot invariants.
\newblock {\em Amer. J. Math.}, 131(6):1679--1713, 2009.

\bibitem{MSSchurWeyl}
V.~Mazorchuk and C.~Stroppel.
\newblock {$G(\ell,k,d)$}-modules via groupoids.
\newblock {\em J. Algebraic Combin.}, 43(1):11--32, 2016.

\bibitem{Musson}
I.~M. Musson.
\newblock {\em Lie superalgebras and enveloping algebras}, volume 131 of {\em
  Graduate Studies in Mathematics}.
\newblock American Mathematical Society, Providence, RI, 2012.

\bibitem{Nazarov}
M.~Nazarov.
\newblock Young's orthogonal form for {B}rauer's centralizer algebra.
\newblock {\em J. Algebra}, 182(3):664--693, 1996.

\bibitem{NDS}
M.~Noumi, M.~Dijkhuizen, and T.~Sugitani.
\newblock Multivariable {A}skey-{W}ilson polynomials and quantum complex
  {G}rassmannians.
\newblock In {\em Special functions, {$q$}-series and related topics
  ({T}oronto, {ON}, 1995)}, volume~14 of {\em Fields Inst. Commun.}, pages
  167--177. Amer. Math. Soc., Providence, RI, 1997.

\bibitem{Noumi}
M.~Noumi and T.~Sugitani.
\newblock Quantum symmetric spaces and related {$q$}-orthogonal polynomials.
\newblock In {\em Group theoretical methods in physics ({T}oyonaka, 1994)},
  pages 28--40. World Sci. Publ., River Edge, NJ, 1995.

\bibitem{OR}
R.~Orellana and A.~Ram.
\newblock Affine braids, {M}arkov traces and the category $\mathcal{O}$.
\newblock In {\em Algebraic groups and homogeneous spaces}, Tata Inst. Fund.
  Res. Stud. Math., pages 423--473. Tata Inst. Fund. Res., Mumbai, 2007.

\bibitem{QR}
H.~Queffelec and D.~Rose.
\newblock Sutured annular {K}hovanov-{R}ozansky homology.
\newblock {\em Trans. Amer. Math. Soc.}, 370(2):1285--1319, 2018.

\bibitem{Rouquier}
R.~Rouquier.
\newblock 2-{K}ac-{M}oody algebras.
\newblock {\em arXiv preprint arXiv:0812.5023}, 2008.

\bibitem{Rui}
H.~Rui.
\newblock A criterion on the semisimple {B}rauer algebras.
\newblock {\em J. Combin. Theory Ser. A}, 111(1):78--88, 2005.

\bibitem{Sartori}
A.~Sartori.
\newblock Categorification of tensor powers of the vector representation of
  {$U_q(\mathfrak{gl}(1|1))$}.
\newblock {\em Selecta Math. (N.S.)}, 22(2):669--734, 2016.

\bibitem{SarTu}
A.~{Sartori} and D.~{Tubbenhauer}.
\newblock {Webs and $q$-Howe dualities in types
  $\mathbf{B}\mathbf{C}\mathbf{D}$}.
\newblock {\em ArXiv e-prints}, Jan. 2017.

\bibitem{Tim}
T.~Seynnaeve.
\newblock Koszulity of type {D} arc algebras and type {D} {K}azhdan--{L}usztig
  polynomials, 2017.
\newblock Master Thesis University of Bonn.

\bibitem{SS}
T.~Seynnaeve and C.~Stroppel.
\newblock Koszulity of type {D} arc algebras and type {D} {K}azhdan--{L}usztig
  polynomials.
\newblock in preparation.

\bibitem{Sperv}
W.~Soergel.
\newblock Kategorie {$\mathcal O$}, perverse {G}arben und {M}oduln \"uber den
  {K}oinvarianten zur {W}eylgruppe.
\newblock {\em J. Amer. Math. Soc.}, 3(2):421--445, 1990.

\bibitem{SoergelKL}
W.~Soergel.
\newblock Kazhdan-{L}usztig polynomials and a combinatoric for tilting modules.
\newblock {\em Represent. Theory}, 1:83--114, 1997.

\bibitem{Stroppel}
C.~Stroppel.
\newblock Category {$\mathcal{O}$}: gradings and translation functors.
\newblock {\em J. Algebra}, 268(1):301--326, 2003.

\bibitem{StHC}
C.~Stroppel.
\newblock A structure theorem for {H}arish-{C}handra bimodules via coinvariants
  and {G}olod rings.
\newblock {\em J. Algebra}, 282(1):349--367, 2004.

\bibitem{StDuke}
C.~Stroppel.
\newblock Categorification of the {T}emperley-{L}ieb category, tangles, and
  cobordisms via projective functors.
\newblock {\em Duke Math. J.}, 126(3):547--596, 2005.

\bibitem{VV}
M.~Varagnolo and E.~Vasserot.
\newblock Canonical bases and {KLR}-algebras.
\newblock {\em J. Reine Angew. Math.}, 659:67--100, 2011.

\bibitem{Watanabe}
H.~{Watanabe}.
\newblock {Crystal basis theory for a quantum symmetric pair
  $(\mathbf{U},\mathbf{U}^{\jmath})$}.
\newblock {\em ArXiv e-prints}, Apr. 2017.

\bibitem{CTP}
B.~Webster.
\newblock Knot invariants and higher representation theory {I}: diagrammatic
  and geometric categorification of tensor products, 2010.
\newblock arXiv:1001.2020.

\bibitem{Wenzl}
H.~Wenzl.
\newblock On the structure of {B}rauer's centralizer algebras.
\newblock {\em Ann. of Math. (2)}, 128(1):173--193, 1988.

\bibitem{Arik}
A.~{Wilbert}.
\newblock {Topology of two-row {S}pringer fibers for the even orthogonal and
  symplectic group}.
\newblock {\em ArXiv e-prints, to appear in Trans. Amer. Math. Soc.}, Nov.
  2015.

\bibitem{Geordie}
G.~Williamson.
\newblock Singular {S}oergel bimodules.
\newblock {\em Int. Math. Res. Not. IMRN}, (20):4555--4632, 2011.

\end{thebibliography}
\end{document}